\documentclass[A4,14pt,openany]{book}

\def\drafts{0} 
\usepackage{mathtools}\usepackage{etoolbox}

\pretocmd{\includegraphics}{%
  \typeout{GRAPHICS FILE: #1}%
}{}{}

\usepackage{fancyhdr}
\usepackage{mdframed}
\usepackage{etoolbox}
 \usepackage{ifthen}
\usepackage{makecell}
\usepackage{booktabs}
\usepackage{rotating}
\usepackage{thumbpdf,lmodern}
 \usepackage{bigints}
\usepackage{tcolorbox}
\usepackage{hyperref}
\usepackage{subcaption}
\usepackage{relsize}
\usepackage{multirow}
\usepackage{minitoc}
 \usepackage{pdfpages}
\usepackage{amsmath,amsfonts,amssymb,amsthm,mathrsfs,mathabx}
\usepackage{stmaryrd}\usepackage{latexsym}
\usepackage{graphicx}
\usepackage{verbatim}

\usepackage{ulem}
\usepackage{tikz}
\usepackage{environ}
 \usepackage[latin1]{inputenc}  
\usepackage[T1]{fontenc}  
\usepackage{comment}
\usepackage{cleveref}
\usepackage{draftwatermark}

 \newcommand{\maxscale}{b}
  \newcommand{\what}[2]{ {\widehat{ #1}}^{ \hspace{-.1cm}    ~^{#2}}    }

 \newcommand{\draft}[1]{ \if\drafts1 { {\color{blue}[DRAFT]#1[/DRAFT]   }}\fi }

    \newenvironment{longversion}[1]
{
  \begin{tcolorbox}[
    colback=blue!5,
    colframe=blue!50!black,
    boxrule=0.5pt,
    arc=3pt
  ]
  \color{blue} #1 {\sl (long version)}
}
{
  \end{tcolorbox}
}

\if\drafts1

\SetWatermarkText{DRAFT}
\SetWatermarkScale{.2}
\else
\SetWatermarkText{}

 \fi

  \newcommand{\tconv}{  \overline{*}}  
    
 \newcommand{\GAF}{   \mathsf   F^{ \text{\rm{Pl}}}}  
 \renewcommand{\a}{  {\bf a}}  
 \newcommand{\CC}{  \mathscr  C}  
 \newcommand{\Q}{ {   \mathsf 	 Q}}  
 \renewcommand{\k}{{  \bf k}}  
  \newcommand{\m}{  {\bf m}}  
\newcommand{\authorname}{R. Lachièze-Rey}
\newcommand{\titletext}{Hyperuniform random measures}
 \newcommand{\s}{   \mathsf s}  

 \newcommand{\wsrm}{weakly stationary random measure}  
  \newcommand{\W}{   \mathsf W}  

  \newcommand{\mo}{   {\mathfrak m}  }
 \newcommand{\smallequlaw}{\stackrel{\mathclap{\scalebox{0.7}{$\scriptscriptstyle(d)$}}}{=}}  
 
\author{Rapha\"el Lachi\`eze-Rey}
 \newcommand{\hatn}[1]{ \widehat{ #1}^{  { \lambda } }}
 \newcommand{\hatX}[1]{ \widehat{ #1}^{ \X}}
 \newcommand{\Sc}{  \mathscr  S}  
\usepackage{color}

 \newcommand{\T}[1][\lambda]{   \mathbb  T  _{ #1}^{ d}}  
 \newcommand{\LSI}{$ L^{ 2}_{ \text{\rm{loc}}}$ }  
 \newcommand{\D}{   \mathsf D}  
   
 \newcommand{\F}{  \mathscr  F}  
 \usepackage{geometry}
\newtheoremstyle{myexample}     
  {3pt}                         
  {3pt}                         
  {}                            
  {}                            
  {\bfseries}                   
  {.}                           
  { }                           
  {}                            
 
 \newcommand{\E}{  \mathbf  E}  
 \renewcommand{\sc}{  \mathscr  S}  
   
 \newcommand{\Leb}[1][d]{  \mathscr  L^{ #1}}  
\theoremstyle{myexample}

\newtheorem{theorem}{Theorem}[chapter]
\newtheorem{lemma}{Lemma}[chapter]
\newtheorem{proposition}{Proposition}[chapter]
\newtheorem{remark}{Remark}[chapter]
\newtheorem{question}{Question}[chapter]
\newtheorem{example}{Example}[chapter]
\newtheorem{corollary}{Corollary}[chapter]

\newtheorem{definition}{Definition}[chapter]

  \newcommand{\Poi}{\text{\rm{Poi}}}  

\newtheoremstyle{pasimportant} 
    {\topsep}                    
    {\topsep}                    
    {\color{magenta}}                   
    {}                           
    { \bf}                   
    {.}                          
    {.5em}                       
    {}  %
    
\theoremstyle{pasimportant}

\newtheorem{exercice}{Exercice}
\newtheorem{exercise}{(Draft) Exercise} 

      \if\drafts0
       \RenewEnviron{longversion}{ 
  { 
  }
  }  
  
     \RenewEnviron{exercise}{ 
  { 
  }
  }  
\fi

 \if\drafts0
\RenewEnviron{exercise}{{}}

\fi

  \newcommand\bloc[2]{%
\expandafter\gdef\csname bloc#1\endcsname{#2}%
\label{#1}#2}

\newcommand\repeatbloc[1]{\csname bloc#1\endcsname}

\renewcommand{\P}{\mathsf{P}}
\newcommand{\M}{\mathsf{M}}
\renewcommand{\S}{\mathsf{S}}
\newcommand{\C}{\mathsf{C}}

 \newcommand{\N}{  \mathscr  N}  
  \newcommand{\equlaw}{ \stackrel{(d)}{ = }}

\author{ Raphael Lachieze-Rey\thanks{INRIA Paris, Team Mathnet, Lab. MAP5, Université Paris Cité, raphael.lachieze-rey@math.cnrs.fr}}

\title{Hyperuniformity of random measures, transport and rigidity}

\begin{document} \thispagestyle{empty}
\pagenumbering{gobble} 
\newpage
~
\vspace{2em}
\begin{center}
    {\LARGE \textbf{Hyperuniform   random measures, \\
    transport and rigidity}\\
\vspace{.5cm} 
    \small {\it Second version} \par}
    \vspace{1em}
    {\large Raphael Lachieze-Rey\footnote{INRIA Paris, Team Mathnet, Lab. MAP5, Université Paris Cité, raphael.lachieze-rey@math.cnrs.fr} \par}
    \vspace{1em} 
\end{center}

\vspace{2em}

 \begin{quote}
     
 {\bf Abstract.} This survey explores the foundational theory and recent developments in the study of hyperuniformity. We present a comprehensive mathematical framework in the context of weakly stationary random measures, emphasizing spectral characterizations and second order asymptotics. Classical examples - including determinantal point processes, Gibbs measures, and zero sets of Gaussian analytic functions - are presented in depth to illustrate core principles. We also highlight recent progress connecting hyperuniformity with optimal transport and rigidity phenomena, pointing to emerging directions in the field.

 \end{quote}
\thispagestyle{empty}

%

\IfFileExists{\jobname.toc}{

\def\noauxfile{0}
  \tableofcontents
}{

 \tableofcontents
 \def\noauxfile{1}

 toc does not exists, leave 2 pages blank  to avoid document shift
 
\newpage ~

} 

\if\drafts1
\section*{\draft{TOC}}
\fi
\label{toc}

\vspace{2cm} 

 \chapter{Introduction and examples}
 \label{ch:intro}
  
\pagenumbering{arabic}

When one has to generate a  sample of {\it  random points} in a window $ W$ of $ \mathbb{R}^{ d}$, the default strategy across
 many fields is to choose a large number $ n$ and draw $ n$  independent points uniformly in $ W$. The asymptotic  mathematical object as $ n\to \infty $ is the celebrated  {\it homogeneous Poisson process}, used as a universal reference, especially in theoretical studies, due to its nice mathematical properties. 
 Such finite ($ n$  i.i.d points) or infinite (Poisson) samples of independent points have some downsides,  such as the tendency to leave large empty spaces, or on the contrary, regions  cluttered with too many points (see the picture on the right, below), but this is an inevitable consequence of total randomness.

Hyperuniformity is a property exhibited by many mathematical models presenting instead a  {\it regular} spatial arrangement, remedying some flaws of independent samples. This type of arrangement is reminiscent of the way particles subject to mutual repulsive forces would be distributed. Many natural models from statistical physics, biology, or other fields, exhibit hyperuniform behavior.

\begin{figure}[h]
\centering
  \begin{subfigure}{0.28\textwidth}
    \raggedright
    \includegraphics[width=\linewidth]{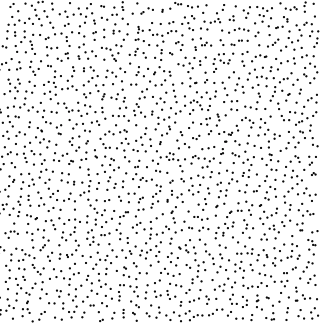}
    \caption*{A hyperuniform sample}
  \end{subfigure}%
  ~\;
  \begin{subfigure}{0.28\textwidth}
    \raggedleft
    \includegraphics[width=\linewidth]{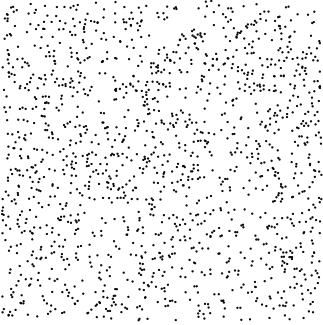}
    \caption*{A Poisson sample}
  \end{subfigure}
\end{figure}
Figure  \ref{fig:three_figs}  (left) shows the photoreceptor locations of a bird's eye, a class of species renowned for their excellent long-distance vision. 
This sample can be categorised as   hyperuniform due to its spatial statistical characteristics  \cite{avian-eyes}. In image analysis and optimal transport, hyperuniformity is also present, sometimes under the term {\it blue noise}, because regular samples can be useful for many tasks, such as texture synthesis or dithering,  and variance reduction is an essential feature of blue noise samples \cite{BlueNoise,DitheringBlueNoise,VarianceMC,deGoes}.   Figure  \ref{fig:three_figs}  (right), for instance, has been obtained by replacing greyscale levels by blue noise samples with the corresponding density, and the hyperuniform behaviour is essential for the nice visual impression  \cite{deGoes}. 
Birds photoreceptors can be compared with insect photoreceptors (Figure  \ref{fig:three_figs} - Middle) which are arranged in a periodic manner. For many applications, the samples should actually be  {\it disordered}, i.e. non-periodic, it can otherwise cause in Monte Carlo integration or image processing undesired aliasing or structured artifacts   \cite{VarianceMC}.

\begin{figure}[h!]
    \centering
    \begin{subfigure}[b]{0.28\textwidth}
        \includegraphics[width=\textwidth]{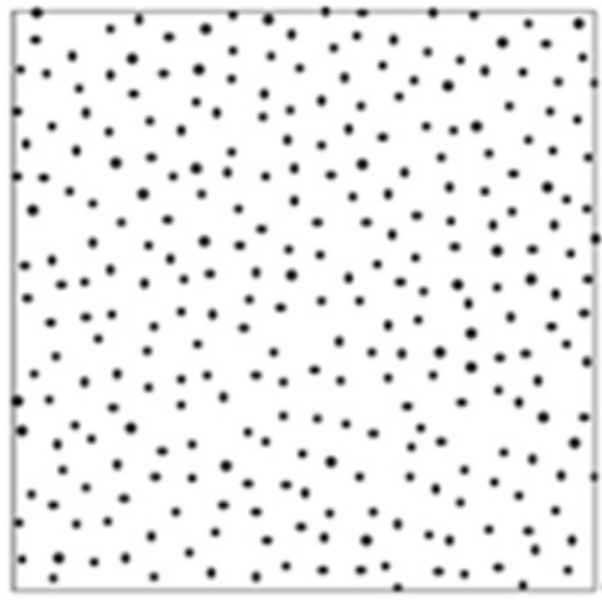} 
    \end{subfigure}
    \hfill    
    \begin{subfigure}[b]{0.3\textwidth}
        \includegraphics[width=\textwidth]{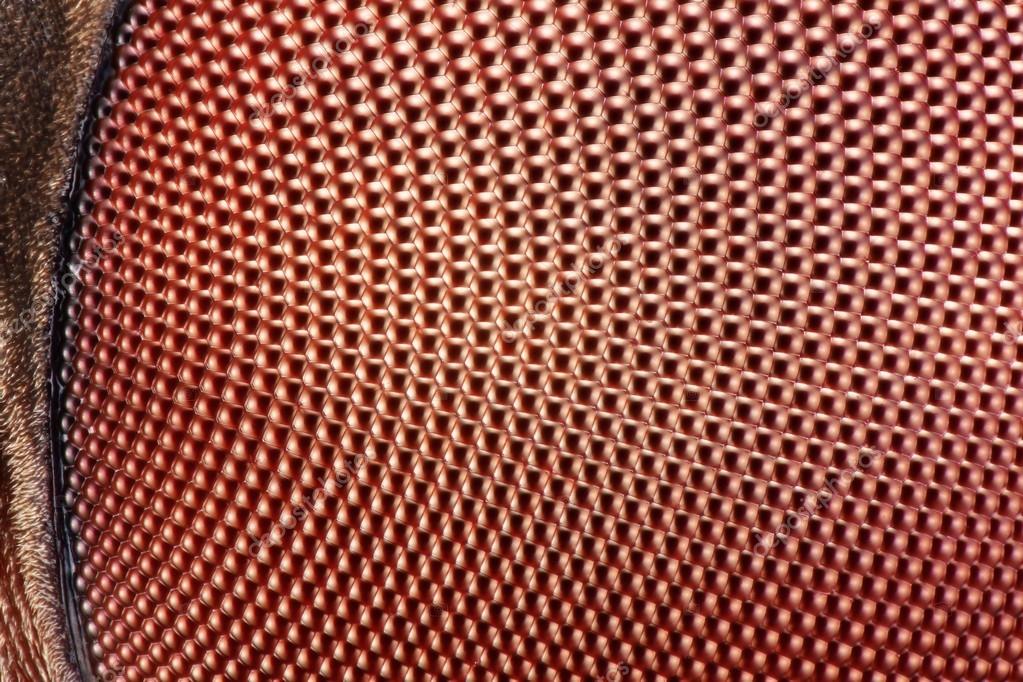}
    \end{subfigure}
    \hfill  \hspace{1.5cm} 
    \begin{subfigure}[b]{0.28\textwidth}
        \includegraphics[width=\textwidth]{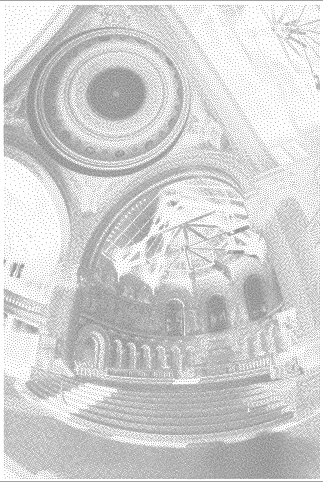}
    \end{subfigure}
 \caption{  {\it Left.} Disordered hyperuniform receptors \cite{avian-eyes}.  {\it Middle.} Periodic ``ordered'' photoreceptors.  {\it Right.} Dithering - Greyscale levels replaced by  hyperuniform ``blue noise'' samples  \cite{deGoes}, ACM Trans. Graph.}
    \label{fig:three_figs}
\end{figure}

This {\it regularity} is difficult to define rigorously in a non-ambiguous way, but the good fortune of mathematicians and the reason why this field of study exists is that hyperuniformity is a very natural and universal way to mathematically define a certain form of regularity. Roughly speaking, a sample is hyperuniform if the variance of the number of points in each large window is smaller than if the points were independent  (Poisson or i.i.d), see Figure  \ref{fig:reduced-fluct} for a numerical illustration of this variance reduction. Hyperuniformity is a simple second-order assumption, which does not involve higher order structure of the process, but it surprisingly implies many macroscopic phenomena, related to optimal transport, or rigidity. It is also quite universal since it can equivalently be defined by low variance for not-too-irregular linear statistics, and the whole theory extends to general random measures, including for instance Gaussian fields, spin systems and nodal domains of random fields.

\begin{figure}[h!]
\begin{center}
\includegraphics[width = 10cm]{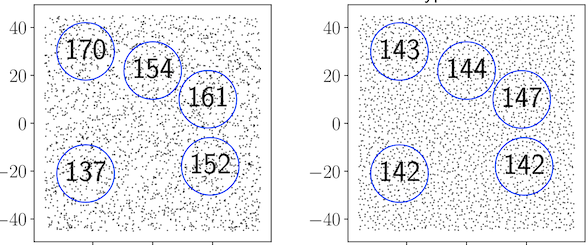}
\caption{Non-hyperuniform VS  hyperuniform}
\label{fig:reduced-fluct}
\end{center}
\end{figure}

The reduced charge fluctuations in some particular models, also called  {\it sum rules},  have been noticed a long time ago for particle systems \cite{martin1980charge, Leb}, or for zeros of random polynomials and eigenvalues of random matrices  \cite{ForresterHonner,Forrester}. 
The systematic investigation of this phenomenon  started in the 2000's with the team of S. Torquato at Princeton  \cite{TS03}, who popularised the term  {\it hyperuniformity,} or J. Lebowitz at Rutgers University, sometimes under the terminology of {\it superhomogeneity}. We will see that hyperuniformity of a medium   is related to the  {\it blue noise} property in the Fourier domain, i.e.  the transparency to very small wavelengths, a property with a universal utility. The parallel investigation of this phenomenon in the aforementionned  fields, mostly experimental,    has remained largely unnoticed.
 
 Besides its usefulness and appearances in other sciences, many popular mathematical models turned out to be  hyperuniform, in random matrices, statistical physics, random polynomials, quasicrystals, see the surveys  \cite{Tor18,GosLeb} for an in depth collection. Reading the literature gives the impression that  hyperuniform  point processes can be categorised in two classes: the class of periodic  or almost periodic, that eventually undergo small perturbations,  having properties similar to those of crystalline or quasi-crystallines structures, and the class of particle systems that look like spontaneous organisation of particles  arranging themselves after a pairwise repulsive force, and conserve some sort of local disorder. To draw a parallel with the way animal visual receptors sample space, the latter models seem visually more disordered, somewhat like the bird photoreceptors in Figure \ref{fig:three_figs} (left), whereas one can make a parallel between crystalline models and the regular arrangement of the eyes of insects (middle figure). A remarkable property of disordered hyperuniform processes is that they often display the same large-scale properties as their crystalline counterparts, which is why physicists sometimes subtitle hyperuniformity as {\it global order and local disorder}.
 
  The scope of this survey is to study hyperuniformity and its consequences under a mathematical perspective.
 We also present most stationary models for which hyperuniformity has been proven rigourously: determinantal processes, zeros of random Gaussian functions, Coulomb gases, quasicrystals ...   We give a first definition in Section  \ref{sec:intro-pp} and discuss the concept of  {\it disordered} sample, we give some emblematic examples in Sections  \ref{sec:perturbed-intro},  \ref{sec:intro-examples}. In Chapter 2, the mathematical core of this book, we explore hyperuniformity from the spectral viewpoint, which allows for a practical and universal characterization in the direct space through smooth linear statistics variance. The most natural framework for studying hyperuniformity is in fact that of (weakly) stationary random measures, generalising point processes. In this setting, we also provide universal bounds on the variance of linear statistics, useful for parameterizing measures by their hyperuniformity exponent.
Section 3 is devoted to a deeper study of certain classes of hyperuniform point processes, which requires delving into the world of Gaussian analytic functions (GAFs), determinantal point processes (DPPs), and quasicrystals. Sections 4 and 5 present surprising macroscopic properties of hyperuniform processes in dimension $ 1$ and $ 2$: in Section 4, we show that they enjoy good optimal transport properties. This large scale order allows crystalline/periodic structures to be merged with the class of disordered hyperuniform processes into a single continuum. It more broadly explores optimal transport between a hyperuniform measure and Lebesgue measure, or equivalently the link between hyperuniformity and random fair partitions of the space. Section 5 deals with another property, {\it rigidity}, which states that for many hyperuniform processes, the number of particles in a region of space can be completely inferred by observing the process on the rest of the space. This rigidity can take more extreme forms as the hyperuniformity exponent increases, leading us to the study of   {\it stealthy processes}, with infinite exponent, showcasing even more fascinating properties.
 
\subsection*{Context and objectives}

 These notes were written at the occasion of the mini-course   {\it Hyperuniformity of random samples} given at the 2025 GeoSto conference at Grenoble-Alpes University, the slides can be found at \url{https://helios2.mi.parisdescartes.fr/~rlachiez/recherche/talks/slides-hu.pdf}. This version is intended to take part in the Springer series  {\it Stochastic Geometry}. The lecture notes  \cite{Lr-DPPs} were written at the occasion of a Master course at Paris Dauphine Université, they give more details for some aspects related to random matrices and determinantal processes.\\
 
 There already exists general studies about  hyperuniformity. 
 The   surveys  \cite{Tor16,GosLeb}  list many different physical models experiencing  hyperuniformity at different orders, and presents most important properties, and some conjectures on  hyperuniform and stealthy processes. The more mathematical discussion of Coste \cite{Coste} contains some of the material treated here. Since the last version, there  have been several theoretical advances that we report here, notably concerning spectral characterization  \cite{HartBjo},  transport properties  \cite{LrY,BDG,HueLebl,KLLY,DFHL,Flimmel,LotzKlatt},  rigidity  \cite{DFHL,Lr24,rigid-companion,LotzKlatt}, linear statistics  and limit theorems \cite{MBL,jalowyBox,KrishYogesh}, Gibbs measures  \cite{DereudreFlimmel,Leble,leble-stationary}, and others. The excellent book \cite{BKPV} describes several models of hyperuniform DPPs and random GAF zeros, and provides their properties, it constitutes an essential source for this work. \\
 
 Between the first and the second version of this survey, the main additions concern quantitative aspects (Chapter \ref{ch:numerics}), finite volume hyperuniformity (Section  \ref{sec:hu-finite}), infinite volume optimal transport (Section \ref{sec:finite-Coulomb}), hyperuniform constructions based on fair partitions (Section  \ref{sec:fair-partitions}), hyperuniformity of specific models of Gibbs measures (Section  \ref{sec:gibbs}), and the removal of many erros and typos, more or less serious (many thanks for all the feedback received!).
 
  \newcommand{\B}{  \mathscr  B}  
 
  \section{Point processes formalism and first definition of HU}
  \label{sec:intro-pp}
  
  We shall give a quick introduction to the formalism of point processes here. We refer the reader to the reference textbook  \cite{DVj08}, or to the   recent monographs  \cite{baccelli2020random,BKPV} more connected to the models of interest in this work.
  Even though many general results will be stated without additional cost to general random measures, the main objects of this branch of literature are  {\it simple point processes}. To define them properly, 
introduce the space of configurations $ \N = \N(  \mathbb R^{d})$ whose elements are the atomic measures $ P = \sum_{i}\delta _{ x_{ i}}$, where the $ x_{ i}$ are countably many isolated points in $ \mathbb{R}^{ d}$. A configuration $ P$  can unambiguously be assimilated to its support and we often  use  set-related   notation such as, for $ P,P'\in \N,$ 
\begin{align*}
 P\cup P' =  { \rm supp}(P)\cup  { \rm supp}(P'), P\cap P'     =  { \rm supp}(P)\cap  { \rm supp}(P')        ,P  \setminus A =   { \rm supp}(P)  \setminus A,   ...
\end{align*}
Endow $ \N$ with the  {\it vague} topology, generated by the mappings 
\begin{align*}
\varphi _{ f}:P\in \N\to P(f): =  \int_{}fdP
\end{align*}
for $ f$  continuous  with compact 
support. The corresponding Borel $ \sigma $-algebra is denoted by $ \B(\N)$, it is also generated by mappings $ \varphi _{ 1_{ A}}$ for $ A\subset \mathbb{R}^{ d}$ bounded measurable.  A  {\it simple point process}, or just   {\it point process} in the following, is a random element $ \P$ of $ (\N, \B(\N))$. Equivalently, it is a family of measurable mappings
\begin{align*}
 A  \mapsto  \P(A)\in \mathbb{N}\end{align*} for $ A\subset \mathbb{R}^{ d}$ bounded  measurable, and the family of laws $ \{\P(A);A\subset \mathbb{R}^{ d}\}$ uniquely defines the law of $ \P$  as a probability measure over $ \N$ (\cite[Th. 9.2.XII]{DVj08b}).  An essential assumption for the modelisation of homogeneous structures 
is that of stationarity. Call $ \tau _{ x}$ the operator of shift by $ x\in \mathbb{R}^{ d}$, lifted to a set $ P\subset \mathbb{R}^{ d}$ with $ \tau _{ x}P = \{y + x;y\in P\}$, and say that a   point process $ \P$ is  {\it stationary} if for all $ x\in \mathbb{R}^{ d}$ $
 \tau _{ x}\P \smallequlaw   \P $,
using the set notation. For $ A\subset \mathbb{R}^{ d}$, we use the notation $  \mathscr  N(A)$ for the $ P\in  \mathscr  N$ with $ P\subset A$, endowed with the restriction topology and Borel $ \sigma $-algebra.\\

Let us introduce the   {\it Poisson  process}, fundamental brick in the realm of point processes. Given a non-zero non-negative locally finite measure $ \mu $ on $ \mathbb{R}^{ d}$ without atoms,   $ \P^{ \Poi(\mu) }$ is defined as the unique process satisfying $$ \P^{ \Poi(\mu) }(A) \stackrel{(d)}{=}  \text{\rm{\color{black} Poisson}}(\mu (A)), A\subset \mathbb{R}^{ d},$$ where $ \text{\rm{\color{black} Poisson}}(\lambda )$ denotes the law of a Poisson  variable with parameter $ \lambda \in \mathbb{R}\cup \{\infty \}$ ($ \Poi(\infty ) = \infty $ a.s. by convention).  One way to explicitly build $ \P^{ \Poi(\mu ) }$ is to start from  i.i.d.   variables uniform in the ball $ B_{ n}$ centred in $ 0$ with volume $ n$, i.e. $ \P_{ n}^{ \mu }: = \{X^{ (n)}_{ 1},\dots ,X^{ (n)}_{ n}\}$   with law $ \frac{ 1}{\mu (B_{ n})}\mu 1_{ B_{ n}}$ (for $ n$ sufficiently large). We have for $ A\subset \mathbb{R}^{ d}$ bounded as an easy exercise
\begin{align*}
 \#\{k:X^{ (n)}_{ k}\in A\} \xrightarrow[ n\to \infty ]{\text{\rm{Law}}} \text{\rm{\color{black} Poisson}}(\mu (A)),
\end{align*}hence $ \P^{ \Poi(\mu) }$ exists in $ \N$ as the  weak limit of the $ \P_{ n}^{ \mu }$ in the  vague topology. Let $ \Leb$ be Lebesgue measure. To obtain a stationary model, one must necessarily choose $ \mu  = \lambda \Leb$ for some $ \lambda >0$, and $ \lambda $ is called the \text{\rm{intensity}} of $ \P.$ More generally, for any stationary point process $ \P$, the intensity $ \lambda $ is defined by 
\begin{align*}
 \lambda  = \frac{ \mathbf E \P(A)}{\Leb(A)},A\subset \mathbb{R}^{ d}\text{\rm{ bounded non-negligible}},
\end{align*}
and this definition does not depend on $ A$; the finiteness of $ \lambda $ is by no means automatic, but we will implicitly assume it by default. Since we mainly conduct here second order analyses of such processes, we will in fact always assume  {\it local square integrability} (denoted by $ L^{ 2}_{ \text{\rm{loc}}}$), i.e.  $ \mathbf E \P(B)^{ 2}<\infty $ for $ B$ bounded.

As a Poisson variable, the variance of the number of points in the ball $ B_{ R}$  centered in $ 0$ with radius $R>0$ for a Poisson process is the volume of the ball
\begin{align*}
  \textrm{Var}\left(\P^{\Poi(\lambda )}(B_{ R})\right) =  \textrm{Var}\left( \text{\rm{\color{black} Poisson}}(\lambda \Leb(B_{ R}))\right) = \lambda \Leb(B_{ R}) = \kappa _{ d}\lambda  R^{ d}\text{\rm{ with }}\kappa _{ d} = \Leb(B_{ 1}).
\end{align*}
We write for short $ \P^{ \Poi(\lambda )} = \P^{ \Poi(\lambda \Leb)}$.
In general, a random measure with variance proportional to the volume on large domains is said to be  {\it extensive}, and is in fact expected for most   natural stationary point processes where particles only interact locally.

The study of perturbed lattices, random matrices, particle systems,  random polynomials, and many other natural objects, that will be the main topic of the current work, made emerge a substantial class of stationary processes where there is no extensivity, and some cancellation seems to equilibrate fluctuations of points, in what we call a  {\it hyperuniform}, or  {\it superhomogeneous} behaviour. The reasons for this compensation, also called  {\it screening} in statistical physics, are not always clear, and generally different for each system.
\begin{definition}
A stationary  point process $ \P$ of $  \mathbb R^{d}$ is  hyperuniform if 
\begin{align}
\label{eq:intro-hu}
\lim_{R\to \infty }\frac{  \textrm{Var}\left({\P}({B_{R}})\right)}{  \Leb(B_{R})} = 0.
\end{align}

\end{definition}

%
%
%
%
%
%
%
%
%
   \section{Perturbed lattices }
   \label{sec:perturbed-intro}
    \newcommand{\Z}{   \mathsf Z}  
   The most basic example of a  hyperuniform infinite sample is the  {\it shifted lattice}, i.e. in $ \mathbb{R}^{ d}$
\begin{align*}
 \Z  ^{ d}: = \{\k + U;\k\in \mathbb{Z} ^{ d}\} ,
\end{align*}
where $ U$ has the uniform distribution on  $ [0,1]^{ d}$, denoted by $  \mathscr  U_{ [0,1]^{ d}}$. The shift by $ U$ ensures stationary, i.e. invariance in law under $ \mathbb{R}^{ d}$ translations. This model is not very rich from the mathematical point of view, it still serves as a reference or as a counter-example for many phenomena.
Any other Bravais lattice, i.e. obtained through a linear mapping applied to $\mathbb{Z} ^{ d}$, would do as well; for simplicity we mostly  consider $ \Z^{ d}.$
 We shall introduce the more general concept of  {\it independently perturbed lattice} (IPL).
\begin{example}[IPL]
\label{ex:IPL-intro} For $ \mu $ a probability measure  on $ \mathbb{R}^{ d}$, let
$\Z^{ d,\mu }: =  \{\k + U + U_{ \k};\k\in \mathbb{Z} ^{ d}\}$ where the $ U_{\k}$ are i.i.d.   with law $ \mu ,$ called IPL with law $ \mu .$
\end{example} 

The hyperuniformity of $ \Z^{ d,\mu }$ is not trivial,  especially when $ \mu  = \delta _{ 0}$, i.e. $\Z^{ d,\mu } =  \Z^{ d}$, where it is related to  {\it Gauss's circle problem}.
A general proof in the spectral domain  is a corollary of Theorem \ref{thm:general-hu}. 
We can still give some geometric intuition when the $ U_{ \k}$ are bounded and not deterministic: there are approximately $ O(R^{ d-1})$ points $ \k + U_{ \k}$ close to $ \partial B_{ R}$, hence likely to cross the boundary under application of the shift $ U$, and they would cross approximately independently of other distant points. The variance of the number of particles inside is  the sum of variances of indicators for such points, which therefore gives a sum with $ O(R^{ d-1})$ uniformly bounded terms,   indeed negligible  with respect to $ R^{ d}.$

This model can be refined by introducing dependency among the $ U_{ \k}$, but to ensure stationarity we will always require the perturbations to form a stationary field of $ \mathbb{Z} ^{ d}$, i.e. for $ \m\in \mathbb{Z} ^{ d},$ $$ 
\{U_{\m +  \k};\k\in \mathbb{Z} ^{ d}\} \stackrel{(d)}{=}   \{U_{ \k};\k\in \mathbb{Z} ^{ d}\}.$$
 {At Chapter  \ref{chap:transport}, we will see that hyperuniformity persists if the assumption of independence of the $ U_{ \k}$ is replaced by a mixing assumption.
 More surpsisingly, in dimension $ d\geqslant 2$, most  hyperuniform processes can be written as a (non-mixing) stationary perturbating field  applied to a lattice. In this framework, the $ U_{ \k}$ can be interpreted as a transport between $ \Z^{ d}$ and the obtained point process $ \P.$ }

\section{Disordered samples}
\label{sec:disordered}   
Like particles of a gas, or  chicken photoreceptors (Figure \ref{fig:three_figs}), many hyperuniform processes observed in physics or biology seem to be  {\it disordered}.  Physicist  sometimes present disorder as  the absence of peaks in the spectrum, which can be reminiscent of an underlying periodic structure, and also ask from a disordered model to be  {\it isotropic},  i.e. no direction is priviledged: $ \P = \sum_{i}\delta _{ x_{ i}}$ is isotropic if for any orthogonal matrix $ O$ of size $ d$, $$ O\P: = \sum_{i}\delta _{ Ox_{ i}}\smallequlaw \P.$$ 
 It is sometimes additionally assumed that the covariance measure has  finite total mass (see   the formal definition at \eqref{def:covariance}).
The latter assumption, and most of the results in this survey, pertain to second order analysis, i.e. variance and covariance behaviour, and is referred to as  {\it weak disorder} here. Still, one might have ``disorder'' at order $ 2$ and order from a more global perspective, see for instance the example of cloaked lattices  \cite{cloak},   Example  \ref{ex:IPL}. It is not hard to build counter-examples which satisfy the above properties but cannot be categorised as  {\it disordered}, but they are probably physically unnatural.
A more satisfying mathematical concept is that of  {\it mixing}. This property is another interpretation of disorder where the behaviour of the model at distant locations should be asymptotically independent.   
\begin{definition}
Say that a stationary point process $ \P$ is mixing if for $ A,B\subset \mathbb{R}^{ d}$ bounded Borel sets, 
\begin{align*}
 \mathbf P (\P(A) = 0,\P(\tau _{ x}B) = 0)\xrightarrow[ x\to \infty ] {}\mathbf P (\P(A) = 0)\mathbf P (\P( B) = 0).
\end{align*}

\end{definition}  
The fact that empty intersection events characterise the law comes from the Renyi-Monch theorem for point processes \cite[Corollary 9.2.XIII]{DVj08}.
Mixing extends to   general events $ \Omega ,\Omega '$ of $ \B(\N)$ (see  \cite[Lm. 12.3.II]{DVj08}). If $ \P$ is mixing, we have 
\begin{align*}
 \mathbf P (\P\in \Omega ,\tau _{ x}\P\in \Omega ')\to \mathbf P (\P\in \Omega )\mathbf P (\P\in \Omega ').
\end{align*}   
Another advantage of the mixing property is that, by basic principles of ergodic theory, it prevents the spectrum, formally defined at  \eqref{eq:structure-def}, to have any atoms, also called  {\it Bragg peaks}. In many physics papers, disordered samples are said  {\it amorphous}, which often means isotropy and absence of Bragg peaks  \cite{Tor18}.   See a precise statement at Proposition  \ref{prop:mixing-spectral}.

\begin{longversion}{}
 mixing implies continuous spectrum?
\end{longversion}

\begin{longversion}{
This definition is not completely satisfactory either as some models, such as  the  stationary Poisson line intersection process satisfies it and still   exhibits very long range dependency, see  \cite{KlattLast} and references therein. Still, it excludes strongly structured models such as independently perturbed lattices. We introduce at Section \ref{sec:CLT-mixing} the concept of Brillinger mixing, which seems ideal from many points of view, but hard to verify in practice.

Let us finally note that, as a consequence of a general principle from ergodic theory, the mixing property forbids the presence of  peaks in the spectrum, fulfilling a requirement often present in the physics  literature, see  \cite{sinai-ergodic}. The converse is not true, as one can ``erase'' the peaks in the spectrum with a procedure that can   seem artificial from the ergodic theory point of view, see the discussion at Example \ref{eq:IPL-SF}.}

This concept discriminates well between models that are

Mixing is a good mathematical definition for (long range) disorder. A counter example might be the following.

\begin{example}
Let $   \mathsf L$ a Poisson stationary isotropic line process, and $ \P = \cup _{ L\neq L'\in    \mathsf L}L\cap L'.$ Then $ \P$ is mixing because 
\begin{align*}
 |  \mathbf P (\P(A) = 0,\P(\tau _{ x}B) = 0)-  \mathbf P (\P(A) = 0)\mathbf P (\P(\tau _{ x}B) = 0) | \leqslant \mathbf P (  \exists  L\in   \mathsf L: L\cap A\neq \emptyset ,L\cap \tau _{ x}B\neq \emptyset ).
\end{align*}
\end{example}  

This example has long range interaction because if, say, in 2D, 2 points $ x\neq y$ are in $ \P,$ then there is a positive probability that other points are on the (Lebesgue-negligible line) spanned by $ x,y$, and this could be expressed with factorial moment measures. This can be avoided by introducing a perturbation $ \P^{ \mu }$ as in ..., but it does not suppress long range interaction. A stronger notion of mixing is defined with factorial moment measures (Brillinger. Does it discard $ \P^{ \mu }$?)

{Poisson line intersection process}
 
\begin{exercice}
Poisson line intersection process, Cox and ...
\end{exercice}

Rk: One can easily isotropize a mixing process (?), so isotropic might be nonnecessary to ask.

Show that IPLs are not mixing?

\begin{exercice}
Random intersection of stationary line networks are not HU?
\end{exercice}
\end{longversion}

 \newcommand{\Gin}{  \text{\rm{Gin}}}  
   
    \section{Three emblematic examples}  
    \label{sec:intro-examples}
    We present here three  examples that emerge from different models. The first examples come from  random matrices, more precisely they are the scaling limit of points in the bulk of the eigenvalues of two prominent models. Two of them,  the GUE (in dimension $ d = 1$) and Ginibre ensemble ($ d = 2$), are also  {\it determinantal processes}, which will lead us  to introduce this very important class at Chapter \ref{chap:proofs}. The third example comes from the  field of random polynomials and functions.
It still bears a flavour similar to the Ginibre ensemble in that they are naturally defined on the Complex plane, through Gaussian Standard Complex Variables, and are
 connected to the theory of analytic functions through the complex covariance $ C(z,w) = e^{ z  \bar w}$. Together with the Ginibre ensemble, they  really are the two seminal examples for which have been uncovered in first the universal properties of  hyperuniform processes such as rigidity or good transport properties, partly because they  are  tractable, up to a certain point, among  the jungle of all physically relevant point processes. They are just at the right place in the world of mathematical particle models, between  relevancy and tractability.

    We say a random complex variable $ G$ is a  {\it Standard Complex Gaussian} (SCG), denoted $ G\sim \N_{  \mathbb C }(0,1),$  if it has density
\begin{align*}
 \frac{1}{\pi  }e^{ - | z | ^{ 2} },z\in  \mathbb C .
\end{align*}
Equivalently, $ G =  X + iY$, where $ X,Y$ are i.i.d.   with variance $ 1/2$ Gaussian law, denoted by $ \N(0,1/2)$. Note the simplicity of this definition, without square root or factor $ 2$, and the easy computation of the normalisation constant with Gauss's integral.

    \subsection{The $ \text{\rm{Sine}}_{ \beta }$ processes.}
    
Let $ \beta >0$. Consider the random vector $  (\Lambda _{ 1},\dots ,\Lambda _n)$ on $ \mathbb{R}^{ n}$  with joint density 
\begin{align}
\label{eq:beta-ensemble}
 \propto  \prod_{1\leqslant i<j\leqslant n}  | \lambda _{ i}-\lambda _{ j} | ^{ \beta } \prod_{i = 1}^{ n}\exp(-\beta \lambda _{ i}^{ 2}/4),
\end{align} where the symbol $ \propto$ means  {\it proportional to}, which essentially allows to avoid mentionning the renormalising constant. This density can be rewritten  $\propto
 \exp(-\beta H(\lambda _{ 1},\dots ,\lambda _{ n}))$
with the Hamiltonian
\begin{align*}
 H(\lambda _{ 1},\dots ,\lambda _{ n}) =- \frac{ 1}{2}\sum_{i\neq j}\ln(|\lambda _{ i}-\lambda _{ j}|)  +  \frac{ 1}{ 4}    \sum_{i}\lambda _{ i}^{ 2}.
\end{align*}
This can be interpreted in terms of a system,   called 
 $ \beta $-ensemble, or $ \beta $-log gas, where particles are individually  attracted to $ 0$ due to the  {\it confinement term} $ \exp(-\beta \lambda _{ i}^{ 2}/4)$ term, and this tendancy is compensated by the pairwise  {\it repulsion terms} $  | \lambda _{ i}-\lambda _{ j} | ^{ \beta }$, that favor configurations where particles are not too close from one another. 
A fundamental point   is that   $ \P_{ n}^{ \beta }: = \{\Lambda _{ 1},\dots ,\Lambda _n\}$ can also be seen as the set of eigenvalues of a random matrix:\begin{itemize}
\item If $ \beta  = 1$, $ \P_{ n}^{ 1 }$ has the same law as the set of  eigenvalues of the Gaussian Orthogonal Ensemble (GOE). The GOE is the random     matrix $  \M^{ 1,(n)} = (\M_{ i,j})_{ 1\leqslant i,j\leqslant n}$ in the space $  \mathscr  S_{ n}( \mathbb{R})$ of $ n\times n$  symmetric matrices, where the $ \M_{ i,i}$ are i.i.d.   with law $ \N(0,2)$, and $ \M_{ j,i} = \M_{ i,j},1\leqslant i<j\leqslant n$ are i.i.d.   with law $ \N(0,1)$ (proved at Section \ref{sec:change-of-variables} through a change of variables). The reason for a different variance on the diagonal is that the   density of the law of $  \M^{ 1,(n)}$ at  some $ M\in  \mathscr  S_{ n}( \mathbb{R})$ has a neat  expression in terms of $ M$'s eigenvalues $ \lambda _{ 1},\dots ,\lambda _{ n}$: the density is by independence the product of densities of variables involved:
\begin{align}
\notag
\propto  \prod_{i<j}\exp(-M_{ i,j}^{ 2}/2)  \prod_{i}\exp(-M_{ i}^{ 2}/4)   = &  \prod_{i\neq j}\exp(-M_{ i,j}^{ 2}/2)^{ 1/2}  \prod_{i}\exp(-M_{ i}^{ 2}/4) \\
 \label{eq:gue-particle}= &\exp(-  \text{\rm{Tr}}(MM^{ T}) /4)\\
=& \notag \exp(-\sum_{i = 1}^{ n}\lambda _{ i}^{ 2}/4).
\end{align}
This expression differs from  \eqref{eq:beta-ensemble} as it is the density of the matrix itself, not its eigenvalues (see  Section \ref{sec:change-of-variables}).
It is clear under this form that the law of $  \M^{ 1,(n)}$ is invariant under conjugation by the orthogonal group, which is the reason for the name  {\it orthogonal ensemble}: for $ O$ an orthogonal matrix, 
\begin{align*}
 O  {\M^{ 1,(n)}} O^{ T} \stackrel{(d)}{=}    \M^{ 1,(n)}.
\end{align*}
\item If $ \beta  = 2$, the $ \Lambda _{ i}'$s are the eigenvalues of the Gaussian Unitary Ensemble (GUE), the random matrix $   \M^{2, (n)} = (\M_{ i,j})_{ i,j}$ whose entries are independent  \underline {complex}   Gaussian variables  with variance $ 1$ on the upper diagonal and    \underline {real}  Gaussian with variance $ 1$ on the diagonal.  We then define a Hermitian model by imposing $  \M _{j, i}: =  \overline{\M _{ i,j}}$ for $ i<j$. Similarly as for the GOE, the matrix $   \M^{ 2,(n)}$ has a density in each Hermitian matrix $ H$
\begin{align*}
 \propto  \exp(-\sum_{i}\lambda _{ i}^{ 2}/2) = \exp(-  \text{\rm{Tr}}(H  \bar H^{ T}) /2),
\end{align*}
invariant under the conjugation by a unitary matrix.
\item The case $ \beta  = 4$ involves matrices of quaternions and is called the Gaussian Symplectic Ensemble (GSE), but we will not explicit further cases $ \beta \notin \{1,2\}$.
 
\item For any $ \beta >0$, the  $ \beta $-ensemble  has been showed by Dimitriu and Edelmann   \cite{betaEns} to constitute the eigenvalues of an explicit tridiagonal matrix model $   \mathsf M^{ \beta ,(n)}$, that we do not study further either. \end{itemize}

\begin{remark}

Some might find it more convenient to describe the GOE, GUE (and GSE) ensembles under the form $ X + X^{ *}$ where $ X$ is a $ n\times n$ (non-Hermitian) matrix with $ n^{ 2}$  i.i.d entries in the proper body ($ \mathbb{R},  \mathbb C $ or $   \mathbb{ H}$).
\end{remark}

    Recently,  Valko and Vir\`ag \cite{ValkoVirag} derived the construction for each $ \beta >0$ of the  {\it Brownian carrousel}, a set of SDEs whose limit points form a  point process of $ \mathbb{R}$, and which   is the weak limit of the $ \beta $-ensembles as $ n\to \infty $. Under the formulation  \eqref{eq:beta-ensemble}, the mean number of particles per unit volume  goes to infinity, which is why a rescaling by $ \sqrt{n}$ is necessary:
    
    \begin{theorem}[\cite{ValkoVirag}]\label{thm:sine-beta} 
     For $ \beta >0,$ there is a stationary point process $ \P^{ \beta }\subset \mathbb{R}$, called $ \text{\rm{Sine}}_{ \beta }$ process, such that 
\begin{align*}
\sqrt{n} \P_{ n}^{ \beta } \xrightarrow[ n\to \infty ]{\text{\rm{Law}}}\P^{ \beta }.
\end{align*}
Furthermore, $ \P^{ \beta }$ is  hyperuniform.
     \end{theorem}  
     The scaling $ \sqrt{n}$ is not immediate to justify from  \eqref{eq:beta-ensemble}. Denote $ a_{ n}\asymp b_{ n}$ if $ a_{ n}\leqslant cb_{ n}$ and $ b_{ n}\leqslant c'a_{ n}$ for finite constants $ c,c'.$ Let us compare with   i.i.d.   points  $ X_{ 1},\dots ,X_{ n}$ uniform on $ [-n,n]$, where indeed the mean number of points per unit volume remains constant: 
\begin{align*}
 \mathbf E \sum_{i}X_{ i}^{ 2}\asymp  n^{ 3},
\end{align*}
  which matches the rescaled eigenvalues
\begin{align*}
 \mathbf E \sum_{i}(\sqrt{n}\lambda _{ i})^{ 2} = n\mathbf E \sum_{i,j}(M_{ i,j}^{ \beta ,(n)})^{ 2}\asymp n^{ 3}.
\end{align*}

   Most proofs of those results are pretty involved, and often parts of classical textbooks about random matrices, so we will mostly omit them.    See also more detail in the lectures notes  \cite{Lr-DPPs}, sharing a lot of notation with the current survey. We  provide at Section  \ref{sec:change-of-variables} a proof that for $ \beta  = 1$, \eqref{eq:beta-ensemble} indeed is the density of the eigenvalues of $   \M^{ 1,(n)}$, to illustrate the fundamental link between statistical physics and random matrices. The process $ \text{\rm{Sine}}_{ 1}$ turns out to be a member  of the class of  {\it Pfaffian point processes}  \cite{BKPV}. The case $ \beta  = 2$ is also special as  $ \text{\rm{Sine}}_{ 2}$ process is a member of the class of Determinantal Point Processes (DPPs), important in the theory of  hyperuniformity, which we will prove at Section  \ref{sec:gue-as-dpp}.    Determinantal   point processes, central in random matrix theory, are probably the main source of mathematically  tractable hyperuniform point processes in any dimension.

    \subsection{The Ginibre ensemble}    The next example is again a determinantal system of particles, and at the same time the eigenvalues of a random matrix model and the law of a repulsive system of particles, but in dimension $ 2$, more naturally in $  \mathbb C $. 
    Let $   \mathsf G_{ i,j} \sim \N_{  \mathbb C }(0,1)$   i.i.d, $ 1\leqslant i,j\leqslant n,$  and
 the (non-Hermitian) random matrix $  \Gin_{ n} = (   \mathsf G_{ i,j})_{ 1\leqslant i,j\leqslant n}$.
 Let $ \P_{ n}^{ \Gin}\subset  \mathbb C $ the random subset of $  \mathbb C $ formed by the a.s.  distinct $ n$ eigenvalues of $ \Gin_{ n}$.
  A change of variable yields that $ \P^\Gin_{ n}$ yields an interpretation in terms of statistical physics, namely it corresponds to the equilibrium state of $ n$ particles with the so-called  {\it Coulomb interaction potential}:
 
 \begin{proposition}
$ \P_{ n}^{ \Gin}$ has density
\begin{align}
\label{eq:ginibre}
\propto  \prod_{1\leqslant i<j\leqslant n} | z_{ i}-z_{ j} | ^{ 2}\exp(-\sum_{i} | z_{ i} | ^{ 2}), z_{ 1},\dots ,z_{ n}\in  \mathbb C  .
\end{align}

 \end{proposition} 

Here again, the density  \eqref{eq:ginibre} translates an antagonism between an individual confinement term and a repulsive pairwise interaction, and this system is also characterised as the $ 2$-dimensional one component plasma, or  {\it Coulomb gas}, at inverse temperature $ \beta  = 2.$ Its determinantal nature makes it much more tractable than $ \beta $-Coulomb gases  for $ \beta \neq 2$.

\begin{theorem}
\label{thm:ginibre}
The point processes $ \P_{ n}^{ \Gin}$ converge weakly in the vague topology to a 
point process $ \P^{ \Gin} \subset  \mathbb C $ that is  stationary, isotropic,   hyperuniform.
 \end{theorem}  
 Note the absence of rescaling, which can again be justified by comparing with i.i.d.   variables $ X_{ 1},\dots ,X_{ n}$ uniform on $ B_{ \sqrt{n}}$: 
\begin{align*}
 \mathbf E\left[
 \sum_{i}X_{ i}^{ 2}
\right]\asymp n^{ 2} \asymp \mathbf E \left[
\sum_{i,j}G_{ i,j}^{ 2} 
\right]= \mathbf E  \left[
 \text{\rm{Tr}}(\Gin_{ n} \overline{ \Gin_{ n}}^{ T})
\right] .
\end{align*}
 The proof of Theorem  \ref{thm:ginibre} is at Section  \ref{sec:ginibre}, it relies as for $ \text{\rm{Sine}}_{ 2}$ in dimension $ 1$ on the fact that $ \P_{ n}^{ \Gin}$ is a determinantal point process, those two proofs are actually very similar.

\subsection{Zeros of the planar GAF}
\label{sec:intro-GAF}

 \newcommand{\Gaf}{  { \mathsf{ GAF}}}  
 
Another important class of point processes, or more generally random measures, is that of nodal sets of random functions, i.e. $ \P = \{x:   \mathsf F(x) = 0\}\subset \mathbb{R}^{ d}$ for some random $   \mathsf F:\mathbb{R}^{ d}\to \mathbb{R}^{ q}$. In general these systems are extensive, i.e. they present no  hyperuniformity  \cite{Lac20,Gass-cancellation}. A notable exception is the zero set of the planar Gaussian Analytic Function (GAF).
Let $   \mathsf G_{ k},k\geqslant 1,$ i.i.d.    $ \N_{  \mathbb C }(0,1)$ distributed variables, and the random function 
\begin{align*}
   \GAF (z) = \sum_{k\geqslant 1}\frac{   \mathsf G _{ k}}{\sqrt{k!}}z^{ k},
\end{align*}  
where a.s. the series converges absolutely.
Let its zero set be
\begin{align}
\label{eq:zero-gaf}
 \P^{ \Gaf} = \{z:  \GAF (z) = 0\}.
\end{align}
 
\begin{theorem}
\label{thm:gaf}
 $ \P^{ \Gaf}$ is a stationary  hyperuniform isotropic point process.
 \end{theorem}  
 A surprising point is that the law of $ \GAF$ is not invariant under $  \mathbb C $-translations, but its zero set is. An easy way to prove stationarity is that $ z  \mapsto  e^{ - | z | ^{ 2}/2} | \GAF(z) | $ is a stationarity field having the same zeros than $ \GAF$, but analycity and gaussianity are lost in the process. 
  One can also define $ \P^{ \Gaf}$ as the weak limit $ \P_{ n}^{ \Gaf}$ of the   zeros of the $ n$-degree  {\it Weyl polynomials} $     \sum_{k = 1}^{ n}\frac{   \mathsf G _{ k}}{\sqrt{k!}}z^{ k}.$  
More background and results about GAFs and a proof are provided at Section  \ref{sec:GAF}, based on  \cite{BKPV}.

\begin{figure}[h!]
\begin{center}
\includegraphics[width = 10cm]{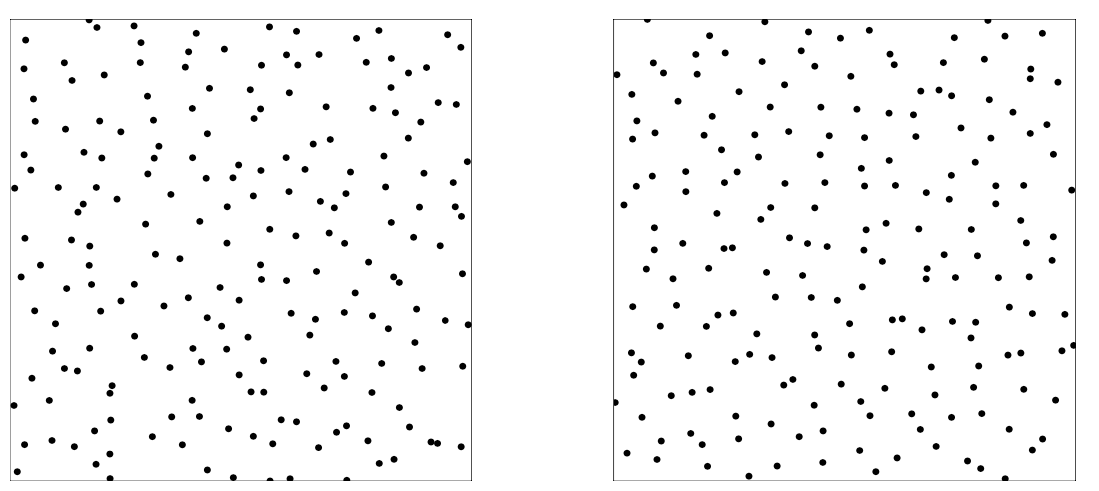}
\label{fig:DPPvsGAF}\caption{  {\it Left.} Ginibre ensemble.  {\it Right.} GAF zeros.}
\end{center}
\end{figure}

\begin{longversion}{}

   \section{Global order and local disorder}

We have seen in this chapter two classes of examples, that correspond to two ways of sampling the space. The first one ...
 
 A line of results that we present shows that the general behaviour of  hyperuniform point processes is very different if $ d\leqslant 2$ or if $ d>2.$ For instance:\begin{itemize}
\item Good matching properties
\item Rigidity
\item Gaussian fluctuations of linear statistics [Mastrilli 2025]
\end{itemize}
Also, zeros of entire functions only in $ \mathbb{R}$ and $  \mathbb C $.

In general, $ d = 1$ and $ d = 2$ behave differently also. 
\end{longversion}

\chapter{Mathematical hyperuniformity of random measures}
 \label{ch:maths}

 Physicists often define  the hyperuniformity of 	 a  point process as the property that its Fourier transform vanishes at $ 0$. 
 We provide here the mathematical material to justify this assertion in full generality and deduce other characterisations of  hyperuniformity easier to handle technically, and prepare some of the proofs of Chapter  \ref{chap:proofs}.
 
Even though  point processes provide the main motivation and featured examples, the theory can be applied most generally in the framework of  {\it weakly stationary measures}, which will allow us to illustrate technical considerations with examples drawn from random Gaussian fields, spin systems, or nodal lines. We also give universal lower bounds on the variance and Central Limit theorems, and other insights about general  hyperuniformity.
 
  \section{Weakly stationary random measures and spectral measure}
  
  \label{sec:wsrm}
  
  \newcommand{\MM}{  \mathscr   M}  
  
Following Kolmogorov's isomorphism theorem in the 1930's,  
several authors in the 1950's and 1960's, such as Doob, Yaglom, or Gelfand, derived a spectral theory for generalised processes, with works from Bartlett in the 1960's targetted towards point processes  \cite{Bartlett}. It relies on the fact that continuous positive definite bilinear forms can be represented by a non-negative measure, building on the work of Bochner in the 1930's, generalised by Schwartz to tempered distributions. See for instance Daley and Vere-Jones'book  \cite[Section 8]{DVj08} for a brief history about random measures, or  \cite{AT07} for Gaussian fields.  
  
  Let us make a brief introduction,  details and precise definitions can be found at Appendix  \ref{app:wsrm}, in particular with Theorem  \ref{thm:wsrm-X}.  We consider $ \M$ a random signed measure on $ \mathbb{R}^{ d}$, and note $ \M(f) = \int_{ }fd\M$ when it makes sense. Assume $ \M$ is locally square integrable (denoted by $ L^{ 2}_{ loc}$), i.e. $ \mathbf E  | \M | (A)^{ 2}<\infty $ for $ A$ bounded measurable. Also assume that $ \M$ is   {\it weakly} stationary, i.e. $$  \mathbf E \left[
\M(\tau _{ x}f)
\right] = \mathbf E \left[
\M(f)
\right],\;\;\textrm{Var}\left(\M(\tau _{ x}f)\right) \stackrel{(d)}{=}   \textrm{Var}\left(\M(f)\right)$$ for $ x\in \mathbb{R}^{ d},f\in \mathscr B_{ c}(\mathbb{R}^{ d})$.
$ \M$ it is furthermore   {\it strongly} stationary, if $ \M(\tau _{x}f)  $  has the same law as $\M(f)$ for such $ x,f$.

Under weak stationarity, the  {\it covariance measure} $ \C$, sometimes called  {\it reduced covariance measure,} is characterised by  
\begin{align}
 \label{def:covariance}
  \textrm{Cov}\left(\M(f),\M(g)\right) = \int_{}f(x)   \overline{   g(x + y)}\C(dy)dx, f,g\in\mathscr B_{ c}(\mathbb{R}^{ d}).
\end{align}
When $ f,g$ have complex values, recall that we consider the complex covariance $  \textrm{Cov}\left(U,V\right) = \mathbf E \left[
U  \bar V
\right]-\mathbf E [U] \mathbf E  [  \bar V].$ 
Taking for $ f,g$ approximations of Dirac masses in resp. $ 0$ and some $ y\in \mathbb{R}^{ d}  \setminus \{0\}$, we see that $ \C(dy)$ measures the covariance between infinitesimal masses around $ 0$ and $ y.$

The covariance measure is semi-definite positive in the sense that 
\begin{align*}
 \int_{}f(x)  \bar f(x + y)\C(dy)dx =   \textrm{Var}\left(\M(f)\right)\geqslant 0
\end{align*} for $ f\in\mathscr B_{ c}(\mathbb{R}^{ d})$, hence Bochner's theorem    yields the existence of a non-negative measure $ \S$ such that 
\begin{align}
\label{eq:structure-def}
  \textrm{Var}\left({ \M}(f)\right) = (2\pi )^{ -d}\int_{\mathbb{R}^{ d}} | \hat f(u) | ^{ 2}\S(du)
\end{align}
where  the Fourier transform is defined by
\begin{align}
\label{eq:L2loc-signed}
 \hat f(u): = \int_{\mathbb{R}^{ d}}f(x)e^{ -i u\cdot x}dx,
\end{align} see Theorem  \ref{thm:wsrm-X} for details. One can also view $ \C$ as a tempered distribution on $ \mathbb{R}^{ d}$, and then $ \S = \F\C$, where $ \F$ denotes the Fourier transform in the space of tempered distributions (see Appendix  \ref{app} for a refresher). 

$ \S$ is called the  {\it spectral measure}, or  {\it Bartlett's spectrum} in the context of point processes. Physicists often call it  {\it structure factor} when it is continuous  with respect to $ \Leb.$ 
 The   identity  \eqref{eq:structure-def}  is really what is needed in most of this work, we give several examples below.  

 We give at Appendix \ref{app:wsrm} a general version of these results, valid for complex values measures on a locally compact Abelian group, and we use for instance the results on the torus at Section \ref{sec:hu-finite}.

Investigating through the spectral domain is fruitful for many purposes:\begin{itemize}
\item It can be shown that the asymptotic behaviour of the number variance can be deduced in all generality from the behaviour of smooth linear statistics (see Theorem  \ref{thm:general-hu} and Proposition  \ref{prop:hyp-expo}).
\item One way to measure the regularity of a point configuration, in the sense that points are well distributed accross space, is through the optimal transport distance of its subsamples to Lebesgue measure, and here again a good strategy goes through the Fourier domain (see Theorem \ref{thm:BDG}).
\item The problem of least-square errors prediction, related to the topic of rigidity (Chapter \ref{chap:rigid}), have to be conducted in the spectral domain, like for time series or any other prediction problem in a framework of invariance, see Chapter  \ref{chap:rigid}.
\end{itemize}
  \renewcommand{\sc}{   \mathsf c}

\begin{example}
[Gaussian processes]
\label{ex:gaussian}
Any finite measure $ \S$ is the  Fourier transform of some continuous covariance function $ \sc$, and there exists a random Gaussian process $ G(x),x\in \mathbb{R}^{ d}$, and the corresponding random measure $ \M(dx) = G(x)dx$, characterised by 
\begin{align*}
 \textrm{Cov}\left(G(x),G(y)\right) = \sc(x-y),
\end{align*} see for instance  \cite[Th. 5.4.2]{AT07}.
Minimal regularity assumptions on $ \C$ (or tail decay on $ \S$) imply that $ G$ can be chosen to have continuous sample paths, see for instance  \cite[Th.1.4.2]{AT07}.  
\end{example} 

\begin{example}[Point processes]
A \LSI random measure $ \P$ taking only integer values is called a  point process as it can a.s. be represented as $ \P = \sum_{i}n_{ i}\delta _{ x_{ i}}$ for some isolated points $ x_{ i}$ and $ n_{ i}\in \mathbb{N}^{ *}$. As in the previous chapter, we shall generally require here that the process is  {\it simple}, i.e. $n_{ i} =  \P(\{x_{ i}\}) = 1$ , so that $ \P$ can be unambiguously associated with its support, we often abusively write $ \P(A) = \#\P\cap A$. Local square integrability implies that  $  { \rm supp}(\P)   $ is a.s. locally finite.  

The most important example is certainly the unit intensity homogeneous Poisson point process (defined at Section  \ref{sec:intro-pp}), which satisfies 
 $
  \textrm{Var}\left(\P^{ \Poi(1)}(A)\right) = \Leb(A),
 $
hence \eqref{eq:structure-def},\eqref{def:covariance} readily imply $ \C = \delta _{ 0},\S = \Leb.$ This is the translation that there is no interaction between different locations (the Dirac mass in $ 0$ is an artifact of the atomic nature of  point processes).   
A general disordered  point process, supposed to have asymptotic independence for distant points, is expected to have a covariance measure of the form $ \C = \delta _{ 0} + g\Leb$ for some integrable $ g$, and a structure factor $ \S = \s\Leb$ where $ \s-1$ is expected to be integrable.
\end{example}

\begin{example}[Shifted lattices]
\label{ex:shifted-lattices}
Following up on Section   \ref{sec:perturbed-intro}, we have for test functions $ f,g$

\begin{align*}
 \mathbf E\left[
 \Z^{ d}(f)\Z^{ d}(g) 
\right]= &\sum_{\k,\m\in \mathbb{Z} ^{ d}}\mathbf E\left[
 f(\k + U)g(\m + U)
\right]\\
 = &\sum_{\k,\m\in \mathbb{Z} ^{ d}}\int_{[0,1]^{ d}}f(\k + u)g(\m + u)du\\
  = &\sum_{\k\in \mathbb{Z} ^{ d}}\int_{[0,1]^{ d}}f(\k + u)
  {\sum_{  {\bf l}\in \mathbb{Z} ^{ d}}g(\k +  {\bf l} + u)}du
\end{align*}
hence we have
\begin{align*}
  \textrm{Cov}\left(\Z^{ d}(f),\Z^{ d}(g)\right) = \int_{}f(x)g(x + y)\sum_{    {\bf l} \in \mathbb{Z} ^{ d}}\delta _{ {\bf l} }(y)dx-\int_{}f(x)g(x + y)dxdy
\end{align*} and the covariance is $ \C = \sum_{  {\bf l}\in \mathbb{Z} ^{ d}}\delta _{   {\bf l}}- \Leb $ (remark that $ \C(\mathbb{R}^{ d})$ is not well defined). We then use the Poisson summation formula 
\begin{align*}
 \sum_{  {\bf l}\in \mathbb{Z} ^{ d}}f(  {\bf l}) = \sum_{\k\in 2\pi \mathbb{Z} ^{ d}} \hat f(\k)
\end{align*}
to have by the Parseval formula, with $ \S = \F\C,$
\begin{align*}
(2\pi )^{ d} \langle \S, f \rangle = \langle \C , \hat f\rangle =  \sum_{  {\bf l}\in \mathbb{Z} ^{ d}} \hat f(  {\bf l})- \int_{} \hat f = \sum_{ \k\in 2\pi \mathbb{Z} ^{ d}} (2\pi )^{ d}f(\k)-(2\pi )^{ d}f(0),
\end{align*}
i.e. $ \S = \sum_{\k\in 2\pi \mathbb{Z} ^{ d}  \setminus \{0\}}\delta _{ \k}.$ 
We will see at the next section that this form of the spectral measure,  in particular the gap around $ 0,$ neatly proves the  hyperuniformity of $ \Z^{ d}$, a fact that is not obvious through direct geometric computations.

\end{example} 

\begin{example}[Independently perturbed and cloaked lattices]
\label{ex:IPL}Following up on Example \ref{ex:IPL-intro},
let us now give the spectral measure for the perturbed lattice $ \Z^{ d,\mu }$, where $ \mu $ is a probability measure on $ \mathbb{R}^{ d}$. Let $ \psi  (u) =  \int_{}e^{ -iu\cdot t}d\mu (t).$ We have 
\begin{align}
\label{eq:IPL-SF}
 \S(du) = (1- | \psi (u) | ^{ 2})du + \sum_{\m\in 2\pi \mathbb{Z} ^{ d}  \setminus \{0\}} |\psi  (\m) | ^{ 2}.
\end{align}
This is a particular case of the more general Proposition  \ref{thm:cluster-process} where a point process is perturbed by  i.i.d. finite point processes. The first time this formula appears seems to be in Gabrielli's work  \cite[(24)]{Gabrielli04}. We can observe that the periodic structure of the lattice is present through the atomic component in the second term, while the continuous component expresses the slight disorder introduced in the system.

As for shifted lattices with no perturbations, we observe that the spectral measure vanishes around $ 0$. Still using the next section, this shall imply the  hyperuniformity of such models.\\

A nice observation by Kim, Klatt and Torquato  \cite{cloak} is that if $ \mu $ is $   \mathscr U _{ [0,1]^{ d}}$, $ \F \mu (2\pi \mathbb{Z} ^{ d}) = \{0\}$, and the singular component vanishes. It implies by Fourier duality that the covariance measure has a compact support. It is a perfect example of a system which is highly ordered but has the second order marginals of a disordered system.

 By the multiplicative structure of the characteristic function for sums of independent variables, the same phenomenon occurs when the  perturbations are  i.i.d of the form $ U_{ \k} = V_{ \k} + W_{ \k},\k\in \mathbb{Z} ^{ d},$ where $ V_{ k} \sim	 \mathscr  U_{ [0,1]^{ d}}$ and $ W_{ \k}$ is the ``genuine perturbation'', so it is a good practice to consider independently perturbed lattices of this form if one does not want Bragg peaks in the spectrum. Such perturbed lattices, called  {\it cloaked}, are clearly a good benchmark model for anyone wishing to experiment on large  hyperuniform samples without aliasing artifacts.  The cloaking implies that the periodic structure is not detectable by a  second order analysis. It is still present at higher orders, in the sense of factorial moment measures defined at Section \ref{sec:factorial}; more generally it is likely not mixing, and the cloaking of the second order periodic structure does not kill  the anisotropy of $ \mathbb{Z} ^{ d}.$

\end{example}

  \subsubsection{Translation boundedness}
  
  An essential feature of spectral measures is that they are translation bounded  \eqref{eq:adhikari}:
   {\begin{lemma}
\label{lm:trans-bd}
Let $\M$ be a $L^2_{\rm loc}$ weakly stationary random measure on
$\mathbb R^d$, with spectral measure $\S$. Then, uniformly over $x\in\mathbb R^d$ and $T\ge 1$,
\begin{align}
\label{eq:adhikari}
 \S(B(x,T))\le C_{d,\M}T^d,
\end{align}
where
\[
C_{d,\M}
=   c_{d}\,
\mathbf E\!\left[|\M|(B_1)^2\right]
<\infty .
\]
\end{lemma}
The proof of this lemma and a tractable value for $ c_{ d}$ are given at Appendix  \ref{app:balls}. This result implies in particular that $ \S$ is a tempered measure.
}

 \subsubsection{Disordered processes and spectral atoms}   
 In several parts of this work, we have an interest in  {\it disordered} processes, a concept discussed at Section  \ref{sec:disordered}. Physicists often demand the absence of spectral atoms, but the previous example of cloaked lattices shows that stronger conditions are necessary to discard crystalline structures, and we advocated at Section  \ref{sec:disordered} the use of the mixing condition, which is strictly stronger.

 \begin{proposition}\label{prop:mixing-spectral}  
 Let $ \P$ a \wsrm~with spectral measure $ \S$. Then if $ \P$ is mixing, or more weakly if $ \P$ has uncorrelated linear statistics, i.e. for $ \varphi \in  \mathscr  C_{ c}^{ \infty }(\mathbb{R}^{ d})$, 
\begin{align*}
  \textrm{Cov}\left(\P(\varphi ),\P(\tau _{ x}\varphi )\right)\to 0,
\end{align*} then $ \S$ has no atoms.
 \end{proposition}

 \begin{proof}
 Let $ \varphi \in \CC_{c}^{ \infty }(\mathbb{R}^{ d})$ such that $ \hat \varphi >0$ on $ \mathbb{R}^{ d}$. Let us work   with the smoothed field 
\begin{align*}
 G(x) = \P(\tau _{ x}\varphi ) ,
\end{align*}
or $ G = \P \ast \varphi $ in terms of convolution of tempered distributions. Note that 
\begin{align*}
 \textrm{Var}\left(G(x)\right) =  \textrm{Var}\left(\P(\varphi )\right)<\infty ,
\end{align*} and the spectral measure $ \S_{ G}$ of the random measure $ G(x)dx$ is finite because $ \varphi $ has fast decay and $ \S$ is translation bounded (Lemma  \ref{lm:trans-bd}):
\begin{align*}
 \S_{ G}(\mathbb{R}^{ d})  =  \textrm{Var}\left(G(0)\right) = \displaystyle\int_{} | \hat \varphi  | ^{ 2}\S(dx)<\infty .
\end{align*} Basic properties of Fourier transforms for distributions yield that $ \S_{ G} = \S \hat \varphi ^{ 2}$ (\cite{Rudin2}). If $ \S$ has an atom, then so does $ \S_{ G}$.

The assumption implies that $ \P(\varphi ),\P(\tau _{ x}\varphi )$ are asymptotically decorrelated as $ x\to \infty $, i.e. 
\begin{align*}
C(x): =   \textrm{Cov}\left(G(0),G(x)\right)  = \textrm{Cov}\left(\P(\varphi ),\P(\tau _{ x}\varphi ))\right)\to 0.
\end{align*}
Since $ C$ is bounded and goes to $ 0$, for any $ u\in \mathbb{R}^{ d}$, as $ R\to \infty ,$
\begin{align*}
 \frac{ 1}{\Leb(B_{ R})}\displaystyle\int_{B_{ R}}C(x)e^{ -ix\cdot u }dx\to 0.
\end{align*} 
If $ \S_{ G} =  \hat C$ had an atom at some $ u_{ 0}\in \mathbb{R}^{ d}$, the previous average could not converge to $ 0$ using the Fourier inversion $ C (-x)= (2\pi )^{ d}\F \S(x)$ (see  \cite{Rudin2}).
 \end{proof}

  \section{Spectral characterisation of  hyperuniformity}
\label{sec:spectral-char}
                     
Physicists noted early that the hyperuniformity of a point process amounts to the vanishing in $ 0$ of its structure factor, under density assumptions of the covariance  \cite{TS03}.
Coste \cite{Coste} derived a more general spectral characterisation of  hyperuniformity of a weakly disordered \LSI  point process, i.e. when $ \C$  is integrable or with constant sign \cite[Prop. 2.2]{Coste}: hyperuniformity is equivalent to  $ \S(B_{ \varepsilon }) = o(\varepsilon ^{ d})$ as $ \varepsilon \to 0 $; Bj\"orklund and Hartnick
 \cite{HartBjo} removed this assumption. This is in particular useful to show that  hyperuniformity can be equivalently characterised using smooth linear statistics instead of discontinuous ball indicators.  
 
 For $ f:\mathbb{R}^{ d}\to \mathbb{C}$, let $ f_{ R}(x) = f(x/R),R>0$. 
 Recall that \eqref{eq:structure-def} holds for Schwartz functions and bounded measurable functions with compact support, but often it also holds for a wider class of functions. Without discussing this further, we  call  {\bf $ \S$-admissible} an integrable function $ f$ such that \eqref{eq:structure-def} holds for all $ f_{ R},R>0$, but the two afore-mentionned classes are sufficient for most purposes.

\begin{theorem}
\label{thm:general-hu}
Let $ \M$ a $ L^{ 2}_{ loc}$ \wsrm~with spectral measure $ \S$. The three following are equivalent.\begin{itemize}
\item (i) $ \M$ is  hyperuniform, i.e. $
  \textrm{Var}\left(\M(B_{ R})\right) = o(R^{ d})$ as $ R\to\infty .$
\item (ii) there exists $ f$  $ \S$-admissible such that $ \int_{}f\neq 0$  and   
$
  \textrm{Var}\left(\M(f_{ R})\right) = o(R^{ d}) $.
\item (iii) We have  {\it spectral  hyperuniformity}:\begin{align*}
\lim_{\varepsilon \to 0}\frac{ \S(B_{ \varepsilon })}{\varepsilon ^{d}} = 0.
\end{align*}
\end{itemize}
 
 \end{theorem}    
 
 {The most surprising implication is perhaps (ii) $  \Rightarrow\;$ (i), and it really requires to go through (iii). The proof can be summarised by the decomposition of the phase space formula \eqref{eq:structure-def} for $ f = 1_{ B_{ R}}$ into low and high frequencies. When $ f$ is irregular, e.g. $ f = 1_{ B_{ 1}}$, the low-order frequencies are equivalent to the low-order frequencies contribution for $ f$ smooth, provided by Assumption (ii). High-order frequencies are dealt with thanks to the translation boundedness of $ \S$ (Lemma  \ref{lm:trans-bd}).}

Point (iii) immediately implies   that  (independently perturbed) shifted lattices $ \Z^{ d,\mu }$ from Examples  \ref{ex:shifted-lattices},\ref{ex:IPL} are hyperuniform, since the structure factor vanishes at the origin.  
 For a  hyperuniform system, the class of   functions giving reduced variance fluctuations, i.e. satisfying (ii), includes the indicator of a ball by Lemma \ref{lm:fourier-ball}, but not the indicator of any shape, as for instance a direct geometric reasoning yields the hyperfluctuating variance
\begin{align*}
  \textrm{Var}\left(\Z^{ d}([-n,n + 1/2]^{ d})\right) \asymp n^{ 2(d-1)}
\end{align*}for $ n\in \mathbb{N}$ going to infinity. Geometrically, the large variance comes from the possibility of large groups of $ O(n^{ d-1})$ points to cross the border of a large cube in the same direction without being compensated.
This irregularity hence does not come from the sharp corners of the rectangles,  rather from its flat edges; a similar reasoning yields that  the indicator of the rectangle with ``rounded corners'' $ W = \{x + y:x\in [-10,10]^{ d},y\in B(0,1)\}$ does not satisfy (ii) either.  \cite[Theorem 3.6]{HartBjo} shows that the number variance cancellation holds for so-called  {\it Fourier smooth} shapes, having a uniform decay of the Fourier transform in $ (1 + \|u\|)^{ -(d + \gamma )/2 },\gamma >0$.  {One needs an additional disorderness assumption so that the kernel shape does not modify the linear statistic asymptotic variance, see  \cite[Proposition 2.2]{Coste}.

\begin{proposition}
\label{prop:universality-C-integrable}
Let $ \M$ a $ L^{ 2}_{ loc}$ stationary random measure with integrable covariance $ \C$, i.e.  $  | \C | (\mathbb{R}^{ d})<\infty $. Then $ \S$ has a continuous density $ \s$, and (i),(ii),(iii) are equivalent to \begin{itemize}
\item (iv) $ \s(0) = \C(\mathbb{R}^{ d}) = 0$
\item (v) For any $ f\in \mathscr B_{ c}(\mathbb{R}^{ d})$, $
\textrm{Var}\left(\M(f_{ R})\right) =  o(R^{ d}).$
\end{itemize} 

\end{proposition}
 The proof actually extends immediately to any $ f\in L^{ 2}(\mathbb{R}^{ d})$ such that   \eqref{eq:scaled-cov} holds below.
\cite[Proposition 2.2]{Coste}   gives a version where the $ \M(f_{ R})$ are replaced by $ \M(D_{ n})$ where the $ D_{ n},n\geqslant 1$ form a nested sequence of convex open sets (his method is used in the proof below).

\begin{proof}
As the Fourier transform of the finite measure $ \C$, $ \S$ is continuous  with respect to Lebesgue's measure and it has  the bounded density   $$ \s(u) = \int_{\mathbb{R}^{ d}}e^{ -ix\cdot u}\C(dx)du\leqslant |\C|.$$ Lebesgue's theorem yields that $ \s$ is continuous. Then (iii) is  equivalent to (iv): $ \s(0) =\C(\mathbb{R}^{ d}) =  0$.  

To treat (v), define for some $ f\in \mathscr B_{ c}(\mathbb{R}^{ d})$ 
\begin{align*} I_{ R}(y): = \int_{}f_{ R}(x)f_{ R}(x + y)dx =R^{ d} \int_{}f(x)f(x + y/R)dx.
\end{align*}
By
 \eqref{def:covariance},
\begin{align}
\label{eq:scaled-cov}
  \textrm{Var}\left(\M(f_{ R})\right) = &\int_{\mathbb{R}^{ d}}I_{ R}(y) \C(dy),
\end{align}
and $   | I_{ R}(y) | \leqslant   R^{ d}\|f\|_{ 2}^{ 2}$ by Cauchy-Schwarz inequality.
It is a consequence of the continuity in $ 0$ of the convolution that for fixed $ y\in \mathbb{R}^{ d},$ $ R^{ -d}I_{ R}(y)\to  \|f\|_{ 2}^{ 2} $
 as $ R\to \infty .$ Therefore  Lebesgue's theorem yields
\begin{align}
\label{eq:var-extensive-C-int}
  \frac{  \textrm{Var}\left(\M(f_{ R})\right)}{ R^{ d}}\to \|f\|_{ 2}^{ 2}\C(\mathbb{R}^{ d}).
\end{align}
We indeed have equivalence of (iv) and (v).

\end{proof}

}

Possible magnitudes for number variance fluctuations are further discussed at Section \ref{sec:beck}.\\

Let us now turn to the proof of the theorem. 
It is obvious from \eqref{eq:structure-def} that, in order to compute the number variance over ball indicators, we need an estimate of their Fourier transforms. 
Its expression and some estimates are given at Appendix \ref{app:balls}.
It is sufficient for most purposes to use the estimate
\begin{align}
\label{eq:bessel}
 |  \widehat{ 1_{ B_{ 1}}}(u) |   = O((1 + \|u\|)^{ -\frac{ d + 1}{2}})\text{\rm{ as }}u\to \infty .
\end{align}

 {We also show that the variance bounds extend to kernel more irregular than ball indicators. For $ \gamma >0,A\subset \mathbb{R}^{ d}$, let $ \F_{ \gamma }(A)$ be the class of $ \varphi \in \B_{ c}(A)$ such that for some finite $   \mathsf c_{ \varphi },$
\begin{align}
\label{eq:def-F-gamma}
 \hat \varphi (u) \leqslant   \mathsf c_{ \varphi }
(1 + \|u\|)^{ -\frac{ d + \gamma }{2}}
,
\end{align}and note $ \F_{ \gamma }: = \F_{ \gamma }(\mathbb{R}^{ d})$. For $ a\in (0,d/2)$, let $ \gamma  = d-2a$. Let $ h\in \CC_{ c}^{ \infty }(\mathbb{R}^{ d}),x_{ 0}\in \mathbb{R}^{ d}$, and the function with a singularity 
\begin{align*}
 \varphi (x) = h(x)\|x-x_{ 0}\|^{ -a}.
\end{align*}$ \varphi $ 
is square integrable and approximating it by functions from $ \F_{ \gamma }(\mathbb{R}^{ d})$ yields that it satisfies  \eqref{eq:def-F-gamma} (see  \cite[Thm. 2.3]{Lr-HU-finite}). Let $ \F_{ \gamma }^{ *}(A)$ the space generated by functions of this form supported by $ A$ and functions from $ \F_{ \gamma }(A)$ (hence $ \F_{ \gamma }^{ *} = \F_{ \gamma }$ for $ \gamma \geqslant d$). One can often show by approximation that such functions are also $ \S$-admissible.

}

  We need also to understand how the behaviour in $ 0$ of the spectral measure and the variance of linear statistics influence each other. We give a general lemma that will also be useful later on. For $ \xi  :\mathbb{R}_{  + }\to \mathbb{R}_{ + },\gamma >0$, 
define  
\begin{align}
\label{eq:def-xi-sigma}
 \sigma _{\xi  ,\gamma } (R): = \int_{R^{ -d-\gamma }}^{ 1}\xi  (R t^{ \frac{ 1}{d + \gamma }})t^{ -\frac{ d}{d + \gamma }}dt.
\end{align} 
Note that if $ \xi  (R)\to 0$ as $ R\to \infty $, $ \sigma _{\xi  ,\gamma }(R)\to 0$ as well by Lebesgue's theorem.  
  
   { 
\begin{lemma}
\label{lm:new-Adhik-SF} 
   Let $ f$ a $ \S$-admissible function 
    {from $ \F^{ *}_{ \gamma }(\mathbb{R}^{ d})$} for some $ \gamma >0$. 
  \begin{enumerate}
\item There is $ \rho _{ f}>0$ such that for $R>0$,   
\begin{align*}
 \frac{ 1}{ 2}   \left|
\int_{}f
\right| ^{ 2}\S(B_{ \rho _{ f}R^{ -1}})\leqslant  (2\pi )^{ d} R ^{ -2d}  \textrm{Var}\left(\P(f_{ R})\right),
\end{align*}
\item   Let $ R\geqslant 1.$ If $ \S(B_{ \varepsilon })\leqslant \varepsilon ^{ d}\xi  (\varepsilon ^{ -1})$  for all $\varepsilon \in [R^{ -1},1]$  for some $ \xi :\mathbb{R}_{  + }\to \mathbb{R}_{  + }$,  then     
\begin{align*}
 \textrm{Var}\left(\M(f_{ R})\right)\leqslant  (2\pi )^{ -d}R^{ d}  { \mathsf c_{ f}^{ 2}}(C_{ d, \M}\frac{ d + \gamma }{\gamma }R^{ -\gamma } + \xi  (R)  + \sigma  _{ \xi ,\gamma   }(R)),
\end{align*}
where $ C_{ d,\M}$ is from   Lemma \ref{lm:trans-bd}.
\end{enumerate}
  
  \end{lemma}
  }

  Let us first prove the Theorem.
  \begin{proof} [Proof of Theorem  \ref{thm:general-hu}] 
 We immediately  have (i) implies (ii) since $ 1_{ B_{ 1}}\in  \
\mathscr B_{ c}(\mathbb{R}^{ d})$ and hence is admissible. \begin{itemize}
\item (ii)$  \Rightarrow\;$(iii). 
(ii) means that $  \textrm{Var}\left(\M(f_{ R})\right) = R^{ d}\sigma (R)$ with $ \sigma (R)\xrightarrow[ R\to \infty ]{}0.$ Since $ f$ is integrable, $ \hat f$ is continuous with $ \hat f(0) = \int_{}f\neq 0$. Hence  { by Lemma  \ref{lm:new-Adhik-SF}-(1), (iii) is proved:
\begin{align*}
 \S(B_{ \varepsilon }) = O(\varepsilon ^{ d}\sigma ( \varepsilon ^{ -1})) = o(\varepsilon ^{ d}).
\end{align*} }
\item (iii) $  \Rightarrow\;$(i).  {  {Recall that $ 1_{ B_{ 1}}\in \F_{ 1}.$ Hence we apply} Lemma  \ref{lm:new-Adhik-SF}-(2), with $ \S(B_{ \varepsilon }) = \varepsilon ^{ d}\xi (1/\varepsilon )$ with $\xi  (R)\xrightarrow[ R\to \infty ]{}0.$ Hence  \begin{align*}
  \textrm{Var}\left(\M(B_{ R})\right) = O(R^{ d}(R^{ -1} + \xi  (R) + \sigma _{ \xi ,1 }(R))).
\end{align*}
Lebesgue's theorem yields that $ \sigma  _{\xi ,1 }(R)\xrightarrow[ R\to \infty ]{}0$, which concludes the proof.
}

 \end{itemize}
 
\end{proof}

Let us now prove the lemma. 

\begin{proof}[Proof of Lemma  \ref{lm:new-Adhik-SF} ]

 We have with    \eqref{eq:structure-def}, \begin{align}
\notag
 (2\pi )^{ d} \textrm{Var}\left(\M(f_{ R})\right) =&\int_{} | \hat f_{ R} | ^{ 2}\S(du).
  \end{align} We  have  the scaling identity $ \widehat{ f_{ R}}  = R^{ d} \hat f(R\;\cdot )$ for admissible $ f\in L^{ 1}(\mathbb{R}^{ d}).$ 
  \begin{enumerate}
\item  Let $ \rho _{ f}>0$ such that $  | \hat f(u) | \geqslant  | \hat f(0) | /\sqrt{2} =   | \int {f} |/\sqrt{2}>0 $ for $u\in  B_{ \rho_{ f} }$. Then the conclusion comes from
\begin{align*}
 (2\pi )^{ d}  \textrm{Var}\left(\M(f_{ R})\right) = \int_{} | \hat f_{ R} | ^{ 2}d\S \geqslant  \frac{ 1}{ 2}   R^{ 2d} \S(B_{ \rho _{ f}/R})    \left|
\int_{}f 
\right| ^{ 2}.
\end{align*}
\item We have by   {definition of $  \F_{ \gamma }^{ *}$}
\begin{align*}
 |  \widehat{ f_{ R}}(u)  | =   R^{ d}  | \widehat{ f}(Ru) | \leqslant   \mathsf c_{ f }R^{ d}(1 + \|Ru\|)^{ -\frac{ d + \gamma }{2}}.
\end{align*} We   have 
\begin{align*}
 (2\pi )^{ d} \textrm{Var}\left(\M(f _{ R})\right)\leqslant & R^{ 2d}  \mathsf c_{ f}^{ 2}\S(B_{ 1/R}) + R^{ 2d}   \mathsf c_{ f}^{ 2} \int_{B_{ 1/R}  ^{ c}}(  \|u\|R)^{ -d-\gamma    }\S(du)
\end{align*}
  \begin{itemize}
\item The first term is  bounded by  $ \mathsf c_{ f}^{ 2}R^{ d  }\xi   (R).$
\item For the second term, we use the layer cake formula with $ \varphi (u) = (R\|u\|)^{ -d-\gamma   }$
\begin{align*}
  \int_{B_{ 1/R}^{ c}} \varphi (u)\S(du) = \int_{0}^{ 1    }\S(\{u:  \varphi (u)>t\})dt = \int_{0}^{ 1    }\S\left(
B(0,{ t^{ -\frac{ 1}{d+\gamma  }}/R})
\right)dt.
\end{align*}

We divide further in $ t<R^{ -d-\gamma   }$ and $ t\geqslant R^{ -d-\gamma   }$.
 \begin{itemize}
\item 
Recall  Lemma \ref{lm:trans-bd} for the large radius  contribution: the $ [0,R^{ -d-\gamma }]$ part is bounded by a constant times
\begin{align*}
  \int_{0}^{ R^{ -d-\gamma   }}R^{ -d}t^{ -\frac{ d}{d+\gamma  }}dt =  \frac{d +  \gamma }{ \gamma } R^{ -d} (R^{ -(d + \gamma )   \frac{ \gamma }{d+\gamma  } })=    \frac{d +  \gamma }{ \gamma } R^{ -d-\gamma   }.
\end{align*}

\item For the other contribution,  $ t^{ -\frac{ 1}{d + \gamma }}/R\geqslant 1/R$, hence the assumption yields
\begin{align*} 
\int_{R^{ -d-\gamma }}^{ 1}\S(B(0,{ t^{ -\frac{ 1}{d + \gamma }}/R}))dt \leqslant   & \int_{R^{ -d-\gamma }}^{ 1}\xi  (Rt^{ \frac{ 1}{d+\gamma  }})R^{ -d}t^{ -\frac{ d}{d+\gamma  }}dt = R^{ -d}\sigma  _{ \xi  ,\gamma }(R).
\end{align*} \end{itemize}

 \end{itemize} 
Gathering all terms gives the result
\begin{align*}
(2\pi )^{ d}  \textrm{Var}\left(\M(f_{ R})\right) \leqslant  \mathsf c_{ f}^{ 2}\left(
R^{ d}\xi  (R) + C_{ d,\M}  \frac{d +  \gamma }{\gamma } R^{ d-\gamma   }+R^{  d} \sigma  _{ \xi ,\gamma } (R))
\right).
\end{align*} 

\end{enumerate}
  
  \end{proof}

  \section{Universal variance lower bounds and non-spherical windows}
  \label{sec:beck}
  
A celebrated result of Beck \cite{Beck87} states that for a deterministic point configuration $ P$, the fluctuations of the number of points in a large sliding  window $ B_{ R}$ are at least of the order $ \sqrt{R^{d-1}}$, which leads in general to a variance lower bound of the order $ R^{ d-1}$ for arbitrarily large $ R$ and a stationary  point process. This principle is  not restricted to atomic measures, as we shall see here, stating the generalisation of  \cite[Theorem 1.1]{bylehn}. 

\begin{theorem}[Beck]
\label{thm:beck}
 Let $ \M$ a  \wsrm~that is not identically $ 0$ a.s.. Then for some $ c >0$, for $  {\bf R} $ sufficiently large,
\begin{align*}
 \int_{1}^{  {\bf R}}\frac{  \textrm{Var}\left(\M(B_{ R})\right)}{R^{ d-1}}dR\geqslant c  {\bf R}
\end{align*}
which in particular yields
\begin{align*}
\limsup_{ R\to \infty }\frac{  \textrm{Var}\left(\M(B_{ R})\right)}{R^{ d-1}}>0.
\end{align*} 
 
 \end{theorem}       
  \begin{proof}
  
Let $ f=1_{B_{1}}$. By Lemma \ref{lm:fourier-ball}, for some finite constants $ c_{ d},c_{ d}'>0,$
\begin{align*}
 \widehat{ 1_{ B_{ 1}}}(u) =c_{ d} \|u\|^{ -\frac{ d + 1}{2}}\sin(\|u\|-c'_{ d})(1 + o_{ u\to \infty }(1)).
\end{align*}
Hence for some $ \kappa >0,$ $  | \hat f | ^{ 2}$ is larger than $\kappa (1 + \|u\|)^{ -{ d-1}{}}$ ``on average'', i.e. there is $  \rho _{0}>0,  {\bf R}_{0}>0$ such that $ \S(B_{\rho _{0}}^{c})>0$ and   for $ u\notin B_{\rho _{0}}, {\bf R}>  {\bf R_{0}},$ 
\begin{align}
\label{eq:avg-fourier-large}
 \int_{1}^{  {\bf R}} R^{ d + 1} | \hat f(Ru) | ^{ 2}dR\geqslant &\kappa  \int_{1}^{  {\bf R}}R^{ d + 1}(1 + \|u\|R)^{ -d-1}dR.
\end{align}
 
   We then have by \eqref{eq:structure-def}, for $  {\bf R}>  {\bf R_{0}},$
\begin{align*}
(2\pi )^{ d}\int_{1}^{  {\bf R}}R^{ 1-d}  \textrm{Var}\left(\M(f_{ R})\right)dR   \geqslant    &\int_{\mathbb{R}^{ d}  \setminus B_{\rho _{0}}}\int_{1}^{  {\bf R}}R^{d+1}  | \hat f(Ru) | ^{ 2}dR\S(du) \\
\geqslant & \kappa \int_{\mathbb{R}^{ d}  \setminus B_{\rho _{0}}} \int_{1}^{  {\bf R}}R^{ d + 1}(1 + \|u\|R)^{ -d-1}dR\S(du)\\
 =&\kappa  \int_{\mathbb{R}^{ d}  \setminus B_{u_{0}}} \|u\|^{ -d-2}\int_{\|u\|}^{ \|u\|  {\bf R}}\rho ^{ d + 1}(1 + \rho )^{ -d-1}d\rho \S(du)\\
\geqslant& \kappa \int_{\mathbb{R}^{ d}  \setminus B_{\rho _{0}}} \|u\|^{ -d-2}(1 + 1/\rho _{ 0})^{ -d-1}\|u\|  ({\bf R}-1)\S(du).
\end{align*}  
This completes the proof.
\end{proof}


   The proof above actually works for any kernel $ f$ satisfying  \eqref{eq:avg-fourier-large}. 
   Since the ball is the most regular shape in many aspects, one could imagine that it has the lowest Fourier transform in some sense, hence this lower bound could hold for any shape. This is in fact not the case for rectangular shapes, whose Fourier transform is indeed larger in some directions, but also smaller in others. The class of admissible shapes, called  {\it Fourier smooth,} is discussed around \cite[Theorem 3.6]{HartBjo}. 
Let us make some further remarks:\begin{itemize}
\item {(balls number variance vanishing $ \liminf$)}  \cite[Prop. 2.2.3]{bylehn} and \cite{Parnovski} show that,  iff  $ d \equiv 1  \text{\rm{ mod }}  4 $,
\begin{align*}
 \liminf_{ R\to \infty }\frac{  \textrm{Var}\left(\Z^{ d}(B_{ R})\right)}{R^{ d-1}} = 0.
\end{align*}
\item {(General shape dependence)} The $ \limsup$ bound can be far off on rectangular windows when one studies non-atomic random measures, even with short range dependency. Consider for instance the spectral measure  $ \S = 1_{S}\Leb$ where $ S=\{(u_{i})\in \mathbb{R}^{d}: \forall i, 2>| u_{i} | >1\}$, and the unique centred Gaussian field whose spectral measure is $ \S$ (Example \ref{ex:gaussian}). One can directly show that the variance is bounded: from \eqref{eq:structure-def}
\begin{align*}
  (2\pi )^{ d}\textrm{Var}\left(\M([-R,R]^{ d})\right)  = & \int_{S} R^{ 2d}  |  \prod_{i = 1}^{ d} \sin(Ru_{ i})/(Ru_{ i}) | ^{ 2}du\\
  \leqslant &\int_{S}  c\prod_{i}u_{i}^{ -2}du<\infty .
\end{align*}
One can refine this example by taking $ \S$ with support on all $ \mathbb{R}^{ d}$, as long as it vanishes sufficiently fast close to the axes and at infinity.
\item {(shape-dependance for point processes)} In a private communication, M. Bylehn mentions that there are some  {\it admissible orthogonal transformations} $  {\bf O}$ such that the previous $ \limsup$ bounds does not hold for the  rotated lattice $  {\bf O}\Z^{ d}$, using a bound of  Skryganov \cite{Skryganov}:
\begin{align*}
{  \textrm{Var}\left(  {\bf O}\Z^{ d}([-R,R]^{ d})\right)} = O(\ln(R)^{ 2(d-1)}).
\end{align*}
\item {(no shape-dependence for disordered point processes)} See also Proposition  \ref{prop:universality-C-integrable}.  Nazarov \& Sodin  \cite[Lemma 1.6]{nazarov2012correlation} show that in dimension $ d = 2,$ under the  {\it weak disorder} assumption that $ \C-\delta _{ 0}$   is integrable for some point process $ \P$, for any bounded window $ W$ with non-empty interior, $  \textrm{Var}\left(\P(RW)\right)$ satisfies Beck's lower bound.
\end{itemize}
  
  \begin{remark}
  When the number variance of a point process is subextensive, i.e. in $ o(R^{ d})$, then under some integrability conditions, it  behaves necessarily in $ R^{ d-1}$, see  \cite[Proposition 2]{martin1980charge}. Hence other rates  indicate a weak asymptotic decorrelation (or no decorrelation).
  \end{remark}

\section{Hyperuniformity  exponent and classification}
\label{sec:hyp-expo}

The speed of decay of the structure factor in $ 0$ actually matters, for the decay of smooth linear statistics, but also for other phenomena, such as rigidity (Chapter  \ref{chap:rigid}).
 It is traditionally said that some number $ \alpha>0 $ is a  hyperuniformity exponent of $ \S$ (or $ \M$) if $ \S(du)\sim c\|u\|^{\alpha }du$ as $ u\to 0$ for some $ c>0$. We shall more generally say without requiring a density that $ \S$ admits hyperuniformity index $ \alpha>0 $ if $ \S(B_{ \varepsilon }) = O(\varepsilon ^{ d + \alpha })$ as $ \varepsilon \to 0.$ We show here that $ \alpha $ also determines the maximal speed to which the variance of linear statistics can go to $ 0.$ 
\begin{proposition}
\label{prop:hyp-expo}   Let $ \S$ the spectral measure of a $ L^{ 2}_{ loc}$ \wsrm~ $ \M.$
Assume $ \S$ admits exponent $ \alpha>0 $. Let $ f$ an $ \S-$admissible function  {of $ \F_{ \gamma }^{ *}(\mathbb{R}^{ d})$. Then} 
\begin{align*}
 \textrm{Var}\left(\M(f_{R})\right) =   
 \begin{cases} 
 O(R^{d-\min(\alpha ,\gamma )  })$  id $\alpha \neq \gamma  .\\
 O(R^{ d-\alpha }\ln(R)))$  if $\alpha  =  \gamma . 
  \end{cases}
\end{align*}
In this case we say that $ \M$ is  {\it  $ \alpha $-hyperuniform.}

If conversely $  \textrm{Var}\left(\P(f_{ R})\right) = O(R^{ d-\alpha })$ holds for some $ \S$-admissible function $ f$ with $ \int_{}f\neq 0$, then $ \S$ admits exponent $ \alpha $.
\end{proposition}

Clearly, if $ \M$ is $ \alpha $-hyperuniform, it is $ \alpha '$-hyperuniform for any $ \alpha '\leqslant \alpha .$\\

A version of the direct implication for Schwartz functions appears in  \cite{MBL} under an integrability assumption on $ \S$. Variance estimation for linear statistics is also a central topic in  \cite{KrishYogesh}, under the running assumption that the covariance measure  $ \C$ is integrable, or has finite higher order moments. They notice in particular that the decay exponent is even when the covariance is integrable. Due to the reciprocal relation with the hyperuniformity exponent exhibited by the previous proposition, one can also see it as a consequence of the symmetry of $ \S$. They also give sufficient conditions for $ \S$ to admit exponent $ \alpha $, and discuss the class of admissible kernels $ f.$


 { 
 \begin{proof}
 
 We apply Lemma \ref{lm:new-Adhik-SF}-(2) with $\sigma  (R ) =R^{- \alpha  }.$ We then have
\begin{align*}
 \xi _{ \sigma ,\gamma }(R) = \int_{R^{ -d-\gamma }}^{ 1} R^{ -\alpha } t^{ \frac{ -\alpha }{d + \gamma }}t^{ -\frac{ d}{d + \gamma }}dt = R^{ -\alpha } \times  
 \begin{cases} 
\left(
(R^{ -(d + \gamma) })^{ 1 - \frac{ \alpha  + d}{d + \gamma }}dt
\right) = O(R^{  -\gamma })$  if $\alpha >\gamma  \\
 O(\ln(R))$  if $\alpha  = \gamma 
\\
O(1)$  if $\alpha <\gamma .
  \end{cases}
\end{align*}
We indeed have
\begin{align*}
  \textrm{Var}\left(\M(f_{ R})\right)  = 
  \begin{cases} 
 O (R^{ d-\min(\gamma ,\alpha )})$  if $\alpha  \neq \gamma  \\
  O(R^{ d-\alpha }\ln(R))$  if $\alpha  = \gamma  .
   \end{cases}\end{align*}
   
   If conversely $  \textrm{Var}\left(\M(f_{ R})\right) = O(R^{ d-\alpha })$, Lemma \ref{lm:new-Adhik-SF}-(1) with $ \sigma (R) = R^{ -\alpha }$ yields indeed $ \S(B_{ \varepsilon }) = O(\varepsilon ^{ d + \alpha }).$
 
 \end{proof}}
   
Example  \ref{ex:irrational-lattices} and  Theorem  \ref{thm:expo-prime}
 provide examples of point processes with arbitrarily fast decay for $ \S$ around $ 0.$
The optimal exponent $ \alpha $ is related to a classification of  hyperuniform point processes  relevant in the physics literature  \cite{Tor18}, depending on  the number variance behaviour (i.e. for $ f = 1_{ B_{ 1}}$ and $\gamma  = 1$). 

 { \begin{itemize}
 \item For $ \alpha  = 0,  \textrm{Var}\left(\P (B(x ,R))\right) \asymp R^{ d}$  corresponds to the the rate of  a Poisson process, or any other ``standard'' process (zeros of random stationary Gaussian functions, cluster or Cox processes, short range Gibbs measures, ...)
\item The case $ \alpha >0$ implies a variance reduction, or main term cancellation, i.e.  {\it hyperuniformity}. More precisely, the assumption means that the system is at least ``$ \alpha $-hyperuniform''. For $ \alpha \in (0,1)$, the exponent is not truncated, this is the case for instance for Riesz gases with exponent $ s\in (0,1)$, see Section \ref{sec:riesz}.
\item We see that a transition occurs at $ \alpha  = 1$.   The $ \ln(R)$ term for $ \alpha  = 1$, corresponding to ``$ 1$-hyperuniformity''  is sharp, in the sense that many models for which $   1 $ is the largest possible exponent, such as the GUE ensemble in dimension $ d = 1$ or its limit the $ \text{\rm{Sine}}_{ 2}$ process, have a logarithmic number variance over large intervals. The case $ \alpha >1$ in 1D corresponds to a bounded variance, it is the case for instance for the one-dimensionnal Coulomb gas \cite{GosLeb}.
\item  Torquato \cite{Tor18} divides hyperuniform systems in three classes $ \alpha >1$  (class I)$,\alpha  = 1$  (class II)$,\alpha \in (0,1)$ (class III), reflecting number variance behaviour over growing balls. It is also formally possible to build  hyperuniform point processes falling in none of these classes, i.e.  for which the variance reduction is subpolynomial, see Example  \ref{ex:irrational-lattices} or   \cite{DFHL}.
  \item For $ \alpha >1$ (class I), the ball number variance growth is in surface order $ R^{ d-1}$. By Beck's lower bound principle (Section \ref{sec:beck}), this is a universal lower bound for balls number variance.
The precise value of $ \alpha $ is therefore uninformative for the   number variance over balls. It does not mean that all class I systems are similar, they incompass both periodic and disordered systems. They can in particular behave very differently when investigated over smooth kernels, especially for large $ \alpha $, where it can be useful in numerical integration.
  \item The most standard case is $ \alpha  = 2$, corresponding to a sufficiently disordered  hyperuniform model, or more formally to a sufficiently decaying covariance function (\cite{GosLeb}), such as for the 1D Coulomb gas or the 2D ginibre ensemble, see Chapter \ref{chap:proofs}.

\item   For the variance of irregular linear statistics others  than the ball indicator, i.e. for $ \gamma \neq 1$, a similar transition occurs at $ \alpha  = \gamma $ under some mild conditions.  For most test functions $ \varphi \in \F_{ \gamma },$ a similar bound lower than Beck in $ O(R^{ d-\gamma })$ applies under broad assumptions, by mimicking the proof of Theorem \ref{thm:beck}.

\item While most known models have index $ \alpha  = 2$, or $ \alpha  = 1$ (more often in dimension $ d = 1$), there exists a unique known disordered model for which $ \alpha >2$; it is the zero set of Weyl polynomials, or in the infinite volume limit the zero set of the planar Gaussian Analytic Function, for which $ \alpha  = 4$  \cite{ST06}, see Section \ref{sec:GAF}. There also exist some toy models with arbitrarily high $ \alpha $, see  \cite{torquato-tilings,Lr24,LotzKlatt}.

\item The purely atomic nature of particle systems is not used, and results   apply to other classes of random measures.
\end{itemize} 
}

Let us conclude with a lemma that yields that most point processes have exponent at most $ 2.$

   \begin{lemma}
   \label{lm:quadra-spectral}
Let $ \mu $ be a  symmetric  non-negative finite  measure which is not supported by a hyperplane. 
 Then there is $ \sigma >0,\rho _{ 0}>0$ such that   for $ \|u\|<\rho _{ 0},$
\begin{align*}
 \left|
\int_{}(1-e^{- iu\cdot x})\mu (dx)
\right|\geqslant \sigma \|u\|^{ 2}.
\end{align*}
This conclusion also holds if the assumption of symmetry is dropped, provided $ \mu $ is instead assumed to have a finite second moment and be centred, i.e. $ \int_{}xd\mu(x) =0.$
\end{lemma}

\begin{proof}
  Let $ R>0$ such that 
\begin{align*}
 \int_{B(0,R)}\|x\|^{ 2}\mu (dx)>0
\end{align*} and $ \rho >0$ such that $ 1-\cos(t)\geqslant t^{2}/4$ for $   | t |  \leqslant R\rho ,$ and let   $ u\in B(0,\rho )$. 
We have
\begin{align}
\notag
\s(u): = \int_{}(1-e^{ iu\cdot x})\mu (dx)= &\int_{}(1-\cos( x\cdot u))\mu (dx)\\
 \label{eq:bound}&  { \geqslant \int_{B(0,R)}\frac{ | x\cdot u | ^{ 2}}{4}\mu (dx).}
\end{align} 
 {If $ \mu $ is not symmetric but centred with a second moment, we use the expansion $ e^{ -it} = 1-it-  \frac{ 1}{ 2}   t ^{ 2}(1 + \xi (t))$ where $ \xi $ is bounded and vanishes in $ 0.$ Therefore 
\begin{align*}
 \s(u) = &iu\cdot\underbrace{ \int_{}x\mu (dx) }_{ = 0}+ \int_{}  \frac{ 1}{ 2}    | x\cdot u | ^{ 2}\mu (dx)  + O\left(
\|u\|^{ 2}\int_{}\|x\|^{ 2}\xi (x\cdot u)\mu (dx)
\right)
\end{align*}
and the last term is in $ o(\|u\|^{ 2})$ using Lebesgue's theorem. Therefore in both cases it suffices to show $ \int_{B(0,R)} | x\cdot u | ^{ 2}\mu (dx)\geqslant \sigma \|u\|^{ 2}$ for $ u$ sufficiently small. }

Call $ C_{ \varepsilon  }(u)$ the  cone  of $ x$'s such that $  | x\cdot u | \geqslant \varepsilon  \|x\|\|u\| $ (and $ C_{0}(u)=\mathbb{R}^{d}$). 
Define the function $$\Sigma (v,\varepsilon ): =\int_{C_{ \varepsilon }(v)\cap B(0,R)}\|x \|^{ 2}\mu (dx),v\in \partial B(0,1).$$
Since by assumption $ \mu $ is not supported by the hyperplane orthogonal to $ v\in \partial B(0,1)$, $\Sigma (v,\varepsilon _{ v})>0 $
for some $ \varepsilon _{ v}>0.$
We wish to show by contradiction that
\begin{align*}
  \exists  \sigma_{ 0} >0,\varepsilon >0:  \forall  v\in \partial B(0,1) ,\Sigma (v,\varepsilon )\geqslant \sigma _{ 0}.
\end{align*}

If it is not the case, there is $ v_{ n}\in \partial B(0,1) $ such that $ \Sigma (v_{ n},1/{ n})\leqslant 1/n$. By compacity we can choose $ v_{ n}$ that converges to some $ v,$ and it yields a contradiction when $ C_{ \varepsilon _{ v}}(v)\subset C_{ 1/n}(v_{ n})$. In consequence, for  $ u\in B(0,\rho )$,  we can conclude the proof with
\begin{align*}
\s(u)\geqslant \varepsilon ^{ 2} \|u\|^{ 2}\sigma_{ 0} /4.
\end{align*}
\end{proof}

This lemma implies  that stationary determinantal processes have exponent at most $ 2$ (see Theorem  \ref{Thm:DPP-rigid}), and with  \eqref{eq:IPL-SF} yields  the following corollary for some IPLs:
\begin{corollary}
Let $ \mu $ a non-null probability measure on $ \mathbb{R}^{d}$ which is symmetric. 
 Then the independently perturbed lattice $ \Z^{d,\mu }$ has exponent at most $ 2.$
\end{corollary}
Even if $ \mu $ is supported by a hyperplane, doing the same analysis on a subspace of minimal dimension supporting $ \mu $, we have $ \s(u)\geqslant \sigma  | u_{ i} | ^{ 2}$ at least for one coordinate $ i,$ which forbids exponent more than $ 2.$

    \section{CLTs and Brillinger mixing}
    \label{sec:CLT-mixing}

We investigate the Central Limit Theorem for the mass of a large ball $ \M(B_{ R})$   for a \wsrm~$ \M$. The cumulants method, classical for extensive systems, fortunately extends to  hyperuniform systems, due to the universality of Beck's lower bound (Theorem \ref{thm:beck}). 
 Recall that the $ m$-th cumulant of a real random variable $ X$ is the $ m$-th order derivative in $ 0$, when it exists, of the log of the moment generative function 
 $K_{ X}(t) = \ln \mathbf E e^{ tX},$ i.e. $$ \kappa _{ m}(X) = \frac{ d^{ m}}{dt^{ m}}K_{ X}(t) | _{ t = 0}.$$
We have a familiar interpretation for low order cumulants: for $  \bar X = X-\mathbf E X,$
\begin{align*}
 \kappa _{ 1}(X) = &\mathbf E X\\
 \kappa _{ 2}(X) =  &\textrm{Var}\left(X\right)\\
 \kappa _{ 3}(X) = &\mathbf E  \bar X^{ 3}\\
 \kappa _{ 4}(X) = &\mathbf E  \bar X^{ 4}-  \textrm{Var}\left(X\right)^{ 2}.
\end{align*}

For instance for $ X\sim \Poi(\lambda ),\lambda >0$, we have easily $ K_{ X}(t) = \lambda (e^{ t}-1) = \lambda \sum_{k\geqslant 1}\frac{ t^{ k}}{k!}$, hence for all $ m$
\begin{align*}
 \kappa _{ m}(X) = \lambda .
\end{align*}
The Gaussian variables are characterised as those variables such that $ \kappa _{ k}(X) = 0$ for $ k\geqslant 3$, indeed $ K_{ X}(t) = t^{ 2}$ for $ X\sim\N(0,1)$. Marcinkiewicz refined this result by showing that a variable is Gaussian
as soon as only finitely many cumulants do not vanish.

Recall that a sequence of variables $ X_{ n},n\geqslant 1$ converge to a standard Gaussian variable $ X$ if all moments converge, i.e. for each $ m\geqslant 1,$ $ \mathbf E X_{ n}^{ m}\to \mathbf E X^{ m}$ as $ n\to \infty .$ Since moments are linear combination of cumulants, the convergence still holds if for each $ m\geqslant 1,$ $   \kappa _{ m}(X_{ n})\to  \kappa _{ m}(X)$. Using Marcinkiewicz theorem, one can show that this convergence holds  if one only assumes   $  \textrm{Var}\left(X_{ n}\right)\to 1$ and $ \kappa _{ m}(X_{ n})\to 0$ for all $ m\geqslant m_{ 0}$, for some $ m_{ 0}\geqslant 3$, see  for instance \cite[Lemma 3]{soshnikovCLT}.

Define in general 
\begin{align*}
  \widetilde{ \M(B_{ R})} = \frac{ \M(B_{ R})-\mathbf E \M(B_{ R})}{\sqrt{  \textrm{Var}\left(\M(B_{ R})\right)}}.
\end{align*}
The previous method applies to the number of points of  a Poisson point process in a large ball because $ \P^{ \Poi(\lambda )}(B_{ R})  \smallequlaw \text{\rm{Poisson}}( \lambda \Leb(R^{ d}))$. Hence $  \textrm{Var}\left(\P^{ \Poi(\lambda )}(B_{ R})\right) \asymp R^{ d}$ and 
\begin{align*}
 \kappa _{ m}\left(  \widetilde{ \P^{ \Poi(\lambda )}(B_{ R})}
\right) =  \textrm{Var}\left(\P^{ \Poi(\lambda )}(B_{ R})\right)^{ -m/2} \kappa _{ m}(\P^{ \Poi(\lambda )}(B_{ R}))\asymp R^{ -md/2} R^{ d},
\end{align*}
it indeed goes to $ 0$ for $ m\geqslant 3$.
 More generally, it applies to many standard and  hyperuniform random measures:
\begin{theorem}
\label{thm:CLT}
Let $ \M$ a \wsrm~ in dimension $ d\geqslant 2$ having finite moments of all orders on a compact set.   
 Assume that for some $ m_{ 0}\geqslant 3$, for $ m\geqslant m_{ 0}$, the cumulants have Poisson / sub-Poisson decay
 \begin{align}
\label{eq:subpoisson-cumulant} \kappa _{ m}(\M(B_{ R})) =& O(R^{ d}) .
\end{align}
Then we have the CLT for some sequence $ R_{ n}\to \infty $
\begin{align*}
\widetilde{ \M(B_{ R_{ n}})}\to \N(0,1).
\end{align*}
In   dimension $ 1$, if \eqref{eq:subpoisson-cumulant} holds  and $  \textrm{Var}\left(\M(B_{ R})\right)\geqslant cR^{  \alpha }$ with $ \alpha >0,$  as $ R\to \infty ,$
\begin{align*}
  \widetilde{ \M(B_{ R})}\to \N(0,1).
\end{align*}
 \end{theorem}

 \begin{proof}
Let $ X_{ n} = \M(B_{ R_{ n}}), \widetilde{ X_{ n}} =   \widetilde{ \M(B_{ R})},n\geqslant 1$.
Recalling Beck's Theorem  \ref{thm:beck},  $  \textrm{Var}\left(X_{ n}\right)\geqslant cR_{ n}^{ d-1}$ for some $ R_{ n}\to \infty $ and 
\begin{align*}
 \kappa _{ m}(  \widetilde{ \M(B_{ R_{ n}})})  =  \frac{ \kappa _{ m}(\M(B_{ R_{ n}}))}{  \textrm{Var}\left(\M(B_{ R_{ n}})\right)^{ m/2}}\leqslant c\frac{ R_{ n}^{ d}}{R_{ n}^{ (d-1)m/2}}.
\end{align*}
We see that indeed for $ d> 1$ and $ m$ sufficiently large, the right hand side goes to $ 0$, and Marcinkiewicz's theorem allows to conclude to the CLT. In dimension $ d = 1$, the bound is $ R^{ d-\alpha k/2}$, and we can conclude similarly.
 \end{proof}
 
 Remark that the $ R^{ d-1}$ lower bound in Theorem \ref{thm:beck} actually holds for a set of radii $ R$ with positive Lebesgue density, it is not just a marginal sequence of $ R_{ n}$'s.
  For other linear statistics $ \M(f_{ R})$, Beck's bound might not hold, see in particular Proposition \ref{prop:hyp-expo}. It can still happen that there is a CLT even if the variance is bounded or, surprisingly, goes to $ 0$, see for instance  \cite{ST1}, but this is rather exceptional. See  \cite[Section 4.1]{KrishYogesh} for a discussion, and for more general results than Theorem  \ref{thm:CLT}.

Assumption \eqref{eq:subpoisson-cumulant}  is the manifestation of a property which is expected for strongly disordered point processes, named  {\it Brillinger mixing}, but hard to prove apart from some well understood classes such as Poisson, determinantal  and permanental processes, or zeros of random Gaussian fields.   One can strictly weaken this assumption to $ \kappa _{ m}(\M(B_{ R}))  = o(R^{ m(d-1)/2})$ (for $ m$ above some $ m_{ 0}\in \mathbb{N}$) but there is no immediate interpretation for the relevancy of such an hypothesis. 

This theorem has been applied successfully to many  linear statistics over point processes, but also to random measures. In  \cite{BYY}, the authors consider more generally a geometric functional over a point process $ \P$ under the form
\begin{align*}
 X_{ n} = \sum_{x\in \P\cap B_{ n}}\xi (x,\P)
\end{align*}for some  {\it score function $ \xi $} that does not only depend on the location $ x$ and is invariant under simultaneous translations of its two arguments. It can be interpreted as a linear statistic over the stationary random measure 
\begin{align*}
 \M = \sum_{x\in \P}\delta _{ x}\xi (x,\P).
\end{align*}
 Under some assumptions of stabilisation and dependency decay related to Brillinger mixing on $ \P$ and $ \xi $, they are able to show a CLT for $ \M(B_{ R})$. 

\begin{longversion}{}
 {
 clearly the difficult point in applying this theorem, but it is possible for DPPs, which is the occasion to recall that DPPs are stochastically   dominated by Poisson (...).
 \begin{theorem}
  Let $ \P^{ \det}$ a determinantal process with kernel $ K$ bounded on $ A\times A$ for $ A\subset \mathbb{R}^{ d}$ and $ \P^{ \text{\rm{Poiss}}}$ a Poisson process with intensity ... Then there is a coupling such that a.s. 
\begin{align*}
 \P^{ \det} \subset \P^{ \text{\rm{Poiss}}}.
\end{align*}
  \end{theorem}
  We have the following corollary, first in [GAF book]
  \begin{corollary}
   On $ \mathbb{R}^{ d}$, if $ K$ is locally bounded, it clearly implies
   $ \kappa _{ k}(\P^{ \det}_{ K}(A))\leqslant \kappa _{ k}(\P^{ \text{\rm{Poiss}}}_{ \|K\|})$.  Hence the DPP has a CLT.
  \end{corollary}
  They actually present a simpler proof not using  hyperuniformity based on the fact that $ \P(B_{ R})$ can be represented as a sum of i.i.d.   Bernoulli variables, and as they notice, one can deduce further asymptotic results in this way.

   Another way to prove ..., is to prove Brillinger mixing 
   
   ...
   
   The coupling?
}
\end{longversion}

  \section{Hyperuniformity of finite particle systems}  
  \label{sec:hu-finite}
  
  The infinite stationary systems studied above are often idealised models for large finite systems of $ N$ particles, that are usually confined around some growing set $ \Sigma  = N^{ 1/d}\Sigma _{ 0}$, where $ \Sigma _{ 0}\subset \mathbb{R}^{ d}$ is bounded. Typical examples include Gibbs measures, eigenvalues of random matrices, and zeros of random polynomials, see Section \ref{sec:intro-examples}.
  Finite systems present in general more complexity because, inherently to their finite nature, they are not invariant under translations, even locally.  There are in general also different behaviour for the particles close to   $\partial  \Sigma $, due to boundary effects, and particles from  {\it the bulk}, i.e. deep inside the point cloud. The latter will be easier to study  because they  asymptotically are comparable to infinite volume limits.

  The geometric definition of  hyperuniformity in  \eqref{eq:intro-hu} can be adapted to a  sequence of point processes $ \P_{ N},N\geqslant 1$.
  Surprisingly, the power of the spectral analysis of stationary random measures can be reproduced for such models, provided they present sufficient spatial homogeneity, i.e. asymptotic stationarity.

The main point here is to  establish a variance transfer principle stating that $ \alpha $-hyperuniformity for smooth kernels transfers to irregular kernels, as in Theorem \ref{thm:general-hu} and Proposition \ref{prop:hyp-expo}.

\begin{definition}
 We say that $ \{\P_{ N};N\geqslant 1\}$ is asymptotically $ \alpha $-hyperuniform on the scales $ I_{ N}: = [R_{ N}^{ -},R_{ N}^{  + }]\subset [1,N^{ 1/2}]$  if for some $ f\in \CC_{ c}^{ \infty }(B_{ 1})  \setminus \{0\},$ there is $ C<\infty $ such that \begin{align}
 \label{eq:def-finite-HU}
  \sup_{ R\in I_{ N},x:B(x,R)\subset \Sigma }\textrm{Var}\left(\P_{ N}\tau _{ x}f_{ R}\right) \leqslant CR^{ d-\alpha }.
\end{align}

\end{definition}  
Note that this rate can only be obtained for $ f$ sufficently smooth (as in Proposition \ref{prop:hyp-expo}).
Clearly, $ \alpha $-hyperuniformity implies $ \alpha '$-hyperuniformity for $ \alpha '\leqslant \alpha .$
It is for many systems easier to prove  hyperuniformity for a smoother kernel than for the variance of  $ \P_{ N}(B(x,R))$  \cite{CES,serfaty-HU}.
The case that is made at  \cite{Lr-HU-finite}, reproduced here,   is that if  hyperuniformity is proved for such smooth $ f$, it can often be transferred to the number variance over ball indicators, or other irregular kernels. Such a result is based on a spectral analysis of translation invariant random measures on the torus.
  
  \subsection{Spectral measure of periodic systems}

For $ \lambda >0$, let $ \T = \mathbb{R}^{ d} / (\lambda \mathbb{Z} )^{ d}$ the torus with volume $ \lambda ^{ d}$, and $ \M_{ \lambda }$ a random measure with $ \mathbf E \left[
 | \M_{ \lambda } | (\T)^{ 2}
\right]<\infty $.
    Denote the toric spectral density by  
\begin{align}
\label{eq:def-spectral-density}
   \s_{ \M_{ \lambda}}(u) := \frac{ 1}{\mathbf E \left[
\M_{ \lambda} (\mathbb  T_{\lambda }^{ d})
\right]}   \textrm{Var}\left(
 \F \M_{ \lambda}(u)
\right),u\in \mathbb{R}^{ d}
\end{align}where $ \F\M_{ \lambda}(u) = \M_{ \lambda}(e^{- i \langle u,\cdot  \rangle })$ is the Stieltjes-Fourier transform. In general, up to rescaling, we work under the unit intensity assumption, i.e. $ \mathbf E\left[
 \M_{ \lambda}(  \mathbb  T_{\lambda }^{ d})
\right] = \lambda ^{ d}.$ At $ u\neq 0,$ the spectral density is often close to the expectation of the  {\it scattering intensity} 
\begin{align}
 \label{eq:def-scatt-intensity}
\hat \s_{ \M_{ \lambda}}(u) :=
\lambda ^{ -d}| \F\M_{ \lambda}(u) | ^{ 2},\end{align} numerically useful for estimation and further discussed at Chapter  \ref{ch:numerics}. Another possible renormalisation is by $ \M_{ \lambda}(  \mathbb  T_{\lambda }^{ d})$, instead of its expectation, but this yields a more noisy estimator.
The toric spectral density will also be useful to compute optimal transport distances  on finite samples (Chapter \ref{chap:transport}).
For $ \M$ a \LSI stationary point process on $ \mathbb{R}^{ d}$,  and $ \M_{ \lambda}$ the restriction on $  \mathbb  T_{\lambda }^{ d}$,  
$   \s_{ \M_{ \lambda}} $ provides a simple and natural approximation of the spectral measure $ \S$ of $ \M$. See for instance  \cite[Th. 5.1]{Coste} that shows weak convergence  if $ \M$ is ergodic, i.e. for $ f $ a Schwartz function on $ \mathbb{R}^{ d}  \setminus \{0\},$
\begin{align}
\label{eq:coste-weak-cvg-SI}
\int_{ \mathbb{R}^{ d}}f (u) \s_{ \M_{ \lambda}}(u)du
\xrightarrow[ n\to \infty ]{}\S(f). \end{align}
The proof in \cite{Coste} is for point processes on spherical domains with random normalisation, but it extends to the current framework. See Proposition \ref{prop:SI-grid} for a more quantitative statement. \\

 Let $ \tau _{ x}^{ \lambda },x\in \T$ the  translation operators in $ \T$. A random measure $ \M$ on $ \T$ is $ \T$-stationary if for $ x\in \T,$
\begin{align*}
 \tau _{ x}^{ \lambda }\M \stackrel{(d)}{=}  \M.
\end{align*}

We in general only require  {\it weak stationarity}, i.e. for $ x\in \T,f\in \CC_{ c}(\T)$
\begin{align*}
 \mathbf E \left[
\M[\tau _{ x}f]
\right] = \mathbf E \Big[
\M[f]
\Big],  \textrm{Var}\left(\M(\tau _{ x}f)\right) =  \textrm{Var}\left(\M(f)\right).
\end{align*}
For such systems,  $ \s_{ \M}$  define a spectral measure similar to that  of translation invariant systems on $ \mathbb{R}^{ d}$. We need to introduce the dual space 
$ \mathbb{Z} _{ \lambda }^{ d}: = 2\pi \lambda ^{ -1}\mathbb{Z} ^{ d}$ and the corresponding Fourier transform: for $ f\in  \mathscr  C_{ c}(\T)$
\begin{align*}
\what{f}{\lambda }(u) = \displaystyle\int_{\T}f(x)e^{ -iu\cdot x}dx, u\in \mathbb{Z} _{ \lambda }^{ d}.
\end{align*}

 \begin{proposition}
 \label{lm:spectral}
 Let $ \M$ a \LSI weakly  stationary random measure on $ \T.$
 There exists a non-negative measure $ \S = \S_{ \M}$ on $ \mathbb{Z} _{ \lambda }^{ d}$ such that
for $ f,g: \T\to  \mathbb C $ bounded,
\begin{align}
\label{eq:PS}
  \textrm{Cov}\left(\M(   f),\M(g)\right) =    \int_{ \mathbb{Z} _{\lambda }^{ d}}  \what{f}{\lambda }  \overline{  \what{g}{\lambda }}d\S.
\end{align}
$ \S$ is called the  {\it spectral measure} of $ \M.$ 
 \end{proposition}

    {As a trivial example, take $f(y) := 1_{  \mathbb  T_{\lambda }^{ d}}(y)e^{ iu_{ 0}\cdot y}$ for some $ u_{ 0}\in \mathbb{Z} _{ \lambda }^{ d}.$ For $ u\in \mathbb{Z} _{ \lambda }^{ d}$, $ \hat f(u) =  \lambda ^{ d} \delta _{ u = u_{ 0}}$. Hence, 
\begin{align}
\label{eq:exact-spectr}
\lambda ^{ 2d}\S(\{u_{ 0}\}) =  \S (  | \hat f | ^{ 2}) =  \textrm{Var}\left( \F\M ( u_{ 0} )\right).
\end{align}}

This proposition is a consequence of the more general Theorem  \ref{thm:wsrm-X} on abelian compact topological groups. 
Extend $ \M$ to functions on $ \mathbb{R}^{ d}$ through periodisation: for $ \varphi \in L^{ 1}(\mathbb{R}^{ d})$, let
\begin{align*}
 \M(\varphi ) := \sum_{\k\in \mathbb{Z} ^{ d}}\int_{ \T}\varphi (\lambda \k + x)d\M(x),
\end{align*}
well defined a.s. by stationarity because, denoting by $a  = \M(\T)$ the intensity, Fubini identity yields
\begin{align*}
 \mathbf E\left[
 \M( | \varphi  | )
\right] = a \sum_{\k\in \mathbb{Z} ^{ d}}\int_{ \T}  | \varphi (\lambda \k + x) | dx =a  \int_{ } | \varphi  | <\infty .
\end{align*}

 Then one can state a result similar to Theorem \ref{thm:general-hu} and Proposition \ref{prop:hyp-expo}.

 \begin{theorem}[  \cite{Lr-HU-finite}, Thm 2.3]
 \label{thm:main-periodic}
  Let $\lambda >1,$  and $ \M_{\lambda   }$ a weakly stationary random measure on $ \T$,   
with $ \mathbf E\left[
 \M_{\lambda  }(\T)
\right] = \lambda ^{ d}$, and let $ \alpha \in \mathbb{R},\maxscale \in [0,1].$ Assume   that for some non-negative $ f\in \CC_{c}^{ \infty }(\mathbb{R}^{ d})  \setminus \{0\}$   there is $ R^{ \circ}\geqslant 1,C_{ f}>0$ such that for $ R\in [ R^{ \circ},\lambda ^{ \maxscale }]$
\begin{align*}
  \textrm{Var}\left(\M_{\lambda   }(f_{ R})\right)\leqslant C_{ f}R^{d -\alpha }.
\end{align*}
Then for $ \gamma >0,\varphi \in \F_{ \gamma }^{ *},$  there is $ C_{ \varphi ,f} >0$ such that for $R\in [ R^{ \circ},\lambda ^{ \maxscale }]$
\begin{align*}
  \textrm{Var}\left(\M_{\lambda   }(\varphi _{ R})\right)\leqslant
  C_{ \varphi,f}V_{ \alpha ,\gamma }(R).
\end{align*}

  \end{theorem}    
The proof is  based on the decomposition in low and high frequencies in  \eqref{eq:PS}. 
 High frequencies  rely on the translation boundedness of $ \S$ on $ \Z_{ \lambda }^{ d}$, which can be proved like in Lemma  \ref{eq:adhikari}.
We give at  Theorem \ref{thm:riesz} an example of  hyperuniformity with 1-dimensional $ s$-Riesz gases in dimension $ 1$, for $ s\in (0,1).$  They are $ (1-s)$-hyperuniform, meaning they satisfy the previous proposition with $ \alpha  = 1-s$, on all the scales, i.e. with $ b = 1.$

What makes it also work is that on the dual group, one can replace the toroidal Fourier transform by the one on $ \mathbb{R}^{ d}$, and exploit its scaling properties:
 \begin{lemma}
 \label{lm:scaling} Let $ \M$ a \LSI weakly stationary random measure on $ \T.$
 For $f \in \B_{ c}^{}(\mathbb{R}^{ d}), R>0$,
\begin{align*}
 \textrm{Var}\left(\M(f_{ R})\right)=  R^{ 2d}  \int_{ \mathbb{Z} _{ \lambda }^{ d}} | \widehat{ f }(R\xi) | ^{ 2}\S(d\xi ) .
\end{align*}
 \end{lemma}

 \begin{proof}   
 First recall  that $ \M(f) = \M( \tilde f)$ where $ \tilde f$ is defined on $ \T$ through
\begin{align*}
 \tilde f(x) = \sum_{\k\in \mathbb{Z} ^{ d}}f(x + \lambda \k).
\end{align*}

 We start from  \eqref{eq:PS}   after a $ R$-rescaling:
\begin{align*}
  \textrm{Var}\left(\M(f_{ R})\right) =  \textrm{Var}\left(\M( \widetilde{ f_{ R}})\right) =   \displaystyle\int_{\mathbb{Z} _{ \lambda }^{ d}}\left|
{\widehat{  \widetilde{  f_{ R}}}}
^{\lambda }(u) \right|^{ 2}
\S(du ).
\end{align*}
 
For $ u\in  \mathbb{Z} _{ \lambda }^{ d}$
\begin{align*}
 \what{ \widetilde{ f_{ R}}  }{\lambda } (u)=& \int_{  \T} \widetilde {f_{ R}}  (x)e^{ -iu\cdot x}dx\\
  = & \sum_{\k\in \mathbb{Z} ^{ d}}\int_{  \T} f_{ R}  (x + \lambda \k)e^{- iu\cdot x}dx
  \\ 
  = & \sum_{\k\in \mathbb{Z} ^{ d}}\int_{  \mathbb  T _{\lambda } ^{ d}} f_{ R}  (x + \lambda \k)e^{ -iu\cdot (x + \lambda \k)}dx\text{ \rm{\color{black} because }}e^{- iu\cdot \lambda \k} = 1
  \\ 
   = &\int_{ \mathbb{R}^{ d}}f _{ R} (y)e^{ -iu\cdot y}dy\\
   = &\widehat{f_{ R}  } (u).
\end{align*}
Therefore, on $ \mathbb{Z} _{ \lambda }^{ d},$   $ \what{  \widetilde{ f_{ R}}}{\lambda } = \widehat{ { f_{ R}}} =R^{ d} \hat f(R\cdot ) $ by standard considerations, which completes the proof.
  
 \end{proof}

\subsection{Hyperuniformity of asymptotically homogeneous particle systems}
\label{sec:hu-finite-asymp}

For a particle system $ \P_{N }\subset \Sigma $  not defined on the torus, a strategy is to consider a restriction $ \tilde \P_{ N } =  \tau _{ x_{N }}\P_{ N}\cap   \T$ for some $ \lambda \ll N^{ 1/d}$ and apply the previous result. Unfortunately, $ \tilde \P_{ N }$ is in general not $ \T$-stationary, even weakly, but in many cases it is asymptotically homogeneous, in the sense that $$ \mathbf E \left[
\tilde \P_{ N }(\tau _{ x}f_{ R})
\right] \approx \displaystyle\int_{}f_{ R}$$ uniformly in $ x$. Therefore, the strategy is to compare $ \tilde \P_{N }$ to its stationarised version $ \tau _{ U_{ \lambda }}^{ \lambda } \tilde \P_{ N}$, at the second order, where $ U_{ \lambda }$ is independent and uniformly distributed in $ \T$. In this case, similarly as in previous results, variance rates for smooth linear statistics can be passed on to irregular linear statistics as in Theorem  \ref{thm:main-periodic}, with some caveats. In particular,  boundary effects kill the rates on large scales, and one cannot obtain in general  hyperuniformity on the full scales $ I_{ N } = [1,N ^{ 1/d}]$, but on some reduced scales $ [1,N ^{ \frac{ 1}{d(d + 2)}}]$. See Section  \ref{sec:gibbs}  for  applications to Coulomb gases and Girko random matrices.

 \section{Non-Euclidean  hyperuniformity}  
  
   \newcommand{\X}{  \mathbb X}  
  
{
Another direction of generalisation is on an abstract metric space $ \X$.  
  This has been explored in the PhD dissertation of M. Bylehn on the hyperbolic space in  \cite{bylehn}, and on more general metric spaces in  \cite{bjorklund2025hyperuniformity}. 
  A natural generalisation of  hyperuniformity for a point process $ \P$ in a metric space $ X$ endowed with a measure $ m$  would be 
\begin{align*}
 \lim_{ r\to \infty } \frac{  \textrm{Var}\left(\P(B_{ X}(0,r))\right)}{ m(B_{ X}(0,r))} =  0.
\end{align*}
This is unfortunately impossible in the hyperbolic space, as the volume of hyperbolic balls and their boundaries have the same exponential growth. Bylehn and Bj\"orklund  are still able to define an analogue of spectral hyperuniformity in some metric spaces endowed with a group structure satisfying the so-called  {\it Gelfand property}, which allows to perform in some sense a polar decomposition of the Fourier transform. In this framework, they define  hyperuniformity as a sub-Poissonian decay of the Bartlett spectral
measure near the endpoint, in analogy to Theorem \ref{thm:general-hu}-(iii). They identify some lattices invariant to group transformations which satisfy this property. Independently perturbed lattices, on the other hand, do not seem to fulfill this property, nor do Determinantal processes with reproducing kernel.}


\if\noauxfile1{

  {\color{red} 
 toc does not exists

}\fi
 
 \chapter{Hyperuniform  point processes}  
 \label{chap:proofs}
 
 In this chapter, we develop the examples from the introduction and give some proofs. We start by defining factorial moment measures, central in the study of determinantal processes. We then discuss  the planar GAF zeros, as it is a striking and isolated example, easy to define formally. Random matrices and DPPs form a core class of  hyperuniform  point processes, but require more context and preparation. Random matrix eigenvalues and non-integrable Gibbs measures are discussed after, but their high mathematical complexity leaves less room for rigourous results, let alone proofs. We conclude with  models presenting an aperiodic order, such as quasicrystals.

 \section{Factorial moment measures}  
 \label{sec:factorial}
 
 We saw that the law of a point process $ \P = \sum_{i}\delta _{ x_{ i}}$ is characterised by the laws $ \P(f)$ for $ f\in\mathscr C_{ c}^{b }(\mathbb{R}^{ d})$. Factorial moment measures give a more analytic way to decompose   $ \P$'s  law in projections of orders $ 1,2, ...$ and characterise it  in the same way that the law of a reasonable random variable is characterised by its moments of every order. Let $ \mu_{ \P}^{(m)}$ the  {\it $ m$-th factorial moment measure} (FMM) of $ \P$, defined through non-negative test functions $ f\in \mathscr B_{ c}((\mathbb{R}^{ d})^{ m})$ applied to $ m$-tuples of distinct $ x_{ i}$'s with
\begin{align*}
 \mathbf E\left[
 \sum_{i_{ 1},\dots ,i_{ m}\text{\rm{ distinct}}}f(x_{ i_{ 1}},\dots ,x_{ i_{ m}}) 
\right]= \int_{}f\mu_{ \P}^{(m)}.
\end{align*}
The number of terms in the sum is determined by the number of Dirac masses in $ \P = \sum_{i}\delta _{ x_{ i}}.$ When $ \mu _{ \P}^{(m)}$ has a density  with respect to $ \Leb[m]$, it is denoted by $ \rho _{ \P}^{(m)}.$ To further characterise this notion, we assume henceforth that the point process is  {\bf simple}, i.e. $ x_{ i}\neq x_{ j}$ for $ i\neq j.$ This  immediately implies that the FMMs vanish on the diagonal 
\begin{align*}
 \mu _{ \P}^{(m)}(\{(x_{ 1},\dots ,x_{ m}): x_{ i} = x_{ j}\text{\rm{ for some }}i\neq j\}) = 0.
\end{align*}
Another characterisation of $ \mu _{ \P}^{(m)}$ is that it vanishes on the diagonal and that over disjoint bounded measurables sets $ A_{ 1},\dots ,A_{ m}\subset \mathbb{R}^{ d}$, 
\begin{align*}
 \mu _{ \P}^{ (m)}(A_{ 1}\times \dots \times A_{ m}) = \mathbf E \left[
\P(A_{ 1})\dots \P(A_{ m})
\right].
\end{align*}
\begin{itemize}
\item For $ k = 1$, one retrieves the intensity $ \mu_{ \P}^{ (1)}(A) = \mathbf E \P(A)$ for $ A\subset \mathbb{R}^{ d}$. If $ \P$ is  stationary, $ \mu_{ \P}^{ (1)}$ is invariant under translations, hence  $\rho _{ \P}^{ (1)} \equiv  \lambda $, with the  {\bf intensity} $ \lambda \geqslant 0$.

\item In view of Section \ref{sec:wsrm}, the first and second order properties of stationary  $ \P$ can be equivalently described by the couple $ (\lambda ,\mu _{ \P} ^{ 2} )$ or by the couple $ (\lambda ,\C)$ (or obviously the couple $ (\lambda ,\S)$): combine \eqref{def:covariance} with  
\begin{align}
\notag
  \textrm{Cov}\left(\P(f),\P(g)\right) = & \mathbf E \left[
\sum_{i,j}f(x_{ i})  \bar g(x_{ j})
\right]-\mathbf E \left(
\P(f)
\right)\mathbf E (\P( \bar g))\\
\notag
   = &\mathbf E \left[
\sum_{i}f(x_{ i})  \bar g(x_{ i})
\right] + \mathbf E \left[
\sum_{i\neq j}f(x_{ i}) \bar g(x_{ j})
\right] -    \int {f}(x)\lambda dx\int{ \bar g(x)}\lambda dx\\
    = & \label{eq:fact-cov}\int_{}f\otimes  \bar g\rho ^{ 2}_{ \P} + \lambda \int_{}f \bar g -\lambda ^{ 2}\int_{}f\int{ \bar g}.
\end{align}
\item The  {\it factorial} in the name might refer to the formula obtained for a simple point process $ \P$     when $ f  = 1_{ B^{ m}}$ for a bounded set $ B$:  \begin{align}
\label{eq:factorial}
 \mu_{ \P}^{(m)}(B^{ m})  = \mathbf E\left[
 \sum_{x_{ 1}\in B\cap \P}\sum_{x_{ 2}\in B\cap \P  \setminus x_{ 1}}\dots { \sum_{x_{ m}\in B\cap \P  \setminus \{x_{ 1},\dots ,x_{ m-1}\}}}1
\right] = \mathbf E \left[
  \P(B)^{ (m)}
\right]
\end{align}where for a number $ x\geqslant 0,x^{ (m)}: = x(x-1)\dots (x-m + 1).$ 
Notice that  the Newton formula
\begin{align*}
{\bf 1}\{ k = 0 \}  = (1-1)^{ k}= \sum_{m = 0}^{ \infty }(-1)^{ m}\frac{ k^{ (m)}}{m!}
\end{align*}  gives the inclusion-exclusion formula  whenever the sum converges absolutely
\begin{align}
\label{eq:inclusion-exclusion}
 \mathbf P (\P(B) = 0) = \sum_{m = 0}^{ \infty }\frac{ (-1)^{ m}}{m!}\mu_{ \P}^{(m)}(B^{ m}).
\end{align}
This can be useful as, by standard results on  {\it random closed sets,} the law of a simple point process $ \P$ can equivalently be characterised by the values of the functional    $B  \mapsto   \mathbf P (\P(B) = 0),B\subset \mathbb{R}^{ d}$, called  {\it capacity functional}  \cite{Mol05}.
\end{itemize}

For a process $ \P$ with a fixed number $ n$ of points $ x_{ 1},\dots ,x_{ n}$, the law of $ \P$ is equivalently described by the joint law  of the random vector $ (x_{ 1},\dots ,x_{ n})$, assuming an independent labelling of the points.
Up to a combinatorial multiplicative constant, this law is proportional to $ \mu  ^{ (n)}_{ \P}$, and we abusively call $ \mu  ^{ (n)}_{ \P}$ (or its density $ \rho ^{ n}_{ \P}$) the  {\it density} of $ \P$ in this case. It entirely determines the law of $ \P$, and lower order FMMs are obtained by projection.

 \begin{lemma}
 \label{lm:projection-FMM}
 Let $ \P_{ n}$ a point process with $ n$ distinct points a.s.. Then  for $ 1\leqslant k<n$, we have the projection
\begin{align*}
 \mu ^{ (k)}_{ \P_{ n}}  = \frac{ 1}{(n-k)!} \mu ^{(n)}_{ \P_{ n}}(\cdot ,\dots ,\cdot , \mathbb{R}^{ d}\times \dots \times \mathbb{R}^{ d}).
\end{align*}
 
 \end{lemma}
 
 \begin{proof} 
 For each test function $ f$ symmetric in $ k$ arguments let 
\begin{align*}
 \tilde f(x_{ 1},\dots ,x_{ n}) = \sum_{\text{\rm{distinct }}(i_{ 1},\dots ,i_{ k})}f(x_{ i_{ 1}},\dots ,x_{ i_{ k}})
\end{align*}
Then
 \begin{align*}
 \mu ^{k}(f) = &\mathbf E \sum_{(x_{ 1},\dots ,x_{ k})^{ \neq }}f(x_{ 1},\dots ,x_{ k})\\
  = &\mathbf E  \tilde f(Z_{ 1},\dots ,Z_{ n})\\
   = &\frac{ 1}{n!}\mathbf E \left[
\sum_{(x_{ 1},\dots ,x_{ n})\neq }f(x_{ 1},\dots ,x_{ n})
\right]\\
   = &\frac{ 1}{n!}\int_{ } \tilde f(x_{ 1},\dots ,x_{ n})\mu ^{ (n)}(dx_{ 1}\dots dx_{ n})\\
    = & \frac{ 1}{n!}\sum_{(i_{ 1},\dots ,i_{ k})\neq }\int_{ } f(x_{ i_{ 1}},\dots ,x_{ i_{ k}})\mu ^{ (n)}(dx_{ 1}\dots dx_{ n})\\
     = & \frac{ 1}{n!} \underbrace{\#\{\text{\rm{ordered } $ k$-tuples}\} }_{\binom{k}{n}\times k!}\times \int_{ }f(x_{ 1},\dots ,x_{ k}) \tilde \mu ^{ (k)}(dx_{ 1}\dots dx_{ k})
\end{align*}
where
\begin{align*}
 \tilde \mu ^{ (k)}(dx_{ 1},\dots ,dx_{ k}) := \int_{ }\mu ^{ (n)}(dx_{ k + 1}\dots dx_{ n})
\end{align*}
using that $\mu ^{ (n)}$ is symmetric, i.e. invariant under permutation of its arguments. Seeing that the prefactor is $ (n-k)!^{ -1}$, it completes the proof.
\end{proof}

 \subsection{Law characterisation and convergence}
 
 See  \cite[Chap. 9]{DVj08b} more a more complete treatment.
 It is classical that the law of a real random variable $ X$   is characterised by its moments $ \mathbf E X^{ m},m\geqslant 1$ if $ \mathbf E \exp(t | X | )<\infty $ for some $ t>0$. Similarly, if some  point process $ \P$ satisfies $ \mathbf E\left[
 \exp(t_{ A} \P( A ) )
\right]<\infty $ for some $ t_{ A}>0$ for all $ A$ bounded, in which case we say that  {\it $ \P$ has some exponential moments}, then the law of the $ \P(A),A$ bounded, and hence the law of $ \P$ (Section  \ref{sec:intro-pp} or \cite[Th. 9.2.XII]{DVj08b}), is characterised by the moments $ \mathbf E \left[
\P(A)^{ m}
\right],m\geqslant 1, A$ bounded. In turn, the $ \mathbf E \left[
\P(A)^{ m}
\right],m\geqslant 1$ can be recovered from the $ \mu_{ \P}^{(m)},m\geqslant 1$ with  \eqref{eq:factorial}. We hence proved the following:
\begin{proposition}
\label{prop:charact-factorial}
The law of a  point process $ \P$ having some exponential moments is characterised by the $ \mu  ^{(m)}_{ \P},m\geqslant 1.$
\end{proposition}

If for instance the factorial moment measures are known to satisfy $ \mu_{ \P}^{(m)}(B^{ m})\leqslant c_{ B}^{ m}$ for some $ c_{ B}<\infty $, it implies finite exponential moments on $ B$: for $ u>0,$
\begin{align}
\label{eq:exp-moments}
 \mathbf E\left[
 (1 + u)^{ \P(B)}
\right] = \sum_{m = 0}^{ \infty }u^{ m}\frac{\mathbf E  \left[
\P(B)^{ (m)}
\right]}{m!}\leqslant  \sum_{m} \frac{| u c_{ B} | ^{ m}}{m!}<\infty ,
\end{align}
hence under such an assumption for all $ B$, the $ \mu_{ \P}^{(m)}$  define at most one distribution.
This will in particular allow to define properly the class of determinantal point processes in Section  \ref{sec:dpp} through their factorial moment measures.\\

Similarly, the convergence between random variables $ X_{ n}\to X$ for $ X$ with some exponential moment is implied by the convergence of the $ m$-th moment $\mathbf E  X^{ m}_{ n}\to \mathbf E X^{ m}$ for each $ m\geqslant 1$. Recall that the convergence between  point processes $ \P_{ n} \xrightarrow[ n\to \infty ]{\text{\rm{Law}}}\P$    is implied by   $ \P_{ n}(A)\xrightarrow[ n\to \infty ]{\text{\rm{Law}}}\P(A)$ for each bounded $ A$  (Section  \ref{sec:intro-pp} or \cite[11.1.VII]{DVj08b}). Hence if for all $ A$ bounded, for all $ m\geqslant 1, \mathbf E \P_{ n}(A)^{ m}\to \mathbf E \P(A)^{ m}$, we have indeed the weak convergence $ \P_{ n}\to \P$. Finally, since $ \mathbf E \left[
\P(A)^{ m}
\right]$ is a linear combination of the $  \mu_{ \P}^{ (k)}(A^{ k}),1\leqslant k\leqslant m$, we have:

\begin{proposition}
\label{prop:cvg-cumulant}
If for some random measures $ \P_{ n},\P$ with $ \P$ having some exponential moments, we have $ \mu _{ \P_{ n}} ^{ (m)}(A^{ m})\to \mu _{ \P}^{(m)}(A^{ m})$,  for each compact  $ A$, then $ \P_{ n}\xrightarrow[ n\to \infty ]{\text{\rm{Law}}}\P.$  
\end{proposition}

 \subsection{Repulsivity and negative dependence} 
  
 Hyperuniform processes are sometimes believed to exhibit local repulsion, probably due to the fact that the most famous examples, DPPs and zeros of random functions, indeed experience a natural local repulsion.
  Mecke's formulas  \cite{baccelli2020random} exactly mean that for $ \P$ a homogeneous Poisson process with intensity $\lambda  = 1$, $ \mu ^{ (m)}_{ \P} = (\Leb)^{ m}$, hence what is considered  {\it repulsion} for a  point process $ \P$ is when the difference  $\mu   ^{ (m)}_{ \P}- (\Leb)^{m}$, measuring somehow  the deviation to neutrality, takes negative values. 
  We say in particular that $ \P$ is completely repulsive if for all $ m\geqslant 1$, $ \mu ^{ (m)}_{ \P}\leqslant (\Leb)^{ m}$ as measures. One often talks about local repulsion when the inequality holds locally for $ m = 2$. It holds much more rarely at large scales, and for all $ m$; the class of DPPs is probably the sole tractable class of useful models having such a property, the GAF zeros do not  \cite{BKPV}. Repulsivity can be seen as negative dependence, in the sense that a positive mass at some location discourages mass in other locations.
   
    The  reduced variance at large scale still induces some negative dependance, as there should necessarily be compensation of large batches of particles. For instance in a large rectangular window $ W$ that can be decomposed in two disjoint congruent rectangles $ W_{ 1},W_{ 2}$,  if there is a large concentration of particles in some half, the concentration in the other half should be below average to ensure the low discrepancy guaranteed by  hyperuniformity. This can be quantified by negative asymptotic correlations
    \begin{align*}
\lim_{ R\to \infty }\text{\rm{Corr}}(\P(RW_{ 1}),\P(RW_{ 2}))<0.
\end{align*}
    {This phenomenon has been first formally studied in  \cite{AdiGhoshLeb} in the discrete setting, and recently in the continuum setting in  \cite{KrishYogesh,jalowyBox,SodinElectric}. }
         
\subsection{Strong and weak decorrelation}

The variance cancellation necessarily implies wide scale compensation,   and hence long range interactions in a sense or another. Still, even if  hyperuniformity is a second order property, this does not necessarily imply long range interactions for second order marginals. Let us illustrate this through two  hyperuniform examples:\begin{itemize}
\item The cloaked lattices of  \cite{cloak}, see Example  \ref{ex:IPL}. The spectral measure $ \S$ is very smooth with possibly fast decay, hence so does the the reduced covariance. Long range interactions still occur through the lattice periodicity, and prevent this model to be mixing.
\item The infinite Ginibre ensemble: we will see that the second order FMM has a very fast decay:
\begin{align*}
 \rho _{ 2}(x,y) \propto \exp(-\|x-y\|^{ 2}/2).
\end{align*}
On the other hand, long range interactions can be characterised in the Coulomb gas representation of the Ginibre process (\eqref{eq:ginibre} or Section  \ref{sec:gibbs}), which does not prevent Ginibre, and many other determinantal point processes, to be mixing (\cite{soshnikov}).
\end{itemize}

\section{Zeros of the planar Gaussian analytic function}
\label{sec:GAF}

We give here some more context about Gaussian fields and Gaussian Analytic Functions (GAFs), based on the excellent reference  \cite{BKPV}. This section is a summary of some results of their Section 2 under the angle of Euclidean Gaussian fields and  point processes.

In general, a Gaussian field $    \mathsf F:\mathbb{R}^{ d} \to \mathbb{R}^{ q}$ is a collection of random vectors $   \mathsf F(x) = (   \mathsf  F_{ 1}(x),\dots ,  \mathsf F_{ q}(x))\in \mathbb{R}^{ q},x\in \mathbb{R}^{ d}$ such that for   $ (x_{ 1},\dots ,x_{ m})\in (\mathbb{R}^{ d})^{ m}$, $ (  {   \mathsf F}(x_{ 1}),\dots ,  {   \mathsf  F}(x_{ m}))$ is a $ (\mathbb{R}^{  q})^{ m}$-valued Gaussian vector. In particular each coordinate field $ \{   \mathsf F_{ i}(x),x\in \mathbb{R}^{ d}\},1\leqslant i\leqslant q$ defines a Gaussian signed measure as in Example \ref{ex:gaussian}. We saw that each finite measure $ \S$ on $ \mathbb{R}^{ d}$ uniquely defines the law of a stationary Gaussian field $  {   \mathsf F}:\mathbb{R}^{ d}\to \mathbb{R}$ with   covariance function $ \C =   \F^{ -1}\S$. For $   {   \mathsf  F}$   a vector-valued field, each of the $ q(q + 1)/2$ pairwise covariance functions $ \C_{ i,j}(x,y)  =  \textrm{Cov}\left(   \mathsf F_{ i}(x),   \mathsf F_{ j}(y)\right),i\leqslant j,$ must be specified to determine the law of $  {   \mathsf  F}$ uniquely in the class of Gaussian fields.

The class of GAFs is a subclass of the class of Gaussian fields from $  \mathbb C $ to $  \mathbb C $ (with the identification $  \mathbb C \equiv \mathbb{R}^{ 2}$) such that a.s. each sample path is  a.s. an analytic function. Hence  one should in principle specify the $ 3$ covariance functions $ \C_{ 1,1},\C_{ 2,2},\C_{ 1,2}$ to characterise the law of the field.
The class of GAFs crucially holds an additional requirement: each finite linear combination $ \sum_{j}a_{ j}  {   \mathsf  F}(z_{ j})$ must be a complex Gaussian variable, which means it should have i.i.d.   (Gaussian) real and imaginary parts. Hence GAFs are not just Gaussian fields which are a.s. analytic. 
  A very convenient and fundamental gain from this requirement is that the law of a centred GAF $   \mathsf F$ is uniquely determined by its  {\it complex covariance}, replacing the $ \C_{ i,j}$ of the real Euclidean representation. 
  
  \begin{proposition}
  Let $   \mathsf F$ a GAF with complex covariance \begin{align*}
 \C(z,w) := \mathbf E \left[
  {   \mathsf F}(z)   \overline{     {   \mathsf  F}(w)}
\right].
\end{align*}
Then any other GAF with the same complex covariance has the same law as $   \mathsf F$.
  \end{proposition}

  \begin{proof}
 This is a consequence of the fact that the distribution of $ \F$ is determined by its finite dimensional distributions $ (\F(z_{ 1}),\dots ,\F(z_{ m})).$The central concept is in fact that of centred {\it complex Gaussian vector} (CGV), which is a random vector of the form $ \mathsf G M   $ where $M$ is a deterministic complex-valued matrix, and $   \mathsf G$ a random vector formed by  i.i.d. $ \N_{  \mathbb C }(0,1)$ variables. It is an easy exercise (\cite[Exercise 2.11]{BKPV}) that the law of $\mathsf G M   $ is characterised by the complex covariance matrix $ \C_{ M}: = \mathbf E [(  \mathsf G M )  (  \mathsf GM)^{ *}] = M M^{ *}$. Hence the result follows from the fact that for $ z_{ 1},\dots ,z_{ m}$, $ (\F(z_{ 1}),\dots ,\F(z_{ m}))$ is a CGV which complex covariance matrix is $ (\C(z_{ i}),\C(z_{ j}))_{ 1\leqslant i,j\leqslant m}$.
  
  \end{proof}

By  \cite[Lemma 2.2.3]{BKPV}, a general recipe to construct GAFs is to use models of the form 
\begin{align*}
   {   \mathsf  F}(z) = \sum_{k = 1}^{ \infty }\mathsf G_{ k}\psi _{ k}(z)
\end{align*}for i.i.d.   $ \N_{  \mathbb C }(0,1)$-distributed $ \mathsf G_{ k}$ and analytic functions $ \psi _{ k}$ such that a.e.
\begin{align*}
 \sum_{k} | \psi _{ k}(z) | ^{ 2}<\infty .
\end{align*}
Such a representation can generally be obtained using the theory of compact operators.
We are   interested here in the planar GAF,  also sometimes called planar GEF (Gaussien Entire Function) defined by 
\begin{align*}
  \mathsf  \GAF(z) = \sum_{k = 1}^{ \infty }  \mathsf G_{ k}\frac{ z^{ k}}{\sqrt{k!}},
\end{align*}
and the analycity of $  \mathsf  \GAF$ yields that its zero set $ \P^{ \Gaf}$ contains a.s. isolated points, it is indeed a point process. 
An easy computation gives the complex  covariance 
\begin{align}
\label{eq:gaf-cov}
 \C(z,w) = \mathbf E \left[
 \GAF (z)   \overline{  \GAF (w)}
\right] = e^{ z  \bar w}.
\end{align}

\begin{proof}[Proof that the $ \P^{ \Gaf}$ is stationary and isotropic]

The   statistical invariances of $ \P^{ \Gaf}$ come from the following conjugation property under shifts and rotations:

\begin{lemma}
\label{lm:conjugate-expo-kernel}
For $ \theta \in \mathbb{R}, z,w,v\in  \mathbb C$
\begin{align*}
 \C(e^{ i\theta }z,e^{ i\theta }w) = &\C(z,w),\\
 \C(z + v,w + v) = &e^{ i\varphi (z,v)}\C(z,w)e^{ -i\varphi (w,v)}
\end{align*}where  $ \varphi(z,v)  =-i z  \bar v/2 + i  \bar zv /2 $ is real valued because $  \bar \varphi  = \varphi $.
\end{lemma}

The proof of this lemma is a straightforward computation. The invariance under rotation of $ \P^{ \Gaf}$ comes from the fact that the rotated field $ z  \mapsto   \GAF (e^{ i\theta }z)$ has the same complex covariance $ \C(z,w)$  as $  \GAF $, hence they have the same law, and their zero sets have the same law as well.

As for invariance under complex translations, $\tau _{ v}  \GAF  $  has  the same covariance as the GAF $    \mathsf  F^{ v} : z\mapsto e^{ \varphi (z,v)}  \GAF (z)$, hence these two GAFs have the same law. Since $   \mathsf F$ and $     \mathsf  F ^{ v}$ have the same zero set, it indeed yields that $ \P^{ \Gaf}$ and $ \tau _{ v}\P^{ \Gaf}$ have the same law, and $ \P^{ \Gaf}$ is indeed stationary. 
\end{proof}

An alternative approach is to show that the field $ z  \mapsto  e^{ - | z | ^{ 2}/2}  | \GAF(z) | $, which has the same zeros as $ \GAF$, is a stationary field  $  \mathbb C  \to    \mathbb C $, but analycity and Gaussianity are lost.

 \begin{remark}[Non-Euclidean GAFs] The theory  of GAFs is very general and takes its full power on domains endowed with a non-euclidean metric, such as the sphere or the hyperbolic disk. It yields the  {\it spherical} and  {\it hyperbolic} families $   \mathsf F_{ L}^{ \text{\rm{Sph}}},   \mathsf F_{ L}^{ 	\text{\rm{Hyp}}}$ of GAFs, each invariant under natural isometries, parametrized by a density parameter $ L >0 $, which is an analogue of the scaling $ z  \mapsto    \GAF (Lz)$ on the complex plane. A further link with DPPs was uncovered by Peres and Vir\`ag  \cite{hypDPP}:  the zero set of  $   \mathsf F^{ \text{\rm{Hyp}}}_{ 1}$ is a DPP on the unit disk, curiously it is the unique DPP among the zeros of all aforementionned GAFs. \end{remark}

 \begin{remark}
 $ \P^{ \Gaf}$ and $ \P^{ \Gin}$ troubly share quite many features, especially on a  local scale analysis. Krishnapur and Vir\`ag  \cite{KrishVirag} give an interesting explanation: $ \P^{ \Gin}$ can be written as the zero set of a GAF with a random complex covariance function (this is stronger than being the zero set of a random analytic function, which is true a.s. for any point process by Weierstrass's theorem).
 \end{remark}

\subsection{Hyperuniformity of $ \P^{ \Gaf}$}  

Let us give a proof of  hyperuniformity that shows how the harmonic nature of analytic functions is at the core of the  hyperuniform behaviour of $ \P^{ \Gaf}$, borrowed from  \cite{ST1}, and originating in the pioneer work of Forrester and Horner  \cite{ForresterHonner}.

\begin{proposition}
\label{lm:gaf-Delta}There is a constant $ C$ such that
for $ f$ in the  class $  \mathscr  C^{ 2}_{ c}(  \mathbb C )$ of $  \mathscr  C^{ 2}$-smooth  functions with compact support, 
\begin{align*}
  \textrm{Var}\left(\P^{ \Gaf}(f)\right)\leqslant C\|\Delta f\|_{ L^{ 2}(  \mathbb C )}^{ 2}.
\end{align*}
Also, $ \P^{ \Gaf}$ is $ 4$-hyperuniform.
\end{proposition}
The bound entails that as $ R\to \infty $,
\begin{align*}
  \textrm{Var}\left(\P^{ \Gaf}(f_{ R})\right)  = O(R^{ -2}) .
\end{align*}
The  hyperuniformity therefore follows directly from Theorem \ref{thm:general-hu}-(ii), and with $ R^{ -2} = R^{ d-4},$ the exponent $ 4$ follows from Proposition \ref{prop:hyp-expo}  (one can also prove that $ 4$ is actually optimal). This is unusual in the sense that most other  known hyperuniform processes   have index at most $ \alpha  = 2$, even perturbed lattices, due in particular to  Lemma  \ref{lm:quadra-spectral}.  {We give in Section \ref{sec:fair-partitions} constructions of other $ 4$-hyperuniform processes built from fair partitions of the space.}

  \begin{proof}

Morally, the proof follows  from the a.s. analycity  of $ \GAF$ and the stationarity of $ \P^{ \Gaf}$. It relies on two claims: a (non-random) analytic function $ F$ on $ D\subset  \mathbb C $ with zero set $ P $ satisfies  
\begin{align}
\label{eq:claim1-gaf}
P(f) =\frac{ 1}{2\pi }\int_{D}\Delta f(z)\ln | F(z) | dz
\end{align}
and the zero set $ \P$ of a GAF $   \mathsf F$ satisfies for $ f\in \mathscr C ^{ 2}_{ c}(  \mathbb C )$
\begin{align}
\label{eq:var-gaf}
 \textrm{Var}\left(\P(f)\right)  = \frac{ 1}{4\pi ^{ 2}} \int {\Delta f(z)\Delta f(w)}  \textrm{Cov}\left(\ln  |  \widetilde{ \mathsf F}(z) |, \ln  |   \widetilde{   \mathsf  F}(w) | \right) dzdw,
\end{align}
where $  \widetilde{ \mathsf F}(z) =  \textrm{Var}\left(   \mathsf F(z)\right)^{ -1/2}   \mathsf F(z).$
For the first claim, the starting point is the harmonicity of the log on the complex plane: $ \Delta \ln( | \cdot | ) = \frac{ 1}{2\pi }\delta _{ 0}$ in the distributional sense, i.e.  for $f \in \mathscr C_{ c}^{ 2}(  \mathbb C )$, 
\begin{align}
\label{eq:harmonic-ln}
f(0) =\frac{ 1}{2\pi } \int_{}\Delta f(z)\ln ( | z | )dz.
\end{align}
Hereafter, fix $ f$ and denote by $ \Lambda $ its support. A non null holomorphic function $ F$ has finitely many zeros $ z_{ i}$ in $ \Lambda $, and the logarithm has an analytic determination on $ \Lambda $, hence $ F$ can be written 
\begin{align*}
F(z) =  e^{ g(z)}  \prod_{i} ( z-z_{ i})dz,z\in \Lambda ,
\end{align*}
for some analytic function $ g$.
 Therefore with $ P = \sum_{i}\delta _{ z_{ i}}\in \N ( \mathbb C ),$ the smooth linear statistic  can be expressed
\begin{align*}
P( f) =  \sum_{i = 1}^{ n}f(z_{ i}) = \frac{ 1}{2\pi }\sum_{i = 1}^{ n}\int_{}\Delta f (z)\ln ( | z_{ i}-z | )dz  =\frac{ 1}{2\pi } \int_{}\Delta f (z) \ln  | F(z) e^{ -g(z)} | dz = \frac{ 1}{2\pi }\int_{}\Delta f (z)\ln  | F(z) | dz
\end{align*}which yields  \eqref{eq:claim1-gaf},
exploiting the fact that the real part of an analytic function $ g$ is harmonic:
\begin{align*}
 \int_{}\Delta f (z)  \mathscr  Rg(z)dz  = 0.
\end{align*}

Let us apply this to a GAF $    \mathsf F$. Denoting by $ \kappa (z)^{ 2} = \mathbf E   |{\sf F}(z) |^{ 2} ,$  the variable $ \tilde {\sf F}(z) = {\sf F}(z) /\kappa (z)$ is  complex Gaussian with constant variance, hence $ \ell: = \mathbf E \ln | \tilde  {\sf F}(z) | $ is constant as well, and 
\begin{align*}
2\pi \mathbf E \P(f) =& \int_{\Lambda }\Delta f(z)\mathbf E \ln | {\sf F}(z) | dz = \int_{\Lambda }\Delta f(z)\mathbf E \ln | \tilde {\sf F}(z) | dz + \int_{}\Delta f(z)\ln(\kappa (z))dz =  {  0}+ \int_{\Lambda }\Delta f(z)\ln(\kappa (z))dz\\
 4\pi ^{ 2} \mathbf E \left[
 | \P(f) | ^{ 2}
\right]
  =  &\int_{\Lambda ^{ 2}}\Delta f(z)\Delta f (w)\mathbf E[ \ln  | {\sf F}(z) | \ln  | {\sf F}(w) |] dzdw\\
 = &\int_{\Lambda ^{ 2}}\Delta f(z)\Delta f(w)\mathbf E[ \ln  | \tilde {\sf F}(z) | \ln |  \tilde {\sf F}(w) | ]dzdw + \int_{}\Delta f(z)  \Delta f(w) \ln(\kappa (z))\ln(\kappa (w)) dzdw + 0 + 0\\
  = &\int_{}\Delta f(z)\Delta f(w) \textrm{Cov}\left( \ln  | \tilde {\sf F}(z) | ,\ln |  \tilde {\sf F}(w) | \right)dzdw + 0 + 0 + \int_{}\Delta f(z)  \Delta f(w) \ln(\kappa (z))\ln(\kappa (w)) dzdw \\
\end{align*}which yields  \eqref{eq:var-gaf}.
Let us apply this to $  \GAF $. The final idea is that $ | \tilde  {   \mathsf F}^{ \text{\rm{Pl}}}(z) |,  |  \tilde  {   \mathsf F}^{ \text{\rm{Pl}}}(w) |  $ have a small correlation when $ z,w$ are far away.  We have the general inequality (\cite[Lemma 3.5.2]{BKPV})   for 
$ \N_{  \mathbb C }(0,1)$ variables  $ Z,Z'$, $  \textrm{Cov}\left(\ln |  Z | ,\ln |   Z' | \right)\leqslant  \frac{ 1}{ 2}    | \mathbf E   Z  \overline{ Z'} | ^{ 2} $, hence
$$ \textrm{Cov}\left( \ln  | \tilde  {   \mathsf F}^{ \text{\rm{Pl}}}(z) |, \ln |  \tilde  {   \mathsf F}^{ \text{\rm{Pl}}}(w) | \right)\leqslant \frac{ 1}{2} | \mathbf E       \tilde {    \mathsf F}^{ \text{\rm{Pl}}}(z)   \overline{  \tilde  {   \mathsf F}^{ \text{\rm{Pl}}}(w)}  |  ^{ 2}.$$  
By  \eqref{eq:gaf-cov}, the right hand side is $   \frac{ 1}{2}e^{ - | z-w | ^{ 2}} =: \sigma (z-w)$.
Hence by Cauchy-Schwarz inequality
\begin{align*}
4\pi ^{ 2} \textrm{Var}\left(\P^{ \Gaf}(f)\right)  = & \int_{}\Delta f(z)\Delta f(w)\sigma (z-w)dzdw  
 \leqslant \|  \Delta f\|_{ L^{ 2}(  \mathbb C )} \|    \Delta f  \ast  \sigma \|_{ L^{ 2} (  \mathbb C )}\leqslant  \| \Delta f\|_{ L^{ 2}(  \mathbb C )}^{ 2}\| \sigma \|_{ L^{ 1}( \mathbb C )}^{ 2}
\end{align*}
which concludes the proof.
%
%
%
%
 
  \end{proof}
\subsection{Other hyperuniform Gaussian nodal measures}

 We conclude the study of planar GAF zeros by showing that they form the unique such stationary set  up to a rescaling, by a rigidity principle discovered by Sodin:
 
\begin{theorem}[\cite{Sodin-Calabi}]
 Let $  {   \mathsf  F,   \mathsf G}$  two GAFs on some domain $ D\subset  \mathbb C $, such that their respective zero sets $ \P^{  {   \mathsf  F}}, \P^{  {   \mathsf  G}}$ have the same intensity: $ \rho ^{ (1)}_{ \P^{  {\bf F}}} = \rho ^{ (1)}_{ \P^{  {\bf G}}}$. Then there exists a deterministic function $ \varphi :D \to  \mathbb C $ not vanishing on $ D$ with $  {   \mathsf F} \equiv\varphi   {   \mathsf G}$ a.s..
 \end{theorem}  
 Hence for $ \lambda >0$, there is a unique set of GAF zeros $ \P$ with $ \rho _{ \P}^{ (1)} = \lambda \Leb$, obtained by properly rescaling $ \P^{ \Gaf}.$ 
To the author's knowledge, no other Gaussian field with stationary hyperuniform zeros has been found,  such a phenomenon is known to be impossible for one-dimensionnal stationary fields  \cite{Lac20}, or for isotropic higher-dimensional stationary fields  \cite{Gass-cancellation}. The complex version of Berry's Gaussian random wave model, experiencing many cancellation phenomena, has hyperfluctuating zeros \cite{NPR}.  The general cancellation phenomena for Gaussian random measures and their chaotic projections has been analysed by Gass  \cite{Gass-cancellation}.

In  \cite{KolianderEtAl}, inspired by the Euclidean GAF and motivated by signal processing, the authors exhibit   charged particles located on  the zero set of the  {\it  Gaussian Weyl-Heisenberg field}  $   {   \mathsf  F}$, i.e. a random measure of the form 
\begin{align*}
 \M = \sum_{z:  {   \mathsf F}(z) = 0}\delta _{ z}\kappa _{ z},
\end{align*}
where $ \kappa _{ z}$ is the sign of the Jacobian determinant $ \det D  {   \mathsf F}(z)$. Such measures (with or without charges) are stationary, and the authors give their intensity and prove that under some assumptions, the net charge experiences  {\it screening}, which implies a  hyperuniform behaviour of class I (Section \ref{sec:hyp-expo}), see their Theorem 1.14.

 \subsection{An invariant deterministic dynamics for GAF zeros}
 
 It is customary and convenient for analysis to identify a dynamics that leaves invariant in law a given probabilistic model. There is a long chain of results deriving dynamics that leaves invariant the distribution of random matrix eigenvalues, random polynomial zeros, or interacting particles, with a seminal work by Dyson, see for instance recent works of Osada  \cite{Osada-isde} and references therein. These dynamics always include an independent Gaussian drift for each particle, compensated by force fields generated by the rest of the system.
 
Recent findings from Hall et. al  \cite{Jalowy-GAF} indicate that the GAF zeros actually admit such dynamics, with the peculiarity that no  additional randomisation is needed. Let $ \P^{ \Gaf}_{ t} = \{z_{ i}^{ (t)};i\geqslant 0\}$ be a random evolution where the initial law is $ \P^{ \Gaf}_{ 0} = \P^{ \Gaf}$, and each particle follows the dynamics
\begin{align*}
 \frac{ d}{dt} z_{ i}^{ (t)} = \sum_{j\neq i}\frac{ 1}{z_{ i}^{ (t)}-z_{ j}^{ (t)}}.
\end{align*}
The right-hand side is actually the Coulomb force field (or electro-magnetic field in dimension $ 2$) generated by other particles. Regarding hyperuniformity of 2D systems, this field is connected to Coulomb energy and optimal transport  (see Section \ref{sec:finite-Coulomb}) and to fair partitions of the space (see Section \ref{sec:coulomb-alloc}).
 Then \cite[Prop. 2.12]{Jalowy-GAF} show that these dynamics are well defined and that $ \P^{ \Gaf}_{ 0} \smallequlaw   \P^{ \Gaf}_{ t}$ for $ t>0.$
  
 \section{Determinantal point processes and random matrices} 
 \label{sec:dpp}
 
 We provide here partial proofs of Theorems    \ref{thm:sine-beta} and   \ref{thm:ginibre}. It is the occasion to make more explicit the fundamental link between random matrix models and particle systems such as log gases and OCPs, through astute changes of variables.
Determinantal processes in the continuous space arise mainly as eigenvalues of random matrix models. We are interested here  in stationary examples (GUE, Ginibre) where the key point is to prove that they are projector DPPs.

\subsection{Change of variables}
\label{sec:change-of-variables}

It is claimed in the introduction  that the $ \beta $-ensembles for $ \beta  = 1,2$ in dimension $ d = 1$ and the Ginibre ensemble can be written out as the set of eigenvalues of resp. the GOE, GUE, and Ginibre random matrix models. As an elementary and prototypical example, we provide here an argument for the GOE ensemble, similar reasonings are possible for other matrix models, but more involved. We refer to the monographs  \cite{GuionnetBook,Mehta,BKPV} for other models.\\

\begin{proof}[Proof that the eigenvalues   of $ \M^{ 1,(n)}$ have density \eqref{eq:beta-ensemble} for $ \beta  = 1$] 
Since $ \M^{ 1,(n)}$ is a symmetric matrix, there is a.s. a random  matrix $ \Q_{ n} $  in the orthogonal space $ \mathscr O_{ n}$ and a random matrix $ \D_{ n} =\text{\rm{diag}} (\Lambda _{ 1},\dots ,\Lambda _n)$ in the space $  \mathscr  D _{n}$ of diagonal matrices   such that $ \M^{ 1,(n)} = \Q_{ n}\D_n\Q_{ n}^{ T}$. $  \mathscr  D_{ n}$ is assimilated to $ \mathbb{R}^{n}$  with the Euclidean metric and endowed with $ \Leb[n].$

 The fact that the eigenvalues of  $ \M^{ 1,(n)}$ are a.s. distinct  is a basic result of random matrices. Briefly, it is a consequence of the fundamental theorem of symmetric polynomials: the discriminant $ \Delta (\Lambda _{ 1},\dots ,\Lambda _{ n}) =  \prod_{i\neq j}(\Lambda _{ i}-\Lambda _{ j})$ is symmetric in the $ \Lambda _{ i}'s$, hence it can be written as a deterministic polynomial in the coefficients of the random  polynomial $ \lambda \mapsto \det(\M^{ 1,(n)}_{ n}-\lambda I_{ n})$, and the  latter coefficients are themselves polynomials in the entries of the matrix. The conclusion comes from the fact that the zero set of a polynomial in several variables is Lebesgue-negligible, see for instance  \cite{Mehta,GuionnetBook} or the brief lecture notes  \cite{Lr-DPPs} for details.  
 Hence $ \M ^{ 1,(n)}$ is .a.s
 in  $  \mathcal S_{ n}^{ *} $ the set of real symmetric matrices with distinct eigenvalues, and  $ \D_n $  is a.s. in $  \mathcal D_{ n}^{ *}$  the class of diagonal  matrices with distinct eigenvalues.
 
Our interest is the exact density of $ \D_n $. 
There is no unicity for the law of $ \Q_{ n}$. Let us first prove that it can be chosen to be $ \sigma $,  the unique probability measure on $  \mathscr O_{ n}$ that is invariant under the action of the multiplication by any  $ Q\in  \mathscr O_{n},$ or  {\it Haar measure}.
 Let $ \Q$ with law $ \sigma  $  and independent of $(\M^{ 1,(n)}, \Q_{ n},\D_n)$. 
The expression \eqref{eq:gue-particle} yields that the law of $\M^{ 1,(n)}$ is invariant under the conjugation action $
 M  \mapsto  Q MQ^{ T} $
for $ Q\in  \mathscr O_{ n}$, hence
 $ \Q\Q_{ n}$ is another possible random orthogonal matrix in the decomposition of $ \M^{1,(n)}$. Since   $  \Q \Q_{n}$ also has law $ \sigma $ and is independent of $ \D_{ n}$, it proves the claim, and we indeed choose $ \Q_{ n}$ with law $ \sigma $ independent of $ \D_{ n}$. Hence the law of $ (\Q_{ n},\D_{ n}) $ is $  \sigma \times \mu _{ 1}$ for some measure $ \mu _{ 1}$ on $  \mathcal D_{ n}^{ *}$ that we seek to explicit.

The idea of the proof is to compute the Jacobian of the mapping $(Q,D)\mapsto M  = QDQ^{ T}$ for  $ (Q, D)\in  \mathscr  O_{n}\times  \mathscr  D_{n}^{*}.$ It is tricky to directly perform a change of variables on the non-Euclidean manifold $  \mathscr  O_{ n}$, we therefore locally linearise it first
with the space of antisymmetric matrices $  \mathscr  A_{ n}$, assimilated to $ \mathbb{R}^{n(n-1)/2}$ and endowed with $ \Leb[n(n-1)/2]$. Recall that the exponential function induces a bijective mapping from some neighbourhood $ U$ of $ 0$ in $  \mathscr  A_{ n}$ to $ \exp(U)$ in $  \mathscr  O_{ n}$. Let $ \nu $ the associated pullback measure, i.e. for a test function $ \psi $ supported by $\exp(U)$
\begin{align}
\label{eq:nu-sigma}
\mathbf E \psi (\Q) = \int_{}\psi  (Q)\sigma (dQ) = \int_{U}\psi  (\exp(A))\nu (dA).
\end{align}
 We introduce the  $ \CC^{ \infty }$ mapping $ \Gamma :  \mathscr  A_{ n}\times  \mathscr  D_{ n}^{ *}\to  \mathscr  S_{ n}^{ *}$
\begin{align*}
 \Gamma (A,D): = \exp (A)D\exp (A)^{ T}.
\end{align*}
Let us compute the absolute  Jacobian determinant $ J_{ \Gamma }(\cdot,\cdot  )$ in $ (0,D_{ 0})$.

\begin{lemma}
\label{lm:jacobian}
For $ D_{0}= \text{\rm{diag}}(\lambda _{1},\dots ,\lambda _{n})\in  \mathscr  D_{n}^{*}$, the  Jacobian matrix of $ \Gamma $ has absolute determinant in the Vandermonde form $$  J_{\Gamma }(0, D_{0})=c_{n}\prod_{i<j} | \lambda _{i}-\lambda _{j} |\text{\rm{ for some $ c_{n}>0.$}}$$
\end{lemma} This lemma is proved later. In particular, the determinant does not vanish.
By the inverse function theorem, $ \Gamma $ is a $ \CC^{ \infty }$-diffeomorphism on a neighbourhood $ V\times U_{ D_{ 0}}$ of $ (0,D_{ 0})$ in $   \mathscr  A_{ n}\times \mathscr  D_{ n}^{ *}$, and for $M\in \Gamma (V\times U_{ D_{ 0}})$, there is a unique $ (A_{ M},D_{ M})\in V\times U_{ D_{ 0}}$ satisfying $ M = \Gamma (A_{ M},D_{ M})$ (there is no such unicity on  $  \mathscr  S_{ n}^{ *}$). Up to reducing $ V$, we assume $ A_{ M}\in U.$
 We can therefore perform a change of variables, for $ \varphi $ supported by $\Gamma ( V\times U_{ D_{ 0}}),$ \begin{align}
 \label{eq:CdV-Sn}
 \int_{\mathscr  S_{ n}^{ *}}\varphi (M) dM = \int_{V\times U_{ D_{ 0}}}\varphi (\Gamma  (A,D))J_{\Gamma }(A,D)d\Leb[n(n-1)/2](A)d\Leb[n](D).
\end{align}

On the other hand, by  \eqref{eq:gue-particle} and  \eqref{eq:nu-sigma},
\begin{align*}
 \mathbf E \left[
\varphi (\M^{ 1,(n)})
\right] \propto& \int_{  \mathscr  S_{ n}^{ *}}\varphi (M)e^{ -\text{\rm{Tr}}(MM^{ T})/4}dM \\
 =& \int_{}e^{ - \text{\rm{Tr}}(DD^{ T})/4}\int_{}\varphi (QDQ^{ T})\sigma (dQ)d\mu _{ 1}(D) \\=& \int_{}\varphi ( \Gamma (A,D))e^{ - \text{\rm{Tr}}(DD^{ T})/4}d\nu (A)d\mu _{ 1}(D).
\end{align*}
We can therefore identify  with  \eqref{eq:CdV-Sn}
\begin{align*}
 e^{ - \text{\rm{Tr}}(DD^{ T})/4}J_{ \Gamma }(A,D)d\Leb[n(n-1)/2](A)d\Leb[n](D) \propto d\nu (A) e^{ - \text{\rm{Tr}}(DD^{ T})/4}d\mu _{ 1}(D).
\end{align*}
It yields that
 $ \mu _{ 1}$ has a density $ f$ around $ D_{ 0}$. With Lemma  \ref{lm:jacobian}, we have $ f (D_{0})\propto \prod_{i<j} | \lambda _{i}-\lambda _{j} | $, and this Jacobian form is valid for any $ D \in U_{ D_{ 0}}$.  For a test function $ \tilde \varphi $ on $ U_{D_{0}},$ and $ \varphi (M):= \tilde \varphi (D_{ M}), M\in \Gamma (V\times U_{ D_{ 0}})$,
\begin{align*}
\mathbf E\left[
 \tilde \varphi ( \D_{n})
\right]=&\mathbf E\left[
 \varphi (\M^{1,(n)})
\right]\\
\propto &\int_{} \tilde \varphi (D)e^{- \text{\rm{Tr}}(DD^{T})}d\nu (A)  d\mu _{1}(D)\\
\propto& \int_{} \tilde \varphi (\lambda _{1},\dots ,\lambda _{n})e^{-\sum_{i}\lambda _{i}^{2}/4}  \prod_{i<j} | \lambda _{ i}-\lambda _{ j} | d\lambda _{1}\dots d\lambda _{n},
\end{align*}
which concludes the proof.

Let us finally  prove Lemma  \ref{lm:jacobian}. 
We see $ \Gamma $ as a function from $ \mathbb{R}^{n(n-1)/2}\times \mathbb{R}^{n}$ to $ \mathbb{R}^{n(n+1)/2}$, and denote its components $ \Gamma _{ i,j},1\leqslant i\leqslant j\leqslant n$.
Let $ \delta = \text{\rm{diag}}(\lambda _{1},\dots ,\lambda _{n})\in  \mathscr  D_{n}$, $ H=(H_{i,j})_{1\leqslant i,j\leqslant n}\in U.$
\begin{align*}
 \Gamma ( H,D_{ 0} + \delta ) = &(I+H+o(H))(D_{0}+\delta )(I-H+o(H))\\
 =&\Gamma (0,D_{0})+\delta +HD_{0}-D_{0}H+o((H,\delta )).
 \end{align*}
 Note that $ HD_{0}-D_{0}H=(H_{i,j}(\lambda _{j}-\lambda _{i}))_{1\leqslant i,j\leqslant n}$ vanishes on the diagonal, and at the opposite $ \delta $ is supported by the diagonal.
For $ i<j$, we can read the partial derivatives on the lower diagonal
\begin{align*}
\frac{\partial \Gamma _{ i,j}}{\partial H_{k,l}}(0,D_{0})= & \lambda _i-\lambda _{j}\text{\rm{ iff }}(i,j)=(k,l)\text{\rm{, for }}k<l\\
\frac{\partial \Gamma _{ i,j}}{\partial \lambda _{k}}(0,D_{0})= & 0,1\leqslant k\leqslant n,
\end{align*}
and for $ 1\leqslant i\leqslant n$
\begin{align*}
\frac{\partial \Gamma_{i,i}}{\partial \lambda _{ k}} (0,D_{0})=  &\delta _{ k,i}\\
\frac{\partial \Gamma _{i,i}}{\partial H_{ k,l}} (0,D_{0})= &0, k<l.
\end{align*}
Seeing the Jacobian matrix $ \nabla \Gamma $ as blocks of dimension $ n(n-1)/2$ or $ n$, the $ n\times n$ block gives determinant $ 1$, the $ n(n-1)/2$ block is diagonal and gives $  \prod_{i<j}(\lambda _{ i}-\lambda _{ j})$, and the other blocks vanish; this gives the desired expression.
 
    \end{proof}
\begin{longversion}{}{
Let us indicate how these work for the two other models of interest:}

  {\bf GUE Case} 
 Differences: $ X = U(\Lambda  + T)U^{ *}$ (Schur) where\begin{itemize}
 \item $ T$ is strictly upper triangular
\item $ U\in  \mathcal U_{ n} $ is complex unitary.

This time $ X$ is invariant under the application of $ U_{ 0}\in  \mathcal U_{ n}$, hence as before we can choose $ \Lambda ,U$ independent?
\begin{align*}
 (U + dU)(U + dU)^{ *} =U +  dU + dU^{ *} = 0
\end{align*}hence $ dU$ is skew Hermitian, i.e. $ dU = O_{ \mathbb{R}} + iO_{  \mathbb C }$ with $ O_{ \mathbb{R}}^{ T} =  O_{ \mathbb{R}},O_{  \mathbb C }^{ T} = -O_{  \mathbb C }$.
\item Even if eigenvalues are distinct and ordered, the decomposition is unchanged by the application $ V  \mapsto  \theta V,T  \mapsto  \theta ^{ *}T\theta $ where $ \theta  = \text{\rm{diag}}(z_{ 1},\dots ,z_{ n})$ where the $ z_{ i}$ have modulus $ 1$. To have uniqueness, one can require $ V_{ ii}\geqslant 0.$
\end{itemize}

 {\bf Ginibre case:} non-Hermitian matrix with density 
\begin{align*}
 \exp(-Tr(XX^{ *})).
\end{align*}
 Remark that $ \exp(-Tr(XX^{ *}))$ cannot be expressed solely in terms of eigenvalues of $ X$. 
 
 The proof is complicated, the simplest way seems to be that of Dyson [from GAF] to use the Schur representation $ X = V(Z + T)V^{ *}$ where $ V$ is unitary, $ Z$ is triangular, $ T$ is strictly upper triangular. To have unicity of the representation, we must impose $ V_{ ii}\geqslant 0.$ The law of $ X$ is again invariant under $ X  \mapsto  VXV^{ *}$ for any unitary $ V$. As before, it means $ V$ is independent from $ X,Z,T$ in such a decomposition, and it is enough to investigate the case $ V = { n}$; i.e. we investigate  the Jacobian $ J(Z,T)$ satisfying 
\begin{align*}
 \mathbf E \varphi (eig(X)) = \int_{}\varphi (eig(X))e^{ -Tr(XX^{ *})}dX = \int_{}\varphi (Z)e^{ -Tr((Z + T)(Z + T)^{ *})}J(Z,T)dZdT.
\end{align*}
We have then from [GAF] with $ V = { n}$
\begin{align*}
 dX = dV(Z + T)-(Z + T)dV + dZ + dT,
\end{align*}
as before $ dV^{ *} = -   \widebar { dV}$ is skew-Hermitian. 
 
 Instead, [Mehta] does writes $ X = UTU^{ *}$ for unitary $ U$ and upper triangular $ T$, again we take $ U = { n}$, and it yields 
\begin{align*}
 dX = dT + i(dHT-TdH)
\end{align*}
where $ dH = -iU^{ *}dU.$ He says he can impose $ n$ constraints 
\begin{align*}
 dH_{ ii} = 0,
\end{align*}
etc... not necessarily clearer.

\end{longversion}

 \subsection{Determinantal processes}
 \label{sec:dpp}  
 
 In order to prove results for the GUE and Ginibre process, we must derive some basic definitions and concepts related to the theory of determinantal processes. Even though instances of determinantal processes occur throughout probability theory and statistical physics, the general concept in the continuous space was introduced by Macchi  \cite{Mac75}, and then Soshnikov  \cite{soshnikov}, see references in  \cite{BKPV}. See also the lecture notes  \cite{Lr-DPPs} for more details, in particular for the Ginibre and GUE ensembles.
 
 \begin{definition}
Let $ K: \mathbb{R}^{ d}\times \mathbb{R}^{ d} \to  \mathbb C $ measurable. A  point process $ \P$ on $ \mathbb{R}^{ d}$ is a DPP with kernel $ K$ if it admits factorial moments densities   of the form, for $ m\geqslant 1,$
\begin{align}
\label{eq:def-dpp}
 \rho _{ K}^{(m)}(x_{ 1},\dots ,x_{ m}): = \det((K(x_{ i},x_{ j}))_{ 1\leqslant i,j\leqslant m}), x_{ 1},\dots ,x_{ m}\in \mathbb{R}^{ d}.
\end{align}
 \end{definition}  
The first example, sometimes considered degenerate, is the homogeneous Poisson process with intensity $ \lambda >0,$ for which $ K(x,y) = \lambda {\bf 1}\{ x = y \}.$
More generally, $ \rho _K^{ (1)}(x) = K(x,x)\in \mathbb{R}_{  + }$ on $ \mathbb{R}^{ d}$.  For existence and unicity questions, it is enough to perform a local analysis, in the sense that  $ \P$ is a DPP with kernel $ K$  if and only if for each bounded $ B,$ $ \P1_{ B}$ is a DPP with kernel $ K1_{ B\times B}$, and unicity and existence of $ \P$ need   only be solved  on such $ B$. 
Implicitly in the definition, $ \P$ should have local moments of every order, and we shall henceforth assume that $ K$ is locally square integrable as it will be fruitful to consider the associated operator in $ L^{ 2}(B)$
\begin{align*}
   \mathsf L_{ K}f(x): = \int_{}f(y)K(x,y)dy.
\end{align*}
 Remark that by  \eqref{eq:def-dpp}, each minor of $ K$ has a positive determinant, hence each submatrix  $ (K(x_{ i},x_{ j}))$ is necessarily semi-definite positive. Therefore, Hadamard's inequality yields for compact measurable $ B\subset \mathbb{R}^{ d}$
 $$  | \rho _{ K}^{(m)}(B^{ n}) | \leqslant \int_{B^{ n}}  \prod_{i}K(x_{ i},x_{ i})dx_{ 1}\dots dx_{ m} = \left(
\int_{B}K(x,x)dx
\right)^{ m}\leqslant \Leb(B)^{ m/2} \left(
\int_{B}K(x,x)^{ 2}
\right)^{ m/2}<\infty .$$
 By  \eqref{eq:exp-moments} and Proposition  \ref{prop:charact-factorial}, such a DPP has exponential moments and is uniquely defined by equations  \eqref{eq:def-dpp}. 
This is not an obligation, but in general, to have tractable existence and unicity results, one requires $ K$ to be Hermitian, i.e. for $ x,y\in \mathbb{R}^{ d}$, $ K(y,x) =  \bar K(x,y)$.


A particularly important class  is that of  {\it canonical kernels}, defined to be of the form 
\begin{align}
\label{eq:canonical-K}
K(x,y) =  \sum_{k\geqslant 1}a_{ k}\varphi _{ k}(x)  \bar \varphi _{ k}(y)
\end{align}
 where the $ \varphi _{ k},k\geqslant 1$ form an orthonormal family. The semi-definite positiveness of $ K$ implies that $ a_{ k}\in \mathbb{R}_{  + }$, and it is also necessary that $ a_{ k}\leqslant 1$, see  \cite{Mac75} or  \cite[Th. 4.5.5]{BKPV}.

 \begin{longversion}{} [proof - see  \cite[Th. 4.5.5]{BKPV} - requires Th.4.5.3]
 
\begin{proof}Let us reproduce for completeness the proof of ... Let $ a_{ n}$ the largest eigenvalue, assume $ a_{ n}>1$.
Let $ \P'$ the thinned version where each point is deleted with probability $ 1-1/a_{ n}$, in particular $ \mathbf P (\P' = 0)>0$. Then we can show that $ \P'$  is a DPP with kernel 
\begin{align*}
 K'(x,y) = \varphi _{ n}(x)  \bar \varphi _{ n}(y) + \sum_{k<n}\frac{ a_{ k}}{a_{ n}}\varphi _{ k}(x)  \bar \varphi _{ n}(y)
\end{align*}

\end{proof} 
   It is also possible to see arising in the limit kernels under the integral form 
\begin{align*}
 K(x,y) = \int_{}a_{ x}\varphi _{ x}(y)  \bar \varphi _{ x}(y)dxdy
\end{align*}
with $ \varphi _{ x}(y) = e^{ ix\cdot y}$ and some measurable integrable function $ x\to a_{ x}$. This is for instance the case for the $ \text{\rm{Sine}}$ kernel 
\begin{align*}
 K_{ \text{\rm{Sine}}}(x,y) = \frac{ sin(x-y)}{x-y} = \int_{[-1,1]}e^{ i(x-y)t}dt
\end{align*}that will emerge in the asymptotic study of the GUE process. As whos the Ginibre process, not all asymptotic  kernels need to be under this integral form.

 The positivity of the intensity $ \rho _{ \P}^{ (1)}$ yields that $ \sum_{k}a_{ k} | \varphi _{ k}(x) | ^{ 2}\geqslant 0$, hence each $ a_{ k}$ must be non-negative??

and the semi-definite positiveness of covariances, i.e. $ \rho _{ \P}^{(2)}\geqslant 0$, implies 
\begin{align*}
  \left|\begin{array}{cc}K(x,x) & K(x,y) \\K(y,x) & K(y,y)\end{array}\right|\geqslant 0.
\end{align*}
Integrating out only one variable yields 
\begin{align*}
 \sum_{k,l}a_{ k}a_{ l} | \varphi _{ k}(x) | ^{ 2}(1-\delta _{ k,l}),
\end{align*}
hence all $ a_{ k}$ are of the same sign, that is non-negative. \end{longversion}
 
 Those are typically the form of kernels for finite point processes $ \P_{ n}$ coming from finite matrix models. 
 When all (non-zero) $ a_{ k}$'s equal $ 1$, the kernel is easily seen to enjoy the  {\it reproducing} property:
 \begin{definition}
 Say that $ K$ is reproducing if for  $ x,y\in  \mathbb{R}^{ d},$
\begin{align*}
 K(x,y) = \int_{\mathbb{R}^{ d}}K(x,z)K(z,y)dz.
\end{align*}
A canonical kernel with all $ a_{ k}$'s equal to $ 1$ is a  {\it projection canonical kernel}.
 \end{definition}  
 This property is more conceptually seen as a projection property in the $ L^{ 2}$ space. If  \eqref{eq:canonical-K} is satisfied,  for $ f\in L^{ 2}(\mathbb{R}^{ d})$ with compact support, $   \mathsf L_{ K}f$ is the projection of $ f$ onto the space spanned by the $ \varphi _{ k}$. For a general reproducing kernel that induces an operator on $ L^{ 2}(  \mathbb R^{ d})$
 \begin{align*}
   \mathsf L_{ K}(   \mathsf L_{ K}f)(x) = \int_{}K(x,z)(\int_{}f(y)K(z,y)dy)dz = \int_{}f(y)\int_{}K(x,z)K(z,y)dz = \int_{}f(y)K(x,y)dy =   \mathsf L_{ K}f(x).
\end{align*}
 As a counterexample, the Poisson kernel $ K(x,y) = {\bf 1}\{ x = y \}$ neither satisfies  \eqref{eq:canonical-K} nor is reproducing.

   The fact that the defining DPP property \eqref{eq:def-dpp}  holds for $ \rho _{ \P}^{(m)}$ does not imply automatically that it holds for $ \rho _{ \P}^{ (k)},k\leqslant m$, except for reproducing kernels.

  {
 \begin{proposition}[Dyson identity  \cite{dyson}]
 \label{prop:reprod-kernels}
 For a  reproducing kernel $ K$ on $ \mathbb{R}^{ d}\times \mathbb{R}^{ d}$ such that 
\begin{align*}
 I = \displaystyle\int_{\mathbb{R}^{ d}}K(x,x)dx<\infty ,
\end{align*}
then for $n\in \mathbb{N}, x_{ 1},\dots ,x_{ n-1}\in \mathbb{R}^{ d},$
\begin{align*}
 \displaystyle\int_{}\rho _{ K}^{ (n)} (x_{ 1},\dots, x_{ n})  dx_{ n}= (I-(n-1))\rho _{ K}^{ (n-1)}(x_{ 1},\dots ,x_{ n-1}).
\end{align*}
 \end{proposition}

\begin{longversion}

 \begin{exercise}
In fact, for a canonical kernel, one can represent the number of points as a sum of independent Bernoulli variables with parameters $ a_{ k}$ (see  \cite[Theorem 4.5.3]{BKPV}).

\end{exercise} 
\end{longversion}
In this case, the determinantal property can therefore by passed on to lower order FMMs by induction.
 Hence the reproducing property saves us a lot of effort in the proof that a point process is determinantal, as we will see with the GUE and Ginibre ensembles.

\begin{longversion}

  \begin{exercise}
 \begin{itemize}
\item Let $ K(x,y) = \hat  \varphi (x-y)$ where  $ \varphi  = 1_{ A}$ for some bounded symmetric $ A$. Prove that $ \P_{ K}$ is reproducing.
 \item Prove that the sine kernel $ K(x,y) = \frac{ \sin(\pi( x-y))}{\pi( x-y)}$ is reproducing. This is the kernel of the limit GUE process.
\end{itemize}

 \end{exercise} 
 
\end{longversion}

\begin{proof}[Proof of Proposition  \ref{prop:reprod-kernels}]
 
  We must come back to the representation of the determinant with permutations:
\begin{align*}
 \det(M) = \sum_{\sigma \in \Sigma _{ n}}\varepsilon (\sigma )  \prod_{i = 1}^{ n}M_{ i,\sigma (i)}.
\end{align*}

We decompose the set of permutations $ \Sigma _{ n}$ in $ \Sigma _{ n}^{ 1}$ for which $ \sigma (n) =n$ and $ \Sigma _{ n}^{ *}$ the 
complement. For $ \sigma \in \Sigma _{ n}^{ 1}$, let $ \tilde \sigma \in \Sigma _{ n-1}$ the restriction of $ \sigma $ to $ \llbracket 1,n-1 \rrbracket.$ For $ \sigma \in \Sigma _{ n}^{ *}$, call $ c_{ n}(\sigma )$ the cycle containing $ n$, and $ \tilde \sigma \in \Sigma _{ n-1}$ bypassing $ n$ (i.e. $ \tilde \sigma (\sigma ^{ -1}(n)): = \sigma (n)$, others values do not change). Denote by $ \varepsilon (\sigma )$ the signature of a permutation $ \sigma $.  
 Then for  fixed $ x_{ 1},\dots ,x_{ n-1}\in \mathbb{R}^{ d},$ by the projection property of factorial moment measures for a system with a.s. $ n$ points:
\begin{align*}
\int_{\mathbb{R}^{ d}}\rho _{ K}^{ (n)}&(x_{ 1},\dots ,dx_n)dx_n =   \int_{\mathbb{R}^{ d}}\det((K(x_{ i},x_{ j}))_{ i,j\leqslant n})dx_n \\
 =& \sum_{\sigma \in \Sigma _{ n}^{ 1}}\varepsilon (\sigma )  \prod_{i\leqslant n-1}K(x_{ i},x_{ \sigma (i)})\int_{}K(dx_n,dx_n)dx_n
  + \sum_{\sigma \in \Sigma _{ n}^{ *}}\varepsilon (\sigma )    \prod_{i\notin c_{ n}(\sigma )}K(x_{ i},x_{ \sigma (i)}) \int_{}\prod_{i\in c_{ n}(\sigma )}K(x_{ i},x_{ \sigma (i)})dx_n\\
  = &I\sum_{\sigma \in \Sigma _{ n}^{ 1}}\varepsilon ( \tilde \sigma )  \prod_{i\leqslant n-1}K(x_{ i},x_{ \tilde \sigma (i)}) \\ &\hspace{1cm}  
  + \sum_{\sigma \in \Sigma _{ n}^{ *}}\varepsilon (\sigma )\prod_{i\notin c_{ n}(\sigma )}K(x_{ i},x_{ \sigma (i)})  \prod_{i\in c_{ n}(\sigma )  \setminus \{n,\sigma ^{ -1}(n)\}}K(x_{ i},x_{ \sigma (i)})\times \int_{}K(x_{ \sigma ^{ -1}(n)},dx_n)K(dx_n,x_{ \sigma (n)})dx_n\\
  = &I\sum_{ \tilde \sigma \in \Sigma _{ n- 1}}\varepsilon ( \tilde \sigma )  \prod_{i\leqslant n-1}K(x_{ i},x_{ \tilde \sigma (i)})  \\
  &\hspace{2cm}  + \sum_{\sigma \in \Sigma _{ n}^{ *}}\underbrace{\varepsilon (\sigma )}_{-\varepsilon ( \tilde \sigma )}\prod_{i\notin c_{ n}(\sigma )}K(x_{ i},x_{ \tilde \sigma (i)}) \underbrace{ \prod_{i\in c_{ n}(\sigma )  \setminus \{n,\sigma ^{ -1}(n)\}}K(x_{ i},x_{ \sigma (i)})\times K(x_{ \sigma ^{ -1}(n)},x_{ \sigma (n)})}_{ =   \prod_{i\in c_{ n}(\sigma )  \setminus \{n\}}K( x_{i }, x_{  \tilde \sigma (i)})}\\
   = &I \det((K(x_{ i},x_{ j}))_{ i,j\leqslant n-1})- \sum_{  \sigma' \in \Sigma _{ n-1}}\#\{\sigma\in \Sigma _{ n}^{ *} : \tilde \sigma  = \sigma' \}\varepsilon (\sigma ')  \prod_{i\leqslant n-1}K(x_{ i},x_{ \sigma '(i)}).
\end{align*}
To conclude, notice that for each $   \sigma' $, there are $ n-1$ ways to choose where to insert index $ n$ in permutation $   \sigma' $ to obtain $ \tilde \sigma \in \Sigma _{ n}^{ *}$, in particular it does not depend on $ \sigma '.$ It gives the conclusion that $$ \displaystyle\int_{}\rho _{ K}^{ (n)}(x_{ 1},\dots ,x_{ n})dx_{ n}= (I-(n - 1)) \rho _{ K}^{ n-1}(x_{ 1},\dots ,x_{ n-1}).$$ \end{proof}

 \begin{exercise}
\begin{enumerate}
\item Prove Cauchy-Binet formula 
 \[
\det(\Phi\Phi^*) = \sum_{I\subset\{1,\dots,n\},\,|I|=m} |\det(\Phi_I)|^2,
\]
\item Give another proof of Dyson's identity.
\end{enumerate}
 \end{exercise}

 \begin{longversion}

 Step 1: Write the kernel as a Gram matrix

For a projection kernel of rank \(n\),
\[
K(x,y)=\sum_{\ell=1}^n \varphi_\ell(x)\overline{\varphi_\ell(y)}.
\]

Fix \(x_1,\dots,x_m\) and define the \(m\times n\) matrix
\[
\Phi_{i\ell} := \varphi_\ell(x_i).
\]

Then
\[
(K(x_i,x_j))_{i,j=1}^m = \Phi \Phi^*.
\]

So
\[
\det(K(x_i,x_j)) = \det(\Phi\Phi^*).
\]

---

 Step 2: Use Cauchy Binet

Cauchy Binet gives:
\[
\det(\Phi\Phi^*) = \sum_{I\subset\{1,\dots,n\},\,|I|=m} |\det(\Phi_I)|^2,
\]
where \(\Phi_I\) is the \(m\times m\) submatrix with columns indexed by \(I\).

This representation is the real simplifier.

---

 Step 3: Integrate one variable

Now integrate in \(x_m\). The key observation:

Each term \(|\det(\Phi_I)|^2\) is the squared volume of \(m\) vectors
\[
(\varphi_\ell(x_1),\dots,\varphi_\ell(x_m)) \quad \ell\in I.
\]

When you integrate in \(x_m\), orthonormality gives:
\[
\int \varphi_\ell(x_m)\overline{\varphi_k(x_m)}\,dx_m = \delta_{\ell k}.
\]

What happens is:

- If you expand \(\det(\Phi_I)\) along the last row, every cross-term vanishes after integration.
- Only diagonal terms survive.
- You effectively remove one column from \(I\).

Result:
\[
\int |\det(\Phi_I)|^2 dx_m
= \sum_{\ell\in I} |\det(\Phi_{I\setminus\{\ell\}})|^2.
\]

---

 Step 4: Count how many times each subset appears

Now sum over all \(I\) with \(|I|=m\):

Each subset \(J\subset\{1,\dots,n\}\) with \(|J|=m-1\) appears exactly \(n-(m-1)=n-m+1\) times (you can add one element to \(J\) in \(n-m+1\) ways).

So:
\[
\int \det(K(x_i,x_j))_{i,j=1}^m dx_m
= (n-m+1)
\sum_{|J|=m-1} |\det(\Phi_J)|^2.
\]

By Cauchy - Binet again, the RHS is:
\[
(n-m+1)\det(K(x_i,x_j))_{i,j=1}^{m-1}.
\]

---

 Final identity (Dyson identity)

\[
\boxed{
\int \det(K(x_i,x_j))_{i,j=1}^m\,dx_m
=
(n-m+1)\det(K(x_i,x_j))_{i,j=1}^{m-1}
}
\]

 \end{longversion}
  
  \newcommand{\x}{   \mathsf x}  
  
}
  
The projection property also has consequences for a stationary DPP $ \P$.
Assume up to applying a rescaling that the intensity is $ \lambda  = 1,$ i.e. $ \rho _{ \P}^{ (1)}(x) = K(x,x) = 1.$ 
Since by \eqref{eq:def-dpp}, we have the density $ \rho _{K}^{(2)}(x,y)=(1- | K(x,y) | ^{2})$, stationarity yields that it only depends on $ x-y$, we therefore define $ \kappa (x)= | K(0,x) | $.  A direct computation combining \eqref{eq:fact-cov} and \eqref{eq:structure-def} gives the spectral measure \begin{align}
\label{eq:dpp-SF}
\S= \Leb- \F{ \kappa^{2} }.
\end{align}
As a consequence, if $ \kappa $ is square integrable,
 $ \widehat{ \kappa^{2}} (u)\leqslant 1$, $ \s(0)=1- \widehat{ \kappa ^{2}}(0)$ and  reproducing kernels yield hyperuniform DPPs:
\begin{proposition}
\label{prop:proj-DPP-HU}
Let $ \P$ be a  stationary DPP with   reproducing kernel $ K$ such that $ K(0,x)$ is square integrable. Then $ \S$ has a density $ \s$ satisfying $ \s(0)=0$ and $ \P$ is  hyperuniform. Furthermore, $ \s(u)\geqslant \sigma \|u\|^{2}$ for some $ \sigma >0$, hence the hyperuniformity exponent is at most 2.
\end{proposition}

\begin{proof}
$ \s = 1- \widehat{ \kappa ^{ 2}}$ is uniformly continuous as $ \kappa^{ 2} $ is integrable by assumption.
The reproducing and Hermitian properties yield
\begin{align*}
\widehat{  \kappa^{ 2}} (0)=\int_{}\kappa^{ 2}  = 
\int_{}K(0,x)K(x,0)dx
=K(0,0)=1
\end{align*}which yields hyperuniformity by Theorem \ref{thm:general-hu}. The lower bound is a consequence of  Lemma \ref{lm:quadra-spectral}.
\end{proof}

\begin{longversion}{}
 direct proof - converse?
\begin{proof}Remark that the quantity 
\begin{align*}
 \int_{\mathbb{R}^{ d}}K(x,y)K(y,x)dy=\int_{\mathbb{R}^{d}} | K(x,y) | ^{2}dy
\end{align*}does not depend on $ x$ because the measure $ \rho _{2}(dx,dy)=(1- | K(x,y) | ^{2})dxdy$ is invariant under shifts, denote by $ \kappa $ this quantity.  
We have
\begin{align*}
 \textrm{Var}\left(\P(B_{ R})\right) = &\mathbf E \sum_{x, y\in \P}{\bf 1}\{ x,y\in B_{ R} \}-\left(
\mathbf E \sum_{x\in \P}{\bf 1}\{ x\in B_{ R} \}
\right)^{ 2}\\
 = &\mathbf E \sum_{x}{\bf 1}\{ x\in B_{ R} \} + \mathbf E \sum_{x\neq y}{\bf 1}\{ x, y\in B_{ R}\}-\left(
\int_{}{\bf 1}\{ x\in B_{ R} \}dx
\right)^{ 2}\\
 = &\int_{}{\bf 1}\{ x\in B_{ R} \} + \int_{B_{ R}^{ 2}}\rho _{ \P}^{(2)}(x,y)dxdy-\Leb(B_{ R})^{ 2}\\
  = &\Leb(B_{ R}) + \int_{B_{ R}^{ 2}}(1-K(x,y)K(y,x))dxdy-\Leb(B_{ R})^{ 2}\\
   = &\Leb(B_{ R})-\int_{B_{ R}}\left(
\kappa -\int_{\mathbb{R}^{ d}  \setminus B_{ R}}K(x,y)K(y,x)dy
\right)dx\\
 = & \int_{B_{ R}}(1-\kappa )dx + \int_{B_{ R}}\underbrace{ {\int_{\mathbb{R}^{ d}  \setminus B_{ R}} | K(x,y ) | ^{ 2}dy}}_{u_{ R}(x)}dx.\end{align*}
using Cauchy-Schwarz inequality and the Hermitian property. For each $ x$, $ u_{ R}(x)\to 0$ by the local square integrability property, which yields that the second term is in $ o(\Leb(B_{ R}))$. As a result, $ \P$ is  hyperuniform  if and only if $ \kappa =1,$ which is the case if   $ K$ is reproducing. For the converse,   {\color{red} ???}

\end{proof}

\end{longversion}
 
 \subsection{Ginibre ensemble and random matrices}     
 \label{sec:ginibre}
  
Recall from Chapter \ref{ch:intro} that $ \Gin_{ n}$ is the random element of $  \mathscr  M_{ n}(  \mathbb C )$ with  i.i.d. complex Gaussian entries, equivalently its density in the space of matrices is  proportional to $\exp(-Tr(HH^{ *})), H\in  \mathscr  M_{ n}(  \mathbb C ).$   Let $ \P_{ n} = \P_{ n}^{ \Gin}$ the finite Ginibre ensemble of order $ n$, containing these eigenvalues.

\begin{theorem}
\label{lm:finite-ginibre}
The point process $ \P_{ n}$  is a DPP  with reproducing kernel 
\begin{align*}
 K_{ n}(z,w)  = \frac{ 1}{\pi } \sum_{k = 0}^{ n-1}\frac{(z  \bar w)^{ k}}{k!}e^{ -\frac{  | z | ^{ 2} +  | w | ^{ 2}}{2}}.
\end{align*}\end{theorem}

\begin{remark}
We will see that
 $ K_{ n}$ is of the canonical form 
\eqref{eq:canonical-K} with $ n$ terms and all $ a_{ k}$ equal to $ 1$ because the $ \varphi _{ k}(z): = \frac{ 1}{\sqrt{\pi k!}}z^{ k-1}e^{ - | z | ^{ 2}/2}$ are orthonormal. For $ m>n$, the matrix $ (K_{ n}(z_{ i},z_{ j}))_{ 1\leqslant i,j\leqslant m}$ is a sum of $ n$ rank-1 matrices, hence it has rank $ n$, and $ \rho _{ \P_{ n}} ^{ (m)} \equiv 0$, confirming that $ \P_{ n}$ has never more than $ n$ points, and 
\begin{align*}
 \mathbf E \#\P_{ n} = \int_{}K(z,z)dz = n
\end{align*}confirms $ \#\P_{ n} = n$ a.s.. If for some other point process $ \P$ one of the $ a_{ k} $ was in $ (0,1)$, we would have a random number of points because $ \mathbf E \#\P<n$ but $\rho _{ \P}^{ (n)} \not\equiv 0. $ In fact, for a canonical kernel, one can represent the number of points as a sum of independent Bernoulli variables with parameters $ a_{ k}$ (see  \cite[Theorem 4.5.3]{BKPV}).

\end{remark}

\begin{remark}
\label{rk:dpp-phi}
We only consider  DPPs with some kernel $ K$ on $ \mathbb{R}^{ d}$ endowed with Lebesgue measure but it is easily seen that one can equivalently consider a DPP with kernel $ \tilde K(x,y) = K(x,y) \varphi (x)  \bar \varphi (y) $ for some $ \varphi :\mathbb{R}^{ d}\to  \mathbb C $ on $ \mathbb{R}^{ d}$ endowed with $  | {\varphi (x)} | ^{ 2}dx$, in the sense that for $ x_{ 1},\dots ,x_{ m}\in \mathbb{R}^{ d}$
\begin{align}
\det((K(x_{ i},x_{ j}))_{ 1\leqslant i,j\leqslant m})   | \varphi (x_{ 1}) | ^{ 2} \dots |   \varphi (x_{ n})  | ^{ 2}
  = \det((\tilde K(x_{ i},x_{ j}))_{ 1\leqslant i,j\leqslant m}).
\end{align}
\end{remark}

 \newcommand{\z}{   \mathsf z}  
\begin{proof}We admit here the density representation \eqref{eq:ginibre}, it can be proved in a similar (but more intricate) manner than for the GOE, see Section \ref{sec:change-of-variables}. Let us start by showing that $ \rho _{ \P_{ n}}^{ (n)} = \rho _{ K_{ n}}^{ (n)}$,  with the canonical form \eqref{eq:canonical-K} with	   $ \varphi _{ k}(z) = \alpha _{ k}z^{ k-1}e^{ - | z | ^{ 2}/2},\alpha _{ k} = ((k-1)!\pi )^{ -1/2}$.
The starting point is the Vandermonde determinant. For $ \z = (z_{ 1},\dots ,z_{ n})\in  \mathbb C ^{ n},$
\begin{align*}
  \prod_{i<j} (z_{ i}-z_{ j}) = \det((z_{ i}^{ k-1})_{ 1\leqslant i,k\leqslant n}).
\end{align*} Let us multiply each column by the scalar $ \alpha _{ k}$: let  $ M(\z) =( \alpha _{ k}z_{ i}^{ k-1})_{ 1\leqslant i,k\leqslant n}$, then
\begin{align*}
\prod_{i<j} | z _{ i}-z_{ j} | ^{ 2}   
 \propto  | \det(M(\z)) | ^{ 2} = \det(M( \z)M(\z)^{ *})  = \det(  \tilde K_{ n}(z_{ i},z_{ j}))\end{align*}
 with   $$ \tilde K_{ n}(z,w) := \sum_{k = 1}^{ n} | \alpha _{ k} | ^{ 2}z^{ k-1}  \bar w^{ k-1}.$$
 Using Remark  \ref{rk:dpp-phi}, it proves that   $ \rho _{ \P_{ n}}^{ (n)} = \lambda _{ n}\rho _{ K_{ n}}^{ (n)}$ for some constant $ \lambda _{ n}$ with the kernel $ K_{ n}(z,w) = \tilde K_{ n}(z,w)e^{ - \frac{ | z | ^{ 2} +  | w | ^{ 2}}{2}}.$  Let us check that the $ \varphi _{ k}$ form an orthonormal family:
\begin{align*}
 \int_{  \mathbb C }\varphi _{ k}(z)\varphi _{ j}(z)dz = \alpha _{ k}  \bar \alpha _{ j}\int_{}z^{ k-1}  \bar z^{ j-1}e^{ - | z | ^{ 2}}dz =  \alpha _{ k}  \bar \alpha _{ j} \int_{0}^{ \infty }\int_{0}^{ 2\pi }\rho^ { k + j-2}e^{ i\theta (k-j)}e^{ -\rho ^{ 2}}\rho d\rho d\theta 
\end{align*} 
in polar coordinates. We see in particular that it vanishes for $ k = j$ due to the angular integral, and for $ k = j$, it gives  with the change of variables $ u = \rho ^{ 2}$
\begin{align*}
  2\pi  | \alpha _{ k} | ^{ 2} \int_{0}^{ \infty }\rho ^{2k-1 }d\rho e^{ -\rho ^{ 2}} = \pi  | \alpha _{ k} | ^{ 2}\Gamma (k) = 1.
\end{align*}
 We also saw that $ \rho _{ \P}^{(m)} \equiv 0$ for $ m>n$, and for $ m< n$ we wish to apply Proposition \ref{prop:reprod-kernels} inductively.

    {
One can interate backwards on $ k<n:$
\begin{align*}
 \displaystyle\int_{(\mathbb{R}^{ d})^{ n-k}}\rho _{ K}^{ (n)}(z_{ 1},\dots ,z_{ n})dz_{ k + 1}\dots dz_{ n} =& (I-(n-1))(I-(n-2))\dots (I-k)  \rho _{ K}^{ (k)}(z_{ 1},\dots ,z_{ k}) \\
 =& (n-k)!\rho _{ K}^{ (k)}(z_{ 1},\dots ,z_{ k}).
\end{align*}

It implies with Lemma \ref{lm:projection-FMM} that up to the constant $ \lambda _{ n}$ Ginibre is the $ K_{ n}$-DPP:
\begin{align*}
 \rho _{ \P_{ n}}^{ (k)}(z_{ 1},\dots ,z_{ k}) = \frac{ 1}{(n-k)!}\displaystyle\int_{}\lambda _{ n}\rho _{ K}^{ (n)}(z_{ 1},\dots ,z_{ n})dz_{ k + 1}\dots dz_{ n} = \lambda _{ n}\rho _{ K}^{ (k)}(z_{ 1},\dots ,z_{ n}).
\end{align*}

For $ k = 1$, integrating gives 
\begin{align*}
 n = \mathbf E \P_{ n}(\mathbb{R}^{ d}) = \displaystyle\int_{}\rho _{ \P_{ n}}^{ (1)} = \lambda _{ n}\displaystyle\int_{}K_{ n}(z,z) = n\lambda _{ n}.
\end{align*}
Hence $ \lambda _{ n} = 1.$
This concludes the proof that Ginibre is the $ K_{ n}$-DPP.

}
 
 \end{proof}

\begin{proof}[Proof of Theorem  \ref{thm:ginibre} and $ 2$-hyperuniformity for the Ginibre ensemble.]
 We have the pointwise convergence   
\begin{align}
 \label{eq:ginibre-kernel}
K_{ n}(z, w) \to K(z,w): = \pi ^{ -1}e^{ z   \bar w}\end{align}
Let $ z,w\in  \mathbb C ,r = \max( | z | , | w | )$. We have the bound 
\begin{align*}
  | K_{ n}(z,w)-K(z,w) | \leqslant \sum_{k\geqslant n}\frac{  | z  \bar w | ^{ k}}{
\pi k!}e^{ - | z | ^{ 2}/2- | w | ^{ 2}/2} \leqslant \sum_{k\geqslant n}\frac{ r^{ 2k}}{\pi k!}e^{ - | z | ^{ 2}/2- | w | ^{ 2}/2}.
\end{align*}
Therefore the convergence is
 uniform on each compact.
From there it is not difficult to prove that for each $ m$, $ \rho _{ K_{ n}}^{ (m)}\to \rho _{ K}^{ (m)}$ uniformly on each compact.
 Hence 
by Proposition  \ref{prop:cvg-cumulant}, $ \P_{n}^{ \Gin}$ converges to $ \P^{ \Gin}$, the DPP with kernel $ K$, weakly in the vague topology (the existence of $ \P ^{ \Gin}$ comes from the existence of the limit). The projection property passes to the limit: for $ f\in \CC_{c}^{ \infty }(  \mathbb C )$, the corresponding operators satisfy
\begin{align*}
 L_{ K}(L_{ K}f) = \lim_{ n}L_{ K_{ n}}(L_{ K_{ n}}f) = \lim_{ n}L_{ K_{ n}}f = L_{ K}f.
\end{align*}

The invariance of $ \P^{ \Gin}$ under rotations is inherited from the invariance of $ \P_{ n}^{ \Gin}$ under rotations, because the random matrix $ \Gin_{ n}$ is invariant under the action of the orthogonal group.

Finally, for the stationarity of $ \P$, we recall Lemma  \ref{lm:conjugate-expo-kernel} (which also reproves isotropy):
\begin{align*}
 K(z + v,w + v) = e^{ i\varphi (z,v)}K(z,w)e^{ -i\varphi (w,v)}
\end{align*}for some real-valued function $ \varphi .$
As already observed for the planar GAF, the kernel $ K$ is not invariant under the action of shifts, but it does not prevent $ \P^{ \Gin}$ to be stationary because for each $ m\in \mathbb{N},$
\begin{align*}
 \rho _{ \P^{ \Gin}} ^{ (m)}(z_{ 1} + w,\dots ,z_{ m} + w) = \det( K(z_{ i} + w,z_{ j} + w)_{ i,j\leqslant m}) = \sum_{\sigma }\varepsilon (\sigma )  \prod_{i} K(z_{ i} + w,z_{ \sigma (i)} + w)
\end{align*}
where the sum is over permutations of $ \{1,\dots ,m\}$,
and  with $ i' = \sigma (i)$
\begin{align*}
  \prod_{i} K(z_{ i} + w,z_{i'} + w) = & \prod_{i}e^{ i\varphi (z_{ i},w)} K(z_{ i},z_{ i'})e^{ -i\varphi (z_{ i'},w)}\\
   = & \exp\left(
i\sum_{i}\varphi (z_{ i},w)-i\sum_{i}\varphi (z_{ i'},w)
\right)   \prod_{i} K(z_{ i},z_{ i'})
\end{align*}
and the first exponential equals $ 1.$
Therefore the kernels $ \rho _ { \P}^{ (m)}$ are invariant under the action of shifts, which proves by Proposition  \ref{prop:charact-factorial} that the point processes $ \tau _{ w}\P^{ \Gin} $ and $ \P^{ \Gin}$ have the same law, i.e. $ \P^{ \Gin}$ is stationary.

Then the  hyperuniformity comes from Proposition  \ref{prop:proj-DPP-HU}. Since $ \kappa $ has a light tail, Proposition  \ref{prop:proj-DPP-HU} yields that $ \s$ is smooth and therefore that the hyperuniformity exponent is $ \alpha  = 2$ (Lemma \ref{lm:quadra-spectral}).
\end{proof}

\subsection{The GUE as a $ 1$-hyperuniform DPP}
\label{sec:gue-as-dpp}

We see from \eqref{eq:beta-ensemble} that the GUE has the same density form as Ginibre, but restricted to $ \mathbb{R}$. Everything works the same in the previous proof, except that   the $ \varphi _{ k}$ are not orthogonal on $ \mathbb{R}$. Going back to the Vandermonde determinant, we have another unexploited freedom: we can substract from each column a linear combination of columns with a smaller index: with the appropriate $ \alpha _{ k}$ (different from the Ginibre case),
\begin{align*}
 \det \left(
\left(
\alpha _{ k}\lambda _{ i}^{ k-1}
\right)
_{ 1\leqslant i,k\leqslant n}\right) = \det\Bigg(
\Big(\underbrace{\alpha _{ k}\lambda _{ i}^{ k-1} + \sum_{j = 1}^{ k-1}a_{ j,k}\lambda _{ i}^{ j-1}}_{ = :P_{ k}(\lambda _{ i})}\Big)_{ 1\leqslant i,k\leqslant n}
\Bigg).
\end{align*}
We  apply a Gram-Schmidt orthonormalisation procedure to ensure that the $ P_{ k}$ are orthonormal on $ \mathbb{R}$, i.e. 
\begin{align*}
 \displaystyle\int_{\mathbb{R}}P_{ k}(x)P_{ j}(x)e^{ -x^{ 2}/2}dx = \delta _{ k = j},
\end{align*} this yields that $ \P_{ n}^{ GUE}$ is a $ K_{ n}-$DPP with $ K_{ n}$ a canonical projection kernel. This characterisation clearly yields that they are a renormalised version of the Hermite polynomials. It then proves with the same computations that $ \P_{ n}^{ \text{\rm{GUE}}}$ is a DPP with kernel 
\begin{align*}
 K_{ n}(\lambda ,\lambda ') = \sum_{k = 1}^{ n}P_{ k}(\lambda ) P_{ k}(\lambda ').
\end{align*} 
 We leave as a black box the following result: Uniformly on each compact of $ \mathbb{R},$
\begin{align*}
 K_{ n}(\lambda ,\lambda ')\to K(\lambda ,\lambda '): =   \frac{ \sin(\pi (\lambda -\lambda '))}{\pi (\lambda -\lambda ')}.
\end{align*}
One can consult for instance  \cite{GuionnetBook}.
Hence using Proposition \ref{prop:cvg-cumulant} and the exact same reasoning than for the Ginibre process, $ \P_{ n}^{ \text{\rm{GUE}}}$ converges to an infinite stationary  hyperuniform process called the $ \text{\rm{Sine}}$ process (remark that here the kernel itself is invariant under shifts, not only the factorial moment measures). We recover the unit intensity 
\begin{align*}
 \lambda  = K(0,0)  = 1.
\end{align*}

  {
  
  This time, the reduced kernel 
\begin{align*}
 \kappa (x) =  | K(0,x) |  = \frac{ |  \sin(\pi x) | }{\pi x}
\end{align*}
does not have a light tail, and we can compute the spectral measure with    \eqref{eq:dpp-SF}: $ \S = \s \Leb$ with
\begin{align*}
\s(u) = 1- \widehat{ \kappa ^{ 2}}  (u)= 
\begin{cases} 
 | u | /2\pi $  if $ | u | <1\\
 1$  if $u\geqslant 1.
 \end{cases}
\end{align*}
which   yields that   this DPP is $ 1$-hyperuniform.

 \subsection{Random matrices universality}
 \label{sec:random-matrices}

 Let us give a very brief introduction to random matrices universality  with respect to  hyperuniformity, most results can be found in the standard texbooks  \cite{Forrester,GuionnetBook,Mehta}.
Let \(G_N\) be a non-Hermitian  random matrix,   filled with  i.i.d. entries, sometimes called a Girko matrix, whose entries may be real or complex and need not be Gaussian. The Ginibre ensemble is the special instance obtained when the entries are  i.i.d. Gaussian. If the common law has a density, then \(G_N\) has simple spectrum almost surely. The global empirical spectral measure of \(G_N\) converges, under standard moment assumptions, to the continuous  {\it  circular law}, i.e. the uniform distribution on $ B_{ 1}$. Moreover, after zooming in at a point in the bulk of the limiting spectrum, the local eigenvalue process is universal and converges, under suitable hypotheses, to another point process, the infinite Ginibre ensemble, generalising the result of the Ginibre model, as seen in  the previous sections.

A Wigner matrix is a Hermitian random matrix \(H_N\) whose entries above the diagonal are independent and centred. Under mild regularity assumptions, its empirical spectral measure converges at the macroscopic scale to the semicircle law. In the bulk, after suitable rescaling, the spectrum  converges to the universal sine process in the space of configurations: to the determinantal sine-kernel process in the complex Hermitian/GUE case, as seen just above, and to the Pfaffian \(\mathrm{Sine}_1\) process in the real symmetric/GOE case. More generally, these local limits are often referred to as instances of bulk universality.

Wishart or sample covariance matrices, typically of the form \(XX^*\) with \(X\) of the Girko form, are also Hermitian. Their global empirical eigenvalue distribution converges to the Marchenko--Pastur law, while their bulk local statistics are again universal and given by sine-type point processes, with different universal limits such as Bessel or Airy processes appearing at the hard and soft edges.\\

The Gaussian ensembles GOE, GUE and Ginibre, mentioned in Section~\ref{sec:intro-examples}, are the Gaussian special cases of the real Wigner, complex Hermitian Wigner and complex non-Hermitian  i.i.d.  models, respectively. In these Gaussian cases, many properties of the limiting point processes, including hyperuniformity, are explicit and well understood.
We proved that 1D examples (GOE, GUE) are  1-hyperuniform, and 2D Ginibre ensembles are 2-hyperuniform. It is in general a challenge to prove similar results for non-Gaussian models.

For random Hermitian matrices, such as Wigner and Wishart ensembles, results on linear statistics are typically formulated under fractional Sobolev regularity assumptions on the kernel. The natural threshold corresponds to functions $f \in H^{1/2+\varepsilon}, \varepsilon > 0$, which in Fourier space amounts to the condition
$$
\int_{\mathbb{R}} |\hat f(u)|^2 |u|^{1+2\varepsilon} du < \infty.
$$
If one assumes a power-law decay $ |\hat f(u)| = O(|u|^{-(1+\gamma)/2})$, i.e. $ f\in \F_{ \gamma },$ this condition is satisfied whenever $\gamma > 1$.  The borderline case $\varepsilon = 0$ corresponds to $\gamma = 1$; canonical examples are interval  indicators  $f = \mathbf{1}_{[a,b]}$. For a $ 1$-hyperuniform system like the GOE or GUE models on the real line, also known as  {\it Class II system} (Section \ref{sec:hyp-expo}), it is expected that the variance of linear statistics is bounded for $ \varepsilon  >0$ (i.e. $ \gamma  >1$) and logarithmic in $ N$ for $ \varepsilon  = 0$ (i.e. $ \gamma  = 1$).
A series of works over the last decade, notably   \cite{BEY,HeKnowles}, show that, for broad classes of Hermitian random matrix models, the variance of linear statistics $\P_N(f)$ indeed remains bounded for all $f \in H^{1/2+\varepsilon}, \varepsilon > 0$, and that in some cases the variance is logarithmic for $ f$ an indicator function; this confirms the $ 1$-hyeruniformity behaviour. 
 We mention for completeness that the GOE is part of the class of   Pfaffian processes, sharing some features with DPPs. We shall not develop this theme here, see for instance   \cite{BufPFaff}, or references in \cite[Section 3.3]{Coste}.}

The situation for non-Hermitian ``Girko'' random matrices is less clear. Recently, Cipolloni, Erd\"os and Schr\"oder  \cite{CES} proved bounded variance for smooth $H_{ 0}^{ 2}$ kernels at all scales, i.e. admitting a square integrable gradient a.e.. This suggests with Proposition \ref{prop:hyp-expo} or Definition   \ref{eq:def-finite-HU} that such systems are at least $ 2$-hyperuniform, and in general the variance does not go to $ 0$, which means $ 2$ should be the optimal exponent.
 According to what we observe for translation-invariant kernels on $ \mathbb{R}^{ d}$ (Proposition  \ref{prop:hyp-expo}) or $  \mathbb  T  ^{ d},\lambda >0$ (Theorem  \ref{thm:main-periodic}), we should observe a surface order number variance bound for ball indicators, i.e.
\begin{align*}
  \textrm{Var}\left(\P_{ N}(B(x,R))\right) = O(R)
\end{align*}for $ B(x,R)$ in the bulk.
The first sub-volumic bound was obtained very recently by  Cipolloni, Erdos and Kolupaiev, giving bounds in $ O(R^{ 2-\varepsilon })$ for the indicators of smooth shapes, including the ball, with $ \varepsilon \leqslant \frac{ 1}{20}$ in a very involved work, but this is still far from the expected $ O(R)$ bound.    In  \cite{Lr-HU-finite}, the strategy of Section \ref{sec:hu-finite-asymp} is applied, and the optimal rate $O(R)$ is retrieved. The caveat of the method  is that it  only gives results on a reduced scale range $ I_{ N} =  [1,N^{ 1/8}]$ and  for a locally averaged version: for $ R\in I_{ n},$ there is $ \delta >0$ such that, with $ \varepsilon _{ N} = N^{ -\delta },$ for all $ \eta >0,$ and $ B(x,R)$ contained in the bulk $ B(0,(\pi -\eta)\sqrt{N} ),$
\begin{align*}
  \frac{ 1}{\Leb(\T[\varepsilon _{ N}\sqrt{N}])}\displaystyle\int_{\T[\varepsilon _{ N}\sqrt{N}]}  \textrm{Var}\left(\P_{ N}(B(x,r))\right)\leqslant CR.
\end{align*}
In any case, it formally prevents the number variance in balls to have a power law growth $ N^{ b},b>1$ outside a vanishing set of centers $ x.$ This result extends to less regular kernels of $ \F_{ \gamma },\gamma \in (0,1),$ with a rate in $ O(R^{ 2-\gamma }).$
\\

 \section{Gibbs measures}
 \label{sec:gibbs}
 
 \draft{Ideally $ n\to N$}
  
The $ \beta $-ensembles, GOE, GUE, and Ginibre models, descend from a more general class of models in statistical physics defined through a Hamiltonian.
 Consider a pairwise potential, i.e. a function $ \varphi : (\mathbb{R}^{ d})^{ 2}\to \mathbb{R},$ and the energy function 
\begin{align*}
 H^{ \circ }(x_{ 1},\dots ,x_{ n}) := \sum_{i< j}\varphi (x_{ i}-x_{ j}).
\end{align*}
A  {\it Gibbs measure} with energy $ H^{ \circ }$ is roughly speaking a system of particles that tend to arrange themselves randomly 
while keeping a low $ H^{ \circ }$-energy configuration. Since the more the particles shift apart at infinity the lowest is, in general, the energy, one must add a confinement term to ensure the particles have an antagonising force to localise them. This could be a hard confinement in a ball (or another shape), with radius $ \sim n^{ 1/d}$ so that each particle has approximately a constant volume for itself. There 
are also smoother ways to confine the particles, we will typically consider an energy term of the form  
\begin{align*}
 H(x_{ 1},\dots ,x_{ n}) = H^{ \circ }(x_{ 1},\dots ,x_{ n}) + \sum_{i}V_{ n}(x_{ i})
\end{align*}
and $ V_{ n}$ is called a confinement potential, it is supposed to have compact sublevel sets. A hard confinement  penalisation consists formally in choosing  $ V_{ n}(x_{ i}) = \infty \times {\bf 1}\{ \|x_{ i}\|>n^{ 1/d} \}$ (with $ \infty \times 0 = 0$). A frequent and convenient choice is $ V_{ n}(x_{ i})  =  \;\| n^{ -1/d}x_{ i}\|^{ 2}$, and it emerges naturally in the theory of random matrices as we saw with previous models. We indicate  \cite{Lewin,serfaty} as an introduction to the mathematical aspects of such topics.

The  deterministic configurations minimising this energy are called  {\it ground states}, or  {\it optimal configurations}, and are highly ordered, similar to the crystalline structure of a lattice. 
To add some randomness among low energy configurations and reflect disordered states of matter such as gases and liquids, one balances the energy   by an entropy term, favoring more disordered random states, parametrized by some temperature $ T>0$, or the inverse temperature $ \beta  = 1/T$. A probability law minimising  this combined quantity  called  {\it free energy} is called a Gibbs  measure at inverse temperature $ \beta >0$. Its density  with respect to $ (\Leb)^{ n}$ is\begin{align}
\label{eq:boltz}
\rho _{ H,\beta }^{ (n)} (x_{ 1},\dots ,x_{ n})\propto \exp(-\beta H(x_{ 1},\dots ,x_{ n})).
\end{align}
The point process with this density is denoted by $ \P_{ n}^{ H,\beta }$, its particles tend to approach global minimisers of the energy in the low temperature regime $ \beta \to \infty $, and at the opposite converge towards  independent ``totally disordered'' processes with density $\propto  \prod_{i}e^{ -\sum_{i} V_{ n}(x_{ i}) }$  when $ \beta \to 0$ in the ``high temperature regime''.

A popular example is $ s$-Riesz gases, $ s\in \mathbb{R},$ where $ \varphi (x-y) = \|x-y\|^{ -s}$ (or $- \ln(\|x-y\|)$ for $ s = 0$). When $ s>d$, $ \varphi $ is integrable at $ \infty $ and the model is said to be  {\it short range}. It seems that  hyperuniformity cannot happen as the energy is extensive, i.e. proportional to the volume, see Ginibre inequality  \cite[(28)]{Lewin} for finite systems. 
In this case, there is an unambiguously defined infinite limiting  point process $ \P^{ H,\beta }$, i.e. 
\begin{align*}
 \P_{ n}^{  H,\beta } \xrightarrow[ n\to \infty ]{}\P^{ H,\beta }
\end{align*} where $ \P^{H,\beta }$ can also be described locally through DLR equations,
and Dereudre and Flimmel  \cite{DereudreFlimmel} confirms that  stationary  Gibbs measures with a ``local interaction'' (not necessarily through a pairwise potential) are never  hyperuniform.

 \newcommand{\Coul}{\text{\rm{Coul}}}   
A particular case is obtained with  {\it Coulomb gases}, also known under the terminology  {\it jellium}, or  {\it One Component Plasmas} (OCPs). Call $ \varphi _{ d}$ the Coulomb potential in dimension $ d$, i.e. 
\begin{align*}
 \varphi _{ d}(x) = 
 \begin{cases} 
 -\ln( \|x\|)$  if $d = 2\\
 \frac{1}{d-2}\|x\|^{ 2-d}$ if $d = 1$  or $d\geqslant 3,
  \end{cases}
\end{align*}
which satisfies in the distributional sense  \begin{align}
\label{eq:coul-laplace}
 \Delta \varphi _{ d}(x) =-d \kappa _{ d} \delta _{ 0},
\end{align} 
where $ \kappa _{ d} = \Leb(B_{ 1}).$ Recall that $ d\kappa _{ d}$ is also the $  \mathscr  H^{ d-1}$-Hausdorff measure of $  \mathbb  S  ^{ d-1}$.
In dimension $ 2$ this fact has already been useful for GAF zeros at  \eqref{eq:harmonic-ln}, and it is again useful for gravitational allocation at Section \ref{sec:coulomb-alloc}. 
 Call $ \P_{ n}^{ d,\beta }\in \N(\mathbb{R}^{ d})$ the simple point process with exactly $ n$ points and inverse temperature $ \beta\geqslant 0 $  whose density is 
\begin{align*}
\rho _{ d,\beta }^{(n)}(x_{ 1},\dots ,x_{ n}) \propto  \exp(-\beta \sum_{i\neq j}\varphi _{ d}(x_{ i}-x_{ j})) \exp(-\beta \sum_{i}\|x_{ i}\|^{ 2}).
\end{align*}
It indeed corresponds to  equation  \eqref{eq:boltz} with potential $ \varphi _{ d}.$
We recover some previously encountered examples: In dimension $ 2$, for $ \beta  = 2$, we exactly have the Ginibre distribution. There are unfortunately no convenient random matrix interpretations for other values of $ \beta $. Since the potential $ \varphi _{ 2}$ is the logarithm, such models are sometimes also called  {\it log gases}. This terminology   is also used in dimension $ 1$ with $ \varphi (x) = -\ln( | x | )$, but this is not a Coulomb gas since $ \varphi _{ 1}(x) = - | x | $, it is rather a $ 0$-Riesz gas. The one-dimensional log gases, also called $ \beta $-ensembles, are mentioned at Section \ref{sec:intro-examples}. They are the topic of a substantial literature, due to their connection with random matrices and statistical physics. 
 
 Hence the remaining ``interesting'' problem is wether long range Riesz gases (i.e. for $ s\leqslant d$) are  hyperuniform, which is, in general, expected  \cite{Lewin}, and sometimes proved in dimension $ 1$  \cite{ValkoVirag,boursier,Lewin}.   The long range interactions make it complicated to even define unambiguously an infinite model. 
 
  {
 \subsection{Hyperuniformity of Coulomb gases}  
\label{sec:coulomb}

 Let us give several results in dimension $ 2$ or $ 3.$ 
 The first result concerns a finite set of $ n$ particles $ \P^{ d,\beta }_{ n}$ obeying  \eqref{eq:boltz} with Coulomb interaction potential and  quadratic confinement, but the result would hold with other types of confinement according to  \cite[Appendix A]{Leble}. It is known in this case that the $ n$ particles follow at the macroscopic level a circular law, like for the Ginibre ensemble ($ \beta  =d =  2$): with some $ \rho _{  + }>0$ depending on $ d,$ we have the vague convergence in law 
\begin{align*}
 \frac{ 1}{n}\sum_{x\in \P^{ d,\beta }_{ n}}\delta _{ n^{ -1/d}x}\to \frac{ 1}{\rho _{  + }^{ d}\kappa _{ d}} 1_{ B_{ \rho _{  + }}}\Leb.
\end{align*}
The more general phenomenon of macroscopic convergence to equilibrium for Coulomb gases is described in  \cite[Chapter 2]{serfaty}. In general, particles close to the boundary have a different behaviour than particles close to the center. We usually call the  {\it bulk} particles from $ B(0,\rho _{ -}^{ (n)})$ where $ \rho _{ -}^{ (n)} \ll \rho _{  + }n^{ 1/d}$ is ``sufficiently far'' from $ \rho _{  + }n^{ 1/d}$.
The first hyperuniformity result in dimension $ d\geqslant 2$ is by Leblé, with $ \rho _{ -}^{ (n)} = (\rho _{  + }-\eta )\sqrt{n}$ for some $\eta  >0$:

\begin{theorem}[\cite{Leble}]
\label{thm:Leble}
There is $ c>0$ such that for $ n,R$ large enough   and $  x_{ n}\in \mathbb{R}^{ 2}$ such that   $   B(x_{ n},R)\subset B(0,\sqrt{n/ \pi }(1-\eta  ))$,
\begin{align*}
  \textrm{Var}\left(\P_{ n}^{ 2,\beta }(B{( x_{ n},R)})\right) \leqslant cR^{ 2}\ln(R)^{ -0.6}.
\end{align*}
 \end{theorem}  
 Such a rate indeed means hyperuniformity since $ \ln(R)^{ -0.6}\to 0$, but this does not mean $ \alpha $-hyperuniformity for some $ \alpha >0$.
 Serfaty  \cite{serfaty-HU} gives polynomially reduced variance rates for smooth kernels in dimension $ 2$ and $ 3$:
 \begin{theorem} [ \cite{serfaty-HU}, Corollary 2.2]Let $ d\geqslant 2.$
  Let $ f\in \CC_{ c}^{ 3}(B_{ \rho _{  + }-\eta  })$. Then for $ R\in [1,n^{ 1/d}]$
\begin{align*}
  \textrm{Var}\left(\P^{ d,\beta }_{ n}(f_{ R})\right)\leqslant c R^{ 2d-4}.
\end{align*} 
  \end{theorem}  
 
  In view of Proposition \ref{prop:hyp-expo} and Theorem  \ref{thm:main-periodic}, this would mean $ 2$-hyperuniformity in dimension $ 2$, and $ 1$-hyperuniformity in dimension $ 3$.  The rate is valid for $ d\geqslant 4$ but does not give the expected  reduced variance rate. Such models are expected to be $ 2$-hyperuniform in any dimension, see the variance renormalisation in the central limit theorem of  \cite{serfaty-HU}. Leveraging the previous result and applying  the strategy of  \cite{Lr-HU-finite} (see Section  \ref{sec:hu-finite-asymp}), it is possible to give optimal variance for number variance over balls, up to the caveats explained at  Section  \ref{sec:hu-finite-asymp}.
  \begin{theorem}
\label{thm:3D-Coul-finite} Let $d\in \{2,3\}, \beta^{ *}>0, b,\delta  $ such that $ (d + 2)\beta  + \delta <\frac{ 1}{d}$.   Let $ n\geqslant 1, R\in [1,n^{\maxscale }],\varepsilon _{ n} = R^{ -\delta },\beta _{n}\geqslant \beta ^{*}.$
Then   
\begin{align*}
  \sup_{ x \in B(0,\rho ^{-}_{n})}   \overline{  \textrm{Var}}^{ \varepsilon _{ n}}\left(\P_{ n}^{ d,\beta  }(B(x,R)\right)\leqslant C
  \begin{cases} 
  R $  if $d = 2\\
  R^{ 2}\ln(R) $  if $d = 3.
   \end{cases}\\
   \end{align*}
   where $C$ does not depend on $ R,n.$
 \end{theorem}  
 
  This allows in particular to give results for  {\it infinite volume Coulomb gases}, i.e. limit points  $ \P^{ d,\beta }$ of subsequences in the space of random measures.  Note that by  \cite{ArmSerf} such limit points exist, and by  \cite{leble-stationary} they are necessarily stationary in dimension $ 2.$
\begin{theorem}[Asymptotic  hyperuniformity for $ d = 2,3$, \cite{Lr-HU-finite} ]
\label{thm:3d-coul}
Let $ x_{ n}\in B(0,\rho _{ -}^{ (n)}),n\geqslant 1$, $ \P^{ d,\beta }$ a limit point of $   \{  \tau _{ -x_{ n}}\P^{ d,\beta }_{ n};n\geqslant 1\}$ in the vague topology. Let $ \gamma>0,\varphi \in \F_{ \gamma }$ $ \Leb$-a.e. continuous.
\label{thm:Coul-infinite} 
\begin{itemize}
\item  {\bf 2D:}  $  \P^{ 2,\beta }$ is stationary (by  \cite{leble-stationary})   and $ 2$-hyperuniform: as $ R\to \infty $ 
\begin{align*}
  \textrm{Var}\left(\P^{ 2,\beta }( \varphi _{ R})\right)  = O(R^{ 2-\min(\gamma,2) }\ln(R)^{ {\bf 1}\{ \gamma  = 2 \}}),
 \end{align*}
\item  {\bf 3D:} There is $ R^{ \circ}\geqslant 1$ such that for $R\geqslant R ^{\circ }, L \geqslant  { R^{ 2d + 1-\min(1,\gamma )}},$  
\begin{align}
\label{eq:concl-asymp-main}
\sup_{ x_{ 0}\in \mathbb{R}^{ d}}\displaystyle\int_{ [-L/2,L/2 ]^{ d}}  \textrm{Var}\left(\P^{ 3,\beta }(\tau _{ x_{ 0} + x}\varphi _{ R})\right)\frac{ dx}{L^{ d}}\leqslant C_{ \varphi }R^{ 3-\gamma }\ln(R)^{ {\bf 1}\{ \gamma  = 1 \}}.
 \end{align}
 If $ \P^{ 3,\beta }$ is stationary, it is $ 1$-hyperuniform.
\begin{align*}
  \textrm{Var}\left(\P^{ 3,\beta }(\varphi _{ R})\right) = O(R^{ 3-\gamma }\ln(R)^{ {\bf 1}\{ \gamma  = 1 \}}).
\end{align*}
\end{itemize}
 \end{theorem}

 As we will see later, it implies that the 2D Coulomb gas is number rigid  (Theorem \ref{thm:rigid-simple-nsc}),  has finite Coulomb energy and can be represented as a $ 2$-perturbed lattice (Theorem  \ref{thm:infinite-wass-d3}-(i)), see also  \cite[Section 4.1.2]{Lr-HU-finite}.
}

 {
 
 \subsection{Periodic Riesz gases}
 \label{sec:riesz}
 
Riesz gases are Gibbs measures generalising Coulomb gases. They are obtained through Equation  \eqref{eq:boltz} with an interaction in $ \|x-y\|^{ -s}$ (or $ -\ln(\|x-y\|)$ if $ s = 0$). We consider a version on the torus, hence without confinement ($ H = H^{ \circ}$), but with a periodised version of the potential. We work on the one-dimensional torus $  \mathbb  T  ^{ 1}_{ n} $ with $ n$ particles. 
 The periodic Riesz potential with exponent $ s$ is  $ \psi = \psi _{ n,s} $, defined by
\begin{align*}
 \psi_{ n,s}(x) = \lim_{ q\to \infty }\left(
\sum_{k = -q}^{ q} \frac{ 1}{ | kn + x | ^{ s}}-\frac{ 2q^{ 1-s}}{n^{ s}(1-s)}
\right) .
\end{align*}
It is also characterised in the distributional sense with the fractional Laplacian $$ 
\begin{cases} 
\displaystyle\int_{}\psi _{ n,s} = 0,\\ (-\Delta )^{ \frac{ 1-s}{2}}\psi _{ n,s} = c_{ s}(\delta _{ 0}-\frac{ 1}{n})
 \end{cases}$$ for some constant $ c_{ s}>0$, see  \cite{boursier} for details. Consider the random configuration $ \P_{ n}^{ s,\beta }$ with a.s. $ n$ points on $   \mathbb  T  _{ n}$.

  \begin{theorem}
  \label{thm:riesz}
  Let $ s\in (0,1),\beta >0.$   
  \begin{itemize}
\item Let $ \gamma \in (0,1],\varphi \in \F_{ \gamma }^{ *}$. Then there is $ C_{ \varphi }$ not depending on $ n,R$ such that
\begin{align*}
  \textrm{Var}\left(\P_{ n}^{ s,\beta }(\varphi _{ R})\right) \leqslant C_{ \varphi }R^{ 2-\min(1-s,\gamma )}\ln(R)^{ 1-s = \gamma }.
\end{align*}
\item  Any  limit point $ \P^{ s,\beta }$ of $ \P_{ n}^{ s,\beta }$ is stationary and  $ (1-s)$-hyperuniform, hence for $ \gamma >0,\varphi \in \F_{ \gamma }$, 
\begin{align*}
  \textrm{Var}\left(\P^{ s,\beta }(\varphi _{ R})\right) = O(R^{ 2-\min(1-s,\gamma )}\ln(R)^{ 1-s = \gamma }
).
\end{align*}
\end{itemize}
 
  \end{theorem}  
This result is essentially contained in  \cite{boursier}, but the application of Theorem \ref{thm:main-periodic} allows  more kernels with singularities from $ \F_{ \gamma }^{ *}$, see  \cite{Lr-HU-finite}.

  }

\begin{longversion}{}[Discussion]

 The fact that ... has an infinite (unique?) local limit is established in 1D (only?), and several possible limits might exist depending on the boundary conditions, or confining potential imposed. Another way to characterise infinite Riesz gases are through DLR equations , and a stationary process $ \P$ is said to be DLR-Riesz if the ratio is satisfied for local restrictions, i.e. for any bounded domain $ \Omega $, an external configuration $ P^{ \text{\rm{ext}}}$, and two internal configurations $ P_{ 1}^{ \text{\rm{in}}}\subset \Omega ,P_{ 2}^{ \text{\rm{in}}}\subset \Omega $, 
\begin{align*}
 \frac{ \P(P_{ 1}^{ \text{\rm{in}}} | P^{ \text{\rm{ext}}})}{\P(P_{ 2}^{ \text{\rm{in}}} | P^{ \text{\rm{ext}}})} \propto \exp(H(-\beta P_{ 1}))/\exp(-\beta H(P_{ 2})).
\end{align*}
\end{longversion}

\begin{longversion}{}{JLM etc}
 
   \section{Hole probabilities and large deviations}
   
   When one observes  hyperuniform samples, the arrangement of points seems

\end{longversion}

  \section{Quasi-periodic models and quasicrystals }   
  
  \label{sec:quasicrystals}
  
   Hyperuniformity  is important because it emerges in many different phenomena, especially in condensed matter physics and statistical physics. Mathematically,   disordered particle systems have attracted most of the activity, but many other models, e.g. hard spheres, quasicrystals, or others, would deserve more attention.
   To introduce the topic of aperiodic order, let us start with the following toy example, which bears some similarity with subsequent more physical examples.
   
   \begin{example}[ Irrational Dirac combs]
   
   \label{ex:irrational-lattices}
 Let some positive numbers $ \a = \{a_{ m}>0,m\geqslant 1\}$, pairwise irrational, and i.i.d. variables   $ U_{ m}\sim  \mathscr  U_{ [0,1]^{ d}},m\geqslant 1.$  Let the model obtained by summing independent rescaled versions of the shifted lattice
\begin{align*}
 \P^{ \a} = \sum_{m\geqslant 1}\sum_{\m\in \mathbb{Z} ^{ d}}\delta _{ a_{ m}(U_{ m} + \m)}.
\end{align*}
Let us assume that $ \sum_{m}a_{ m}^{ -d}<\infty $ to have finite intensity and local square integrability. Since the spectral measure is linear in the random measure, we obtain the spectral measure of $ \P^{ \a}$ by using  Example  \ref{ex:shifted-lattices} at the proper scale:
\begin{align*}
 \S^{ \a} =\sum_{m}  a_{ m}^{ -2d}\sum_{\m\in 2\pi \mathbb{Z} ^{ d}  \setminus \{0\}}\delta _{ a_{ m}^{ -1}\m} .
\end{align*}
Remark that by carefully choosing the $ a_{ m}$, one can have an arbitrary decay of $ \S^{ \a}(B_{ \varepsilon })$ as $ \varepsilon \to 0$, and hence an arbitrary exponent of hyperuniformity (or a subpolynomial reduction), see Section \ref{sec:hyp-expo}.
   
   \end{example}

Quasicrystals are broadly speaking non-periodic atomic measures whose spectrum is purely atomic, their study emerged after experimental discoveries in physics in the 80s and is related to many fields, including crystallography, aperiodic tilings, almost periodicity, see the mathematical monograph  \cite{baake-book} and references therein. Such objects are traditionally assumed to be homogeneous in space, and it is thus natural to consider random constructions that are invariant under translations  (\cite{HartBjo,torquato-quasi}).

  Let us introduce the  {\it cut-and-project processes}, also called  {\it model sets},  introduced in  \cite{meyer72}, before the actual discovery of quasicrystals in physics   \cite{schetchman}. They turned out to encompass many known quasicrystal models, such as Penrose tilings. Such a process is  obtained, for some $ d,d'\in \mathbb{N}^{ *}$, from a higher dimensional tilted lattice $ \Gamma _{ 0}$ in $ \mathbb{R}^{ d  + d'}$ that is projected onto a $ d$-dimensional band with a finite width.    By introducing a uniform shift $ U$ in the fundamental cell of $ \Gamma _{ 0}$, and projecting onto a unit ball $ W$ of dimension $ d $ one obtains a random stationary measure  \begin{align*}
\P^{ \Gamma_{ 0} ,d}  = \{x: (x,y)\in \tau _{ U}\Gamma_{ 0} ,y\in W \}.
\end{align*}   

Those models possess the following properties, defining a  {\it mathematical quasicrystal}:\begin{itemize}
\item   {\it Uniformly discrete}, i.e. $ \inf_{ x\neq y\in \P^{ \Gamma _{ 0},d}}\|x-y\|>0$ a.s.,
\item  {\it relatively dense}, i.e. for some $ r>0$, $ \cup _{ x\in \P^{ \Gamma _{ 0},d}}B(x,r) = \mathbb{R}^{ d}$ a.s.,
\item  The spectral measure $ \S$ is purely atomic with a dense support.
\end{itemize}
 Bj\"orklund and Hartnick \cite{HartBjo} show that $ \P^{ \Gamma _{ 0},d}$ is in general  hyperuniform, depending on the diophantine properties of $ \Gamma $.
We refer to  \cite{HartBjo} for a precise definition, but an interesting class of lattices is that of  {\it arithmetic lattices}, whose elements are badly approximable by rational tuples. A typical example is $ \Gamma_{ 0}  = \{(a + b\sqrt{2},a-b\sqrt{2});a,b\in \mathbb{Z} ^{ 2}\}\subset \mathbb{R}^{ 2}$, in the same way that algebraic numbers (i.e. solutions of polynomial equations with integer coefficients, such as $ \sqrt{2},\frac{ 1 + \sqrt{5}}{2}$, etc...) are the worst approximable numbers of $ \mathbb{R}$. The idea is that the diophantine properties of the dual of a lattice will determine its  hyperuniformity behaviour, and bad approximation properties give a more  hyperuniform behaviour. Interestingly, it was established in   \cite{Lr21} the same phenomenon for nodal domains of Gaussian fields whose spectral measure has irrational atoms.
 Below we   call $ \S$ the structure factor of $ \P^{ \Gamma _{ 0},d}$.

\begin{theorem}[\cite{HartBjo}]
\begin{itemize}
\item If the dual lattice of $ \Gamma _{ 0}$ is arithmetic,
\begin{align*}
 \lim_{ \varepsilon \to 0}\frac{ \S(B_{ \varepsilon })}{\varepsilon ^{ d}} = o(\varepsilon ^{ \frac{ d}{d'}}).
\end{align*}
\item For $ \Leb[ (d + d')^{ 2}]$-a.e. $ (d + d')\times (d + d')$ real matrix $ A$, the lattice $ \Gamma_{ 0}  = A\mathbb{Z} ^{ d + d'}$ yields a  hyperuniform process such that  for each given $ \delta >0$
\begin{align*}
\frac{  \S (B_{ \varepsilon })}{\varepsilon ^{ d}} = o(\varepsilon  ^{ \frac{ d}{d'}-\delta }).
\end{align*}
\item For $ d = d' = 1$, there exists a lattice  $ \Gamma_{ 0} $ such that for every $ \alpha >0$, 
\begin{align*}
 \limsup_{ \varepsilon \to 0}\frac{  \textrm{Var}\left(\S (B_{ \varepsilon })\right)}{\varepsilon ^{ \alpha }} = \infty .
\end{align*}
It implies by Theorem  \ref{thm:general-hu}-(iii) that $ \P^{ \Gamma_{ 0},1}$ is not  hyperuniform.
\end{itemize}
 \end{theorem} 
 
This spectrum of different second order behaviours in fact reflects the spectrum of diophantine properties of real numbers, for which we include a refresher (see  \cite{bugeaud}):
\begin{itemize}
\item the worst approximable numbers, such as the algebraic numbers, are characterised as those $ x\in \mathbb{R}$ such that for some $ c>0$
\begin{align*}
 \forall  q\in \mathbb{N}^{ *}, \inf_{ p\in \mathbb{Z} }| x-\frac{ p}{q^{ 2}} | >\frac{ c}{q^{ 2}}
\end{align*}they form a negligible set of $ \mathbb{R}^{ d}$.
\item For $ \delta >0$, $ \Leb[1]$-a.e. $ x\in \mathbb{R}$ is $ (2 + \delta )$-approximable, in the sense that for some $ c>0$, for infinitely many $ q\in \mathbb{N}^{ *},$
\begin{align*}
 \inf_{ p\in \mathbb{Z} } | x-\frac{ p}{q} | <cq^{ -2-\delta }.
\end{align*}
\item There is a negligible set of numbers $ x$ with good approximation properties, such as the  {\it Liouville numbers}, i.e. such that for every $ \delta >0$, there are infinitely many $ q\in \mathbb{N}^{ *}$ with 
\begin{align*}
  | x-\frac{ p}{q} | <cq^{ -2-\delta }.
\end{align*}
\end{itemize}

We describe further general random quasicrystals at Section  \ref{sec:QC-rigid}, we show in particular that they have an extremely behaviour, in that they can be totally reconstructed from the knowledge on an arbitrarily small portion.
 
\begin{longversion}{}
Let us describe a similar organisation for some quasi periodic Gaussian fields.

       \subsection{Almost periodic Gaussian zeros}   
\end{longversion}

        \chapter{Perturbed lattices, matching and optimal transport}   
 \label{chap:transport}

We saw that an easy way to produce  hyperuniform samples is to start from a lattice and perturb its points by i.i.d.   variables $ U_{ \k},\k\in \mathbb{Z} ^{ d}$, but the underlying periodic structure remains detectable in the resulting point process (see Example \ref{ex:IPL}), leaving artifacts and aliasing defects. A softening of this procedure is to allow for  dependency between the perturbations $ U_{ \k}$, and in this chapter we first explore how much dependent  can these variables be and still maintain  hyperuniformity. To maintain  stationarity, we assume that the field $ \{U_{ \k};\k\in \mathbb{Z} ^{ d}\}$ is   invariant under $ \mathbb{Z} ^{ d}$-translations; in particular the $ U_{ \k}$ are identically distributed.

In the previous chapter, we saw   hyperuniform models who have in appearance nothing to do with a lattice structure, they seem to form a class fundamentally disconnected from the class of perturbed lattices. We will see that, somehow surprisingly (at least in dimension $ 2$), these disordered processes can in most cases be written as dependently perturbed lattices, forming in some sort a converse to the assertion that perturbed lattices are hyperuniform. To formulate and prove this result, we  invoke  the theory of optimal transport. It ultimately proves that perturbed lattices and disordered  hyperuniform processes are not a dichotomy, they form two sides of a continuous class, except in dimension $ 1$ where a stronger requirement than  hyperuniformity is required.

This is another illustration of the   formula   {\it global order and local disorder}, in that a  hyperuniform point process, even when it seems locally disordered, such as for a DPP or the zeros of a random function,  exhibits at large scale the same order as a lattice.\\ 

{We also explore the deep connections between hyperuniformity and fair partitions of the space, also rooted in optimal transport, as an optimal transport between a point process and Lebesgue measure, also called  {\it semi-discrete transport,} is equivalently described by a fair partition of the space.}

    \section{Perturbed lattices are  hyperuniform  point processes}  
     \label{sec:perturbed-is-hu}

     We already saw that when a lattice undergoes i.i.d.   perturbations, it is  hyperuniform (Example \ref{ex:IPL} and Theorem \ref{thm:general-hu}-(ii)). When the perturbations are dependent, Dereudre et al.  \cite{DFHL} proved that it still holds   as long as they have a finite second order moment in dimension $ 1$ or $ 2$. In the following,    $   \mathsf U =  \{U_{ \k};\k\in \mathbb{Z} ^{ d}\}$ is a stationary field over $ \mathbb{Z} ^{ d}$, independent of the shifted lattice $ \tau _{U}\mathbb{Z} ^{d}$ where $ U\sim  \mathscr  U_{[0,1]^{d}}$. Denote the resulting  point process by 
\begin{align*}
\Z^{d, \mathsf U}=\sum_{\k\in \mathbb{Z} ^{d}}\delta _{U+U_{\k}}.
\end{align*}
     
     \begin{theorem}[\cite{DFHL}]
     \label{thm:DFHL}
    Assume $ \mathbf E \|U_{ k}\|^{ 2}<\infty $. Then \begin{itemize}
\item In dimension $ d = 1$, $ \sup_{ R}  \textrm{Var}\left(\Z^{d, \mathsf U}(B_{ R})\right)<\infty $, i.e. $ \Z^{d, \mathsf U}$ is  {\it Class 1}  hyperuniform. 
\item In dimension $ d = 2$, $ \Z^{d, \mathsf U}$ is  hyperuniform, but the variance decay can be arbitrarily slow.
\item In dimension $ d \geqslant 3$, for every $ \varepsilon >0,$ there exists a stationary field $    \mathsf U$ with $ \|U_{ \k}\|<\varepsilon $ a.s. and $ R^{ -d}  \textrm{Var}\left(\Z^{d, \mathsf U}(B_{ R})\right)\to \infty .$
\end{itemize}
      
      \end{theorem}  
 \cite{DFHL} also givew counter-examples showing that these results are sharp in general.
 The general idea is still that the more the perturbation is statistically disordered (extreme disorder being the independence), the more it has the chance to maintain the underlying lattice structure, but also the less likely is the obtained point process  to be mixing. In some sense the mixing properties of $   \mathsf U$ and $ \P$ go in reverse directions. This is confirmed by the results of  \cite{Flimmel,KLLY}, which ensures that  hyperuniformity is maintained in any dimension as long as the perturbating field $   \mathsf U$ is sufficiently mixing. Flimmel  \cite[Theorem 2]{Flimmel} uses the $ \alpha $-mixing coefficients, for $   \mathsf U$ a stationary field over $ \mathbb{Z} ^{ d}$,
\begin{align*}
 \alpha (n): = \sup \{| \mathbf P (\Omega )-\mathbf P (\Omega ') | 
\;  : \; {  \Omega \in \sigma (U_{ \k};\k\in A),\Omega '\in \sigma (U_{ \k'};\k'\in B)};\;A,B\subset \mathbb{R}^{ d}:d(A,B)\geqslant n   \},n\in \mathbb{N}.\end{align*} Klatt et al. \cite[Th. 5.5]{KLLY}     introduce in the flavor of $ \beta $-mixing, with $ \mathbf P  _{ X}$ the law of  $ X,$
\begin{align*}
 \beta (\m) = \sup_{ A\subset \mathbb{R}^{ d}\times \mathbb{R}^{ d}\text{\rm{ Borel}}} |\mathbf P _{ (U_{ 0},U_{ \m})}(A)-\mathbf P  _{ U_{ 0}}^{ \otimes 2}   (A) |,\m\in \mathbb{Z} ^{ d}.
\end{align*}

 \begin{theorem}\label{eq:mixing-perturbed}If  $\sum_{\m} \alpha (\|\m\|)<\infty $ or $ \sum_{\m}\beta (\m)<\infty $, then $ \Z^{ d,   \mathsf U}$ is  hyperuniform.
  \end{theorem}  
 \cite{Flimmel} contains a stronger version under some moment assumptions, and  \cite{KLLY} actually considers any stationary random measure (not only point processes), and any kind of perturbation, as long as stationarity is in order.     Those results were generalised in the recent work  \cite{LotzKlatt}.

\begin{remark}[A Poisson transport cannot be mixing]
\label{rk:non-mixing}
Consider an equivariant matching $ T$ between $ \Z^{ d} = \tau _{ U}\mathbb{Z} ^{ d}$ and a homogeneous Poisson process $\P =  \P^{ \Poi(1 )}$, or equivalently let us look $ \P = \Z^{ d,   \mathsf U}$ as a perturbed lattice, with $ U_{ \k} = T(\k + U)-(\k + U)$.
By the result above, the matching  cannot be quantitatively mixing, because $ \Z^{ d,   \mathsf U } = \P$ is not  hyperuniform. 

 \cite{KLLY} give more general similar results also applying to semi-discrete transport, which means that a transport between Poisson and Lebesgue cannot be mixing either. It does not prevent  such matchings to be well behaved in terms of moments  \cite{CPPR,HPPS}.
\end{remark} 
 
\begin{example}[
 {Lattices perturbed by clusters  and  arbitrarily high  hyperuniformity exponent}
] \label{sec:cluster-lattice}
 
 A similar class of dependent perturbations is obtained when each point is replaced by a whole finite point process, as considered in  \cite{Lr24}.  This is not formally a perturbed  lattice as studied above, but the resulting points can still be seen as the result of a perturbation applied to a lattice. We   refer the reader to \cite{KLLY, LotzKlatt} for a general study  of the hyperuniformity  of perturbed random measures. Let us give here a result where the point processes attached to different points are independent.
 
 Let $ \mu $ a probability distribution on $ \N(\mathbb{R}^{ d})$, and $ \P =\sum_{i}\delta _{ x_{ i}}$ a simple stationary point process. Let $ \P_{ i} ,i\geqslant 1,$ i.i.d.   point processes with law $ \mu $, and 
\begin{align*}
 \P^{ \mu }: = \sum_{i}\sum_{j} \tau _{ x_{ j}}\P_{ i}.\end{align*}

\begin{proposition}
\label{thm:cluster-process}
Let   $ \S$ the spectral measure of $ \P$ . Assume $\mathbf E \#\P_{ 1}^{ 2}<\infty  $. Let the random Fourier-Stieljes transform $ \varphi (u) = \int_{}e^{ iu\cdot t}d\P_{ 1}(t)$. Then $ \P^{ \mu }$ is locally square integrable and has the spectral measure 
\begin{align*}
 \S_{ \mu } = \left(
\mathbf E  | \varphi (u) | ^{ 2}- | \mathbf E \varphi (u) | ^{ 2}
\right)\Leb(du) + \mathbf E  | \varphi (u) |^{ 2} \S.
\end{align*}
 \end{proposition}
 Note the absence of assumption on the perturbation moments (only the number of points is constrained).  
 We recover   \eqref{eq:IPL-SF}   when $ \#\P_{ 1} = 1$ a.s.
 We see that in general, if $ \P$ is  hyperuniform, $ \S$ vanishes in $ 0$, hence also $ \S_{ \mu }$ vanishes in $ 0$, and by Theorem  \ref{thm:general-hu}$,  \P^{ \mu }$ is still  hyperuniform.  Remark that Lemma \ref{lm:quadra-spectral} does not always apply to this form of structure factor, and, in contrast to the  i.i.d. perturbation, we can indeed exploit this to have a point process with arbitrarily high hyperuniformity exponent.  {Lotz and Klatt  \cite[(5.30)]{LotzKlatt} give a general formula where a point process is perturbed by a random measure, not necessarily atomic, to build non-negative hyperuniform random measures.}

Among stationary models mathematically treated, the only exception so far to the behaviour $ \s (u)\geqslant \sigma \|u\|^{2}$ around $ 0$ are GAF zeros (Proposition \ref{lm:gaf-Delta}) and shifted lattices (Example \ref{ex:shifted-lattices}). For the latter, at the opposite, the spectral measure vanishes in a neighbourhood of the origin; it is a member of the class of  {\it stealthy processes}, which are the topics of Section \ref{sec:stealthy}. To obtain an intermediary model, that is not formally periodic, but still presents an arbitrarily high level of rigidity, one can use a union of irrational shifted lattices, see Example  \ref{ex:irrational-lattices}.  
Let us give another class of models, with slightly more disorder, in the spirit of perturbed lattices, where the points of a lattice are perturbed by randomly rotated clusters.  
   Let $ \rho >0, p\in \mathbb{N}^{ *}$. Let $ \P_{ i} = \sum_{j = 1}^{ p}\delta _{ 2i\pi j \rho  + \theta_{ i} }$ on $  \mathbb C $ where the $ \theta_{ i}$  are i.i.d.   uniform on $ { [0,2\pi ]}$,  and let $ \mu _{ p}$ the common law of the $ \P_{i}$, and let $ \P=\Z^{d}$, recall the construction of $ \P^{\mu _{ p}}$ from above.

 \begin{theorem} ( \cite{Lr24}, Theorem 3 and  \cite[Example 5.16]{LotzKlatt})
 \label{thm:expo-prime}
  The  point process $ \P^{ \mu_{p } }$ is  $ p$-hyperuniform. \end{theorem}

{On this question of point processes with high  hyperuniformity, let us  point  to Section  \ref{sec:fair-partitions}, where we use fair partitions to obtain higher orders of  hyperuniformity.  }

\end{example}
   
\begin{example}
 
 \label{sec:partial-matching}
 
 Klatt, Last and Yogeshwaran  \cite{KLY} consider a special construction that can be seen as a dependently perturbed lattice. Let $ \P = \P^{ \Poi(a)}$  with intensity $ a>1$, let $ \Z^{ d} = \tau _{ U}\mathbb{Z} ^{ d}$ the usual shifted lattice with $ U\sim   \mathscr  U_{ [0,1]^{ d}}$ independent from $ \P$ (Section \ref{sec:perturbed-intro}). Then build an injective mapping $ T: \Z^{ d}\to \P$ that is invariant under translations, i.e. $ \tau _{ x}T \smallequlaw   T,x\in \mathbb{R}^{ d}$, by performing a  {\it stable matching:}\begin{itemize}
\item First match $ x\in \Z^{ d},y\in \P$ if they are mutual nearest neighbours, i.e. $ y$ is the closest point from $ x$ in $ \P$ and vice-versa, and put $ y = T(x)$. Call $ \Z^{ d}_{ \text{\rm{match}}}$ such matched points from $ \Z^{ d}$ and $ \P_{ \text{\rm{match}}}$ the matched points from $ \P$.
\item Remark that $ \Z^{ d}  \setminus \Z^{ d}_{ \text{\rm{match}}},\Z^{ d}_{ \text{\rm{match}}},\P  \setminus \P_{ \text{\rm{match}}}, \P_{ \text{\rm{match}}}$ are still stationary point processes. Repeat the procedure by matching mutual nearest neighbours of $\Z^{ d}  \setminus  \Z^{ d}_{ \text{\rm{match}}}$ and $ \P  \setminus \P_{ \text{\rm{match}}}$, putting each time $ y = T(x)$ if $ y$ and $ x$ are matched.
\item Then repeat the procedure by matching iteratively mutual nearest neighbours that have not been matched previously.
\end{itemize}
 This  procedure never  terminates globally, but each point of $ \Z^{ d}$ is eventually matched after finitely many iterations, hence $ T$ is well defined, stationarity follows from the stationarity at each step. We then have the following result. For   $ x = \k + U\in \Z^{ d}$, define $ U_{ \k} = T(x)-x$ and $   \mathsf U = \{U_{ \k};k\in \mathbb{Z} ^{ d}\}$, so that indeed $ T(\Z^{ d}) = \Z^{   \mathsf U}$ is a dependently perturbed lattice. Since $ a>1$, a fraction $ (1-a)$ of the Poisson points remain unmatched.
 
 \begin{theorem}[\cite{KLY}]
 $ T(\Z^{ d})$ is a stationary hyperuniform perturbed lattice and the $ U_{ \k}$ have an exponential tail: for some finite $ c>0,$ 
\begin{align*}
 \mathbf P (\|U_{ 0}\|>r)\leqslant ce^{ -cr^{ d}}.
\end{align*}
  \end{theorem}  
  The proof can be also be deduced from Theorem  \ref{eq:mixing-perturbed} by  the same authors. In  \cite{KLY}, they furthermore prove that $ T(\Z^{ d})$ is number rigid, a concept that is introduced at Chapter  \ref{chap:rigid}.
 
 \begin{remark}
 A similar procedure can be conducted with $ a = 1$, and in this case the points of $ \P$ are also eventually exhausted, so that $ T$ is a bijection between $ \Z ^{ d}$ and $ \P$, but the matching properties are much less agreable; in particular, $ \|U_{ 0}\|$ has an infnite $ d$-th moment, whereas other Poisson matchings perform much better \cite{hoffman2006stable,HPPS}.
 \end{remark}
 
\end{example}

 \section{  Hyperuniform point processes are perturbed lattices}   
 \label{sec:hu-is-lattice}

One of the motivations for studying  hyperuniform  point processes is the low variance for linear statistics (Proposition \ref{prop:hyp-expo}), and more generally  that they are ``evenly'' distributed across space, at least they seem so visually. We try to formalise here this impression with the concept of allocation, and more generally of transport.  As we will see, this   provides some kind of converse statement to Theorem  \ref{thm:DFHL}, in that many  hyperuniform processes can be seen as a perturbed lattices.\\

From the point of view of transport, a sample of points is well distributed if one can divide up the space in cells of equal volume such that each cell can be associated to a nearby  point of the process, the ``center of the cell'' (note that this ``center'' can end up being outside the cell). One can think for instance of points as school locations in a city $ \C_{ n}$, and Lebesgue measure represents the distribution of the population, the goal being that each school is associated with an equal volume of population and that each citizen should not have to travel too far to go to the school it has been assigned to (which is not necessarily the closest school, since each school has a limited capacity). The
ideal repartition occurs when points (or schools) are distributed periodically as in a  lattice, we rather investigate here disordered samples. We show in the figure below three samples of points and a corresponding partition of a sphere in cells of equal volume using the so-called gravitational allocation  \cite{NSV} (showed to be close to optimal in magnitude).

\begin{figure}[h!]
\begin{center}
\includegraphics[width = 10cm]{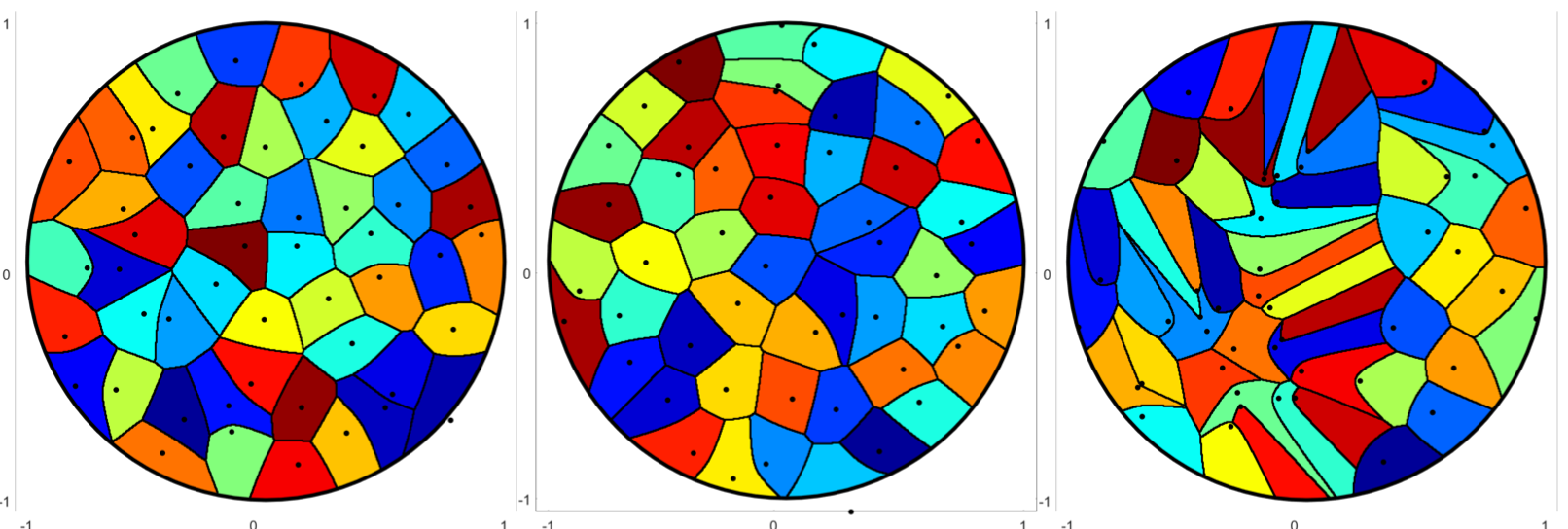}\caption{Sending kids to school.  {\it Left.} Ginibre eigenvalues.  {\it Middle}. Discrete random matrix eigenvalues.  {\it Right.}  i.i.d. points.  {\it Ilustration by J. Jalowy  \cite{jalowy}}}
 
\label{fig:hawat-allocations}
\end{center}
\end{figure}

As we see, The Poisson /  i.i.d.  cells seem more elongated, giving rise to a higher transport cost for the typical individual, compared to the  hyperuniform samples, which is what we try to formalise in the next sections.

\subsection{Optimal transport, matching and allocation}  

We propose in this chapter a quantification of this  concept in terms of matching and allocation.  Let $ p\geqslant 1.$

\begin{itemize}
\item Let $  \mathbb  T  _{ n}^{ d} = [0,n)^{ d}$, and let $ \P_{ n} $ be a sample of $ n^{ d}$ points in  $  \mathbb  T  _{ n}^{ d}$. Call   {\it allocation} a measurable mapping $ T: \mathbb  T  _{ n}^{ d}\to \P_{ n}$ such that a.s. the volume  of the set of points sent to each $ x\in \P_{ n}$ is exactly $ 1$, i.e. $ \Leb(T^{ -1}(\{x\})) = 1$, $ x\in \P_{ n}$. Call $ p$-allocation cost of $ T$
\begin{align*}
 \mathsf C_{p}(T) = \int_{ \mathbb  T  _{ n}^{ d}}\|x-T(x)\|^{ p}dx,
 \end{align*}
 and call optimal $ p$-allocation cost $ \mathsf W _{p,\text{\rm{alloc}}}^{ p}( \mathbb  T  _{ n}^{ d},\P_{ n})$ the infimum of $ \mathsf C_{p}(T)$ over such allocations $ T$, the notation $ \mathsf W _{p}^{ p}$ is explained below. 
\item  Let $ \Z_{ n} =  \mathbb  T  _{ n}^{ d}\cap \mathbb{Z} ^{ d}$ the   lattice points of $  \mathbb  T  _{ n}^{ d}$, note that $\# \Z_{ n} = n^{ d}$. Call $ p$-matching cost of a bijection $ \sigma :\Z_{ n}\to \P_{ n}$ the quantity 
\begin{align*}
 \mathsf C_{p}(\sigma ) = \sum_{\k\in \Z_{ n}}\|\sigma (\k)-\k\|^{ p}.
\end{align*}
Call optimal $ p$-matching cost $ \mathsf W _{p,\text{\rm{match}}}^{ p}(\Z_{ n},\P_{ n})$ of $ \Z_{ n}$ to $ \P_{n}$ the minimum of $ \mathsf C_{p}(\sigma )$ over all such bijections $ \sigma .$ \\
\end{itemize}

 The concept of matching is in fact very close to that of perturbed lattices, as it means that there are variable $ U_{ \k}: = \sigma (\k)-\k$ such that $ \P_{ n} = \{\k + U_{ \k};\k\in \Z_{ n}\}$.   In general, of course, such $ U_{ \k}$ would  not be independent.  
 
  The matching and allocation costs are obviously different in general, but in the large sample asymptotics, we will see that they have the same magnitude for a given stationary point process $ \P$ restricted to $  \mathbb  T  _{ n}^{ d}$. From this perspective, the process $ \P$ will be stationary, i.e. invariant under $ \mathbb{R}^{ d}$-translations, and accordingly the field $   \mathsf U: = \{U_{ \k};\k\in \mathbb{Z} ^{ d}\}$ will be required to be invariant under $ \mathbb{Z} ^{ d}$-translations.
To avoid some technicalities for now, we study more generally a family of  point processes $ \P_{ n},n\geqslant 1$, and study whether there is explosion of the mean $ p$-transport cost per particle.
  
\subsection{Linear cost for large finite samples}  
 \label{sec:linear-cost-Wp}
 
The concepts of matching and allocation costs above are both instances of the  concept of $ p$-Wasserstein distance in optimal transport. Given two non-negative measures  $ \mu ,\nu $ with same mass on $ \mathbb{R}^{ d}$, call  {\it coupling} of $ \mu ,\nu $ a measure $ M$ on $ \mathbb{R}^{ d}\times \mathbb{R}^{ d}$ with 
\begin{align*}
M(\cdot \times \mathbb{R}^{ d}) = \mu  ,\;\;M(\mathbb{R}^{ d} \times \cdot ) = \nu.
\end{align*} Then define the $ p$-th order Wasserstein distance by
\begin{align}
\label{eq:def-Wpp}
 \mathsf W _{p}^{ p}(\mu ,\nu ) =\inf_{ M\text{\rm{ coupling}}}\int_{}\|x-y\|^{ p}M(dx,dy).\end{align}
 Remark that this quantity is symmetric in $ \mu $ and $ \nu .$
 
Call $ \Leb_{n}$ the Lebesgue measure restricted to $  \mathbb  T  _{ n}^{ d}$. Note that $ \#\Z_{ n} =\# \P_{ n} = \Leb_{ n}( \mathbb  T  _{ n}^{ d})= n^{ d}$, which is the reason why we employ a straight cubic window in this chapter. An allocation $T : \mathbb  T  _{ n}^{ d}\to \P_{ n}$ induces the coupling 
\begin{align*}
  M(A\times B) := \sum_{y\in \P_{ n}\cap B}\Leb_{ n}(A\cap T^{ -1}(\{y\})),
\end{align*}
the mass of points from $ A$ sent to $ B.$
A matching $ \sigma :\Z_{ n}\to \P_{ n}$   induces similarly a coupling
\begin{align*}
 M(A\times  B) = \sum_{x\in \Z_{ n}\cap A,y\in \P_{ n}\cap B}{\bf 1}\{ y = T(x) \}.
\end{align*}
Hence the optimal matching and allocation costs correspond to infima taken over specific classes of couplings, hence
\begin{align*}
 \mathsf W _{p}^{ p}(\Leb_{ n},\P_{ n})\leqslant \mathsf W _{p,\text{\rm{alloc}}}^{ p}(\Leb_{ n},\P_{ n}),\;\; \mathsf W _{p}^{ p}(\Z_{ n},\P_{ n})\leqslant \mathsf W _{p,\text{\rm{match}}}^{ p}(\Z_{ n},\P_{ n}).
\end{align*}
We in fact have  the reverse inequalities
$$ \mathsf W _{p}^{ p}(\Z_{ n},\P_{ n}) = \mathsf W _{p,\text{\rm{match}}}^{ p}(\Z_{ n},\P_{ n}),\; W_{ p}^{ p}(\Leb_{ n},\P_{ n}) = W_{ p,\text{\rm{alloc}}}^{ p}(\Leb_{ n},\P_{ n}),$$
see for instance \cite[Thm. 1.7]{Sant}.
 We expect all those costs to be similar, which we formalise in the next result.
 
 \begin{proposition}
 \label{prop:match-alloc}
Let $ \P_{ n},n\geqslant 1$  be  point processes with resp. $ n^{ d}$ points in $  \mathbb  T  _{ n}^{ d}$, and $ p\geqslant 1.$ 
There is a universal constant $ c_{ d,p}>0$ such that 
\begin{align}
\label{eq:lower-bound-Wp}
   \mathsf W_{ p}^{ p}(\Leb_{ n},\P_{ n})\geqslant c_{ d,p}n^{ d}.
\end{align}
Hence the  {\it linear cost} (proportional to $ n^{ d}$), is the least one can expect.
Then the two following are equivalent:\begin{itemize}
\item [(i)]There is a constant $ c_{ \text{\rm{match}}}$ such that
 $
 \mathbf E \mathsf W _{p}^{ p}(\P_{ n},\Z_{ n}) \leqslant c_{ \text{\rm{match}}}n^{ d},$
\item [(ii)] There is a constant $ c_{ \text{\rm{alloc}}}$ such that 
 $
 \mathbf E \mathsf W _{p}^{ p}(\P_{ n},\Leb_{ n})\leqslant c_{ \text{\rm{alloc}}}n^{ d}.
 $
\end{itemize}
\end{proposition}

\begin{proof}The lower bound is a consequence of the isoperimetric inequality, which gives  for any domain $ \Omega $ with volume $ 1$, and any $ y\in \mathbb{R}^{ d}$
\begin{align*}
 \int_{\Omega }\|x-y\|^{ p}dx\geqslant c_{ d,p}: = \int_{B(y,\kappa _{ d}^{ -1/d})}\|x-y\|^{ p}dx\text{\rm{\color{black}    with }}\kappa _{ d} = \Leb(B_{ 1}).
\end{align*}Then just make the summation over the cells $ \Omega _{ i}: = T^{ -1}(x_{ i}).$

The equivalence comes from the general triangular inequality between finite measures: to prove that (i) implies (ii), use  
\begin{align*}
 W_{ p} (\P_{ n} ,\Leb_{ n})\leqslant W_{ p} (\P_{ n},\Z_{ n} ) + W_{ p} (\Z_{ n} ,\Leb_{ n}).
\end{align*}
This inequality is proved in all textbooks about optimal transport, see for instance  \cite{Sant,Vil03}. One can obtain the estimate 
\begin{align*}
 W_{ p}^{ p}(\Z_{ n},\Leb_{ n})\leqslant cn^{ d}
\end{align*}
using the following elementary allocation: cut $  \mathbb  T  _{ n}^{ d}$ into $ n^{ d}$ small cubes with sidelength $ 1$ and define $ T_{ 0}: \mathbb  T  _{ n}^{ d}\to\Z_{ n}$  by
\begin{align*}
 T_{ 0}(\k + [0,1)^{ d}) = \{\k\};\k\in \Z_{ n}.
\end{align*}
It gives $ W_{ p}^{ p}(\Z_{ n},\Leb_{ n})\leqslant C_{ p}(T_{ 0}) \leqslant cn^{ d}.$
Then use the inequality $ (a + b)^{ p}\leqslant c_{ p}(a^{ p} + b^{ p}),a,b\geqslant 0$ to have (i) implies (ii). The same inequality after switching the roles of $ \Z_{ n},\Leb_{ n}$ gives (ii) implies (i). 
 
%
 
\end{proof}

An example of paramount importance in geometric probability is when $ \P_{ n}$ consists in   i.i.d.   points uniformly distributed over $  \mathbb  T  _{ n}^{ d}.$
 
 \begin{theorem}[AKT Theorem  \cite{AKT}]  
 \label{AKT}
Let $ \P_{ n}$ made up of $ n^{ d}$ i.i.d.   points uniform in $  \mathbb  T  _{ n}^{ d}$, and $ p\geqslant 1$. Then  (i) and (ii) hold  if and only if  $ d\geqslant 3$  \end{theorem}  

This result admits many generalisations, it is  in particular valid for processes with asymptotically integrable covariance, see for instance  \cite{LrY,BDG,HPPS}. 
Hence in dimension $ d\geqslant 3$,  hyperuniform or not, a standard  point process has a satisfying behaviour in terms of transport / allocation / matching.

  In dimension $ d = 2$, on the other hand, one cannot find a nice allocation from $  \mathbb  T  _{ n}^{ d}$ to $ \P_{ n}$, in the sense that, by the AKT theorem
\begin{align*}
\lim_{ n} \frac{ \mathbf E \mathsf W_{ 2,\text{\rm{alloc}}}^{ 2}(\P_{ n},\Leb_{ n})}{n^{ d}}  \to \infty .
\end{align*}
This can be observed on Figure \ref{fig:hawat-allocations}, the right sample is formed by independent points, and the cells are elongated, generating a large transport cost.  The other samples,  eigenvalues of random matrices, have more spherical cells, yielding a lower cost. We will see now that the reason is that such samples are asymptotically hyperuniform, as we saw previously.

Pioneer works in this direction are  \cite{jalowy,prod2021contributions}, giving results for the finite Ginibre ensembles.  
  Let us give a result unifying  the works  \cite{BDG,LrY}. 
As we saw at Section \ref{sec:hu-finite}, the  hyperuniformity of a finite sample is preferentially treated in the spectral domain with the discrete spectral measure $ \S_{ \P_{ n}}(u),u\in \mathbb{Z}_n  ^{ d}$, see Proposition  \ref{lm:spectral}.   
 The work  \cite{LrY} is based on Sobolev inequality, inspired by  \cite{BL21}, the method allows to derive  transport bounds for a measure $ \mu $ on $  \mathbb  T  ^{ d}$ in terms of $ \F\mu $ through smoothing and reverse Sobolev inequalities. The drawback is that this bound is valid for the toric Wasserstein distance $  \widetilde{   \mathsf W}_{ p}^{ p}$, i.e. when the distance $ \|x-y\|^{ p}$ is replaced by the toric distance $ d_{  \mathbb  T  _{ n}^{ d}}(x,y)^{ p} = \inf_{ \k\in \mathbb{Z} ^{ d}}\|x-y + \k n\|^{ p}$ in  \eqref{eq:def-Wpp}. When $ n$ goes to $ \infty $, the influence of boundaries in the toric distance will vanish under some additional assumptions (see the next section).  The work of Butez, Dallaporta,  and Garcia-Zelada \cite{BDG} gives a more direct geometric method to upper bound the optimal transport distance. We state the results for a general   random measure $ \M_{ n}$ as the methods generalise immediately.

  \begin{theorem}
  \label{thm:BDG}  Let $n\geqslant 1, \M_{ n}$ a non-negative    random measure on $  \mathbb  T  _{ n}^{ d}$ with $   \M_{ n}(  \mathbb  T  _{ n}^{ d}) = n^{ d}$ a.s.. 
  \begin{itemize}
\item {(a)} \cite{LrY} Assume $ \M_{ n}$ is $  \mathbb  T _{ n}^{ d}$-invariant,  and  that   for some finite  $ c,\varepsilon >0 $, \begin{itemize}
\item {(i)} for some $ f\in \mathscr B_{ c}( \mathbb{R}^{ d})$ such that $ \int_{}f\neq 0$, for $R\in [1,n]$,
$$  \textrm{Var}\left(\P_{ n}(f_{ R})\right)\leqslant c 
\begin{cases} 
R^{ d}$  if $d\geqslant 3\\
 R^{ d}\ln(R)^{ -1-\varepsilon }$  if $d = 1,2\end{cases}$$
\item or (ii) its spectral measure $ \S_{ n}$ satisfies  for  some $ \rho   _{ 0}>0$, for $ 0<\rho   <\rho   _{ 0}$,
\begin{align}
\label{eq:ass-BL}
\S_{ n}(B_{\rho   })\leqslant c\times 
 \begin{cases} 
\rho   ^{ d}$  if $d\geqslant 3\\
\rho   ^{ d}{ | \ln(\rho ) | ^{ -1 - \varepsilon }}$  if $d = 1,2
  \end{cases}
\end{align}
\end{itemize}
holds. Then  for some constant $ c'$  depending  on $ c,\rho _{ 0} ,d$
\begin{align}
\label{eq:linear-cost-finite}
 \widetilde{   \mathsf W}_{ 2}^{ 2}(\Leb_{ n},\M_{ n}) \leqslant c'n^{ d}.
\end{align} 
\item  {(b)} \cite{BDG}  Let $ p\geqslant 1$. Assume that $ \M_{ n}$ has intensity $ \Leb_{  \mathbb  T  _{ n}^{ d}}$, and   that for some finite $c, \varepsilon >0$, for any square $ B\subset  \mathbb  T  _{ n}^{ d}$ with $ \Leb(B)>1,$
\begin{align}
\label{eq:ass-variance-bdg}
\sup_{ n} \mathbf E  \left|
 {  \M_{ n}(B)-\Leb(B)}\right|  ^{ p}\leqslant  { c} \Leb(B)^{ p/2}\times 
 \begin{cases} 
{\ln(\Leb[2](B))^{ -p - \varepsilon }}$  if $d = 2.\\
1$  if $ d\neq 2.
  \end{cases}
\end{align}
\end{itemize}
Then   for some constant $ c'$ not depending  on $ n$
\begin{align*}
   \mathsf W_{ p}^{ p}(\Leb_{ n},\M_{ n})\leqslant c'n^{ d}.
\end{align*}
  
   \end{theorem}

  Point (b) generalises a purely geometric construction from  \cite{prodhomme} focused on the Ginibre process. For this reason, it does not depend  on $ p$ and gives a result in the non-toric metric. It  does not require invariance formally, but the  results require some form of asymptotic invariance up to the second order.
 Hence this result is more flexible, but in the case $ p = 2,d = 2$,  \eqref{eq:ass-variance-bdg} is a log order below (i), so that (b) is more useful for $ p\neq 2$.  Our formulation is not exactly the same as in  \cite{BDG}, but it can be obtained by choosing 
\begin{align*}
 g(t) = 
 \begin{cases} 
 \ln(t)^{ -1-\varepsilon/p }$  if $d = 2\\
 t^{  \frac{ 1}{ d}   -\frac{ 1}{2} }$  if $d\neq 2
  \end{cases}
\end{align*}
 
 {Note that  \cite{HPPS,HueLebl}, discussed at the next section, give results in infinite volume, but also need an intermediary finite volume bound at some point.} Let us also compare the assumptions with the definition of $ \alpha $-hyperuniformity for finite systems at  \eqref{eq:def-finite-HU}, because it is very similar. In fact, the assumptions above do  not even require $ \alpha $-hyperuniformity for some $ \alpha >0$, as only a logarithmic reduction is required.
    Let us give the main ideas for (a), from \cite{LrY}.

   \begin{proof}[Sketch of proof of (a)] Let us first prove (i) implies (ii). Assume (i) holds.
   Let $ \rho >0$ such that $ |  \hat f(u) | > \kappa : = \frac{ 1}{2} | \int_{}f | $ on $ B_{ \rho }$. With $ \hat f_{ R} = R^{ d} \hat f(R.)$, Lemma  \ref{lm:scaling} yields  
\begin{align*}
 \kappa \S_{ n}(B_{  \rho R^{ -1}})\leqslant R^{ -2d}\int_{B_{ \rho R^{ -1} }}  | \hat f_{ R} | ^{ 2}d\S_{  { n}} \leqslant R^{ -2d}\int_{\mathbb{R}^{ d}} | \hat f_{ R} | ^{ 2}d\S_{ { n}} = R^{ -2d}(2\pi )^{ d}  \textrm{Var}\left(\M_{ n}(f_{ R})\right) ,
\end{align*} hence   (ii) is a consequence of (i), so we only assume (ii) hereafter.
   
   The method consists first in smoothing the measure $ \M_{ n}$ on $  \mathbb  T  _{ n}^{ d}$ through a $  \mathbb  T  _{ n}^{ d}$-convolution (denoted $ \ast_{   \mathbb  T  }$) with a  probability distribution $ \mu $. We obtain a non-negative random measure  denoted by $ \M^{ \mu }_{ n} = \M_{ n} \ast_{  \mathbb  T  }  \mu $, still with mass $ n^{ d}.$ Standard considerations on distributional Fourier transform yields that its spectral measure is $ \S_{ { n}}^{ \mu } = \S_{ { n}} | \F\mu | ^{ 2} $ (\cite{Rudin2}). The main ingredient is the general inverse Sobolev bound  (\cite[(7)]{BL21})
\begin{align}
\label{eq:inv-sobolev-0}
  \widetilde{   \mathsf W}_{ 2}^{ 2}(\Leb_{ n},\M^{ \mu }_{ n}) \leqslant c_{ d}n^{ 2} \S_n ^{ \mu }( \| u \| ^{- 2}1_{ u\neq 0}) =  c_{ d}n^{  2}  \S_{ { n}}( \| u \|^{ -2} 1_{ u\neq 0}) | \F_{ \mu }(u)   |  ^{ 2}.
\end{align}
Apply this inequality for a particular choice of $ \mu $ where $ \F\mu $ is supported by $ B(0,1)$ and $ \mu $ has a finite second moment (use for instance $ \mu  = \F (1_{ B_{ 1/q}})^{ q}\Leb$ for sufficiently large $ q$ and Lemma \ref{lm:fourier-ball}). We obtain
\begin{align}
\label{eq:inv-sobolev}
  \widetilde{   \mathsf W}_{ 2}^{ 2}(\Leb_{ n},\M^{ \mu }_{ n}) \leqslant c'_{ d}n^{ 2} \S_{  { n}}(\|u\|^{ -2}1_{ 0<\|u\|\leqslant 1}).
\end{align}
Then the layer cake formula yields  with $ \varphi (u) = \|u\|^{ -2}1_{ 0<\|u\|\leqslant 1}$,
\begin{align*}
 \S_{ { n}}(\varphi ) = \int_{0}^{ 1    }\S(\{u:\varphi (u)>t\})dt = \int_{0}^{ 1    }\S(B_{ t^{ -1/2}})dt.
\end{align*}

Differentiating between the cases $ d = 1,d = 2,d\geqslant 3$, Assumption  \eqref{eq:ass-BL} yields that the right-hand term of  \eqref{eq:inv-sobolev} is in $ O(n^{ d})$.
To have a bound on $  \tilde \W_{ 2}^{ 2}(\Leb_{ n},\M_{ n})$, remark that the convolution defines naturally a transport between $ \M_{ n}$ and $\M^{ \mu }_{ n}   $ with total cost in $ O(n^{ d})$, and then apply the triangle inequality of Wasserstein distance.
   
   \end{proof}
   
   The quantity central in the proof is  $ \S_{ \M_{ n}}(\|u\|^{ -2}{\bf 1}\{ 0<\|u\|\leqslant 1 \})$, coming from the inverse Sobolev inequality. We will see that a similar quantity, related to   {\it  Coulomb energy} in dimension $ 2,$ is  crucial  in the continuous case.
    An issue is that if $ \M_{n}$ is the restriction of a stationary measure $ \M$, $ \M_{n}(  \mathbb  T  _{ n}^{ d})$ is random and the previous bound cannot apply directly, we give in the next section the results of \cite{LrY,HueLebl} where this problem is handled.

 \subsection{Infinite samples and Coulomb energy}  
 
\label{sec:finite-Coulomb}

 We investigate here how to translate the previous results to an infinite stationary    point process $ \P$ with unit intensity, to  obtain the representation  as a perturbed lattice, i.e. $ \P=\P^{   \mathsf U}$ for some $ \mathbb{Z} ^{ d}$-stationary   $   \mathsf U=\{U_{\k};\k\in \mathbb{Z} ^{d}\} $, as announced in the chapter introduction.
Previous results yield that for a Poisson (or  {\it Poisson-like}) process in dimension $ d\geqslant 3$, or a hyperuniform process in dimension $ 2$, on a large window $  \mathbb  T  _{ n}^{ d}$ where $ \P$ has approximately $ n^{d}$ points, there should be a nice matching between $ \P\cap  \mathbb  T  _{ n}^{ d}$ and $ \Z \cap \T[n]$, or a nice allocation from $  \mathbb  T  _{ n}^{ d}$ to $ \P\cap  \mathbb  T  _{ n}^{ d}$. The matching can equivalently be written as a bijection $ \sigma $ between $ \mathbb{Z} ^{d}$ and $ \P$, with $ \sigma (\k):=\k+U_{\k}$. The length  $\| U_{0}\|$ can be interpreted as the law of the distance between a typical point of $ \P$ and the point of $ \mathbb{Z} ^{d}$ it has been assigned to; a good matching is such that $ \mathbf E \|U_{0}\|^{p}<\infty $ for sufficiently high $ p$, we shall typically investigate this question for $ p=2$. 

In analogy with the finite case, we call invariant allocation to $ \P$ a random measurable mapping $ T:\mathbb{R}^{d}\to \P$ whose law is invariant under shifts, i.e. 
 $ 
\tau _{y}T \smallequlaw T,y\in \mathbb{R}^{ d}$, and such that a.s., the cells $ T^{-1}(\{x\});x\in \P$ form a partition of $ \mathbb{R}^{ d}$ by cells of volume $ 1$,   called a  {\it fair allocation}. As before, the law of $ \|T(0)\|$ and its moments are a good indicator of how good an allocation. It is possible to show under mild assumptions that there is a stationary allocation $ T$ with $ \mathbf E \|T(0)\|^{p}<\infty $ for some $ p>0$  if and only if  $ \P$ is a $ p$-perturbed lattice, i.e. $ \P=\Z^{   \mathsf U}$ where $   \mathsf U$ is an invariant perturbation $   \mathsf U$ with $ \mathbf E \|U_{0}\|^{p}<\infty .$ We shall state the next results in terms of matchings, but it could equivalently be stated in terms of allocations, recalling Proposition  \ref{prop:match-alloc}.

As a preliminary remark,  with or without  hyperuniformity, there always exist such a perturbation $   \mathsf U$ or   allocation $ T$ under the minimal assumption that $ \P$ is ergodic, see for instance  \cite{HPPS} where  a matching is built with the Gale-Shapley stable marriage algorithm (see a description at Section  \ref{sec:partial-matching}). This procedure might actually yield suboptimal matchings; it is proved in  \cite{HPPS} that  this stable marriage yields $ \mathbf E \|U_{0}\|^{ d}=\infty $ in dimension $ d\geqslant 3$, when other matching procedures admit exponential moments.

\begin{theorem}
\label{thm:infinite-wass-d3}
Let  $ \P$ be a \LSI stationary point process in $ \mathbb{R}^{ d},d\geqslant 2$.  \begin{itemize}
\item [(i)]  Assume one of the equivalent conditions
 \begin{align}
 \label{eq:coulomb-E-var}
  \int_{ }  R^{ -2d + 1}\textrm{Var}\left(\P(B_{ R})\right)dR<\infty   \\
\label{eq:cont-coulomb-energy}
 \S_{ \P}(\|u\|^{ -2}1_{ 0<\|u\|<1})<\infty .
\end{align}
 Then $ \P$ is a $ 2$-perturbed lattice
\item [(ii)]  Assume $  \textrm{Var}\left(\P(B_{ R})\right) = O(R^{ d})$. Then $ \P$ is a $ p$-perturbed lattice for $ p<d/2$.
\end{itemize}
 \end{theorem}  
 Let us see what result to use for each dimension, having in mind that a standard Poisson-like process satisfies $  \textrm{Var}\left(\P(B_{ R})\right) = O(R^{ d})$, and for  hyperuniform this is $ o(R^{ d}).$
\begin{itemize}
 \item $ d\geqslant 5$. The second result is stronger as it works for some $ p >2$.
\item If $ d  = 3,4$, The first point gives $ p = 2$.
\item If $ d = 2$, we need in both cases some reinforcment of  hyperuniformity, namely the integrability of $ R^{ -3}  \textrm{Var}\left(\P(B_{ R})\right)$, and point (i) can apply.
\end{itemize}

 Point (ii) results from a geometric construction from  \cite{HPPS},   written for Poisson processes. See  Theorem 1-(ii) and \cite[Remark 1]{LrY} for a generalisation. For $ d = 2,3,4$, \cite[Theorem 1]{LrY} gives the result if $ \C$ is integrable, which is a slightly stronger assumption. 
  For $ d = 2$, the first connection with optimal transport was done in 
  \cite{LrY}, who gave  the sufficient condition  \begin{align*}
\int_{\mathbb{R}^{ 2}  \setminus B_{ 1}}  \ln(\|x\|) | \C | (dx) <\infty,
\end{align*}  stronger than  \eqref{eq:coulomb-E-var} but of the right magnitude.
Then \cite{HueLebl} noticed the link with the Coulomb energy and established  that condition  \eqref{eq:cont-coulomb-energy},  i.e. {\it finite Coulomb energy}, is sufficient.
This condition appeared first in   \cite{SodinElectric} in relation to the existence of an electric field generated by $ \P$ (see Section \ref{sec:coulomb-alloc}).     The study of the Coulomb energy is ubiquitous in statistical physics  \cite{serfaty,Lewin}. \cite{SodinElectric}  and \cite{HueLebl} proved also the equivalence with the condition  $$ \sum_{n\geqslant 0} 4^{ -n} \textrm{Var}\left(\P(B_{ 2^{ n}})\right)<\infty, $$ called condition $ (   \mathsf {HU}_{ \ast})$ in \cite{HueLebl}, as a reinforcement of the plain hyperuniformity condition (Theorem \ref{thm:general-hu}-(i)).
  \cite{DFHL} give counter-examples showing that  $ \P$ being a $ L^{ 2}$-perturbed lattice does not imply  \eqref{eq:coulomb-E-var}-\eqref{eq:cont-coulomb-energy} in dimension $ 2$, but it implies plain  hyperuniformity, see Theorem  \ref{thm:DFHL}-2.  


\begin{proof}[Partial proof of Theorem  \eqref{thm:infinite-wass-d3}]

To prove the existence of a global matching with finite expected square distance, one must prove first that finite restricted samples have a good matching distance to Lebesgue measure
as in \eqref{eq:linear-cost-finite}, and then let the window grow to $ \mathbb{R}^{ d}$. Let us examine more closely Equation  \eqref{eq:inv-sobolev}, bounding the Wasserstein distance for a finite sample, recalling that  $ \F\mu $ can be chosen to be bounded and supported by $ B(0,1)$. The general necessary condition to have a linear transport cost  as in  \eqref{eq:linear-cost-finite} for the restriction  $ \P_{ n} = \P\cap \T[n]
$ is  that
\begin{align}
\label{eq:finite-coulomb-discrete}
 \sup_{ n}\S_{ \P_{ n}}(\|u\|^{ -2}1_{ 0<\|u\|\leqslant 1})<\infty ,
\end{align}  analogue to  \eqref{eq:cont-coulomb-energy} in the continuous space.
Let us decompose the proof in lemmae.
\begin{lemma}
\label{lm:4-10}
 \eqref{eq:finite-coulomb-discrete} implies that $ \P$ is a $ L^{ 2}$-perturbed lattice.
\end{lemma}

The second step of the proof is to establish that if  \eqref{eq:coulomb-E-var}
-  \eqref{eq:cont-coulomb-energy} hold, then  so does \eqref{eq:finite-coulomb-discrete}. Surprisingly, starting from  \eqref{eq:coulomb-E-var} gives a more direct argument.

\begin{lemma}
\label{lm:2-transport}
\eqref{eq:coulomb-E-var} implies  \eqref{eq:finite-coulomb-discrete}.
\end{lemma}

We will finally have to prove the equivalence \eqref{eq:coulomb-E-var}
-  \eqref{eq:cont-coulomb-energy}.
 
 \begin{proof}[Sketch of proof of Lemma  \ref{lm:4-10}]
 
  {
 
Once \eqref{eq:linear-cost-finite} holds for all $ n\geqslant 1$, it remains to prove that it implies the existence of an infinite matching.
 One of the difficulties is that $N: =  \#\P_{ n}$ is not equal to  $ n^{ d}$ a.s., hence one applies Theorem  \ref{thm:BDG}  to a renormalised lattice restriction $ \tilde \Z_{ n} $   with random mass $ N$, and shows  it  admits a good matching with $ \P_{ n}$. Then one uses the compacity of the space of matchings between atomic measures with same intensity to show that asymptotically, we have a $ L^{ 2}$ matching between $ \Z^{ d}$ and $ \P.$
 
Let us therefore start by normalising the mass. Denote by $ \tilde \M_{n}$  the renormalised version of $ \P_{n}$, i.e. $$ \tilde \M_{n}=\frac{n^{d}}{\P_{n}(  \mathbb  T  _{n}^{d})}\P_{n}$$ with $ \tilde \M_{n}=\Leb_{n}$ (arbitrarily) if $ \P_{n}(  \mathbb  T  _{n}^{d})=0$. In particular, $ \tilde \M_{n}$ and $ \Leb_{n}$ have the same mass a.s..
  An optimal transport cost for each $ n$  is always achieved by a random measure $ M_{ n}$ on the product space $ \mathbb{R}^{ d}\times  \mathbb{R}^{ d}$ whose marginals are $ \tilde \M_{ n}$ and   $ \Leb_{ n}$,   $$   \widetilde{   \mathsf W}_{ 2}^{ 2}(\Leb_{ n}, \tilde \M_{ n})  = \int_{ ( \mathbb  T  _{ n}^{ d})^{ 2}}d_{  \mathbb  T  _{ n}^{ d}}(x,y)^{ 2}M_{ n}(dx,dy).$$
Since $ \sup_{ n}\mathbf E\left[
 M_{ n}([0,1]^{ 2})
\right]<\infty $, the tightness result from \cite[Thm. 14.16-(iv)]{kallenberg2002foundations} yields a infinite measure $ M$ on $ \mathbb{R}^{ d}\times \mathbb{R}^{ d}$ limit of a subsequence $ M_{ n_{ j}}\to M$. It is not hard to show that as $ n_{ j}\to \infty $, $ M$ inherits the following properties from $ M_{ n}$, see  \cite[Proposition 11]{LrY}:\begin{itemize}
\item $ M$ is stationary
\begin{align*}
 M(\tau _{ x}A,\tau _{ x}B) \equlaw  M(A,B),x\in \mathbb{R}^{ d}.
\end{align*}
\item $ M$ is a  transport between (i.e. a coupling of)  $ \P$ and $ \Leb$: $ M(\cdot ,\mathbb{R}^{ d}) = P, M(\mathbb{R}^{ d},\cdot ) = \Leb$.
\item The assumption  \eqref{eq:cont-coulomb-energy} on the structure factor validates  \eqref{eq:inv-sobolev-0} and yields    by  \eqref{eq:inv-sobolev} a linear transport cost for $ \P_{ n}$, which can be transferred to $ \tilde \M_{ n}$, and passes to the limit: for $ B\subset \mathbb{R}^{ d}$ bounded
\begin{align*}
 \int_{B\times \mathbb{R}^{ d}}\|x-y\|^{ 2}M(dx,dy)<\infty .
 \end{align*}
\end{itemize} 
Some infinite triangle inequality can be applied to show that there is a stationary coupling $ M'$ between $ \P$ and $ \Z^{ d}$ with finite local cost also satisfying $ \mathbf E M'(B,\Z^{ d})<\infty .$ The final point is to show that $ M'$ can be chosen to be an actual matching, i.e. there is a stationary one-to-one mapping $ T:\P_{ n}\to \Z^{ d}$ such that 
\begin{align*}
 M'(A,B) = \#\{x\in \P_{ n}\cap A:T(x)\in \Z^{ d}\cap B\}.
\end{align*}
This is in fact a tedious step because of the renormalisation procedure,  one can use  \cite{erbar2023optimal}.
 More details can be found in  \cite[Proposition 11]{LrY}.
}

 \end{proof}

\begin{proof}[Proof that  \eqref{eq:cont-coulomb-energy} $  \;\Leftrightarrow \;$  \eqref{eq:coulomb-E-var}]
 It is based on the same considerations than  in the proofs of Theorems \ref{thm:general-hu} and  \ref{thm:main-periodic}. We also use the identity
\begin{align*}
  \int_{B_{ 1}}\|u\|^{ -2}\S(du) = \int_{1}^{ \infty }
 \S(B_{ t^{ -1/2}})dt  =  &\int_{1}^{\infty } \S(B_{ R^{ -1}})2RdR.
\end{align*}

 {Use Lemma \ref{lm:new-Adhik-SF}-1 with $ f = 1_{ B_{ 1}},\gamma  = 1 $. The identity above yields\begin{align*}
 \int_{B_{ 1}}\|u\|^{ -2}\S(du)   \leqslant   c\int_{1}^{ \infty }R^{ -2d + 1}  \textrm{Var}\left(\P(B_{ R})\right)dR 
\end{align*}
which gives the direct implication.
For the converse statement, let $ \xi (R) = R^{ d}\S(B_{ R^{ -1}})$. We use Lemma \ref{lm:new-Adhik-SF}-2. Define $
\sigma  (R): = R^{ -1} + \xi  (R) + \sigma _{ \xi  ,1}(R).$ Then   
\begin{align*}
\int_{}  \textrm{Var}\left(\P(B_{ R})\right)R^{ -2d + 1}dR \leqslant &c \int_{}R^{ d}\sigma  (R)R^{ -2d + 1}dR\\
\leqslant & c\int_{1}^{ \infty }R^{ -d}dR + \int_{1}^{ \infty }R^{ -d + 1}\xi (R)dR+ \int_{1}^{ \infty }R^{ -d + 1}\int_{0}^{ 1    }\xi  (Rt^{ \frac{ 1}{d + 1}})t^{ -\frac{ d}{d + 1}}dtdR\\
 \leqslant  &c'  + \int_{ 1}^{ \infty }R \S(B_{ R^{ -1}})dR+ \int_{0}^{ 1    }\left(
\int_{0}^{ \infty }T^{ -d + 1}t^{ \frac{ d-1}{d + 1}}\xi (T)t^{ -  \frac{ d-1}{d + 1}} t^{ -\frac{ 1}{d + 1}}dT
\right)dt\\
\leqslant &c' +  \frac{ 1}{ 2}   \int_{ B_{ 1}}\|u\|^{ -2}\S(du) + \int_{ 0}^{ 1    }t^{ -1 + \frac{ d-1}{d + 1}}dt\int_{ 0}^{ \infty }T^{ -d + 1}\xi (T)dT<\infty 
\end{align*}
which concludes the proof.}

\end{proof}

\begin{proof}[Proof of Lemma   \ref{lm:2-transport}]

Recall that $ \S_{ \P_{ n}}$ is supported by $ \mathbb{Z}^{ d} _{ n}$, so that $ \S_{ \P_{ n}}(B_{ \varepsilon }  \setminus \{0\}) = 0$ for $ \varepsilon <2\pi n^{ -1}$. For $ \varepsilon \geqslant 2\pi n^{ -1}$, recall that with  \eqref{eq:half-max-inequality}, we have $  | \widehat{ 1_{ B_{ 1}}}(u) | \geqslant \frac{ \kappa _{ d}}{2}1_{ B_{ \rho _{ d}}}(u)$, hence   
\begin{align*}
 \S_{ \P_{n}}(B_{ \rho _{ d}\varepsilon }) = \S_{ \P_{n}}(  | 1_{ B_{ \rho _{ d}\varepsilon }} | ^{ 2})\leqslant \frac{ 4}{\kappa _{ d}^{ 2}} \varepsilon ^{2d}\S_{\P_{n}}(  | \widehat{ 1_{ B_{  \varepsilon ^{ -1}}}} | ^{ 2}) = c_{ d}\varepsilon ^{ 2d}  \textrm{Var}\left(\P_{ n}(B_{ \varepsilon ^{ -1}})\right) =  c_{ d}\varepsilon ^{ 2d}  \textrm{Var}\left(\P (B_{ \varepsilon ^{ -1}})\right)
\end{align*}because the restriction of $ \P$ to $ B_{ \varepsilon ^{ -1}} \subset  \mathbb  T  _{ n}^{ d}$ coincides with $ \P_{ n}$. Finally, for some $ c<\infty ,$
\begin{align*}
 \S_{ \P_{n}} (\|u\|^{ -2}1_{ 0<\|u\|<1}) =& \int_{1}^{ (2\pi )^{-2}n^{ 2}}\S_{ n}(B_{ t^{ -1/2}})dt\\
 & \leqslant c \int_{1}^{ \infty }t^{ -d}  \textrm{Var}\left(\P(B_{ t^{  1/2}})\right)dt\\
   =& c\int_{}R^{ -2d + 1}  \textrm{Var}\left(\P(B_{ R})\right)dR<\infty .
\end{align*}

\end{proof}
\end{proof}

In conclusion, the dichotomy between perturbed lattices and disordered  hyperuniform processes seems rather to be a continuum.
 In dimension $ 1$,  hyperuniformity is not enough to guarantee linear $ \W_{2}^{2}$ cost, but we do not investigate this question as 1D optimal transport tools are pretty specific, see  \cite[Th.2-(ii)]{LrY}.\\

  \section{Hyperuniform point processes and fair partitions}

  A countable random collection of closed sets $  \mathscr  P = \{B_{ i};i\geqslant 1\}$  partitionning $ \mathbb{R}^{ d}$ a.s. and whose joint law is invariant under translations is  called an  {\it equivariant tessellation of the space}. We implicitly assume throughout that the $ B_{ i}$ have a negligible boundary. We do not introduce the formalism for random partitions, and explore some examples below; it is a by-product of the theory of random closed sets, see details in  \cite{Mol05}.  If the sets have a.s. the same volume, their union is a.s. $ \mathbb{R}^{ d}$, and their pairwise intersections are a.s. all negligible, they are said to form a  {\it fair equivariant partition}.   
A transport plan between Lebesgue measure on $ \mathbb{R}^{ d}$ and a unit intensity stationary point process $ \P$, also called  {\it semi-discrete transport}, is a random mapping $ T:\mathbb{R}^{ d}\to \P$ assigning to each point of $ \P$ a cell with unit volume, i.e. a.s. $ \Leb(T^{ -1}(y)) = 1,y\in \P.$ In this case, $ T$ is called a  {\it fair allocation} from $ \Leb$ to $ \P$. Equivalently, the $ T^{ -1}(y),y\in \P$ form a fair partition of the space.  The equivariance property in this case means $$ \{T(x + y),y\in \mathbb{R}^{ d}\} \stackrel{(d)}{=}  \{T(y) + x,y\in \mathbb{R}^{ d}\},  \forall  x\in \mathbb{R}^{ d},$$ or for short $ \tau _{ x}T \smallequlaw   T.$    
The relation between  fair partitions and  hyperuniform point processes goes both ways:\begin{enumerate}
\item First, we show how to exploit fair partitions to build  hyperuniform processes, based on  \cite{torquato-tilings,KLLY,LotzKlatt}.
\item Then, we show how to build such an allocation from a point process $ \P$, and in dimension $ 2$, this procedure is more successful if $ \P$ is  hyperuniform.
\end{enumerate}
We can actually cycle through this procedure. Many random partitions of the space by finite-volume cells are built from stationary point processes (typically Poisson). Then one obtains a  hyperuniform process as in item 1, and from there one can build a new partition more efficiently, and sample a new point process, etc... We do not investigate such iterations theoretically, but it is a numerically efficient method to produce  hyperuniform samples, see Chapter \ref{ch:numerics}.  {We also discuss a method to build an equivariant fair partition not originating from a point process, based on  \cite{LotzKlatt}.}

    {

\subsection{Hyperuniform constructions based on fair partitions}
\label{sec:fair-partitions}

Fair partitions seem to offer the most flexible framework to build $ \alpha $-hyperuniform processes with arbitrarily high $ \alpha .$ Gabrielli, Joyce and Torquato  \cite{torquato-tilings} have proposed a general heuristics to build such  process, starting from an equivariant fair partition $  \mathscr  C = \{B_{ i};i\geqslant 1\}$,  and attaching to each cell a cluster of points with the right moment properties. Assume hereafter up to rescaling that all $ B_{ i}$ have volume $ 1$ for notational simplification. We will see that the constructions proposed here are not modified if the $ B_{ i}$ are modified on a negligible set, hence we more generally call fair partition a collection of unit volume sets $ B_{ i}$ with negligible boundary whose interior sets are pairwise disjoint and whose union is $ \mathbb{R}^{ d}$ (note that the $ B_{ i}$ need not be bounded).

 Recently, Klatt et al.  \cite{KLLY} and Lotz and Klatt  \cite{LotzKlatt} have given formal and general frameworks and proofs for the procedure of Gabrielli et al.  \cite{torquato-tilings}, their construction allows furthermore non-convex cells and more randomisation than in   \cite{torquato-tilings}. They claim that the construction of a fair disordered partition based on the STIT model  \cite{NagWei05} works with this construction; yielding  disordered  hyperuniform processes with a high exponent. Let us present this construction in the framework of tilings, i.e. with a  fair partition $  \mathscr  C$ where the $ B_{ i}$ are all unit volume compact convex polyhedra.

For a finite measure $ \mu $  on $ \mathbb{R}^{ d}$, $\m = (m_{ 1},\dots ,m_{ d})\in \mathbb{N}^{ d}$, denote  its $ \m$-th order moment by
\begin{align*}
 \mo_{ \m}(\mu ) = \displaystyle\int_{C}x^{ \m}\mu (dx)\text{\rm{ where }}x^{ \m} = x_{ 1}^{ m_{ 1}}\dots x_{ d} ^{ m_{ d}},
\end{align*} if it exists. For instance, the  {\it barycenter}, also called  {\it centroid} or  {\it center of gravity} of a set $ C$ is defined by the mean of the corresponding uniform random variable 
\begin{align*}
 \text{\rm{centr}}(C) = \mo_{ (1,\dots ,1)}(1_{ C}\Leb) = \displaystyle\int_{C}xdx.
\end{align*}
 {For $ m\in \mathbb{N}$, introduce the tuple of all moments of order $ m$ or less $$ \mo_{ m}(\mu ) = (\mo_{ \m}(\mu ), | \m | \leqslant m)$$ where $  | \m | : = \sum_{i}m_{ i}$ if these moments are well defined and finite. }Given a finite set of points seen as an atomic measure $ P =\sum_{i}\delta _{ x_{ i}}$ and a bounded set $ C\subset \mathbb{R}^{ d}$, say that $ P$ is  a  {\it $ m$-averaging set} for $ C$   if they share all moments up to order $ m$, i.e. $ \mo_{ m}(P) = \mo_{ m}(1_{ C}\Leb)$.  $ P$ is also called a   {\it $ t$-design}, see   \cite{LotzKlatt}  and references therein. They discuss the existence of such sets and show in Proposition C.10 that for $ m\in \mathbb{N}$, there is $ n_{ m,d}$ such that for $ n\geqslant n_{ m,d}$ and $ C\subset \mathbb{R}^{ d}$ a convex polyhedron, there exists an averaging set of cardinality $ n$ for $ C$, the important point being that $ n_{ m,d}$ does not depend on $ C.$
Now, assume one has for some $ m\in \mathbb{N},n\geqslant n_{ m,d}$ and each unit volume convex polyhedron $ C$ a  finite point process $ \P_{ m,n,C}$ with a.s. $ n$ points which is a.s. a $ m$-averaging set of $ C$ (this procedure is also detailed in  \cite{LotzKlatt}). Then the construction of the stationary point process goes as follows:\begin{itemize}
\item For each cell $ B_{ i}$ of $  \mathscr  C$, draw $ \P_{ m,n,B_{ i}}$ independently of   the others $B_{ j}$  and $ \P_{ m,n,B_{ j}},j\neq i$.
\item Define the random counting measure  $
 \P: = \sum_{i}\P_{ m,n,B_{ i}}. $
 
 \end{itemize}
 See  \cite[Figure 1]{LotzKlatt} for a visual illustration of this procedure.
We give below conditions under which $ \P$ is hyperuniform with exponent $ \alpha  \geqslant  2(m + 1)-d$. Let us give a few attention points first:\begin{itemize}
\item It is not required that $ \P_{ m,n,B_{ i}}\subset B_{ i}$, so that only the information $ \{\mo_{ m}(B_{ i}),i\geqslant 1\}$ is needed to implement the construction of  the point process $ \P_{ m,n,B_{ i}}$.   
\item If the points of $ \P_{ m,n,B_{ i}}$ are not contained in $ B_{ i}$, $ \P$ might not correspond to a simple point process, in the sense that points of $ \P_{ m,n,B_{ i}}$ and $ \P_{ m,n,B_{ j}}$ might fall on the same location for some $ i\neq j.$
\end{itemize}
Since the $ B_{ i}$ have negligible boundaries and form a partition, we can introduce the zero cell $ B_{ *}$ as the a.s. unique cell containing   $ 0.$    They introduce the following concept: for $ \alpha \geqslant 0,$ a stationary $ L^{ 2}_{ loc}$ random measure is  {\it beyond} $ \alpha $-hyperuniform, that we denote by $ \alpha ^{  + }$-hyperuniform, if the structure factor satisfies as $ \varepsilon \to 0$
\begin{align*}
 \S(B_{ \varepsilon }) = o(\varepsilon ^{ d + \alpha }).
\end{align*}
In particular, $ 0^{  + }$-hyperuniformity is plain hyperuniformity (Theorem \ref{thm:general-hu}).
\begin{theorem} [\cite{LotzKlatt}, Th. 5.2 and Ex. 5.7]
\label{thm:LotzKlatt}
 Let $  \mathscr  C = \{B_{ i};i\geqslant 1\}$ a random equivariant   fair partition such that each $ B_{ i}$ is a unit volume convex polyhedron. Let $ m\in \mathbb{N},n\geqslant n_{ m,d},\alpha  = 2(m + 1)-d$. Assume that 
\begin{align*}
 \mathbf E \left[
\text{\rm{diam}}(B_{ *})^{ \max(2 (m + 1) ,d)}
\right]<\infty ,
\end{align*}
and for each $ i$ let $ \P_{ m,n,B_{ i}}, i\geqslant 1$  be a point process that is a.s. a $ m$-averaging set for  $ B_{ i}$, independent conditionally on $  \mathscr  C$. Then, provided this construction is measurable,
\begin{align*}
 \P = \sum_{i}\P_{ m,n,B_{ i}}
\end{align*}
is a $ \alpha ^{  + }$-hyperuniform $ L^{ 2}_{ loc}$ stationary random measure, and corresponds to a hyperuniform simple point process   if a.s. $ \P_{ m,n,B_{ i}}\cap \P_{ m,n,B_{ j}} = \emptyset $ for all $ i\neq j$.
 \end{theorem} 
 
 We insist that  hyperuniformity holds even if $ \alpha  = 0$.
 One can notice that any method to draw a point $ Y_{ B_{ i}}$ independently in each cell  $ B_{ i}$ provides a ``cluster'' $ \P_{ 0,1,B_{ i}} = \{Y_{ B_{ i}}\}$, which ultimately yields, in dimension $ 1$ and $ 2$,  a  hyperuniform point process $ \P$. As noted   in  \cite{KLLY}, some choices might actually improve the  hyperuniformity. If for instance $ Y_{ B_{ i}}$ has the uniform distribution in $ B_{ i}$, then the method yields a  hyperuniform process in any dimension, whose  hyperuniformity exponent is $ 2$ under some moment assumption on the diameter of $ B_{ *}$  (this requires a specific proof  \cite{KLLY}, see Theorem  \ref{thm:KLLY} below).  Without entering in the details, the fact that all $ B_{ i}$ have a.s. identical volume yields that $ B_{ *}$ is also the law of the  {\it typical cell}  \cite{skm}.
 
  If on the other hand $ Y_{ B_{ i}}$ is deterministic conditionally on $ B_{ i}$ and placed  at the barycenter, i.e.
\begin{align*}
 Y_{ B_{ i}}: = \mo_{ (1,\dots ,1)}(B_{ i}),
\end{align*}
then $ \P_{ 1,1,B_{ i}} = \{Y_{ B_{ i}}\}$ is a $ 1$-averaging set, which provides a  hyperuniform process, in dimension $ d\leqslant 4$, with exponent $ \alpha \geqslant 4-d$. When one controls the correlation of some geometric quantities between distant cells, one can show that the  hyperuniformity exponent is higher, possibly up to $ 4$  \cite{torquato-tilings}. It is not clear in general what is the optimal  hyperuniformity exponent for this type of construction for a given original point process.

\begin{theorem} {  \cite[Ex. 4.7]{KLLY}}
\label{thm:KLLY}
   Let $  \mathscr  P = \{B_{ i};i\geqslant 1\}$ an equivariant fair partition of $ \mathbb{R}^{ d}$ with unit volume, and conditionally to $  \mathscr  P$, let $ X_{ i},i\geqslant 1$ independent points where $ X_{ i}\sim  \mathscr  U_{ B_{ i}}$. Then $ \P: = \{X_{ i};i\geqslant 1\}$ is a stationary hyperuniform point process.  Furthermore, its spectral measure is
\begin{align}
\label{eq:SF-partition-allocation}
 \S(du) = (1- \mathbf E  | \widehat{ 1_{ B_{ *}}}(u) | ^{ 2})du
\end{align}
where $ B_{ *}$ is the a.s. unique cell containing $ 0$, and $ \P$'s  hyperuniformity exponent is at most $ 2.$ $ \P$ is $ 2$-hyperuniform if furthermore 
\begin{align*}
 \mathbf E \left[
\displaystyle\int_{B_{ *}}\|x\|^{ 2   }dx<\infty 
\right].
\end{align*} \end{theorem}    
   
 We give the proof for completeness.

\begin{proof}
Let $ f$ a non-negative test function. Then the conditional variance formula gives
\begin{align*}
  \textrm{Var}\left(\P(f)\right) = &\mathbf E  \textrm{Var}\left(\sum_{i}f(X_{ i})\Big|  \mathscr  P\right) +  \textrm{Var}\left(\mathbf E\left[
 \sum_{i}f(X_{ i}) \Big| \mathscr  P
\right] \right).
\end{align*}
For the second term, we have conditionally on $  \mathscr  P,$
\begin{align*}
 \mathbf E\left[
 \sum_{i}f(X_{ i})
\right] = \mathbf E  \left[
\sum_{i}\underbrace{\textrm{Vol}(B_{ i})}_{ = 1} \int_{B_{ i}}f(x)dx
\right] = \int_{\mathbb{R}^{ d}}f
\end{align*}which has $ 0$ variance. For the first term, write $ x\sim y$ if $ x,y$ are in the same cell  of $  \mathscr  P.$ The first term gives, using stationarity and conditional independence
\begin{align*}
 \textrm{Var}\left(\P(f)\right) = &\sum_{i}  \textrm{Var}\left(f(X_{ i})\Big |  \mathscr  P)\right)\\
  = &\mathbf E \left[
\sum_{i}\int_{B_{ i}}f(x)^{ 2}dx
\right]-\mathbf E \sum_{i}\left(
\int_{B_{ i}}f
\right)^{ 2} \\
 = &\int_{}f^{ 2}-\mathbf E \left[
\sum_{i}\int_{B_{ i}\times B_{ i}}f(x)f(y)dxdy
\right]\\ 
  = &\int_{}f^{ 2}-\mathbf E \left[
\int_{\mathbb{R}^{ d}\times \mathbb{R}^{ d}}{\bf 1}\{ x\sim y \}f(x)f(y)dxdy
\right]\\
   = &\int_{}f^{ 2}-\int_{\mathbb{R}^{ d}\times \mathbb{R}^{ d}}\mathbf P (x\sim y)f(x)f(y) dxdy\\
    = &\int_{}f^{ 2}-\int_{\mathbb{R}^{ d}\times \mathbb{R}^{ d}}\mathbf P (0\sim z)f(x)f(x + z) dxdz.
\end{align*}
Hence with Definition \eqref{def:covariance}, and $ \mathbf P (0\sim z) = \mathbf P (z\in B_{ *}),$
\begin{align*}
 \C(dz) = \delta _{ 0}-\mathbf P (z\in B_{ *})dz
\end{align*}
which gives the announced spectral measure. 
 We have with  Lemma \ref{lm:quadra-spectral}, given that $ B_{ *}$ has volume $ 1$ a.s., the spectral density
\begin{align*}
\s(u) =  1- \widehat{ \mathbf E \left[
1_{ B_{ *}}
\right]}(u) = &\mathbf E \int_{B_{ *}}(1-e^{ -iu\cdot x})dx, u\in \mathbb{R}^{ d}\\
  = &\int_{\mathbb{R}^{ d}}(1-e^{ -iu\cdot x})\mathbf P (x\in B_{ *})dx\\
  \geqslant & \sigma \|u\|^{ 2}
\end{align*}
for some $ \sigma >0$, using the symmetry $ \mathbf P (x\in B_{ *}) = \mathbf P (-x\in B_{ *})$, provided by translation invariance. When 
\begin{align*}
 \mathbf E \displaystyle\int_{B_{ *}}\|x\|^{ 2   }dx<\infty ,
\end{align*}
doing a Taylor expansion and applying Lebesgue's theorem yields $ \s(0) = 0, \nabla \s(0) = 0$ by symmetry of $ \s$, hence 
\begin{align*}
 \s(u) =\mathbf E  \left[
\displaystyle\int_{B_{ *}}(u\cdot x)^{ 2}dx
\right](1 + o(1)),\end{align*}
indeed the dominating term does not vanish and the exponent is $ 2.$
\end{proof}

 \draft{
 The second procedure, by Gabrielli et al.   \cite{torquato-tilings}, is more powerful, as it potentially gives a  process with  arbitrarily high exponent. This procedure consists in placing a cluster of points in a specific way 
 independently in each cell of a fair allocation. Let us first give a self-contained result for when  one point is placed in each cell. In this case, to produce hyperuniformity, the point has to be placed at the gravity center of the cell. Note that, if cells are not convex, the gravity center might not be contained in the cell.
 
  For $ C\subset \mathbb{R}^{ d}$, if $ \int_{C}\|x\|dx<\infty $, introduce its  {\it centroid} 
 $ \text{\rm{\color{black} centr}}(C) = \int_{C}xdx$.
 For $ x\in \mathbb{R}^{ d}$, denote by $  { i_{ x}} = i_{ x}(  \mathscr  P)$ the (random) index with $x\in B_{ i_{ x}}$, and let $ D_{ x} = X_{ i_{ x}}-x$.  {\color{red} To be able to formulate a sufficient condition, denote the joint law of $ (D_{ x},D_{ y})$ by $ \mu _{0, x-y}$, and that of $ D_{ 0}$ by $ \mu _{ 0}$.}
 
  {\color{red} leave??}
 
  \begin{theorem}
  \label{thm:gravity-center}
  Let $  \mathscr  P =  \{B_{ i};i\geqslant 1\}$ an equivariant fair partition with unit cell volume. Assume \begin{itemize}
\item [(i)] $ \mathbf P (\|D_{ 0}\|>t) = O(t^{ -d-4})$
\item[(ii)]  {
$\displaystyle
 \int_{\mathbb{R}^{ d}}\displaystyle\int_{(\mathbb{R}^{ d} \times \mathbb{R}^{ d})^{ 2}}\|a\|^{ 2}\|b\|^{ 2} | \mu _{ z}-\mu _{ 0} \otimes \mu _{ 0} | (da,db)dz<\infty .
$ }
\end{itemize}
Then the (non-necessarily simple) point process $ \P = \sum_{i}\delta _{ X_{ i}}$ is  hyperuniform with exponent $ \alpha  \geqslant  4.$
  \end{theorem}  
Remark that the exponent can be in theory arbitrarily high, as $ \P$ can be for instance the shifted lattice, but it requires some unexcpected cancellations.

\begin{remark}[Discussion of (i) and (ii)]
\label{rk:poisson-centroids}
Many works give a bound on the tail of the cell's diameter  \cite{HPPS,CPPR,NSV}, see also a moment bound in Theorem \ref{thm:BDG}, from  \cite{BDG}, hence condition (i) seems possible to achieve. Condition (ii) is related to $ \beta $-mixing. It is more complicated to study correlations between cells at distant locations. A construction based on a non-hyperuniform point process $ \P_{ 0}$ is actually unlikely to satisfy (ii), as by Remark  \ref{rk:non-mixing}  (see  \cite{KLLY}) a transport from Poisson (or Poisson-like) to Lebesgue cannot be mixing, and one can see $ \P = \sum_{i}\delta _{ X_{ i}}$ as a perturbation of $ \P_{ 0}$. This procedure is still probably powerful,  \cite{torquato-tilings} argue that   the exponent is $ \alpha \in (4-d,4)$. Adapting the proof below, one can likely get an index $ \alpha >2$ depending on the growth rate of the covariance if (ii) is not satisfied.
\end{remark}

\begin{proof}First note that (i) implies that  a.s., for all $ i$, $ \int_{ B_{ i}}\|x\|dx<\infty $ hence $ X_{ i} = \text{\rm{\color{black} centr}}(B_{ i})$ is well defined.

We use the same method as previously, with $ f$ a test function with compact support.
This time, though, conditionaly on $  \mathscr  P$, the $ f(X_{ i})$ are deterministic. Assume without loss of generality $  { \rm supp}(f)\subset B_{ 1}   $. For a scaling $ R>1, $ let $ U_{ i}^{ R} := {\bf 1}\{ X_{ i}\in B_{ 2R}    \}$.   For a symmetric matrix $ H$ and $ v\in \mathbb{R}^{ d}$, write $ Hv^{ 2} = vHv^{ T}$.

By Taylor formula, since $  { \rm supp}(f_{ R})\subset B_{ R},   $\begin{align*}
\P(f_{ R}) = & \sum_{i} f_{ R}(X_{ i}) = \sum_{i}U_{ i}^{ R}f_{ R}(X_{ i})\\
  = &\sum_{i}U^{ R}_{ i}\int_{B_{ i}}f_{ R}(x)dx + \sum_{i}U^{ R}_{ i}\int_{B_{ i}}(f_{ R}(X_{ i})-f_{ R}(x))dx\\
   = &\int_{\mathbb{R}^{ d} }f_{ R} + \Delta _{ R} + \sum_{i}U^{ R}_{ i}\nabla f_{ R}(X_{ i})\cdot \underbrace{\int_{B_{ i}}(X_{ i}-x)dx}_{ = 0} + \sum_{i}U_{ i}^{ R}\int_{B_{ i}}\int_{0}^{ 1    }t\text{\rm{\color{black} Hess}}_{ f_{ R}}(x + t(X_{ i}-x))(X_{ i}-x)^{ 2}dxdt
\end{align*}
where
\begin{align*}
  | \Delta _{ R} |  = \left|
\sum_{i}U_{ i}^{ R}\int_{B_{ i}}f_{ R}-\sum_{i}\int_{B_{ i}}f 
\right|=\left|
   \int_{B _{ R}}  1_{X_{ i_{ x}}\notin B_{ 2R} }f_{ R}(x)dx  
\right| \leqslant  \|f\|\int_{B_{ R}}{\bf 1}\{ X_{ i_{ x}}\notin B_{ 2R} \}dx.
\end{align*}
Hence with Cauchy-Schwarz inequality and assumption (i)
\begin{align*}
 \mathbf E \Delta _{ R}^{ 2}\leqslant \|f\|^{ 2}\int_{B_{ R}^{ 2}}\sqrt{\mathbf P (\|D_{ x}\|>R)\mathbf P (\|D_{ y}\|>R)}dxdy\leqslant \|f\|^{ 2}\Leb(B_{ R})^{ 2}O(R^{ -d-4}) = O(R^{ d-4}).
\end{align*}

Regarding the other term, recall that $ \text{\rm{\color{black} Hess}}_{ f_{ R}} = R^{ -2}\text{\rm{\color{black} Hess}}_{ f}(R^{ -1}\cdot )$.  Let $ H_{ f}^{ R} (x) =  \text{\rm{\color{black} Hess}}_{ f}(R^{ -1}x)$. Define 
\begin{align*}
 \Xi^{ R}_{ x} = U_{ i_{ x}}^{ R}\int_{0}^{ 1    }t H_{ f}^{ R}(x + tD_{ x})D_{ x}^{ 2}dxdt.
\end{align*}
Then with the previous expansion
\begin{align*}
  \textrm{Var}\left(\P(f_{ R})\right) =& R^{ -4}  \textrm{Var}\left(\int_{\mathbb{R}^{ d}}  \Xi_{ x}^{ R}dx\right) + O(R^{d -4})\\
   = &R^{ -4} \int_{\mathbb{R}^{ d}\times \mathbb{R}^{ d}}  \textrm{Cov}\left(\Xi_{ x}^{ R},\Xi^{ R}_{ y}\right)dxdy + O(R^{d -4}).
\end{align*}
It remains to show that this integral is in $ O(R^{ d}).$

\begin{align*}
  \textrm{Cov}\left(\Xi_{ x}^{ R},\Xi^{ R}_{ y}\right)  = &\int_{0}^{ 1    }\int_{0}^{ 1    }ts \int_{\mathbb{R}^{ d}\times \mathbb{R}^{ d}}1_{x +  a\in B_{2 R}}1_{ y + b\in B_{ 2R}}H_{ f}^{ R}(x + ta)a^{ 2} H_{ f}^R{ (y + sb)}b^{ 2}(    d\mu _{ 0,x-y} -d\mu _{ 0}\times \mu _{ 0} )(a,b)   dtds\\
\int_{\mathbb{R}^{ d}\times \mathbb{R}^{ d}} \textrm{Cov}\left(\Xi^{ R}_{ x},\Xi^{ R}_{ y}\right)dxdy  
  \leqslant  &\|H_{ f}\|^{ 2}\int_{(\mathbb{R}^{ d}\times \mathbb{R}^{ d})^{ 2}}\|a\|^{ 2}\|b\|^{ 2} 
1_{x +  a\in B_{2 R}}1_{ y + b\in B_{ 2R}}
 | d\mu _{ 0,x-y} -d\mu _{ 0}\times \mu _{ 0} |(a,b)  dxdy\\
  \leqslant  &\|H_{ f}\|^{ 2}\int_{(\mathbb{R}^{ d}\times \mathbb{R}^{ d})^{ 2}}\|a\|^{ 2}\|b\|^{ 2} 
1_{x\in B(-a,2R) }   
 | d\mu _{ 0,z} -d\mu _{ 0}\times \mu _{ 0} |(a,b)  dxdz\\
\leqslant &\|H_{ f}\|^{ 2}\Leb(B(0,2R))\int_{\mathbb{R}^{ d}}{\int_{\mathbb{R}^{ d}\times \mathbb{R}^{ d}}\|a\|^{ 2}\|b\|^{ 2} | d\mu _{ 0,z} -d\mu _{ 0}\times \mu _{ 0} |(a,b)  }dz = O(R^{ d}).
\end{align*}
Hence we have $  \textrm{Var}\left(\P(f_{ R})\right) = O(R^{ d-4})$, thereby concluding the proof.
\end{proof}

Without any decay assumption like in (ii), one can still prove that the process is $ \alpha $-hyperuniform for $ \alpha \geqslant 4-d$ (interesting only if $ d = 2,3$).
 }

 \begin{exercise}
 Replace a random partition by a set of probability measures $ \pi  _{ i}$ such that $ \M: = \sum_{i}\pi  _{ i}$ is a \wsrm, then one obtains a WSSPP $ \P = \{Z_{ i}\}$, where $ Z_{ i}\sim \pi _{ i}.$ If $ \M$ is  hyperuniform, then so is $ \P.$
 \end{exercise}
 
 \begin{exercise}
 Instead of putting one $ Z_{ i}$ in each cell $ C(x_{ i})$, put a whole point process $ \{Z_{ i,1},\dots ,Z_{ i,m}\}$. Then it sill gives a  hyperuniform  point process . Or the $ Z_{ i,j}$ have to be independent for a given $ j?$
 \end{exercise}

  We present now a few methods to divide the space into random cells of equal volume, i.e. produce a fair partition. Along the same principles as for  {\it disordered} point processes (Section  \ref{sec:disordered}), we are looking for disorered partitions, not periodic tilings. There are other partitions  that we shall also try to avoid,    {\it  aperiodic tilings}, such as the  Penrose tiling  \cite{penrose-tilings},   also expected to yield a  hyperuniform random measure. In some constructions, we start from a unit intensity point process $ \P$, often Poisson, and build a fair partition adapted to $ \P.$ We present here two procedures where rigourous results were obtained, and give other possible constructions at Chapter  \ref{ch:numerics}.
  
  \subsection{Stable marriage}
  
  There is a simple and general way to build a partition from a stationary point process $ \P$. The construction is based on some parameter $ R$ going from $ 0$ to $ \infty $, which can be seen as the time of evolution, see Figure  \ref{fig:mariage}:\begin{itemize}
\item Grow a ball simultaneously and progressively $ B(x,R)$  around each point $ x\in \P$ as $ R$ grows from $ 0$ to $ \infty $.
\item The territory is progressively allocated to the first point claiming it. Therefore, even when $ R $ is such that $ \Leb(B(0,R)) = 1$, there might be points $ x\in \P$ with a territory volume $ <1$, because of other  (close-by) points getting first some territory contained in $ B(x,R)$.
\item Continue the growth and allocation in this way with $ R\to \infty $ until each point has an effective territory of volume $ 1.$
\end{itemize}
  This process is studied in  \cite{hoffman2006stable}. It is called Gale-Shapley algorithm in a non-spatial context, or  {\it stable marriage}, because it satisfies the following property: for any $ y\in \mathbb{R}^{ d}$ allocated to $ x\in \P$, the match $ y\to x$ is the best possible in the sense that  for any other $ y'\in \mathbb{R}^{ d},x'\in \P$, 
\begin{align*}
  | x-y | \leqslant \min(  | x'-y | , | x-y' | ).
\end{align*}
Note that there is a negligible set of unallocated points that we do not consider. Also, in general, many cells will not be topologically connected.
  The good news is that it converges under minimal assumptions. 
  
  \begin{figure}[h!]
  \begin{center}
  \includegraphics[width = 10cm]{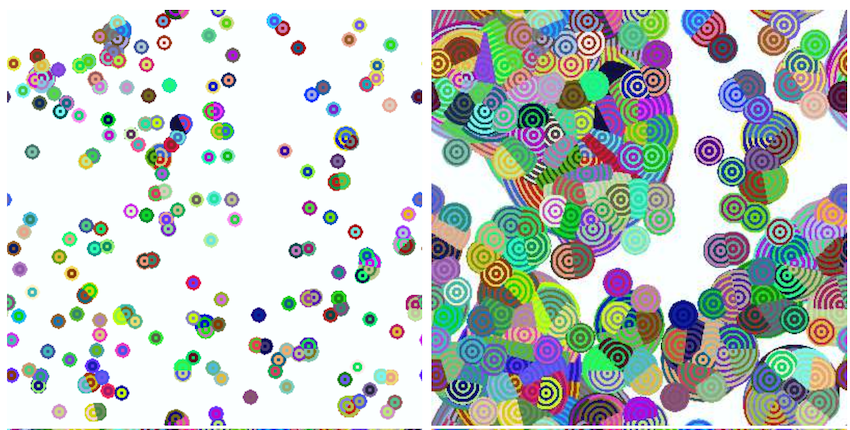}\caption{  \cite{hoffman2006stable}, Fig.2. The construction of the stable marriage allocation for two different values of $ R.$}
  \label{fig:mariage}
  \end{center}
  \end{figure}

  \begin{theorem}[  \cite{hoffman2006stable}, Th. 5] For $ \P$ a stationary point process, almost surely, in the stable mariage allocation, all points of $ \P$ are allocated a territory of volume $ 1$, forming a fair partition. Furthermore, a.s.,\begin{itemize}
\item Each territory is a finite union of bounded open connected sets,
\item Any bounded subset of $ \mathbb{R}^{ d}$ only intersects finitely many territories.
\end{itemize}

   \end{theorem}  
   
   As already explained, a partition is  {\it well-behaved} when one can obtain a good bound on the diameter of the typical cell, ideally when one can prove exponential decay.
   It has been proved that the stable mariage partition is badly behaved, even for a Poisson point process, and even when there are other well-behaved partitions.
 
  \begin{theorem}[  \cite{HPPS}] Let $ \P = \P^{ \Poi(1)}$, and $ X$  the element of $ \P$ to which $ 0$ has been allocated (well defined a.s.). Then \begin{itemize}
\item For $ d = 1,2, \mathbf E \|X\|^{ d/2} = \infty $
\item For $ d \geqslant 3, \mathbf E \|X\|^{ d} = \infty .$
\end{itemize}
   
   \end{theorem}  
   
   The results for $ d = 1,2$ actually align with what is known for optimal transport to Poisson points, see Theorem  \ref{AKT}, but it also known that in dimension $ d\geqslant 3$, there are well-behaved transport for Poisson point processes, see  \cite{AKT,CPPR} and the next subsection. 
   The results of  \cite{HPPS} are actually stated for a one-to-one matching between two independent realisations, but considerations of optimal transport yield that the result holds for allocations, see Proposition  \ref{prop:match-alloc}, or  \cite[(ii), p.270]{HPPS}, \cite{erbar2023optimal} for a general approach to optimal transport questions for stationary random measures.

  \subsection{Optimal transport Laguerre cells / power diagrams}
  \label{sec:laguerre-theor}
  
  Given some (finite or infinite) point configuration $ P = \{x_{ i};i\geqslant 1\}$, and some weights $ \mu  = (\mu _{ i},i\geqslant 1)$, the Laguerre cell of point $ x_{ i}$ within some window $ B\subset \mathbb{R}^{ d}$ is defined as 
\begin{align*}
 \text{\rm{\color{black} Lag}}(x_{ i}; P,\mu ) = \{y\in B : \|y-x_{ i}\|-\mu _{ i}<\min_{ j\neq i}(\|y-x_{ j}\|-\mu _{ j})\}
\end{align*}
the collection of all those cells form a partition of $ \mathbb{R}^{ d}$ up to a negligible set, and is the  {\it Laguerre tessellation} with centers $ x_{ i}$ and weights $ \mu _{ i}$,  {sometimes also called  {\it power diagrams}}. As already noticed, the ``centers'' $ x_{ i}$ need not belong to their own cell.
 For $ \mu _{ i} = 0$, it retrieves the well-known Voronoi tessellation based on $ P$ (in the Voronoi case, $ x_{ i}$ always belongs to its cell).

 A  non-intuitive fact is that the allocation obtained from $ \W_{ p}$-optimal transport, $ p>1$ (see Section  \ref{sec:linear-cost-Wp}), can be written as a Laguerre tessellation. 
For $ P = \sum_{i = 1}^{ n}\delta _{ x_{ i}}$   and a window $ B\subset \mathbb{R}^{ d}$ with volume $ n$, there exists a transport map $ T:W\to P$, unique up to a negligible set, such that 
\begin{align*}
 \W_{ p}^{ p}(\Leb1_{ W},P) = \sum_{i = 1}^{ n}\int_{B_{ i}}\|x-x_{ i}\|^{ p}dx
\end{align*}where $ B_{ i} = T^{ -1}(\{x_{ i}\})$ is the cell of points allocated to $ x_{ i}$, see  \cite{Sarrazin} and references therein.  {This procedure is illustrated with the two first images of Figure  \ref{fig:sarrazin} with $ p = 2$ for a non-homogeneous initial sample $ P$ on a cube $ B$.} The fact that $ T$ is a transport means that it induces a fair partition, i.e. $ \Leb(B_{ i}) = 1,i = 1,\dots ,n.$ 
 It is a consequence of the dual formulation of optimal transport    that, for some weights $ \mu  = (\mu _{ i},i = 1,\dots ,n)$,
\begin{align*}
 B_{ i} = \text{\rm{\color{black} Lag}}(x_{ i};P,\mu ).
\end{align*}
There is no tractable expression for the weights $ \mu _{ i}$, each weight depends on all the points of $ P$, to satisfy the constraint $ \Leb( \text{\rm{\color{black} Lag}}(x_{ i};P,\mu )) = 1$, and in general this induces long range dependency. We will see at Chapter  \ref{ch:numerics} that the cells can be found by gradient descent, or another iterative procedure.  It is also easy to prove that for the $ 2$-Wasserstein distance,  the cells $ B_{ i}$ are separated by hyperplanes, hence they are polygonal (Figure    \ref{fig:sarrazin}).   {When the $ x_{ i},i = 1,\dots ,n$ form a finite random sample, rigourous results on the cells shapes are  corollaries of previous results. In particular, Theorem  \ref{AKT} for  i.i.d. samples in dimension $ d\geqslant 3$ and  Theorem \ref{thm:BDG} for  hyperuniform samples in dimension $ 2$ give conditions under which the $ p$-th moment of the typical cell is finite.}

 Defining the infinite volume $   \mathsf W_{ 2}^{ 2}$-Laguerre tessellation for a stationary point process $ \P = \P^{ \Poi(1)}$ is far from trivial. When the linear transport cost is finite, e.g. for Poisson in dimension $ d\geqslant 3$ (Theorem  \ref{AKT}) or  hyperuniform processes in dimension $ 2$, such a construction is well defined  \cite{erbar2023optimal}. 
Since the $   \mathsf W_{ 2}^{ 2}$-optimal transport cost per unit volume between a Poisson point process in dimension $ 2$ and Lebesgue measure is infinite (Theorem  \ref{AKT}), such a construction is more uncertain. This does not prevent the image processing and optimal transport communities to obtain very good hyperuniformlity experimental results for the $ \W_{ 2}^{ 2}$-Laguerre tessellations for a large number $ n$ of  i.i.d points in a window of volume $ n$  \cite{Sarrazin,deGoes}.
  
  \subsection{The Coulomb fair partition} 
  \label{sec:coulomb-alloc}
  
  After an idea of Tsirelson, Nazarov, Sodin and Volberg \cite{NSV} introduced the idea of building a fair partition based on the harmonicity of Coulomb force fields (also said  {\it gravitational field}, or  {\it electrical field} in dimensions 
   $ 2$ and $ 3$) based on a stationary process $ \P$. We fix here the dimension $ d\geqslant 2$. Given a point  $ x\in \mathbb{R}^{ d}$, the repulsive Coulomb force field generated by $ x$ is $$ F_{ \{x\}}(y) = \tau _{ x}  \nabla \varphi _{ d} = \frac{ y-x}{\|y-x\|^{ d}},y\neq x,$$  where $ \varphi _{ d}$ is the Coulomb potential, see  \eqref{eq:coul-laplace}. The crucial mathematical aspect of this field is that it stems from a harmonic potential on $  \mathbb R  ^{ d}  \setminus  \{x\}$, more precisely, in the distributional sense, 
\begin{align}
 \label{eq:div-Fx}\text{\rm{div}}  F_{ \{x\}} =  \tau _{ x}\Delta \varphi _{ d}  = -d\kappa _{ d} \delta _{ x}
\end{align}  where $ \kappa _{ d} = \Leb(B_{ 1})$. Then, for a point configuration $ P$, the Coulomb force field generated by $ P$ is formally 
\begin{align*}
 F_{ P}(y) = \sum_{x\in P}F_{ \{x\}}(y),y\in \mathbb{R}^{ d}  \setminus P.
\end{align*}
Giving a sense to this summation for an infinite point configuration is far from obvious. For a (non-empty) stationary point process $ \P$, this sum does not converge absolutely, and not even sometimes in a weaker sense.  For instance if $ \P$ is Poisson, the sum converges   only in dimension $ d\geqslant 3$, in a sense described below, as first observed by Chandresakar  \cite{chandresakar}, further studied by Chatterjee, Peled, Peres, Romik   \cite{CPPR}. 
They show that the following limit exists a.s. and defines a $ \CC^{ 1}$ field:\begin{align}
\label{eq:def-electric-field}
F_{ \P}(y): =  \lim_{ R\to \infty }\sum_{x\in B(y,R)}F_{ \{x\}}(y), y\in \mathbb{R}^{ d}.
\end{align}
We would like to pass Property  \eqref{eq:div-Fx} to the sum, i.e. show that $ \text{\rm{div}}\sum_{x}F_{ \{x\}} = \sum_{x} \text{\rm{div}}F_{ \{x\}} = -d\kappa _{ d}\sum_{x}\delta _{ x}$, but this  turns out to be false. The moving center of summation makes the differentiation delicate, and  \cite{CPPR} show the following representation with a fixed summation center and an additional linear term
\begin{align}
\label{eq:fixed-repr-FP}
 F_{ \P}(y) = \lim_{ R\to \infty }\sum_{x\in B(0,R)}F_{ \{x\}}(y) + d\kappa _{ d}y.
\end{align}
This   reveals that in fact the divergence is a.s. in the distributional sense
\begin{align}
\label{eq:div-FP}
 \text{\rm{div}}F_{ \P} = -d\kappa _{ d}\sum_{x\in \P}\delta _{ x} + d\kappa _{ d}\Leb.
\end{align}
Even though the article  \cite{CPPR} is about Poisson processes, the method seems to be generalisable under mild assumptions in dimension $ d\geqslant 3.$
As we saw at chapter \ref{chap:transport}, dimensions $ 1$ and  $ 2$ have a peculiar behaviour  with respect to  hyperuniformity.  In dimension $ 2,$ Sodin, Wennman and Yakir \cite{SodinElectric} investigate the question of the existence of a field satisfying  \eqref{eq:div-FP} in all generality on $  \mathbb C $. Basic complex analysis allows to tackle the question under mild assumptions, but the requirement to have a field $ F_{ \P}$ that is stationary brings   additional difficulties. For the summation to make sense, a quantitative version of the  hyperuniform property is required for $ \P$, namely the  {\it finite Coulomb energy}, also related to properties of optimal transport, see Section \ref{sec:finite-Coulomb}.  Let us summarise the results in all dimensions:

\begin{theorem} \begin{itemize}
\item \cite{chandresakar} , \cite[Proposition 1]{CPPR}: Let $ \P = \Poi(1 )$   in dimension $ d\geqslant 3$. Then  the limits 
 \eqref{eq:def-electric-field} or  \eqref{eq:fixed-repr-FP}
 define a.s. a stationary field of class $ \CC^{1 }$ on $ \mathbb{R}^{ d}  \setminus \P$ whose distributional divergence is given by  \eqref{eq:div-FP}.
\item  \cite{SodinElectric} In dimension $ d = 2$, let $ \P$ a   weakly  stationary point process on $  \mathbb C $. Then there is a stationary field $ F_{ \P}$ satisfting \eqref{eq:div-FP} if $ \P$ has finite Coulomb energy, i.e. if \eqref{eq:cont-coulomb-energy} is satisfied. \end{itemize}
 
 \end{theorem} 
   On $  \mathbb C $,   Sodin et al.
  \cite{SodinElectric} show that the finite Coulomb energy condition is in fact equivalent to the existence of a random stationary meromorphic function $ F _{ \P}$  with poles at $ \P$ with unit residues and such that $ F _{ \P}(z) +2  \pi    \bar z$ is a random constant field with finite first moment, and that it is also equivalent to the a.s.  convergence of $ \sum_{z\in \P}\frac{ 1}{z}$. \\

  To extract a partition from this field, the central object is the trajectory of fictitious mass-free particles moving according to this force field, through the random differential equation
\begin{align}
\label{eq:coulomb-pde}
  \dot Y = -F_{ \P}(Y).
\end{align}
 Cauchy-Lipschitz conditions are satisfied on $ \mathbb{R}^{ d}  \setminus \P$ a.s., hence  there is   local unicity of the solution. One sees that a maximal trajectory of this PDE   necessarily ends at a singularity of the field, i.e. at a particle $ x$ of $ \P$, we say points of this trajectory are ``attracted'' by $ x$. The surprising fact is that if one regroups points of $ \mathbb{R}^{ d}  \setminus \P$ according to the endpoint of their trajectory, we obtain a fair partition.

  \begin{proposition}\label{prop:general-grav}  Let $ \P = \sum_{i}\delta _{ x_{ i}}$ a simple stationary point process of $ \mathbb{R}^{ d}$ with unit intensity for which some stationary $ \CC^{ 1}$ field $ F_{ \P}$ satisfies a.s.  \eqref{eq:fixed-repr-FP}. Call $ B_{ i}$ the  {\it attraction basin} of $ x_{ i}$, i.e. the set of points attracted to $ x_{ i}$
\begin{align*}
 B_{ i} = \{Y (s),s\geqslant 0:Y\text{\rm{\color{black}  satisfies  \eqref{eq:coulomb-pde} and }}\lim_{ t\to \infty }Y(t) = x_{ i}\}.
\end{align*}
 Then a.s. the $ B_{ i}$ all have volume $ 1$ and form an equivariant fair partition.
  \end{proposition}  
  The corresponding allocation $ T:B_{ i}\to \{x_{ i}\}$ is called  {\it Coulomb allocation} based on $ \P.$
  See the illustration of this construction at Figure  \ref{fig:loic-grav}. 
  \begin{figure}[h!]
  \begin{center}
  \includegraphics[width = 8cm]{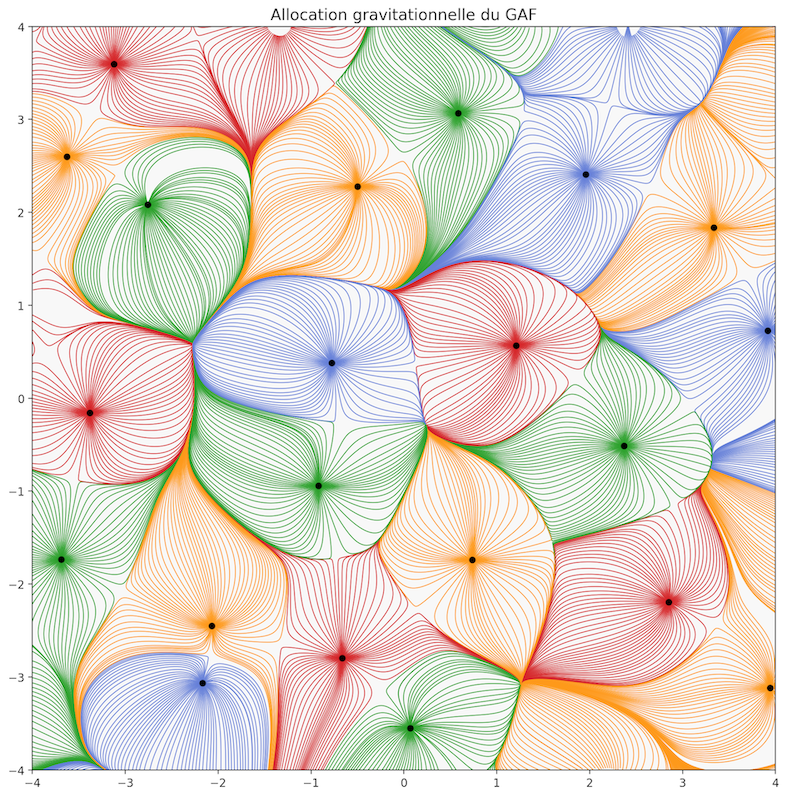}
  \caption{trajectories of Coulomb allocation in dimension $ 2$ for some zeros of Weyl polynomials ({\it simulation by L. Thomassey}).}
  \label{fig:loic-grav}
  \end{center}
  \end{figure}  The proof is at Section 4 of  \cite{CPPR},  based on Liouville's theorem. They also give the following heuristic explanation.
  \begin{proof}[Proof heuristics] Assume that the boundary of each attraction basin is of class $ \CC^{ 1}$, or at least piecewise $ \CC^{ 1}$. The key is to consider 
\begin{align*}
 a_{ i} := \displaystyle\int_{\partial B_{ i}} F_{ \P}(y)\cdot   \mathsf n_{ B_{ i}}(y)  \mathscr  H^{ d-1}(dy)
\end{align*} where $ \mathscr  H^{ d-1}(dy)$ is the $ (d-1)$-dimensional Hausdorff measure and $   \mathsf n_{ B_{ i}}(y)$ is the vector orthogonal to the boundary on the essential boundary $ \partial B_{ i}$ at $ y$, pointing towards the exterior of $ B_{ i}.$ The force flow $ F_{ \P}(y)$ cannot cross the boundary of $ B_{ i}$, otherwise it means that some particles could cross the border, which contradicts the definition of the $ B_{ i}.$ Therefore $ F_{ \P}(y)$ is always tangent to $ B_{ i}$ on the boundary, as illustrated by Figure  \ref{fig:loic-grav}, and $ F_{ \P}(y)\cdot   \mathsf n_{ B_{ i}}(y) = 0$, which implies that $ a_{ i} = 0$. On the other hand, Green formula gives with  \eqref{eq:div-FP}
\begin{align*}
a_{ i} = \displaystyle\int_{B_{ i}} \text{\rm{div}}(F_{ \P})(y)dy  =- \kappa _{ d}\P(B_{ i}) + \kappa _{ d}\Leb(B_{ i}) = -\kappa _{ d} + \kappa _{ d}\Leb(B_{ i}) ,
\end{align*}
since $ B_{ i}$ contains exactly one point of $ \P,$
which concludes with $ \Leb(B_{ i}) = 1.$

  The proof that the complement of the $ B_{ i}'s$, formed by all particles not ending up in a point of $ \P,$ is negligible, relies on a simple application of the mass transport principle: for $A,B\subset \mathbb{R}^{ d}$ introduce the quantity of mass sent from $ A$ to $ B$, additive in $ A$ and $ B:$
\begin{align*}
   M(A,B)  =   
\sum_{i}\delta _{ x_{ i}}(B)\Leb(B_{ i}\cap A).
\end{align*}
We have $ M(\tau _{ x}A,\tau _{ x}B) \smallequlaw M(A,B)$ for $ x\in \mathbb{R}^{ d}$ exploiting the stationarity of $ F_{ \P}.$ Since each $ B_{ i}$ has volume $ 1$, we have a.s. 
\begin{align*}
 \P(B) = \sum_{i}\delta _{ x_{ i}}(B) = M(\mathbb{R}^{ d},B) .
\end{align*}
Denoting by $ Q_{ \m} = \tau _{ \m}[0,1)^{ d}$,   the expected mass starting in $ Q_{ \m}$ does not depend on $ \m$ by stationarity and is $$ \lambda  := \mathbf E\left[
 M(Q_{ \m},\mathbb{R}^{ d})
\right]  =
\sum_{\k\in \mathbb{Z} ^{ d}} \mathbf E \left[M(Q_{ \m},Q_{ \k})
\right].$$ We can use the mass transfer principle:  for $ \k\in \mathbb{Z} ^{ d}$, using the stationarity of  $ \P$ and $ F_{ \P},$
\begin{align*}
 1 = \mathbf E \P(Q_{ \k}) = \sum_{\m\in \mathbb{Z} ^{ d}}\mathbf E M(Q_{ \m},Q_{ \k}) = \sum_{\m\in \mathbb{Z} ^{ d}}\mathbf E M(Q_{ 0},Q_{ \k-\m}) = \sum_{\k'\in \mathbb{Z} ^{ d}}\mathbf EM(Q_{ 0},Q_{ \k'}) = \lambda . 
\end{align*}
Hence $ \lambda  = 1$ and the average mass sent to $ \infty $ (i.e. not on a point $ x_{ i}$) in a unit cube is $1-\lambda  =  0$, and   the total mass of the set of all points sent to $ \infty $ is $ 0$ as well.
  \end{proof}

 We therefore know that a.s., for a.e.  $x\in  \mathbb{R}^{ d}$ the maximal trajectory going through $ x$ terminates in $ \P$, but this does not guarantee that the set of all trajectories ending up in some $ B_{ i}$ is bounded, i.e. there is no guarantee that $  { \rm diam}(B_{ i})<\infty $ a.s., this is a difficult question in general.  
 Let us give the main results of  \cite{CPPR,NSV}, which prove cells boundedness, and furthermore give exponential tails for the   diameter of the typical cell  for two particular models. Recall that $ B_{ *} $  is the cell containing $ 0$ a.s.

\begin{theorem}
\label{thm:grav-GAF-poisson-diameters}
Let $  \mathscr   C= \{B_{ i};i\geqslant 1\}$ the Coulomb partition associated to a point process $ \P.$ There exist $ C_{ -},C_{  + }<\infty $  depending on the dimension such that the following holds.
\begin{itemize}
\item  Case $\P =  \P^{ \Gaf}$  on $  \mathbb C $ \cite{NSV}:    Then a.s. the $ B_{ i}$ are bounded and 
\begin{align*}
C_{ -}\exp(-C_{  + }R\ln(R)^{ 3/2})\leqslant  \mathbf P (  { \rm diam}(B_{*})>R)\leqslant C_{  + }\exp(-C_{ -}R\ln(R)^{ 3/2})
\end{align*}
\item  Let $ \P = \P^{ \Poi(1)}$ in dimension $ d\geqslant 3$   \cite{CPPR}. Then a.s. the $ B_{ i}$ are bounded and
\begin{align*}
 \mathbf P (  { \rm diam}(B_{ *})>R)\leqslant C_{  + }\exp(-C_{ -}R(\ln(R))^{ \alpha _{ d}})
\end{align*}
where $ \alpha _{ 3} > \frac{ 4}{3}$ and $ \alpha _{ d} = \frac{ d-2}{d}$ for $ d\geqslant 4.$
\end{itemize}
 \end{theorem}  
This result hence proves that this fair allocation is well behaved, and confirms that in dimension $ d\geqslant 3$, the Poisson process is close to Lebesgue measure in terms of optimal transport, for any reasonable transport cost (see Theorem  \ref{AKT}). Note that for Poisson points, this rate is not optimal, in the sense that Marko and Timar  \cite{MarkoTimar} give a bound   in $ \exp(-cR^{ d})$, and it is easy to prove that this one is optimal, in the sense that such a bound is impossible if $ c$ is too large, using a nearest-neighbour argument.



%
%

\subsection{The hyperuniform Coulomb partition point process}  

One difficulty of building a  hyperuniform process with the procedure of Theorem \ref{thm:KLLY}   is that determining exactly the cells and drawing a point uniformly inside can be a daunting task in practice, especially for Coulomb partitions where cells are arbitrary non-convex sets with piecewise $  \mathscr  C^{ 1}$ boundary. Yet, there is a very simple procedure to simulate random variables with the uniform distribution in the Coulomb cells, based again on  Liouville's theorem. 
To that end, define for $ i\geqslant 1$  and $\theta \in   \mathbb  S  ^{ d-1}$ the repulsive trajectory $Y_{ i,\theta } (t),t\geqslant 0$, i.e.  satisfying $ \dot Y = F_{ \P}(Y)$, such that \begin{itemize}
\item  $ Y_{ i,\theta }$ starts in $ x_{ i}$: $ Y_{ i,\theta }(0) = x_{ i}$,
\item and in direction  $  \theta _{ i}$: $ \theta _{ i} = \| \nabla Y_{ i,\theta }(0)\|^{ -1} \nabla Y_{ i,\theta }(0) $
\end{itemize} see \cite[Section 8.1]{NSV} for details.
\begin{theorem}
 
\label{thm:PCPM}

Let $ \P = \{x_{ i}:i\geqslant 1\}$ a stationary point process for which a Coulomb partition $  \mathscr  C = \{B_{ i}\}$ as in Proposition  \ref{prop:general-grav}  is well defined.
 \begin{itemize}
\item Let $ \theta _{ i}$ i.i.d.   variables distributed according to the Haar measure $ \sigma _{ d}$ on the sphere $  \mathbb  S  ^{ d-1},i\geqslant 1,$
\item Let $ T_{ i}$ i.i.d.   exponential variables of parameter $ 1,i\geqslant 1$,
\item Let $ Z_{ i} = Y_{ i,\theta _{ i}}(T_{ i}).$ 
\end{itemize} Then, conditionaly on $ \P$, the $ Z_{ i}$ are independent and  $ Z_{ i}$ is uniformly distributed in $ B_{ i}$. By Theorem  \ref{thm:KLLY}, the point process $ \P': = \{Z_{ i};i\geqslant 1\}$ is stationary and hyperuniform.  If $ \P = \P^{ \Gaf}$  on $  \mathbb C $ or $\P = \P^{ \text{\rm{Poi}}(1)}$ in dimension $ d\geqslant 3$, $ \P$ is $ 2$-hyperuniform.

  \end{theorem}   
  
  The fact that  $ Z_{ i}$ is uniformly distributed in $ B_{ i}$
 results from the analysis at  \cite[Section 8.1]{NSV}. It is written for the GAF zeros, but the proof generalises easily. Then one can directly apply Theorem  \ref{thm:KLLY}.
In the Poisson case, the diameter bound from  Theorem  \ref{thm:grav-GAF-poisson-diameters} yields that $ \P'$ is also $ 2$-hyperuniform. 

}

   \chapter{Rigidity and stealthiness}
   \label{chap:rigid}
   
Another intriguing phenomenon has been noticed for some seemingly disordered  hyperuniform processes, that of rigidity. Ghosh and Peres
 \cite{GP17} have shown that, for either $ \P = \P^{ \Gaf}$ or $ \P = \P^{ \Gin}$, we have  {\it number rigidity}: the number of points in a ball can be a.s. guessed from the outside configuration, i.e. for $ R>0$, 
\begin{align}
\label{eq:intro-rigid}
 \P(B_{ R})\in \sigma (\P\cap B_{ R}^{ c}).
\end{align}
   Note that this property seems unrealistic for a standard process; the prototypical example is the Poisson process, where the independence of $ \P 1_{ B_{ R}}$ and $ \P1_{ B_{ R}^{ c}}$ makes  \eqref{eq:intro-rigid} impossible. This number rigidity is once again reminiscent of lattices, in the sense that for ``small enough'' perturbations $ U_{ \k},\k\in \mathbb{Z} ^{ d}$ with law $ \mu $, one expects that $ \Z^{ d,\mu } $ is number rigid. The shifted lattice (obtained for $ \mu  = \delta _{ 0}$) is an extreme example, as not only the number of points, but in fact all of $ \Z^{ d,\delta _{ 0} }1_{ B_{ R}}$ can be inferred from $ \Z^{ d,\delta _{ 0}}1_{ B_{ R}}^{ c}.$
   
   The oldest result in this line of research might be  by Aizenman and Martin \cite{AizMart}, about 1D Coulomb gases, but the systematic study of rigidity really started with the introduction of the notion of  {\it tolerance}   \cite{LyonsSteif,HolSoo}. As it turns out, the two rigid processes $ \P^{ \Gin},\P^{ \Gaf}$ investigated in  \cite{GP17} possess  also the property of  {hyperuniformity}.  Many models have been proven to be rigid since then, mostly in the realm of hyperuniformity: the most investigated class clearly is that of  determinantal point processes   \cite{BufAiry,BufBalanced, BufLinear,Buf-conditional,GhoshComplete,Ghosh-conditional}, due to its importance in many fields, in particular random matrices (see Section  \ref{sec:dpp}). Rigidity of  Coulomb and Riesz systems  has been investigated in \cite{DHLM,ChhaibiNaj,Cha17,DereudreVasseur,AizMart,Thoma}, and   quasicrystals  in  \cite{HartBjo,Lr24}. There are also many works on eigenvalues of random operators  \cite{rigidity-schrodinger}, Pfaffian processes  \cite{BufPFaff}, zeros of Gaussian   processes \cite{GP17,GK21,Lac20}, stable matchings  \cite{KLY},  and others  \cite{GL18,KlattLast}. In fact, the rigidity terminology is  more generally used to describe a strong form of hyperuniformity, i.e. a system of  particles  where the hyperuniformity exponent is one or several orders  below the Poisson behaviour  \cite{Cha17,BBNY,Ganguly-Sarkar}.
Besides the striking nature  of rigidity and its link with hyperuniformity, it has proven to be a useful property in other types of problems, such as continuous percolation  \cite{GKP}, or to establish the singularity of conditional measures  \cite{Ghosh-conditional}.  Osada \cite{OsadaRigid} recently proved a relation with diffusive dynamics of particle systems, Ghosh \cite{GhoshComplete} also exploited number rigidity to show the completeness of random sets of exponential functions, an interesting property in signal processing, Lyons \& Steif \cite{LyonsSteif} study this property to show phase uniqueness for some models from statistical physics.


Stronger forms of rigidity have also emerged. For instance, for the $ \P_{\textrm{GAF}}$ model,  Ghosh and Peres \cite{GP17} showed that not-only the number of particles in $ B(0,1)$ can be determined, but also the first moment, or center of mass $ \sum_{x\in \P_{\textrm{GAF}}\cap B(0,1)}x$, corresponding to {\it $ 1$-rigidity}, and that   no further moment can be determined from the observation of $ \P_{\textrm{GAF}}\cap B(0,1)^{c}$.    They proved  similarly that for the Ginibre process no moment can be determined beyond order $ 0$. Another such result appears in  \cite{DHLM} for Sine$ _{\beta }$ processes above order $ 0.$ Ghosh \& Krishnapur \cite{GK21} generalise this concept to that of $ k$-rigidity, where the $ k$-th order moment is determined, for some non-stationary models of GAF zeros.

We give in this chapter the heuristics behind the link between rigidity and  hyperuniformity, and give the necessary and sufficient conditions for $ k$-rigidity, in terms of the spectral measure. We finally explain some results around stealthy processes and quasicrystals and the phenomenon of  {\it maximal rigidity}, which means $ k$-rigidity for every $ k\geqslant 0.$

   \section{Heuristics and linear rigidity}
   \label{sec:heuristics-rigidity}
   
   Let us see how we can guess $ \P(B_{ R})$ for a  point process with low variance for linear statistics. Let $ f$ smooth with compact support and  $ f(0)=1$. Then as $ R\to \infty $, $ f_{ R}(x) \approx 1$ on $ B(0,1)$, hence  Bienaym\'e-Tchebyshev inequality  yields
\begin{align*}
 \P(B_{ 1}) \approx \P(1_{ B_{ 1}}f_{ R})  = \P(f_{ R})-\P(f_{ R}1_{ B_{ 1}}^{ c}) \approx \mathbf E \P(f_{ R}) \pm \sqrt{  \textrm{Var}\left(\P(f_{ R})\right)}-\P(f_{ R}1_{ B_{ 1}^{ c}}).
\end{align*}
For   $ \P$ sufficiently  hyperuniform (see Proposition  \ref{sec:hyp-expo}), $  \textrm{Var}\left(\P(f_{ R})\right)$ goes to $ 0$ and the right hand side gets arbitrarily $ L^{ 2}(\mathbf P )$-close to an element of $ \sigma (\P1_{ B_{ 1}^{ c}})$. This method is not optimal, in the sense that one need also to optimise on the function $ f$ simultaneously, as in  \cite{GP17,GL-sufficient} in the case where it is assumed that $ \S(du) \sim \sigma  \|u\|^{ 2}du$ as $ u\to 0$ for some finite $\sigma  $. This allows also for each $ R$ to give a consistent estimator of $ \P(B_{ 1})$ based only on finite data from $ \P1_{ B_{ R}  \setminus B_{ 1}}$ with a controllable error. Interestingly, this estimator only depends on $ \sigma  $, and $ \sigma ^{ 2}  = \lim_{ R\to \infty }  \textrm{Var}\left(\P(B_{ R})\right)$ (see  \eqref{eq:structure-def}).

   For $ \P^{ \Gaf}$,  the  hyperuniformity exponent is larger (see Proposition  \ref{lm:gaf-Delta}), Ghosh and Peres \cite{GP17}  show that one can furthermore guess  the first moment of $ \P1_{B_{R}}$ with a similar reasoning and $ \tilde f(x) = xf(x)$:
\begin{align*}
 \mo_{ 1}(\P^{ \Gaf}\cap B_{ 1}) = \int_{B_{ 1}}xd\P^{ \Gaf} \approx R\mathbf E[ \P( \tilde f_{ R})] \pm R\sqrt{  \textrm{Var}\left(\P( \tilde f_{ R})\right)}-R\P( \tilde f_{ R}1_{ B_{ 1}^{ c}}).
\end{align*}
   This phenomenon is called $ 1$-rigidity, since one can infer the moment of order $ 1$ based on the outside configuration. From this perspective, number rigidity can also be called $ 0$-rigidity. 
   
   In both examples, the rigidity can be  described as  {\it linear}, in the sense that the quantity of interest (the moment of order $ 0$ or $ 1 $ on the unit ball) can be approximated by an external linear statistic $ \P(f_{R}1_{B_{1}^{c}})$. Generalising to higher moments leads to the following definitions. Define 
\begin{align*}
 x^{ \k} = x_{ 1}^{ k_{ 1}}\dots x_{ d}^{ k_{ d}},x\in \mathbb{R}^{ d},\k\in \mathbb{N}^{ d}.
\end{align*}
   
 \begin{definition}Let $ R>0$ fixed.
 Say that a  \LSI random measure $ ~\M$ is $ k$-rigid on $ B\subset \mathbb{R}^{ d}$ for $ k\in \mathbb{N}$ iff
\begin{align*}
 \mo_{ k}(\M1_{ B})
\in \sigma (\M1_{ B}^{ c}).
\end{align*}
Say that it is furthermore {\it linearly} $ k$-rigid on $ B$ if for each   $ \k\in \mathbb{N}^{ d}$ such that  $ \sum_{i}k_{ i}\leqslant k$,  some kernels $ h_{ n}\in\mathscr C_{ c}^{ \infty }(B ^{ c})$ give the $ L^{ 2}(\mathbf P )$-approximation  \begin{align*}
\M(h_{ n})\xrightarrow[ n\to \infty ]{} \displaystyle\int_{B}x^{ \k}d\M .
\end{align*}
 \end{definition}

   \subsection{Necessary and sufficient conditions for linear $ k$-rigidity}  
   
   It turns out that linear rigidity on a compact subset $ B$ of $ \mathbb{R}^{ d}$ involves the tempered distributions whose spectrum is supported by $ B$, and those are by the Paley-Wiener theorem a subclass of the  analytic functions of exponential type. Exploiting this connection and the good properties of this class of functions, the following results are derived in  \cite{Lr24}.
   
   Let us first give a condition easy to state in some structurally more simple cases.  We consider hereafter a \wsrm~$ ~\M$. Decompose its  spectral measure $ \S$ according to the Radon-Nykodym theorem: there is a non-negative measurable symmetric function $ \s:\mathbb{R}^{ d}\to \mathbb{R}_{  + }$, with $ \S = \s\Leb + \S_{ s}$ and $ \S_{ s}$ is singular  with respect to $ \Leb, $
 we call $ \s$ the  {\it spectral density} of $ \M$. 
   \begin{theorem}
   \label{thm:rigid-simple-nsc}
    Let   $ k\in \mathbb{N}$.
    Assume that one of the following holds:\begin{itemize}
\item (i) $ k = 0$
\item (ii) $ \s$ is invariant under rotations
\item (iii) $ \s$  is  {\it separable}, i.e. $ \s(u)=\s_{1}(u_{1})\dots \s_{d}(u_{d}),u\in \mathbb{R}^{d}$, for some measurable symmetric $ \s_{i}\geqslant 0.$

\end{itemize}
Then for $ R>0,$ $ \M$ is linearly $ k$-rigid on $ B_{ R}$ if for all $ \varepsilon >0$
\begin{align}
\label{eq:simple-condition-k-rigidity}
 \int_{B(0,\varepsilon )}\frac{ \|u\|^{ 2k}}{\s(u)}du=\infty .
\end{align}
 {As noted in  \cite[Prop. 3.17]{LotzKlatt},  \eqref{eq:simple-condition-k-rigidity} is also implied by the $ (2k + d)$-hyperuniformity of $ \M$.} In particular, a $ d$-hyperuniform process is always number rigid.
    \end{theorem}  
    
    \begin{remark}
    \label{rk:rigid}
\begin{itemize}
\item This condition only bears on the continuous part of the spectral measure. It is reminiscent of Wiener, Kolmogorov and Krein's theorems on time series  \cite{Kolmo-rigid,Wiener,Krein}, and there is indeed a relation, as we shall see at Section  \ref{sec:stealthy}.
\item  We can show in greater generality that rigidity obeys some rules of monotonicity:  for $ \s'\leqslant \s$, $ k$-rigidity for $ \s$ implies $ k$-rigidity for $ \s'$, and for $ k'\leqslant k$, $ k$-rigidity for $ \s$ implies $ k'$-rigidity for $ \s.$ 
\end{itemize}
    \end{remark}

    \begin{example}
    We have built at Sections
 \ref{sec:cluster-lattice} and  \ref{sec:fair-partitions}   classes of models  which can have an arbitrarily high hyperuniformity exponent. Therefore these models can also exhibit an arbitrarily high order of rigidity.    \end{example} 
   
   When none of (i),(ii),(iii) holds, one can still state a more algebraic necessary condition, see  \cite{Lr24}. It is difficult to give in general a necessary condition, there is at \cite[Section 4]{Lr24} the example of a non-atomic $ 0$-rigid random measure which is not  hyperuniform, i.e. whose spectral density does not vanish in $ 0$; wether this is possible for a point process is an open question.
Let us give a framework in which we can assess the necessity of the previous conditions.
 
 \begin{proposition}[  \cite{Lr24}, Proposition 2]
 \label{prop:necessary-rigid}
 Assume the spectral density $ \s$ satisfies for some $ c,p,\varepsilon >0$:\begin{itemize}
\item $ \s$ has finitely many zeros $ u_{ 1},\dots ,u_{ m}\in \mathbb{R}^{ d}$ and they have finite order, i.e. 
\begin{align*}
 \int_{B(u_{ i},\varepsilon )}\frac{ \|u-u_{ i}\|^{ p}}{\s(u)}du<\infty 
\end{align*}
\item $ \s$ does not vanish too fast at $ \infty $: for $ u \notin \bigcup _{ i}B(u_{ i},\varepsilon )$, $ 
 \s(u)\geqslant c(1 + \|u\|)^{ -p} $.
\end{itemize}
Then $ \M$ is not maximally rigid on $ B_{ R}$ for $ R>0$, and for $ k\in \mathbb{N}$, $ \M$ is not linearly $ k$-rigid if  \eqref{eq:simple-condition-k-rigidity} does not hold. 
 \end{proposition}

  This yields the following characterisation of rigidity for DPPs, using  \eqref{eq:dpp-SF} and Lemma \ref{lm:quadra-spectral}:
 
 \begin{theorem}[\cite{Lr24}, Theorem 5]
 \label{Thm:DPP-rigid}
  Let $ \P$ a stationary DPP with   Hermitian kernel $ K$ such that $ x  \mapsto   K(0,x)$ is square integrable. 
  Then $ \P$ is  number rigid in dimension $ d=1,2$  if and only if  $ \P$ is  hyperuniform and for all $ \varepsilon >0,$
\begin{align*}
\int_{B_{\varepsilon }}\frac{1}{\s(u)}du=\infty ,
\end{align*}and is not linearly $ k$-rigid for $ d\geqslant 3$ or $ k>0$.
  \end{theorem}  
  The square integrability condition is only useful for the necessity part, to be able to use the conclusion of Proposition  \ref{Thm:DPP-rigid}.
As seen at Sections  \ref{sec:ginibre} and \ref{sec:gue-as-dpp}, the $ \text{\rm{Sine}}_{1}$ process is a DPP with kernel $$ K(x,y)=\frac{\sin(\pi (x-y))}{\pi (x-y)}$$ on $ \mathbb{R}$, and the Ginibre processes with kernel $$ K(z,w)=\pi ^{ -1}e^{z  \bar w- | z | ^{2}/2- | w | ^{2}/2}$$ on $  \mathbb C .$ Hence one can show with  \eqref{eq:dpp-SF} that they are both number rigid.
An example of a non-rigid  hyperuniform DPP is given by the tensor product of the cardinal sine kernel 
\begin{align*}
 K(x,y) = \text{\rm{sinc}}(x_{ 1}-y_{ 1})  \text{\rm{sinc}}(x_{ 2}-y_{ 2}),(x_{ 1},x_{ 2}),(y_{ 1},y_{ 2})\in \mathbb{R}^{ 2},
\end{align*}
because the spectral measure is continuous with density
\begin{align*}
 \s(u_{ 1},u_{ 2}) 
 =   | u_{ 1} |  +  | u_{ 2} |  + o(u_{ 1} ) + o(u_{ 2} )\text{\rm{ as }}u\to 0,
\end{align*}
hence  $$ \int_{B_{ \varepsilon }}\|u\|^{ -1}du<\infty .$$

   The main reason why DPPs cannot be more then number rigid,  and only in dimension $ 1$ or $ 2$, and why it is in general difficult to find such processes in general, is Lemma  \ref{lm:quadra-spectral}, that yields $ \s(u)\geqslant \sigma \|u\|^{ 2}$ as $ u\to 0$; a similar phenomenon occurs for independently perturbed lattices.\\

 The proof of the above results rely on an analysis in the Hilbert space $ L^{2}(\S)$, and   on the relation with analytic functions of exponential type through Paley-Wiener theorem. In \cite{Lr24}, we derive a general condition for linear $ k$-rigidity in terms of the local behaviour of the spectral measure $ \S$ around $ 0$.

  \subsection{Reconstruction procedure}
  
Linear rigidity on $ A$ implies that a linear statistic supported by $ A$, such as the number of points in $ A,$ can be approximated by linear statistics supported by $ A^{ c}$. Finding them is a matter of Hilbert analysis in $ L^{ 2}(\S)$. Say for instance one wishes to know the mass $ \M(B_{ 1})$ based on the information contained in $ B_1^{ c}$. Then one must project $ { 1_{ B_{ 1}}}$ in a Hilbert basis $ (h_{ k})_{ k\geqslant 1}$ of functions  supported by $ B_{ 1}^{ c}$, where the scalar product is $$ \langle h,h' \rangle _{ \S}=  \textrm{Cov}\left(\M(h),\M(h')\right) = (2\pi )^{ -d}\int_{ } \hat h \hat h' d\S.$$ Then the reconstruction is 
\begin{align*}
 \M(1_{ B_{ 1}}) = \sum_{k} \langle 1_{ B_{ 1}}, h_{ k} \rangle \M(h_{ k}).
\end{align*}
In many cases, it mostly depends on the behaviour of $ \S$ around $ 0$, see Section  
   \ref{sec:heuristics-rigidity}. So, if one assumes $ \S(du)\sim \kappa \|u\|^{ \alpha }$ as $ u\to 0$ for some $ \alpha >0$, one can find the estimator with a more simple form of $ \S$ having the same behaviour around $ 0$.
   \subsection{Proof heuristics}
   \label{sec:prf-rigidity}

The heuristics of Section \ref{chap:rigid} only work under some structural assumptions on $ \S$. We prove here that $ 0$-rigidity on 
$ B_{ 1}$ is equivalent to $ \int_{B(0,\varepsilon )}\s^{ -1} = \infty $ for all $ \varepsilon >0$, under the sole assumption that $ \M$ is not maximally rigid. Maximal rigidity is further studied below, it is quite an extreme phenomenon that can be discarded pretty easily for most standard models, and it does not hold under the assumptions of  Proposition  \ref{prop:necessary-rigid}.
   
Predicting $ \M(B_{ 1})$ by a linear statistic of $ \M1_{ B_{ 1}^{ c}}$       is a problem of linear interpolation, we seek for $ h\in \mathscr B_{ c}(B_{ 1}^{ c})$ minimising 
\begin{align}
\label{eq:1}
 \mathbf E \left[
 | \M( h)-\M(B_{ 1}) | ^{ 2}
\right] = (2\pi )^{- d}\int_{\mathbb{R}^{ d}} | \hat h- \widehat{ 1_{ B_{ 1} } }| ^{ 2}d\S.
\end{align}
In the following, we replace $ 1_{ B_{ 1} }$ by a Schwartz function $ \gamma $ taking value $ 1$ on $ B_{ 1}$; knowing that $ \M1_{ B_{ 1}^{ c}}$ approximates $ \M(\gamma )$ or $ \M(1_{ B_{ 1}})$ is strictly equivalent. To prove that the infimum is $ 0$, we must prove that   $ \hat \gamma $ is in the $ L^{ 2}(\S)$-closue of such $ \hat h$, or equivalently that $ \hat \gamma $ is orthogonal to any $ \varphi \in L^{ 2}(\S)$ that is itself orthogonal to all such $ \hat  h$. Let therefore such a $ \varphi $, meaning for $ h\in \mathscr B_{ c}(B_{ 1}^{ c})$, 
\begin{align*}
 \int_{\mathbb{R}^{ d}} \hat h\varphi d\S = 0.
\end{align*}
With the Parseval formula, it means $ \psi : = \varphi \S$ has a spectrum supported by $ B_{ 1}$. The very important point is Schwartz'Paley-Wiener theorem (\cite[Th. 19-(c)]{Rudin2}), that states that a tempered distribution with bounded spectrum is in fact analytic, i.e. $ \varphi \S = \psi \Leb$ where $ \psi $ is an analytic function, and that furthermore it is of type $ 1$, i.e. for some $ C,N>0,$
\begin{align*}
  | \psi (z) | \leqslant C\|z\|_{  \mathbb C ^{ d}}^{ N}\exp(\|z\|_{  \mathbb C ^{ d}}),z\in  \mathbb C ^{ d},
\end{align*}where $ \|(z_{ 1},\dots ,z_{ d})\|^{ 2}_{  \mathbb C ^{ d}} = \sum_{i}|z_{ i} | ^{ 2}$. Remark that since $ \varphi \S$ is a continuous measure, it does not put mass on the singular part of $ \S$, if it exists, hence we equivalently write $ \psi  = \varphi \s$. The second point is that $ \psi \in L^{ 2}(\s^{ -1})$ because 
\begin{align*}
 \int_{} | \psi  | ^{ 2}\s^{ -1} = \int_{}\varphi ^{ 2}\s <\infty .
\end{align*}

The assumption \eqref{eq:simple-condition-k-rigidity} therefore yields $ \psi (0) = 0.$ Finally, we have
\begin{align*}
\langle \hat \gamma , \psi  \rangle  = \int_{} \gamma  \hat \psi = \int_{B_{ 1}} \gamma \hat \psi  = \int_{} \hat \psi   = \psi (0) = 0,
\end{align*}       
using that $ \hat \psi $ is supported by $ B_{ 1}$. $ \hat \psi $ is indeed orthogonal to all such $ \psi $, meaning the infimum in  \eqref{eq:1} is indeed $ 0.$

\begin{longversion}{The necessity part will be proved in a future version.}
 
For the converse statement, assume $ \int_{B(0,\varepsilon )}\s^{ -1}<\infty $ for some $ \varepsilon >0$. Let us prove non-$ 0$-rigidity. By the discussion below  \eqref{eq:1}, it suffices to find $ \psi \in L^{ 2}(\s^{ -1})$ such that $ \psi (0)\neq 0.$ Since we assumed that there is no maximal rigidity, there exists $ \psi \in L^{ 2}(\s^{ -1})$ 

\end{longversion}
\begin{remark}
Establishing the conditions for $ \k$-rigidity might get more technical. Let us simply mention that under assumption \eqref{eq:simple-condition-k-rigidity}, $ \psi \in L^{ 2}(\s^{ -1})$ implies $ \partial ^{ \k}\psi (0) = 0$, and for $ \gamma(u) $ a Schwartz function coinciding with $ u^{ \k}$ on $ B_{ 1}$, 
\begin{align*}
 \langle \gamma ,\psi  \rangle = (-i)^{ k}\partial ^{ \k}\psi (0) = 0,
\end{align*}
this allows to show the  sufficiency of the condition.
\end{remark}

   \subsection{Linear and non-linear rigidity of perturbed lattices}

   We can deduce from Example  \ref{ex:IPL}, Lemma  \ref{lm:quadra-spectral} and Theorem  \ref{thm:rigid-simple-nsc}  the following.
For $ \mu $ a non-negative measure on $ \mathbb{R}^{ d}$,   call $ \mu _{ c}$ its component continuous  with respect to Lebesgue measure and say $ \mu $ has a continuous part if $ \mu _{ c}$ is not the null measure.

\begin{proposition}
\label{cor:linear-rigid-lattices}
    Let $ \mu $ a symmetric probability measure on $ \mathbb{R}^{ d}$ with a continuous part. Assume furthermore that $ \mu $ is symmetric or has a finite second moment.
   Let $ \Z^{ d,\mu }$ the corresponding perturbed lattice (Example \ref{ex:IPL-intro}). Then $ \Z^{ d,\mu }$ is not linearly  $ k$-rigid if $ k\geqslant 1$ or if $ d\geqslant 3$.
    
   For $ k = 0,d\in \{1,2\}$,  let  $ \psi (u) = \int_{\mathbb{R}^{ d}}e^{- iu\cdot t}\mu (dt)$. 
Then $ \Z^{ d,\mu }$ is linearly number rigid  iff  for all $ \varepsilon >0$
\begin{align*}
 \int_{B_{ \varepsilon }}\frac{ 1}{1- | \psi (u) | ^{ 2}}du = \infty .
\end{align*}  
\end{proposition}

\begin{proof}
We use the spectral measure expression at \eqref{eq:IPL-SF} with Theorem  \ref{thm:rigid-simple-nsc} and Proposition  \ref{prop:necessary-rigid}.
Let $ f$ the density of $ \mu _{ c}$, i.e. $ \mu  = f\Leb + \mu _{ s}$ for some singular measure $ \mu _{ s}$, and
 $ p := \mu  _{ s}(\mathbb{R}^{ d})<1$. The Riemann-Lebesgue lemma yields that 
\begin{align*}
\int_{}e^{ -iu\cdot x}f(x)dx\to 0
\end{align*}
as $ \|u\|\to \infty $, hence for $ \|u\|$ sufficiently large, $  | \psi (u) | ^{ 2}\leqslant p + \varepsilon <1$ for $ \varepsilon >0$ sufficiently small. Furthermore, $  | \psi (u) | \neq 1$ for $ u\in \mathbb{R}^{ d}$ because $ \mu $ cannot be concentrated on $ \{x: \langle x,u \rangle\in 2\pi \mathbb{Z} \}$ since the continuous part of $ \mu $ has a support with positive Lebesgue measure. Therefore, the spectral density $ \s$ satisfies the two  conditions in Proposition  \ref{prop:necessary-rigid}. 
 
  Lemma  \ref{lm:quadra-spectral} yields that for $ \|u\|$ sufficiently small, $  | \psi (u)  |^{ 2} \leqslant  1-\sigma \|u\|^{ 2}$ for some $ \sigma >0$. Hence 
\begin{align*}
 \int_{}\frac{ \|u\|^{ 2k}}{1- | \psi (u) | ^{ 2}}du <\infty 
\end{align*}
if $ k\geqslant 1$ or $ d\geqslant 3$.
Then for $ k = 0$, Theorem  \ref{thm:rigid-simple-nsc}-(i) yields number rigidity iff the integral is infinite.
\end{proof}

In disordered models, most rigidity proofs actually yield linear rigidity  \cite{GP17,ChhaibiNaj,GL-sufficient}, but other   arguments   use a different route. In   \cite{DHLM}, the authors prove number rigidity for Coulomb gases using DLR equations, and  \cite{ChhaibiNaj} proved that this rigidity is actually linear. Proving that a model does not experience  rigidity actually gives information on its covariance decay  \cite{Lr24}.

For perturbed lattices, conversely, many models seem to experience non-linear rigidity. Consider for instance in any dimension a probability measure $ \mu $ supported by $ B_{ 1/4}$. Then clearly the resulting process $ \Z^{d,\mu }$ is number rigid, as one can unambiguously recover the underlying structure and find the perturbed $ \mathbb{Z} ^{d}$-neighbour of each particle, and in dimension $ d\geqslant 3$, by the results above, this number rigidity is not linear.

There is also the notable example of  \cite{PS14} that proves with an ad-hoc argument that in dimension $ d = 3$, a lattice perturbed by Gaussian variables with sufficiently small variance is number rigid, and this rigidity is also necessarily non-linear by Proposition  \ref{cor:linear-rigid-lattices}.

\begin{theorem}[  \cite{GP17}]
Let $ \mu  = \N(0,\sigma ^{ 2}I_{ d})$
on $ \mathbb{R}^{ d}$. Then $ \Z^{ d,\mu }$ is number rigid  if and only if $ d\in \{1,2\}$ or if $ d = 3$ and $ \sigma <\sigma _{ c}$ for some $ \sigma _{ c}\in (0,\infty )$. 
 \end{theorem}  
   \begin{question}
    What are the number rigid independently perturbed lattices in dimension $ d\geqslant 3?$
   \end{question}

   \section{  Stealthy processes and  maximal rigidity}  
   \label{sec:stealthy}
   
   We saw that the speed of decay in $ 0$ of the structure factor determines the degree of rigidity of the random structure. An extreme case is when the spectrum vanishes identically on a neighbourhood of the origin.  
   
   \begin{definition}
   A \wsrm~~is  {\it stealthy} if its spectral density vanishes on a non-empty open set.
   \end{definition}  
   We dropped the usual assumption that the gap contains $ 0$ as it might be non-relevant mathematically, but it is often present in the physics literature. Stealthy point processes have attracted considerable interest in physics, in particular for their peculiar optical properties, see the non-exhaustive bibliographic sample   \cite{TZS,ZST,ZST-I,Tor16b,Tor18,Morse,casiulis} . It bears relations with the concept of  {\it blue noise} in image analysis, numerical integration, optimal transport and quantization   \cite{Klatt-nature,deGoes,Sarrazin},  since they yield  superpolynomial decay for the variance of smooth linear statistics  by Proposition \ref{prop:hyp-expo}.

Stealthy random measures exhibit a very strong form of rigidity, called maximal rigidity, where the restriction of $ \M$ on a subset $ B\subset \mathbb{R}^{ d}$ can be completely inferred from values of $ \M$ outside $ B$ \cite{GL18}, it is also a consequence of Theorem \ref{thm:rigid-simple-nsc}.

\begin{definition}
Say that a \wsrm~$ ~\M$ is  {\it maximally rigid} on $ B\subset \mathbb{R}^{ d}$ if $ \M1_{ B}\in \sigma (\M1_{ B}^{ c})$.
\end{definition}  
  
  When $ B$ is bounded, maximal rigidity could be denoted by $ \infty $-rigidity in the sense that it corresponds to $ k$-rigidity for all $ k\in \mathbb{N}.$
Stealthy measures are actually even more rigid than that, in that the set $ B$ can be non-compact.
   Call strictly convex cone  a closed cone $ C\subset \mathbb{R}^{ d}$ which does not contain both $ x$ and $ -x$ for some $ x\neq 0$, or equivalently such that for some $ x_{0}\in C$, 
\begin{align*}
\inf _{x\in C,\|x\| = 1} x_{0}\cdot x>0.
\end{align*}
   \begin{proposition}
   A stealthy \wsrm~is maximally rigid on any strictly convex cone $ C$.
   \end{proposition}
   This result is proved in    \cite{rigid-companion}, where it is also shown that maximal rigidity on a bounded set persists when the stealthiness hypothesis is relaxed to that of a deep zero (i.e. where all derivatives vanish).  It is argued there that maximal rigidity comes from an instance of the uncertainty principle in harmonic analysis called  {\it annihilating pairs}, or  {\it uniqueness pairs}, putting in relation a set $ B$ with a set $ \hat B\subset \mathbb{R}^{ d}$ such that if a test function $ f$ is supported by $ B$ and $ \hat f$ is supported by $ \hat B$, then $ f\equiv 0.$

   Stealthy processes have other striking features that distinguish them from the common point processes, even the perturbed lattices.
   The bounded holes property  \cite{ZST}, giving somehow an a.s. lower bound on the density,  was  proved in  \cite[Theorem 1.1]{GL18}. Formally it means that there is $ R>0$ such that a.s. in every ball with radius $ R$ there is at least one particle.  \cite[Theorem 1.3]{GL18} also give an upper bound on the density. Let us give a simple statement and proof.
   \begin{proposition}
   \label{prop:stealthy-density}
   Let $ \M$ a stealthy non-negative \wsrm. Then there exist finite $ A ,R,\eta >0 $  such that a.s., for every $ x\in \mathbb{R}^{d}$, $ \M(B(x,1))\leqslant A$ and $ \M(B(x,R))\geqslant \eta .$
   \end{proposition}In the case where $ \M$ is a point process, $ \M(B(x,R))\in \mathbb{N} $, hence the second statement gives the  {\it bounded holes property}: there is a.s. no ball with radius $ R$ without point.
   \begin{proof}Assume without loss of generality unit intensity.
The stealthiness yields  $ \varepsilon >0$ such that,  with  \eqref{eq:structure-def}, for $ \varphi $ a non-negative Schwartz function supported by $ B_{ \varepsilon }$,   $  \textrm{Var}\left(\M( \hat \varphi )\right)=0$. 
Hence a.s. 
\begin{align*}
\M( \hat \varphi )=\mathbf E \M(  \hat  \varphi )=\int_{} \hat \varphi.
\end{align*}
To make sure $ \hat \varphi \geqslant 0$, take $ \varphi  = \varphi _{ 0}  \ast  \varphi _{ 0}$ for some smooth $ \varphi _{ 0}$ supported by $ B_{ \varepsilon /2}$. Assume also $ \varphi _{ 0}\geqslant 0$ and $ \varphi (0)  >0$.  Also,  $ \hat \varphi (0) = \int_{}\varphi >0$ and there is $ \kappa ,a>0$ such that 
\begin{align*}
 \hat \varphi \geqslant \kappa 1_{ B_{ a}}.
\end{align*}If $ a <1,$
one can still shrink $ \varphi $ by factor $ a $, hence $ \hat \varphi  $ expands by a factor $ 1/a$, meaning we can assume without losing the previous sign properties that $ a= 1.$ 
By positivity of $ \M$ and $ \hat \varphi  =   | \hat \varphi _{ 0} | ^{ 2} $,
\begin{align}
\label{eq:bd-hole}
 \kappa \M(B_{ 1})\leqslant \M( \hat \varphi ) = \int_{} \hat \varphi <\infty 
\end{align}which yields the first part of the statement by stationarity. Since $ \varphi $ is non-negative, $ \sup\hat \varphi  =  \hat \varphi (0).$
We have for $ R>0$
\begin{align}
\label{eq:bd-hole2}
0<\int_{} \hat \varphi  = \M( \hat \varphi )\leqslant \hat \varphi (0) \M(B_{R})+\M( \hat \varphi 1_{B_{R}^{c}}).
\end{align}
Since $ \hat \varphi $ has fast decay and the $ B(\k,1) ,\k\in d^{ -1/2}\mathbb{Z} ^{ d},$ cover the space, we have for some $ c>0$
\begin{align*}
\hat \varphi 1_{B_{R}^{c}}(x)\leqslant  \sum_{\k\in d^{ -1/2} \mathbb{Z} ^{d}  \setminus B_{R}}c \|\k\|^{-d-1}1_{B(\k,1)}(x),x\in \mathbb{R} ^{d}.
\end{align*}
It finally yields with stationarity,  \eqref{eq:bd-hole} and  \eqref{eq:bd-hole2}
\begin{align*}
\int_{} \hat \varphi \leqslant  \hat \varphi (0)\M(B_{R})+ \sum_{\k\in d^{ -1/2} \mathbb{Z} ^{d}  \setminus B_{R}}c\|\k\|^{-d-1}\M(B(\k,1 ))\leqslant  \hat \varphi (0)\M(B_{R})+cA\sum_{\k\in d^{ -1/2} \mathbb{Z} ^{d}  \setminus B_{R}}\|\k\|^{-d-1}.
\end{align*}
The series converges, hence the rest goes to $ 0,$ and $ \M(B_{R})\geqslant \eta : =\int_{} \hat \varphi / 2\hat \varphi (0)$ for $ R$ large enough. 
   \end{proof}
   
      The proof can be refined to optimise the values of $ \eta ,R,A$, see  \cite{GL18,Coste}.
      
      \subsection{Further questions on stealthy point processes}  
      
      Spectral considerations therefore give the main properties of stealthy random measures. If on the other hand one focuses on the class of point processes, which is the prominent class studied in the literature,  deeper questions arise.
      
Experimenters in physics and image analysis have generated very large point samples that seem to be disordered and exhibit a stealthy behaviour, i.e. a flat spectrum at the origin  \cite{gaussian-blue-noise,Morse,casiulis}.
   The interest of practitionners is to put in evidence stealthy models of point processes which are also disordered, e.g. isotropic and mixing, through some simulation procedures. At Section \ref{sec:grad-desc}, we give a heuristic explanation why such disordered stealthy samples should exist, confirmed by  the results of numerical experiments.  
   To the author's knowledge, the only known mathematical examples of stealthy point processes on $ \mathbb{R}^{ d}$ are finite unions of shifted lattices as in Example \ref{ex:irrational-lattices} (with only finitely many $ a_{ m}$)  or perturbations thereof by stacked-sliders \cite{ZST-I}. 
   
      \begin{question}
 Are there disordered stealthy  hyperuniform processes in dimension $ d\geqslant 2?$.
\end{question}
Regarding the  meaning of {\it disorder}, see the discussion at Section \ref{sec:disordered}. Mixing would be already be very strong, even if of course Brillinger mixing would be ideal (see Section  \ref{sec:CLT-mixing}).   
Here again, this question is trivial if one considers  random measures in general. A Gaussian field whose spectral measure is a non-negative Schwartz function supported by, say, $ B_{ 2}  \setminus B_{ 1}$, is stealthy and as mixing as possible (it is actually $ 2$-dependent). There are also purely atomic non-negative  random measures  that also satisfy this property,     but they are likely not mixing, either by Proposition \ref{prop:mixing-spectral} because they have spectral atoms  \cite{KurSarnak}, or because the construction originates from a lattice  \cite{LotzKlatt}; still, the latter example has a continuous spectrum, but it is not isotropic.      

Another question concerns the transport properties of stealthy processes. Using the strategies of Section  \ref{sec:hu-is-lattice}, it should be possile to show that for each $ p\geqslant 1$, stealthy point processes are $ L^{ p}$ perturbed lattices, i.e. of the form $ \Z^{   \mathsf U}$ where $   \mathsf U = \{U_{ \k};\k\in \mathbb{Z} ^{ d}\}$ is stationary and $ \mathbf E \|U_{ 0}\|^{ p}<\infty $. 
 \begin{question}
Can a stealthy  point process always be represented as a $ L^{ \infty }$-perturbed lattice, i.e. with the $ U_{ \k}$ a.s. bounded by some finite $ A>0?$
\end{question}
This would corroborate the  upper and lower density bounds of Proposition  \ref{prop:stealthy-density}, and it is not contradicted by Example  \ref{ex:irrational-lattices} with a finite number of $ a_{ i}$, which is actually the more complex examples we have of a stealthy  point process to date.
 
 A final question is about the class of sets where stealthy processes are maximally rigid. We saw that they are maximally rigid on closed strictly convex cones, but they might be interpolable on larger classes, such as complements of convex cones, called major cones, like for quasicrystals (see below). For instance, stealthy processes from Example  \ref{ex:irrational-lattices} are maximally rigid on the complement of an infinite large strip, i.e. on $ (\mathbb{R}  \setminus [-A,A])\times \mathbb{R}^{ d-1}$, for $ A$ sufficiently large.
  \begin{question}
 Are stealthy point processes maximally rigid on major cones? Or on larger sets?
 \end{question}
       
       \section{Quasicrystals}  
       \label{sec:QC-rigid}
       We saw at Section  \ref{sec:quasicrystals} that quasicrystals provide a great deal of  hyperuniform  point processes, more on the ordered side. The cut-and-project models are in general not stealthy, as their spectral measure has a dense support. Nevertheless, as we saw before, it is often the continuous part of the spectral measure, i.e. the spectral density $ \s$, that bears informations on the rigidity behaviour. In this respect, quasicrystals are extremely rigid since their spectral density identically vanishes. We give the following result, which is applicable to random measures with zero spectral density. Note that such examples need not be  hyperuniform.
       \begin{theorem}[\cite{rigid-companion}] Let $ \M$ a \wsrm.   If $ \s$ is purely atomic, $ \M$ is maximally rigid on the complement of any cone with non-empty interior.        
        \end{theorem}  
        Hence in our current knowledge, for quasicrystals, the maximal rigidity situation is even more extreme than for stealthy processes: it is enough to know the process on an arbitrarily small convex cone with non-empty interior to uniquely determine its values on the whole space.  {See  \cite[Proposition 3]{Lr-max-rigid} for an interpretation of this property for  {\it deterministic} quasicrystals}. Once again, one might ask the question of the minimal class of sets ensuring maximal rigidity.

        \begin{longversion}
         
         Quasicrystals might provide models of stealthy measures that are not a finite union of shifted irrational lattices (example ...). Kurasov and Sarnak gave examples of point configurations $ \P  = \sum_{x\in \Lambda }\delta _{ x}$ whose spectrum $ \S$ is a locally finite measure and which are not of the form ... This was thouroughly investigated by Lev and Olevskii and Wayne M. Lawton1
, August K. Tsikh2. [Fourier quasicrystals and discreteness
of the diffraction spectrum] proved that if it is additionally assumed that  $ \P$ is uniformly discrete, then we have a union of shifted lattices, after a conjecture of Lagarias (and also in this case the spectrum is automatically UD).

This does not address the existence of quasicrystals with a continuous part in their spectrum.
         
        \end{longversion}
%
%

       \chapter{Quantitative aspects } \label{ch:numerics}
       
        We shall explore quantitative aspects of  hyperuniformity, such as the problem of estimating the structure factor and detecting mere  hyperuniformity, or simulating large samples. 
      {In this chapter,   we observe a finite measure $ \M_{ \lambda }$ on  the window $   \mathbb  T _\lambda  ^{ d} = [-\lambda /2,\lambda/2 )^{ d},\lambda >0$,   to exploit the spectral analysis of Section \ref{sec:hu-finite} on the dual space $ \mathbb{Z} _{ \lambda }^{ d}$. Even if the model is originally defined in another shape, such as the ball $ B(0,\lambda /2)$, it is still formally possible, and sometimes mathematically useful, to embed it in the torus $ \T$. Sometimes we assume that  $ \M_{ \lambda }$ is invariant under translations of $  \mathbb  T  _{ \lambda }^{ d}$ with periodic boundary conditions, called $  \mathbb  T  _{ \lambda }^{ d}$-invariance.  This is the case if for instance $ \M_{\lambda } = \M1_{  \mathbb  T  _{ \lambda }^{ d}}$ for some stationary process $ \M$ on $ \mathbb{R}^{ d}$, but for many large finite systems $ \{\M_{ \lambda },\lambda >0\}$, stationarity  is only true asymptotically.} We do not discuss the intensity estimation topic and assume that the intensty is known, we make the renormalisation assumption
\begin{align*}
 \mathbf E \left[
\M_{ \lambda }(\T)
\right] = \lambda ^{ d}.
\end{align*}
We give here a brief introduction, definition of main estimators and a basic study of the bias. We refer to  \cite{mastrilli-minimax} for a good introduction to the estimation of the structure factor and the hyperuniformity exponent,  and the related literature.
       
 \section{Spectrum estimation and the \texttt{structure-factor} package}

The objective of this section
is to estimate the second order properties of a random measure $ \M_{ \lambda }$, and the most basic tool is the  {\it scattering intensity}, already introduced at  \eqref{eq:def-scatt-intensity} $$ \hat \s_{ \M_{ \lambda }} (u) = \lambda ^{ -d} | \F\M_{ \lambda }(u) | ^{ 2}, u\in \mathbb{R}^{ d}.$$ 
It is clearly related to the spectral measure of  \eqref{eq:def-spectral-density}. In particular if $ \M_{ \lambda } $ is translation invariant, we have immediately on the dual space 
\begin{align*}
 \mathbf E \left[
\hat \s_{ \M_{ \lambda }}(u) 
\right]= \s_{ \M_{ \lambda }}(u),u\in \mathbb{Z} _{ \lambda }^{ d}  \setminus \{0\}.
\end{align*}
Also, in this case, from  \eqref{eq:coste-weak-cvg-SI}, when $ \M_{ \lambda }$ is the restriction of a stationary measure on $ \mathbb{R}^{ d}$, the scattering intensity is unbiased at the so-called  {\it allowed wavevectors} $ u$, i.e. $ u\in \mathbb{Z} _{ \lambda }^{ d}$. It is shown in  \cite{HGBLr} that it is actually enough to have   one coordinate on $\mathbb{Z} _{ \lambda }^{ 1} =  2\pi \lambda ^{ -1}\mathbb{Z} $, i.e. if $ u$ is in  the set 
\begin{align*}
   \mathbb{ A}_{ \lambda }: = \{u\in \mathbb{R}^{ d} :  \exists  i, u_{ i}\in \mathbb{Z} _{ \lambda }^{ 1}  \setminus \{0\}\}.
\end{align*}

 { \begin{proposition}[\cite{HGBLr}, Prop. 1]
  \label{prop:SI-grid}
Let $ \M$ a weakly stationary random measure  whose covariance measure $ \C$ is integrable, let $ \M_{ \lambda }$ be its restriction on $  \mathbb  T  _{ \lambda }^{ d}$. Then its spectral measure  $ \S$ has a density $ \s$ with respect to  $ \Leb $, and for $ u\in \mathbb{R}^{ d}$
\begin{align*} \left|
\mathbf E\left[
 \hat \s_{ \M_{ \lambda }}(u)
\right]-\s(u)
\right| =    \left | 
\prod_{i:u_{ i}\neq 0}\frac{e^{ iu_{ i}{ \lambda}}-1}{u_{i}\lambda}
\right | ^{2} + \varepsilon _{ \lambda }(u)
\end{align*}
with $ \lim_{ \lambda \to 0}\sup_{ u\in \mathbb{R}^{ d}}\varepsilon _{ \lambda } (u)= 0.$
In particular,
\begin{align*}
\sup_{ u\in    \mathbb{ A}_{ \lambda }}\left|
\mathbf E \hat \s_{ \M_{ \lambda }}(u)-\s (u)
\right|   \xrightarrow[ \lambda \to \infty ]{}0. 
\end{align*}
Also, for $ u\in \mathbb{R}^{ d}$ fixed,  $ \mathbf E \left[
\hat \s_{ \M_{ \lambda }}(u)
\right]\to \s(u)$ 
 if and only if  more than half of $ u$'s coordinates  do not vanish.
  \end{proposition}
  
Hence,   as noted at  \cite[8.4.11b]{DVj08}, the bias vanishes everywhere (except at $ u = 0$) in dimension $ 1.$  The estimation at $ u = 0$	requires a specific technique, and it is in general a bad idea to estimate the spectral density at $0<  \|u\|<2\pi \lambda ^{ -1}$ because the observation window $ \T$ is not large enough to reflect the behaviour at such large wavelengths.

  \begin{proof}  Define $$ f_{ u,\lambda}(x) = 1_{  \mathbb  T _\lambda ^{ d}}(x)e^{ iu\cdot x}$$ so that $   \hat \s_{ \M_{ \lambda}}(u) = \lambda^{ -d}  | \M(f_{ u,\lambda}) | ^{ 2}$.
Let us use   \eqref{def:covariance}: 
\begin{align*} 
\lambda^{ -d}\E | \M(f_{ u,\lambda}) | ^{ 2} = & \lambda^{ -d} \textrm{Var}\left(\M( {f_{ u,\lambda}})\right) + \lambda^{ -d}\left | 
\mathbf E \M(f_{ u,\lambda})
\right | ^{ 2}\\
 = &\lambda^{ -d}\int_{ }  \displaystyle\int_{}1_{ \T}(y)e^{ iu\cdot y}1_{ \T}(x + y)e^{- iu\cdot (x + y)} dy\C(dx) +\lambda^{ -d} \left | 
 \widehat{ 1_{  \mathbb  T  ^{ d}_\lambda}}(u)
\right | ^{ 2}\\
  =&\int_{} e^{ -iu\cdot x} \underbrace{\lambda^{ -d}\Leb( \tau _{ x}\T\cap \T) }_{1 + o(1)}\C(dx) +\lambda^{ -d} \left  |  
\lambda^{d} \widehat{ 1_{ \mathbb  T  ^{ d}_{ 1}}}(u\lambda  )
\right  |  ^{ 2}\\
 =  & \s(u)   + o\left(
\displaystyle\int_{}  | e^{ -iu\cdot x} |  | \C | (dx)
\right)+ \lambda^{  d}  \left | 
\prod_{i:u_{ i}\neq 0}\frac{e^{ iu_{ i}{ \lambda}}-1}{u_{i}\lambda}
\right | ^{2},
\end{align*}
where the $ o(\cdot )$ comes from Lebesgue's theorem, using $  | \C | (\mathbb{R}^{ d})<\infty $.
The last term vanishes if $ u\in    \mathbb{ A}_{ \lambda }$, and otherwise oscillates around $  \lambda^{ 2(d-m)}$ where $ m$ is the number of $ i$ such that $ u_{ i}\neq 0$.
  \end{proof}
  
In the physics literature, it is common to estimate the structure factor and assess hyperuniformity at such lattice points. In conclusion, if $ \M$ is  hyperuniform, $ \s_{ \M_{ \lambda}}(u)$ should be small for $ u\in  \mathbb A_{ \lambda }$ close to $ 0$, in particular for  $    u \in \mathbb{Z} _{ \lambda}^{ d}$.
}

 For pointwise estimation, this estimator yet suffers a few caveats, and  estimators are generally built from averaged smoothed versions of the scattering intensity.      For a more stable and accurate approximation, using a terminology coming from time series, it is recommanded to use  {\it tapered versions}  \cite{Percival+Walden:2020}, i.e. a square integrable properly renormalised function $ \varphi :   \mathbb  T _\lambda ^{ d}\to \mathbb{R}$, and 
\begin{align*}
  \hat \s^{ \varphi }_{ \M_{ \lambda}}(u) =  \left | \int_{  \mathbb  T  _{ \lambda}^{ d}}e^{ -ix\cdot u} \varphi (u)d\M_{ \lambda}(u)\right| ^{ 2}. 
\end{align*}We recover for instance $  \hat \s_{ \M_{ \lambda}} =  \hat \s^{ \varphi_{0}}_{ \M_{ \lambda}}$ with $ \varphi_{ 0} = \lambda^{ -d/2}1_{  { \mathbb  T  ^{ d}_{ \lambda}}}$ ($ \lambda $ is implicit in this notation). A more stable estimation is obtained with multitapered estimators, i.e.  averages of tapered estimators over a finite family of orthogonal tapers $ (\varphi_{ 1},\dots ,\varphi_{ q})$, 
\begin{align*}
  \hat \s_{\M_{ \lambda}}^{ \varphi_{ 1},\dots ,\varphi_{ q}} = \frac{ 1}{q}\sum_{j = 1}^{ q}  \hat \s^{ \varphi_{ j}}_{ \M_{ \lambda}}.
\end{align*}
This procedure requires debiasing, resp. inside or outside the $  | \cdot  | ^{ 2}$, yielding resp. a  {\it directly debiased tapered / multi-tapered estimator} (DDT or DDMT), or  {\it undirectly debiased} (UDT or UDMT).

See  \cite{hawat2023estimating,olhede,mastrilli-minimax} for further mathematical investigations of the convergence of this estimator and its variants. In particular \cite{mastrilli-minimax}  gives optimal minimax bounds for spectral measure estimation and shows that it is impossible to obtain a consistent estimator of $ \s(u)$ in a non-parametric setting.\\

 These estimators are implemented in the \texttt{structure-factor} package  \footnote{\url{https://pypi.org/project/structure-factor/}}, along with several built-in sample distributions to test them out, see the corresponding article  \cite{hawat2023estimating}, in particular Sections 5,6, for an in-depth discussion of their performances. We reproduce at Figure  \ref{fig:SF-package} the results of multi-tapered estimators  $ \hat \s^{ \varphi_{ 1},\dots ,\varphi_{ 4}}_{ \P_{ \lambda}}$  with 4 orthogonal sine tapers, for several models of point processes,  hyperuniform or not, isotropic or not. We show the radial profile of the estimator, i.e. each gray point $ (\rho ,s)$ represents a couple value $ (\|u\| , \hat \s_{ P}(u))$ for the sample $ P\in  \mathscr  N(  \mathbb  T  _{ \lambda}^{ d})$ represented in the first line, for some multi-tapered estimator $ \hat  \s$ applied to $ P$. The last three rows represent three different estimators. Note that, as one can visually detect from the top left square, the partial matching process  (see Section \ref{sec:partial-matching}) on the left is not isotropic, which likely yields some oscillations on the radial projection of the spectral measure.

\begin{figure*}[!ht]
  \vspace*{-0.2cm}
  \begin{tabular}{p{\dimexpr 0.05\textwidth-\tabcolsep}p{\dimexpr 0.22\textwidth-\tabcolsep}p{\dimexpr 0.22\textwidth-\tabcolsep}p{\dimexpr 0.22\textwidth-\tabcolsep}p{\dimexpr 0.22\textwidth-\tabcolsep}}
    \multirow{3}{*}{\rotatebox[origin=l]{90}{Sample}}                                  &
    \raisebox{-\height}{\includegraphics[width=0.9\linewidth]{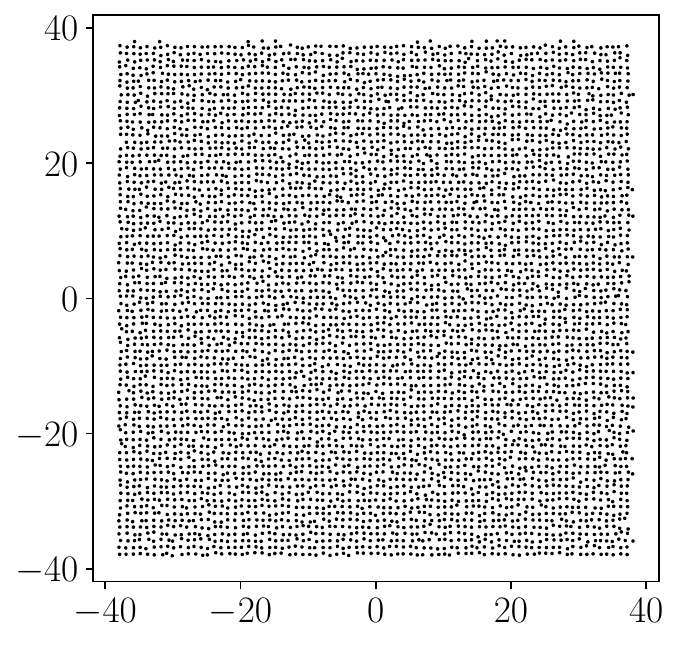}}     &
    \raisebox{-\height}{\includegraphics[width=0.9\linewidth]{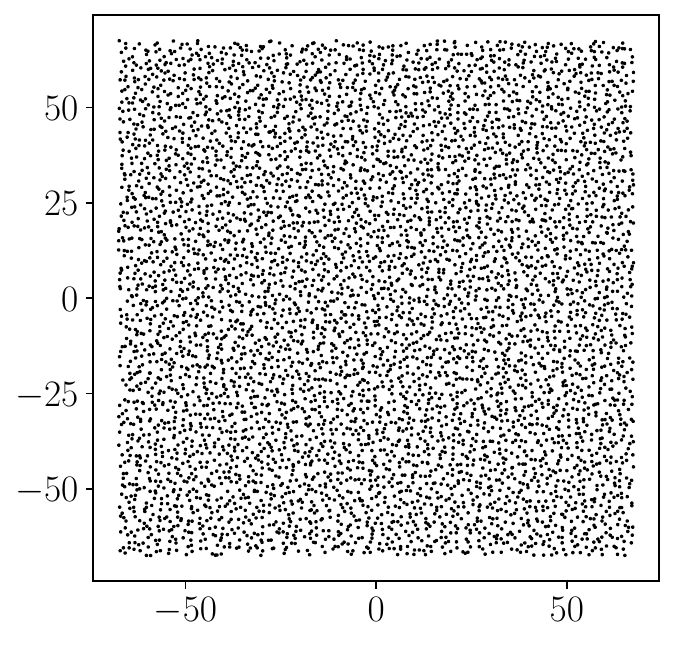}} &
    \raisebox{-\height}{\includegraphics[width=0.9\linewidth]{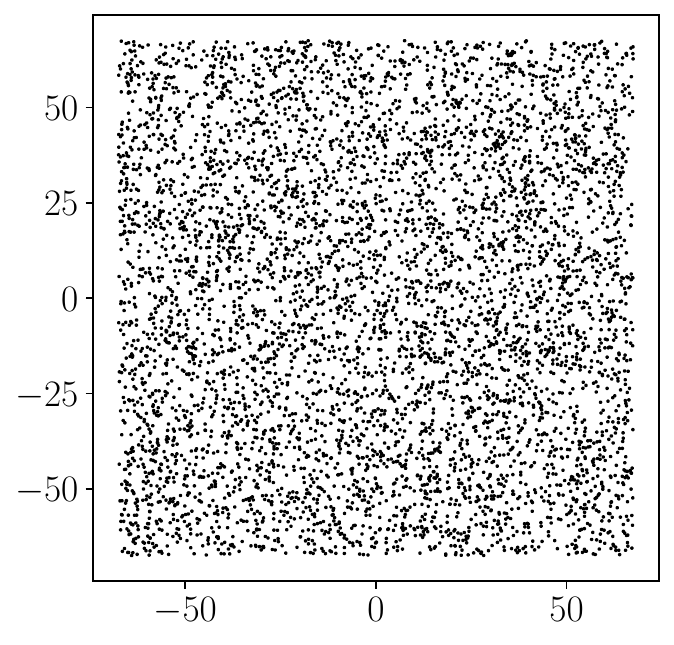}} &
    \raisebox{-\height}{\includegraphics[width=0.9\linewidth]{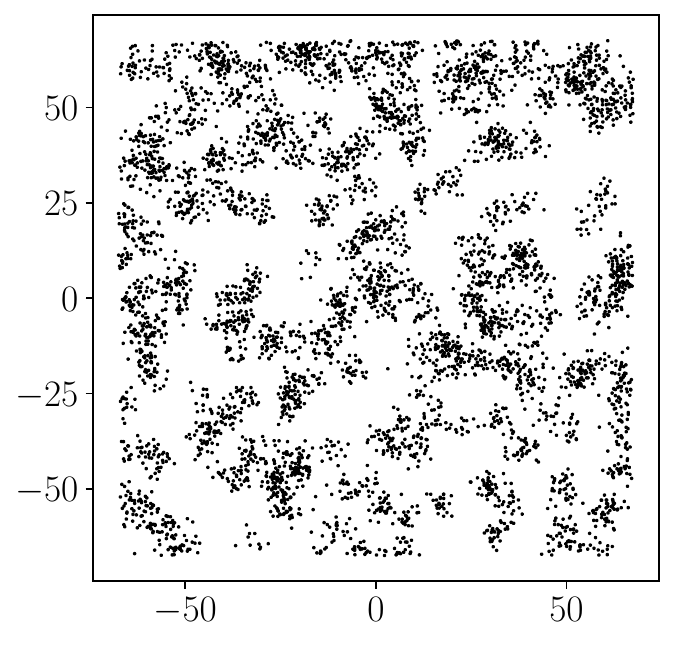}}
  \end{tabular}
  \vspace*{-0.2cm}
  \begin{tabular}{p{\dimexpr 0.05\textwidth-\tabcolsep}p{\dimexpr 0.22\textwidth-\tabcolsep}p{\dimexpr 0.22\textwidth-\tabcolsep}p{\dimexpr 0.22\textwidth-\tabcolsep}p{\dimexpr 0.22\textwidth-\tabcolsep}}
    \multirow{6}{*}{\rotatebox[origin=c]{90}{MT}}             &
    \raisebox{-\height}{\includegraphics[width=1\linewidth]{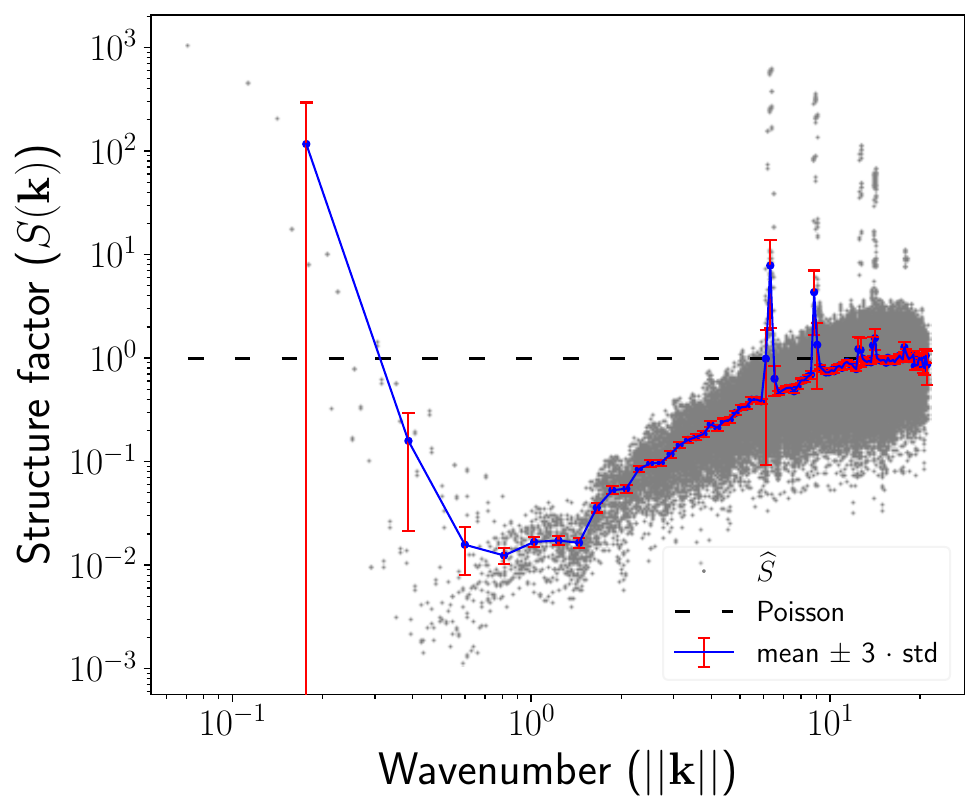}}     &
    \raisebox{-\height}{\includegraphics[width=1\linewidth]{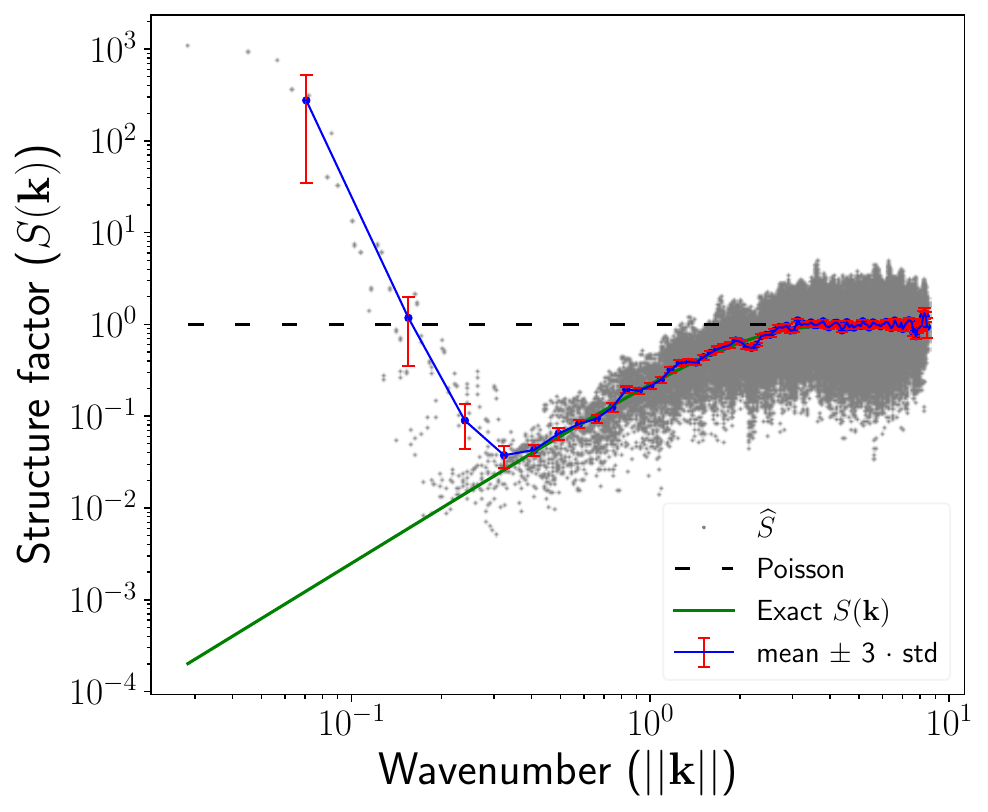}} &
    \raisebox{-\height}{\includegraphics[width=1\linewidth]{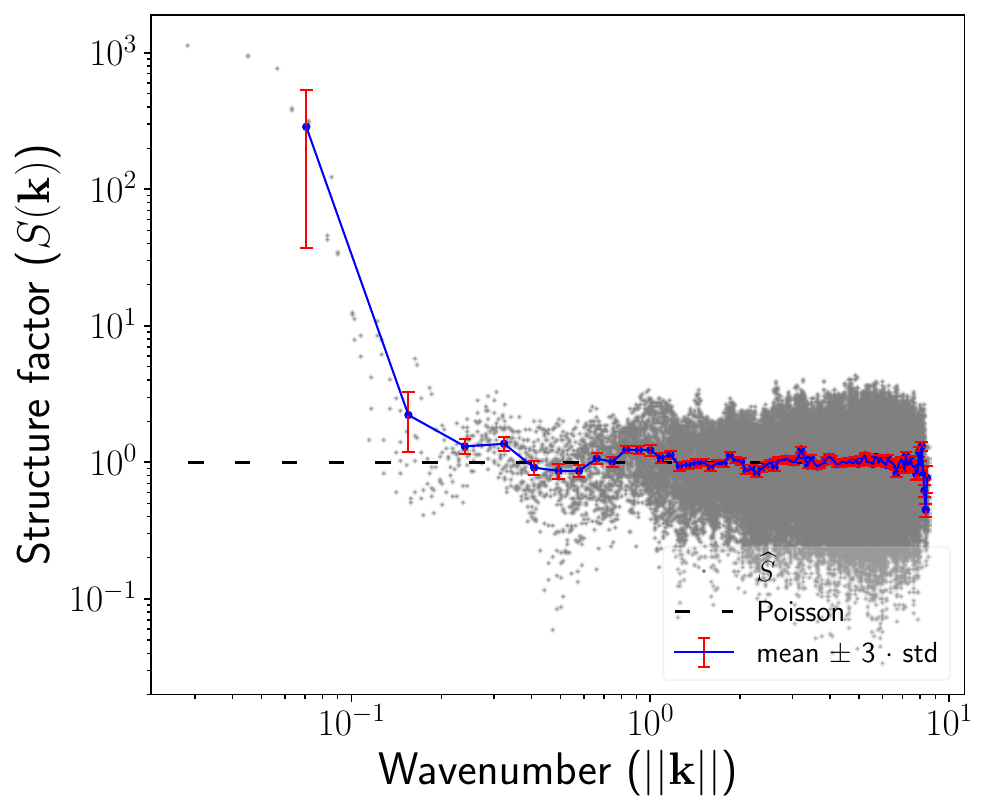}} &
    \raisebox{-\height}{\includegraphics[width=1\linewidth]{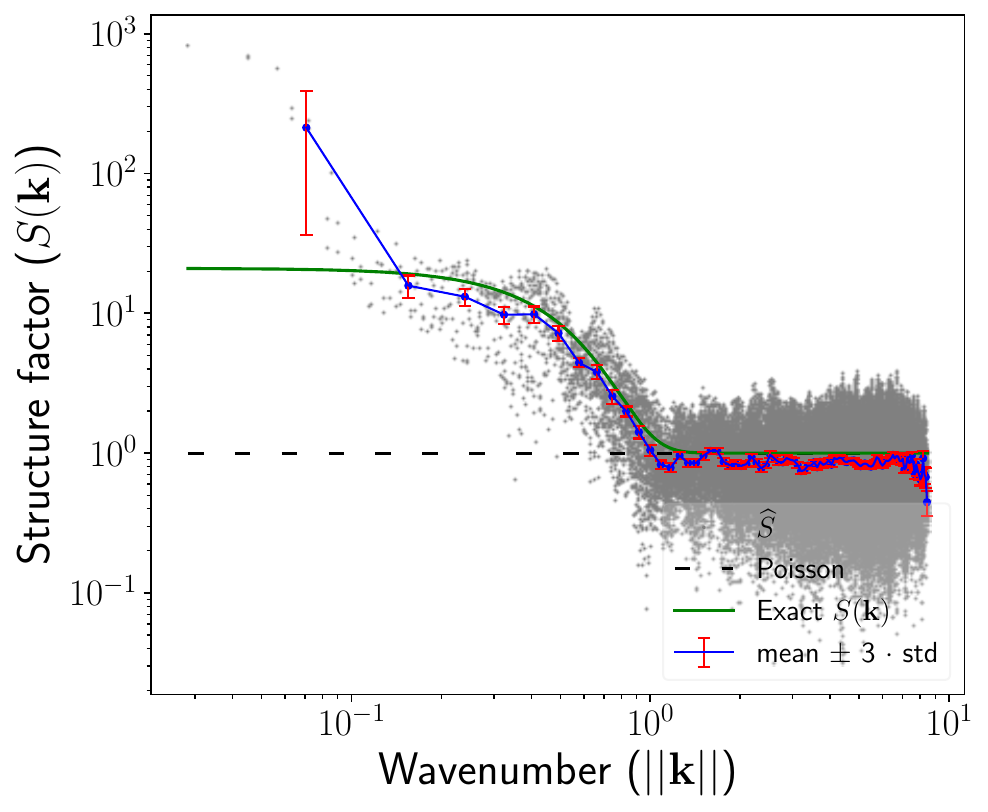}}
  \end{tabular}
  \vspace*{-0.2cm}
  \begin{tabular}{p{\dimexpr 0.05\textwidth-\tabcolsep}p{\dimexpr 0.22\textwidth-\tabcolsep}p{\dimexpr 0.22\textwidth-\tabcolsep}p{\dimexpr 0.22\textwidth-\tabcolsep}p{\dimexpr 0.22\textwidth-\tabcolsep}}
    \multirow{6}{*}{\rotatebox[origin=c]{90}{DDMT}}             &
    \raisebox{-\height}{\includegraphics[width=1\linewidth]{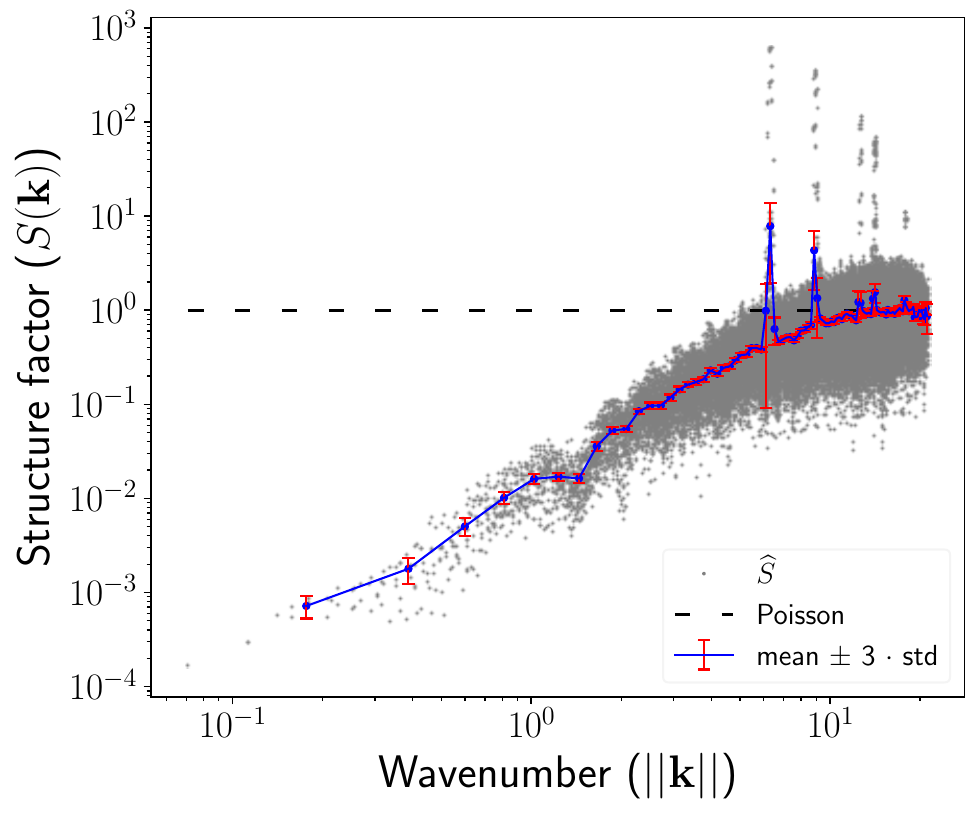}}     &
    \raisebox{-\height}{\includegraphics[width=1\linewidth]{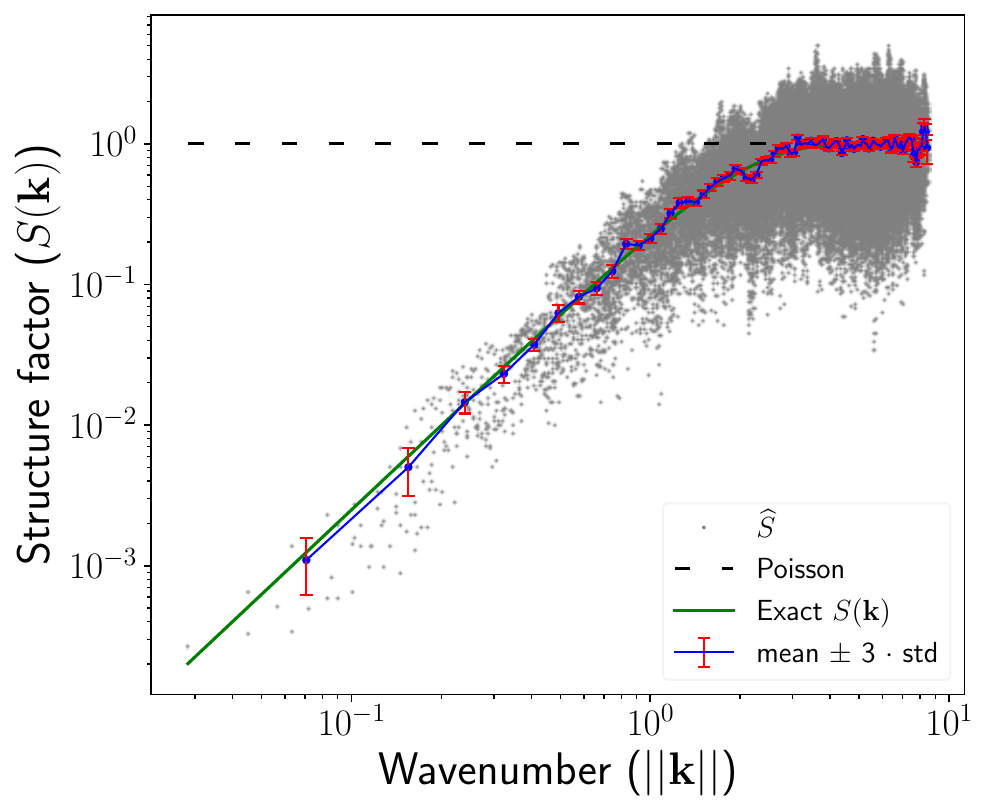}} &
    \raisebox{-\height}{\includegraphics[width=1\linewidth]{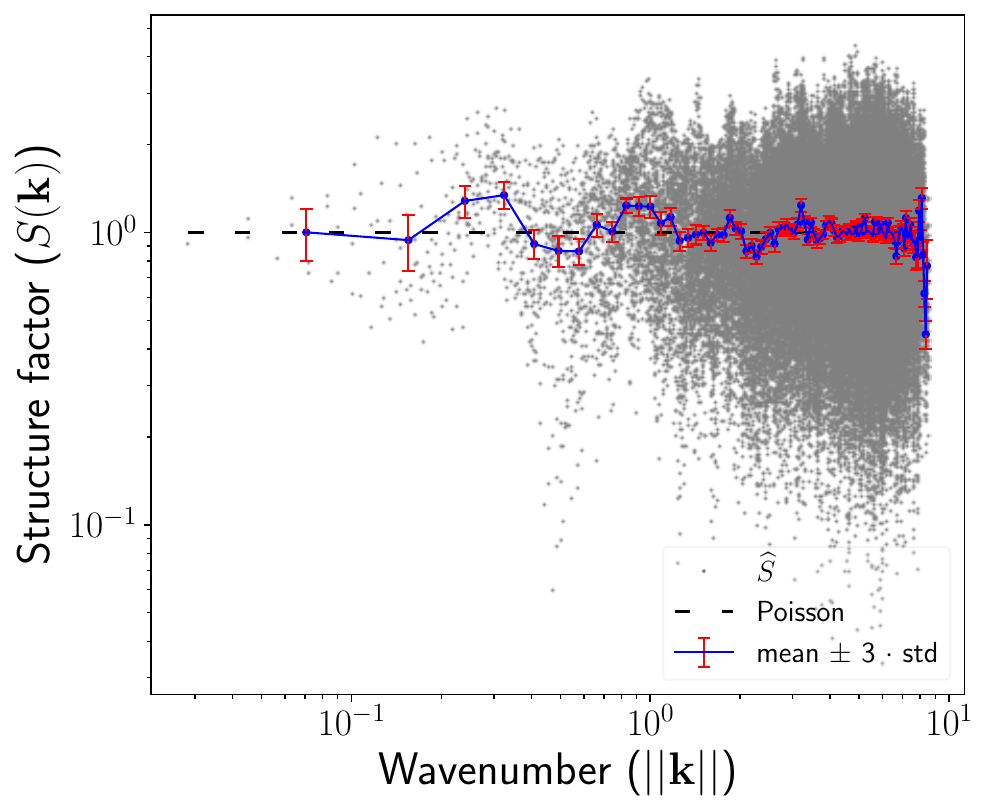}} &
    \raisebox{-\height}{\includegraphics[width=1\linewidth]{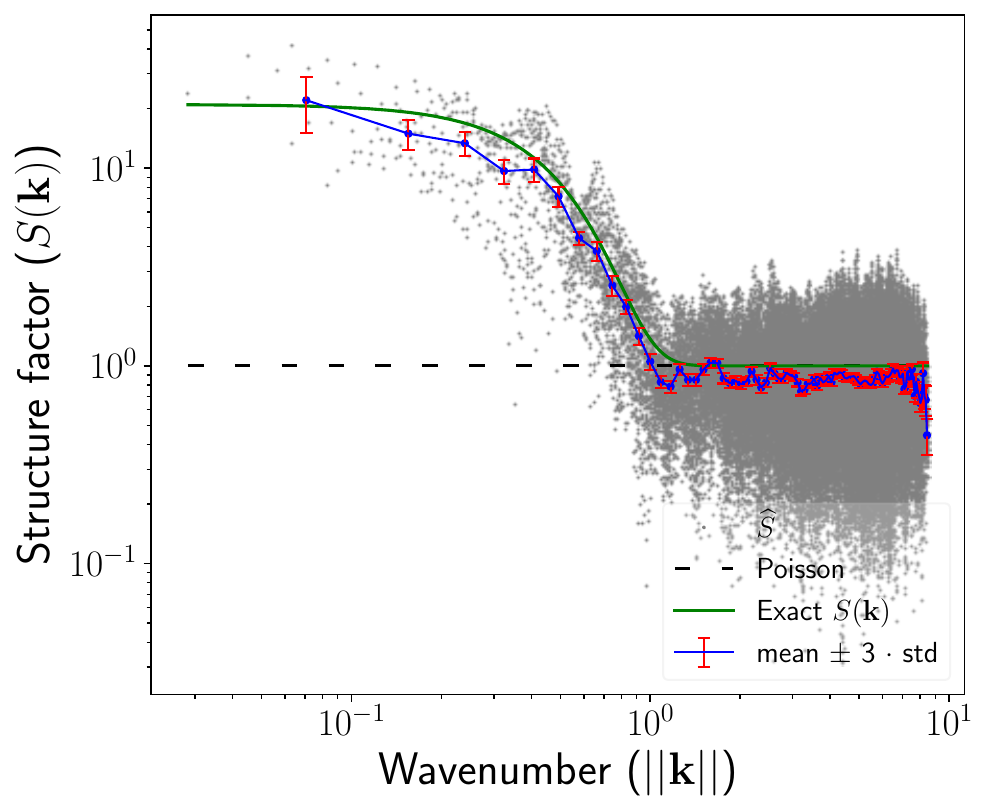}}
  \end{tabular}
  \vspace*{-0.2cm}
  \begin{tabular}{p{\dimexpr 0.05\textwidth-\tabcolsep}p{\dimexpr 0.22\textwidth-\tabcolsep}p{\dimexpr 0.22\textwidth-\tabcolsep}p{\dimexpr 0.22\textwidth-\tabcolsep}p{\dimexpr 0.22\textwidth-\tabcolsep}}
    \multirow{6}{*}{\rotatebox[origin=c]{90}{ UDMT}}             &
    \raisebox{-\height}{\includegraphics[width=1\linewidth]{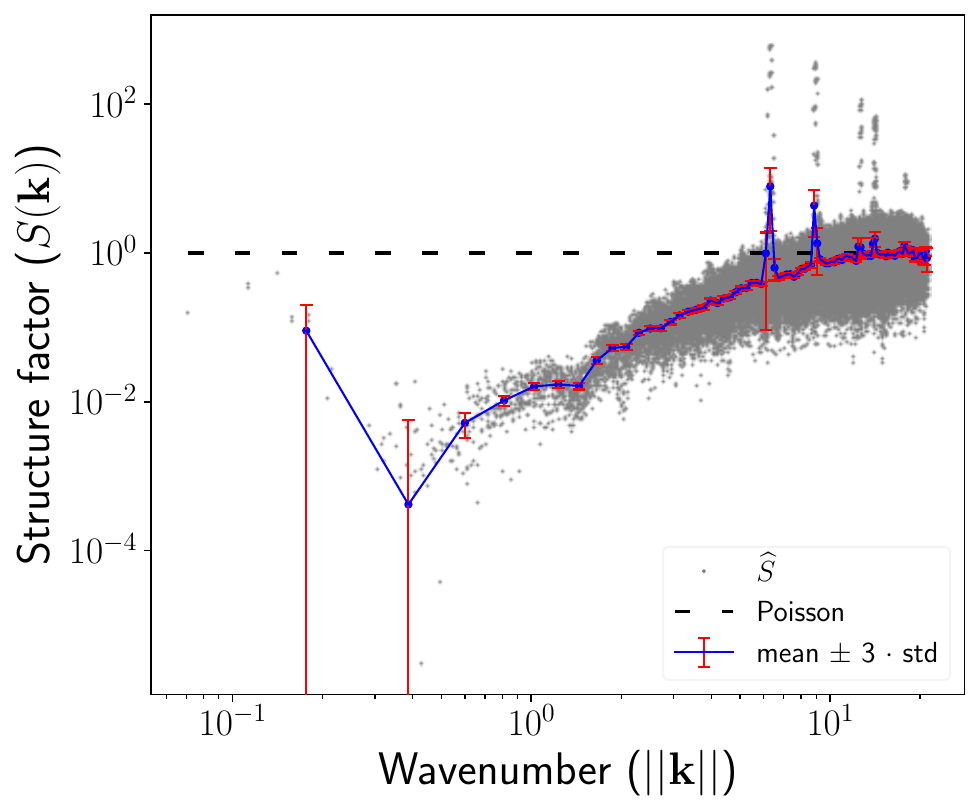}}     &
    \raisebox{-\height}{\includegraphics[width=1\linewidth]{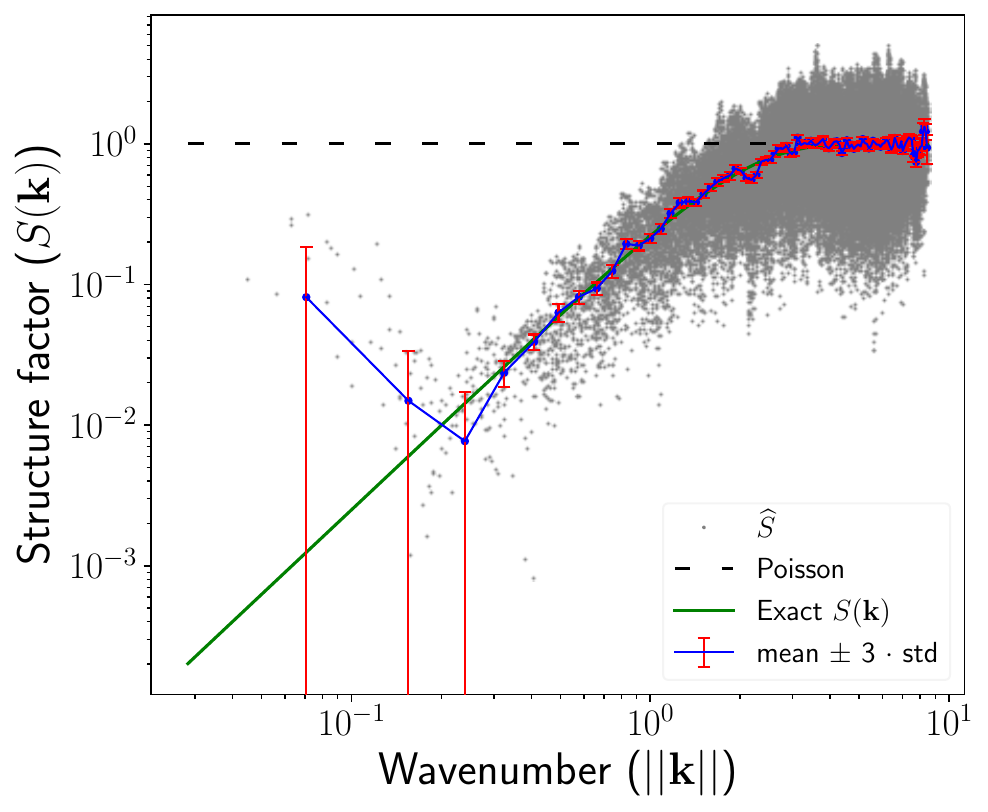}} &
    \raisebox{-\height}{\includegraphics[width=1\linewidth]{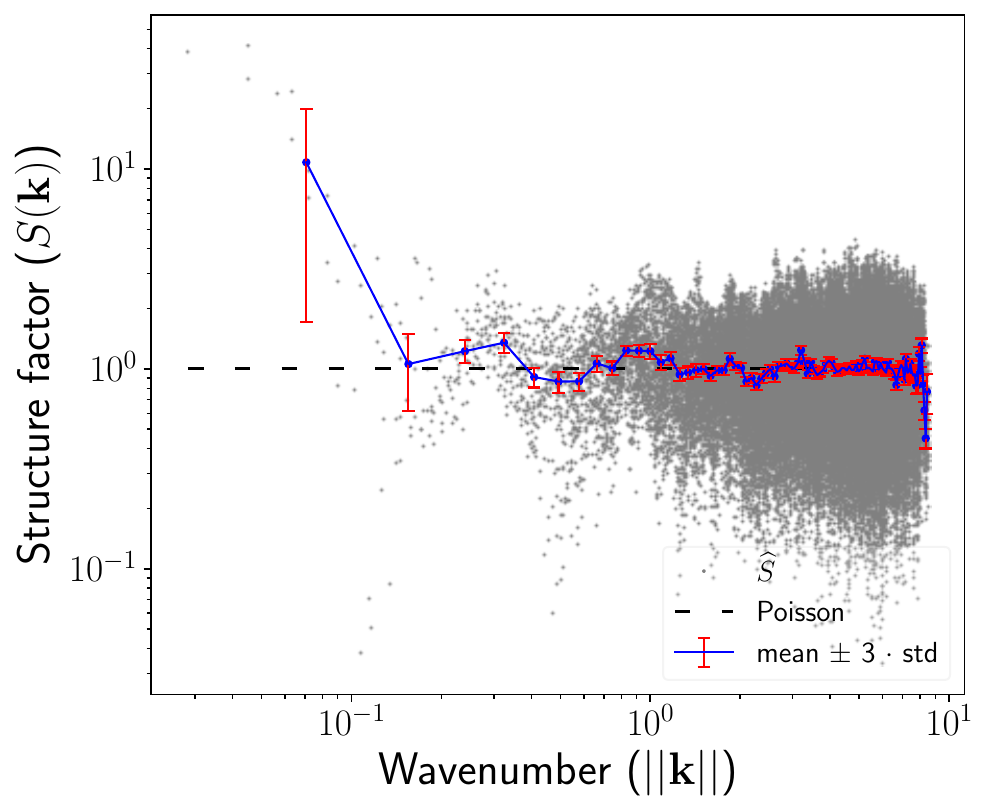}} &
    \raisebox{-\height}{\includegraphics[width=1\linewidth]{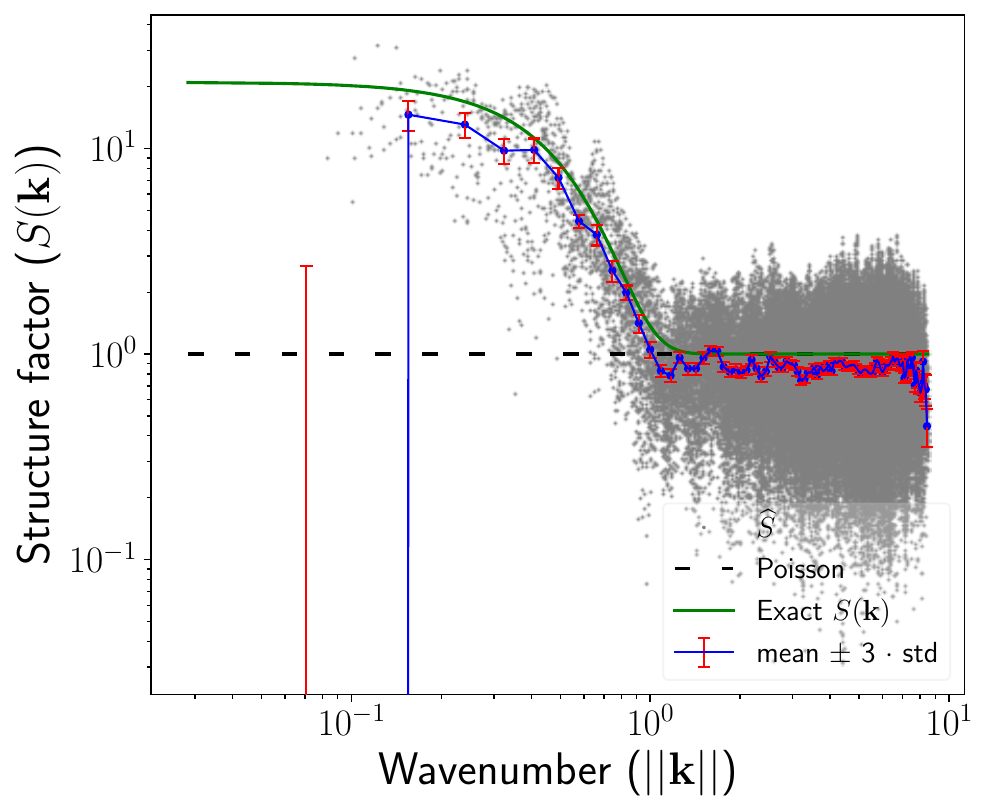}}    \\
    \caption*{}                                                                                         &
    \vspace*{-0.2cm}
    \caption*{{\fontfamily{pcr}\selectfont } Partial matching }                                               &
 
     \vspace*{-0.2cm}\caption*{{\fontfamily{pcr}\selectfont } Ginibre }                                          &
  \vspace*{-0.2cm}
    \caption*{{\fontfamily{pcr}\selectfont } Poisson }                                           &
    \vspace*{-0.2cm}
    \caption*{{\fontfamily{pcr}\selectfont } Thomas }
  \end{tabular}
  \vspace{-0.5cm}
  \caption{  \cite{HGBLr} Estimation of the structure factor of several models of point processes over  several samples with the \texttt{structure-factor} package. From left to right: Partial matching process  \cite{KLY} (Section  \ref{sec:partial-matching}), Ginibre model (Sections \ref{sec:intro-examples},\ref{sec:ginibre}), Poisson point process (Section \ref{sec:intro-pp}), Thomas cluster process (\cite[Section 5.3]{skm}).  {\it Three bottom rows}: multi-tapered estimator; directly debiased; undirectly debiased. Red bars indicate $ \pm 3$ times the standard deviation from the empirical mean.}
  \label{fig:SF-package}
\end{figure*}

We compared the estimators of Figure  \ref{fig:SF-package} numerically through the integrated mean square error over 50  i.i.d  samples $ \P_{ n}^{ (1)},\dots ,\P_{ n}^{ (50)}$, for $ n = 5800$ points, for each model. Specifically, given $ \rho _{ min} = 0.1, \rho _{ max} = 2.8,$ we compute the values $ \hat \s(u),u\in \mathbb{Z}_\lambda  ^{ d}$, for some estimator $ \hat \s$ as above. Then, we use the radial projection $ \hat \s_{ \text{\rm{\color{black} radial}}}(\|u\|) = \hat \s(u)$, and on $ [\rho _{ min},\rho _{ max}]$, $ \hat \s_{ \text{\rm{\color{black} radial}}}$ is interpolated linearly between the values projected from the grid $u\in  \mathbb{Z}_\lambda  ^{ d}$. We compare with the exact radial structure factor $\s_{ \text{\rm{\color{black} radial}}}(\|u\|)  $ that we derive from theoretical computations (e.g.  \eqref{eq:dpp-SF} and \eqref{eq:ginibre-kernel} for Ginibre). Then the integrated MSE is
\begin{align*}
\text{\rm{\color{black}  iMSE}}( \hat \s) = \int_{\rho _{ min}}^{ \rho _{ max}} | \hat \s_{ \text{\rm{\color{black} radial}}}(t)- \s_{ \text{\rm{\color{black} radial}}}(t) | ^{ 2}dt.
\end{align*}
 
 We give the variance  of $ \text{\rm{\color{black} iMSE}}( \hat \s)$ and a confidence interval   for $ \mathbf E [\text{\rm{\color{black} iMSE }}( \hat \s)]$ for the isotropic models investigated and several estimators in Table  \ref{tab:summary_statistics_DSE}.

\begin{table*}[!ht]
  \centering
   \caption{  \cite{HGBLr} variance and confidence interval for iMSE$ ( \hat \s)$
  }
  \label{tab:summary_statistics_DSE}
  \small
  \begin{tabular}{|l|c|c|c|c|c|c|c|}
    \hline
    \rule{0pt}{15pt}
    Estimators
     & Variance
     & confidence
     & Variance
     & confidence
     & Variance
     & confidence
    \\
    \hline
    $ \hat \s_{ \P_{ n}}$
     & 0.32
     & $0.32 \pm 0.02$
     & 1.31
     & $1.34 \pm 0.06$
     & 69.51
     & $70.71 \pm 17.95$
    \\
    \hline
    $\hat \s_{ \P_{ n}}^{ \varphi_{ 0}}$ (DDT)
     & 0.32
     & $0.33 \pm 0.03$
     & 1.44
     & $1.47 \pm 0.1$
     & 72.15
     & $73.63 \pm 26.12$
    \\
    \hline
    $\hat \s_{ \P_{ n}}^{ \varphi_{ 1}} $ (DDT)
     & 0.34
     & $0.35 \pm 0.06$
     & 1.47
     & $1.50 \pm 0.14$
     & 79.29
     & $80.51 \pm 27.20$
    \\
    \hline
    $\hat \s_{ \P_{ n}}^{ \varphi_{ 1},\dots ,\varphi_{ 4}}$
 (DDMT)    & 0.08
     & $ \mathbf{0.08} \pm 0.007$
     & 0.37
     & $\mathbf{0.38} \pm 0.02$
     & 17.90
     & $\mathbf{18.19} \pm 4.19$
    \\
    \hline
     & \multicolumn{2}{|c|}{Ginibre }
     & \multicolumn{2}{|c|}{Poisson }
     & \multicolumn{2}{|c|}{Thomas }
    \\
    \hline
  \end{tabular}
\end{table*}

Both visually and numerically, the directly debiased multi-tapered version seems to give the most accurate estimation of $ \hat \s$, at least in the isotropic case.

Other topics are discussed in  \cite{HGBLr} and implemented with  \texttt{structure-factor}, such as estimation of the covariance function, estimation of the spectral measure under isotropy conditions, test of hyperuniformity and estimation of the  hyperuniformity exponent, which is the topic  of the next section.

 \subsection{Exponent estimation}

 To estimate the  optimal hyperuniformity exponent (Section \ref{sec:hyp-expo}), i.e. $ \alpha $ such that $ \s(u)\sim C\|u\|^{ \alpha }$ as $ u\to 0,$ one can simply perform a regression on the estimated structure factors of Figure  \ref{fig:SF-package}, but a weakness of this approach is to require several  i.i.d samples for a consistent estimation.  This is natural since the spectral density is defined through the variance of a random variable  \eqref{eq:def-spectral-density}, which is in theory not available from a single realisation, but it is not adapted to the common experimental setup where a large  single observation is the only data available. 
 Mastrilli, Blaszczyszyn and Lavancier  \cite{MBL} propose an approach to circumvent this issue.   They  also use a family of orthogonal tapers $ \varphi _{ 1},\dots ,\varphi _{ q}$.
   In analogy to the classical estimation of the long-range memory exponent of time series, they  implement the idea to average over several scales. Namely, for $ J = \{j_{ 1},\dots ,j_{ m}\}\subset (0,1)$, they consider the linear statistics $ \P_{ \lambda }(\varphi _{ i,\lambda ^{ j}}),1\leqslant i\leqslant q,j\in J,$ for a given  point process $ \P$, where $ \varphi _{ i,R} = \varphi _{ i}(\cdot /R)$. Based on the dependency in the exponent  (Proposition  \ref{prop:hyp-expo}), they introduce the estimator of the  slope at scale $ j$ 
\begin{align*}
 \hat c_{ j}  =   \frac{ w_{ j}}{\ln(\lambda)} \ln\left(
\frac{ 1}{q}\sum_{i}\P_{ \lambda^{ j}}(\varphi _{ i})^{ 2}
\right),
\end{align*}and then the averaged estimator $ \hat \alpha  = \sum_{j\in J}w_{ j} c_{ j}$
where the weights $ w_{ j}$ are obtained from a least squares minimisation procedure, under the constraints $ \sum_{j}w_{ j} = 0, \sum_{j}jw_{ j} = 1.$ Then they derive results of consistency and asymptotic normality under integrability assumptions for the covariance. This estimator is implemented in the package \texttt{Estim\_Hyperuniformity}  \footnote{   \url{https://github.com/gabrielmastrilli/Estim_Hyperuniformity}}.
 It is interesting to visualise the influence of the scale $ j$  on the exponent. On Figure  \ref{fig:mastrilli}-left, one has a visual illustration of the orthogonal Hermite tapers related to the size of the sample depending on the scale $ j$. In the middle, they recover the exponent $ \alpha  \approx 2$ for the Ginibre ensemble, over intermediate scales (the scales $ j>1$ suffer from boundary influence, as is illustrated by the dashed line of the Poisson process). They also seem to recover several regimes for the real data of a small marine algae system ($ n = 900$) on the right panel, with $ \alpha  \approx 2.2$ for intermediate scales. The study of ``multi-hyperuniformity'', where the hyperuniformity exponent varies with the scale $ j$, is a rich and yet ill-posed statistical problem.

\begin{figure}[h!]
    \centering
    \begin{subfigure}[b]{0.28\textwidth}
        \includegraphics[width=\textwidth]{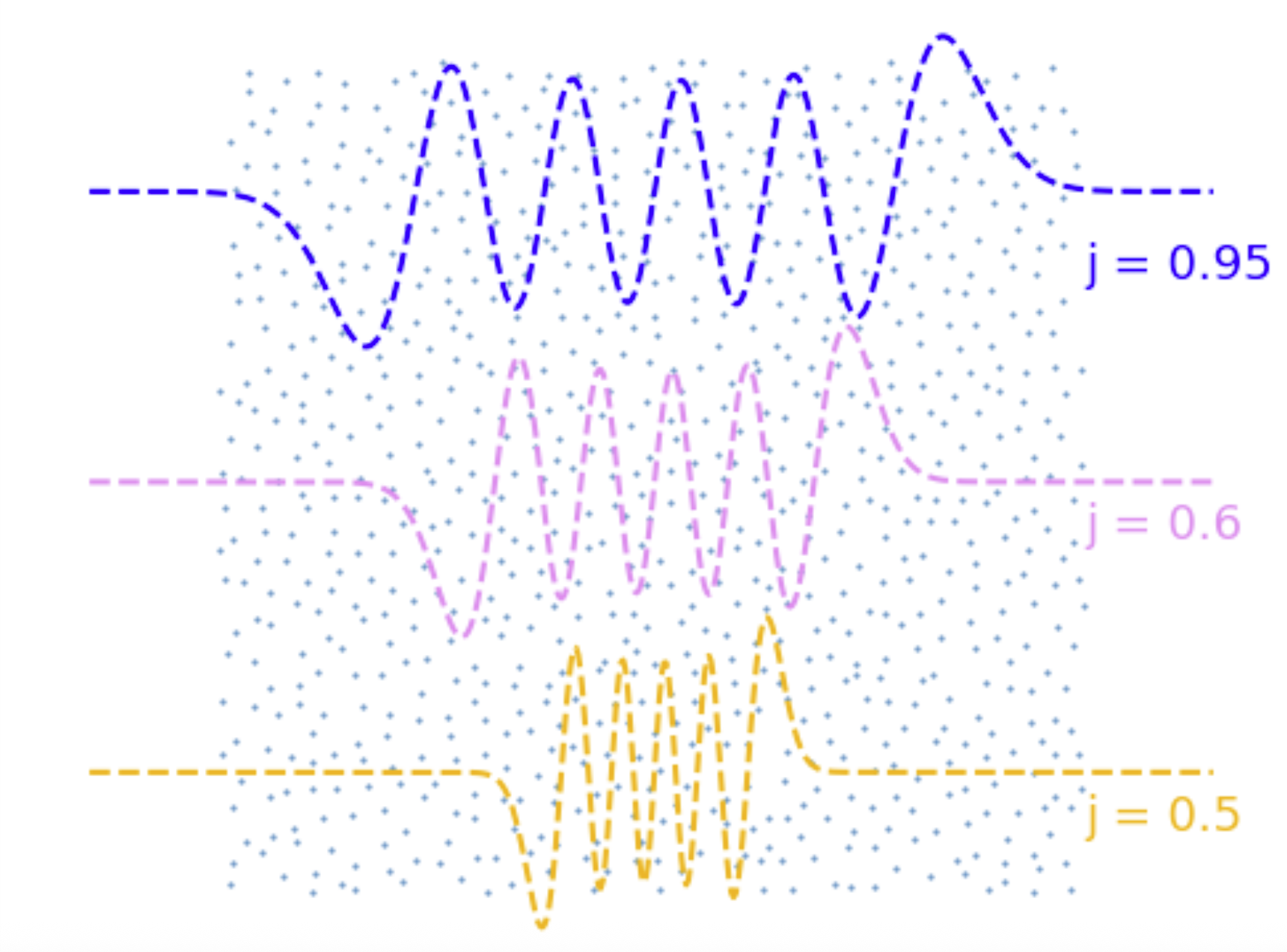} 
    \end{subfigure}
    \hfill    
    \begin{subfigure}[b]{0.3\textwidth}
        \includegraphics[width=\textwidth]{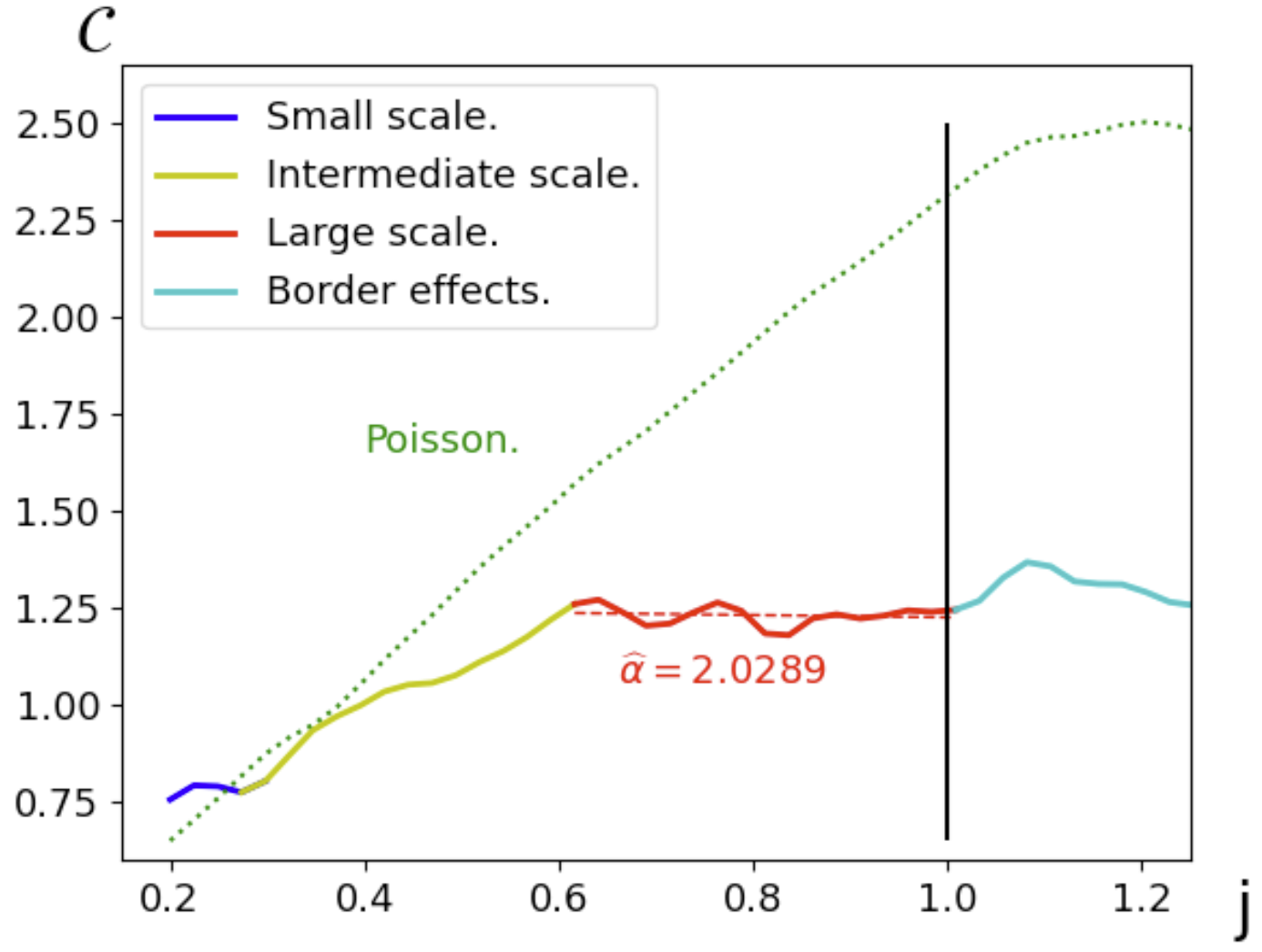}
    \end{subfigure}
    \hfill  \hspace{1.5cm} 
    \begin{subfigure}[b]{0.28\textwidth}
        \includegraphics[width=\textwidth]{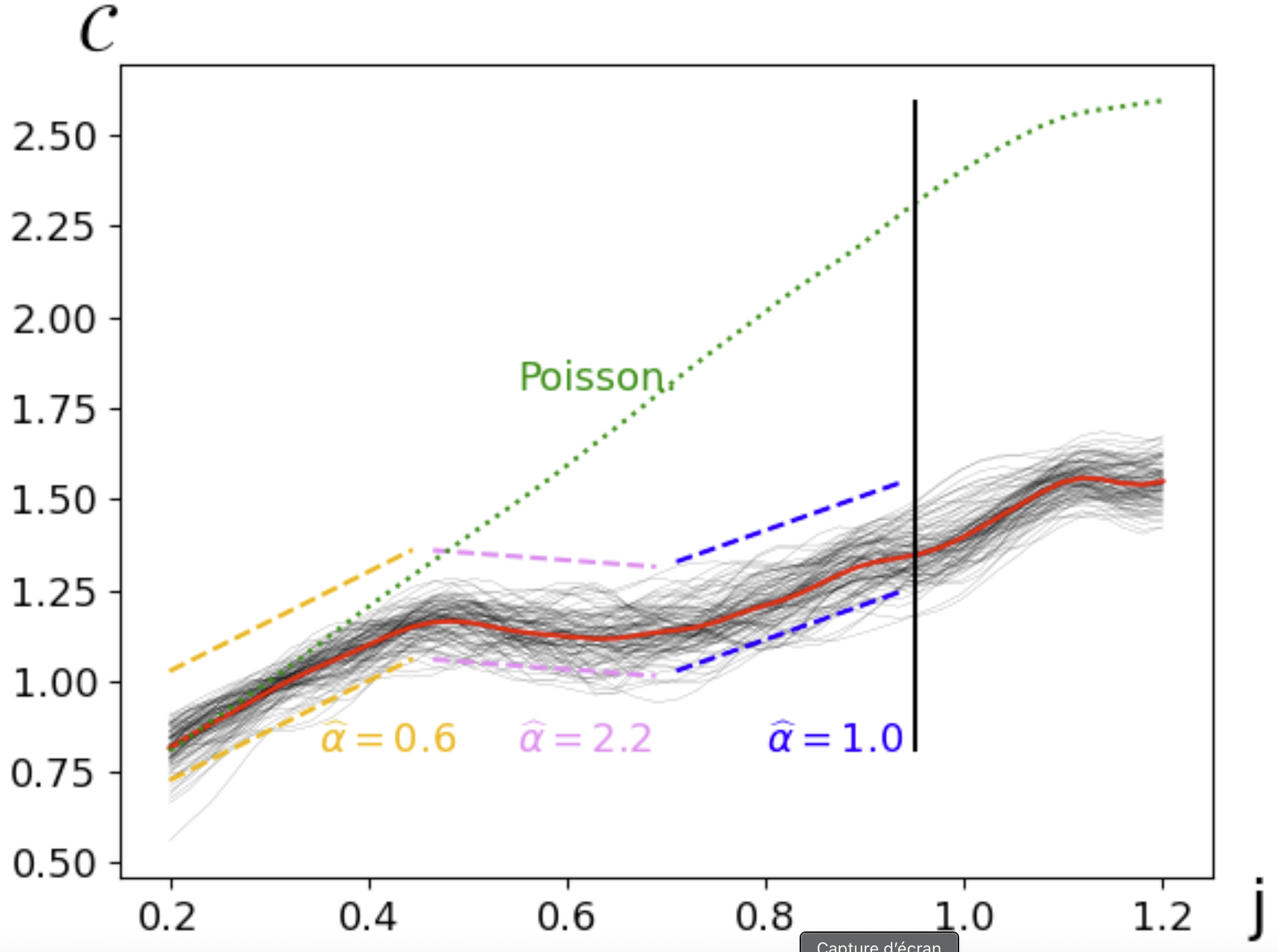}
    \end{subfigure}
 \caption{  {\it Left.} Hermite tapers at different scales $ j$.  {\it Middle.} Exponent estimation for the Ginibre process.  {\it Right.} Exponent estimation for the marine algae data.}
    \label{fig:mastrilli}
\end{figure}

  \section{Simulation of  hyperuniform samples}  
  
   We discuss the question of simulation of models which are proven to be hyperuniform, at least asymptotically. We assume here that the number of points is fixed $$ \#\P =N = \lambda ^{ d} .$$
   Several procedures are encoded in the package \texttt{blue\_sampler}   \footnote{\url{https://github.com/For-a-few-DPPs-more/rgbn}, install with \texttt{pip install blue\_sampler}}, developped by Armand de Caqueray, where the aim is to compare several simulation methods to evaluate their performances in the same set-up.
 
\subsection{Exact models}  
As already seen at Sections \ref{sec:intro-examples}, \ref{sec:ginibre}, \ref{sec:gue-as-dpp},  \ref{sec:random-matrices}, many hyperuniform processes are eigenvalues of a random matrix model.
The complexity of the implementation is essentially that of finding the eigenvalues of a large random matrix, there does not seem to be a trick beyond the known tricks about fast matrix diagonalisation,  it is in general a costly procedure. 
Simulating Coulomb gases / plasma / jelliums might be even trickier, and one often prefers the random matrix representations when they are available. The naive research of eigenvalues of a random Ginibre matrix (with \texttt{np.linalg.eigvals} in \texttt{Python}) on a personnal computer gives the following times, in seconds: \begin{itemize}
\item 3000pts in 17s
\item 4000pts in 35s
\item 5000pts in 62s
\end{itemize}
  Lavancier and Rubak  \cite{LavRubDPP} simulate the  2D $ \beta $-Ginibre ensemble, i.e. the DPP on $  \mathbb C $ whose kernel is of the form
\begin{align*}
 K_{ \beta }(z,w) = \rho \exp\left(\beta ^{ -1}(z  \bar w- \frac{ | z | ^{ 2} +  | w | ^{ 2}}{2})\right), 
\end{align*}hence the classical Ginibre process is obtained for $ \beta  = 1$ (see  \eqref{eq:ginibre-kernel}). Some performances are summarised in  table  \ref{fig:lavancier} ($ \rho $ is the average number of points).
\begin{figure}[h!]
\begin{center}
\includegraphics[width = 8cm,
                     trim=0cm 6cm 0cm 0cm,
                     clip]{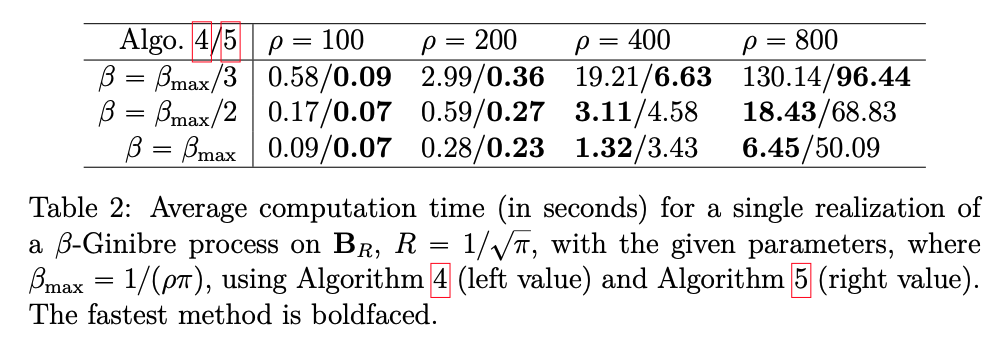}
\end{center}
\caption{Screen capture from \cite{LavRubDPP}, $ \rho $ represents the number of points sampled.}
\label{fig:lavancier}\end{figure}

Gautier et al. \cite{GautBardValko} report the following performances for simulation   2D Jacobi $ \beta $-ensembles with their package \texttt{DPPy}: \begin{itemize}
\item 1000pts: 2s
\item 2000pts: 14s
\item 2500pts: 30s 
\end{itemize}

\subsection{Fair partitions}
\label{sec:numerics-laguerre}

 It seems to have been exploited independently in the physics  \cite{torquato-tilings,Klatt-nature} and signal processing literature \cite{fiume,Balzer,deGoes,gaussian-blue-noise, Sarrazin}  that taking the centroids of a fair partition  conducts  to hyperuniformity and stealthiness (Section \ref{sec:stealthy}). See Theorem \ref{thm:LotzKlatt} for a theoretical justification. The more rigourous and general work in this direction is  \cite{LotzKlatt}, see Section \ref{sec:fair-partitions} for theoretical insights. We already discussed theoretical properties of the typical cell for stable marriage, power diagrams and Coulomb partitions. Let us discuss further their use to produce  hyperuniform samples.
 
 \subsubsection{Laguerre cells centroids and Lloyd's algorithm}   
 \label{eq:numeric-laguerre}

The terminology  {\it centroid} is often used instead of  {\it barycenter}, we recall that it is defined by 
\begin{align*}
 \text{\rm{centr}}(C) = \mo_{ 1}(C) = \displaystyle\int_{C}xdx.
\end{align*}
Recent works often start from optimal transport Laguerre cells, recall the definition  from Section  \ref{sec:laguerre-theor}.
 This procedure can be iterated: choose a Laguerre model of partition,  and denote by $  \mathscr  T$ the mapping   $ (  \mathbb  T  _{ n}^{ d})^{ N}  \mapsto  (  \mathbb  T  _{ n}^{ d})^{ N}$ which to $N$ points associates the centroids of the partition based on these $ N$ points.  See the illustration at Figure  \ref{fig:sarrazin}. Starting from some point configuration $ \P_{ 0}$, iterate this procedure with $ \P_{ n + 1} =   \mathscr  T \P_{ n},n\geqslant 0$. 
 
  {

  The resulting centroid process is called Centroidal Voronoi Tessellation (CVT) when the underlying tessellation is Voronoi, or Capacity Constrained Voronoi Tessellation (CCVT) for Laguerre cells with equal volume. The  importance  of having equal volume is clear from  \cite{torquato-tilings,LotzKlatt}, but iterative procedures such as in  \cite{Klatt-nature} starting from  Voronoi cells with different volumes tend to homogeneize the cells and give them equal volume asymptotically, saving the initial cost of producing cells with equal volumes from the start.  In image processing and optimal transport, quite often the target density is not homogeneous, i.e. we are not working with stationary point processes, especially in relation with the problem of optimal quantization of a given density  \cite{Sarrazin,deGoes,Balzer}, but we will not address this aspect here as the stationary case already allows an in-depth discussion that captures most fundamental aspects. 
 \renewcommand{\a}{  {\bf a}}

}

\begin{figure}[h!]
\begin{center}
\includegraphics[width = 15cm]{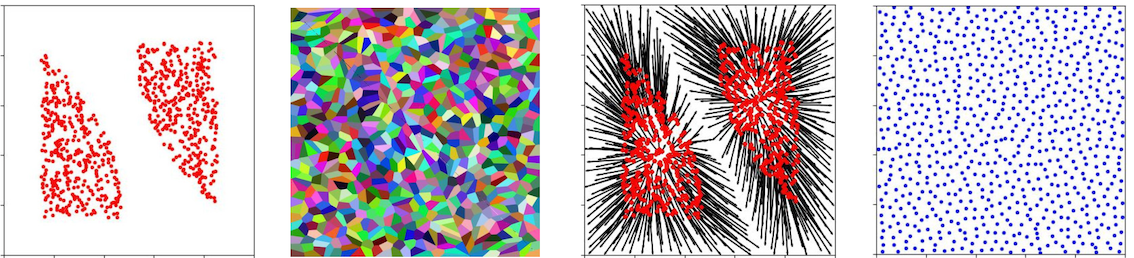}
\caption{  \cite[Fig. 1]{Sarrazin} {\it From left to right}: a point cloud $ \P_{ 0} = \{x_{ i},i = 1,\dots ,N\}$; the $ \W_{ 2}^{ 2}$-optimal transport power diagram $ \{B_{ i};i = 1,\dots ,N\}$; the displacement vectors towards centroids $ x_{ i}- \text{\rm{centr}}(B_{ i})$; the centroid process $ \{\text{\rm{centr}}(B_{ i});i = 1,\dots ,N\}$.}
\label{fig:sarrazin}
\end{center}
\end{figure}

 The centroid   $  x = \text{\rm{centr}}(B_{ i})$ minimises the inertia $ x  \mapsto  \int_{ B_{ i}}\|y-x\|^{ 2}dy.$  The procedure above is also called Lloyd's algorithm, and is related to the $ k$-means algorithm in statistical clustering, seeing each cell as a cluster among the points of $  \mathbb  T  _{ n}^{ d}$, and at each iteration the center of each cluster is updated to be at the position minimising the inertia of the point cloud.  
Lloyd / $ k$-means algorithm has been studied intensively, wether for Voronoi tessellation, or the  Laguerre tessellations.

\subsubsection{Entropic transport}
  
In practice, computing the optimal transport plan, or in our case the optimal transport allocation / Laguerre cells, can be  costly numerically, especially in dimension $ d>5$. Many cheaper procedures exist to build a transport plan that   is  close to optimality. A branch that is flexible and theoretically rich is that of entropic optimal transport, where an entropy term is added in the transport cost to minimise. See  \cite{PeyreCuturi} for a general introduction, giving rigourous results. One hence seeks for a coupling $ M$ minimising
\begin{align*}
\displaystyle\int_{}\|x-y\|^{ p}M(dx,dy)+ \varepsilon    \mathsf E(M | \Leb \otimes P)
\end{align*} where the second term is the entropy of $ M$  with respect to the product measure, computed in terms of the Radon-Nikodym derivative
\begin{align*}
     \mathsf E(M | \Leb \otimes P)
 = \displaystyle\int_{}\ln\left(
\frac{ dM}{d\Leb\otimes P}
\right)dM,
\end{align*} and $ \varepsilon >0$ is a parameter weighting between optimality and efficiency.

 The entropic regularization is implemented by the celebrated Sinkhorn algorithm for computation, which has much better computational complexity than linear programming approaches.
  In this case, an optimal transport plan $ M$ is not necessarily represented by an allocation, but more generally by non-negative kernels $ f_{ i}$ summing to unity:
\begin{align*}
 \sum_{i}f_{ i}(x) = 1, x\in W.
\end{align*} 
$ f_{ i}(x)dx$ represents $ M(dx,\{x_{ i}\})$, the proportion of mass of Lebesgue measure at location $ x$ that has to be sent to point $ x_{ i}.$   The kernels $ f_{ i}$
can be interpreted as "soft assignments" or "fuzzy allocations" as opposed to the hard assignments of classical  optimal transport. An advantage is that this transport plan is less costly to generate, and then a cloud of points can be attached to each kernel $ f_{ i}$. One can use the general Theorem 5.2  of  \cite{LotzKlatt} to make the last step rigourous.

\subsubsection{Fair STIT}  
The previous constructions of Laguerre tessellations, stable marriage, and gravitational allocations, might seem a bit impractical to produce, because of the long range interactions present in the construction and the complicated shape of the cells / kernels.   In dimension $ 2$, the situation is even worse because any tessellation based on a Poisson (or Poisson-like) process will be badly behaved, in the sense that the $ \W_{ 2}^{ 2}$ transport cost per unit volume of the corresponding allocation to the original Poisson points diverges (Theorem \ref{AKT}). Therefore, if one wished to build, say, a mixing  hyperuniform process, one has to start from a mixing  hyperuniform point process, and build the related partition, which seems a bit cyclic.

Lotz and Klatt  \cite{LotzKlatt} propose a finite volume direct construction not involving an original point process, in any dimension. They do not provide a detailed study yet, but this tessellation seems adapted  to build mixing hyperuniform processes with arbitrarily high exponent $ \alpha .$ The construction is inspired by the STIT tessellation introduced 20 years ago by Nagel and Weiss  \cite{NagWei05}. Let $ q\in \mathbb{N}^{ *}.$\begin{itemize}
\item Start from an initial window $ C$ of volume $ 1$,  and a deterministic distribution $ \nu $ on $  \mathbb  S  ^{ d-1}$, seen as the set of orthogonal directions of hyperplanes of $  \mathbb R  ^{ d}.$
\item Draw randomly a direction $ \theta(C)\sim \nu $, and chose a point $ x(C,\theta (C))\in C$ such that the unique hyperplane  containing $ x(C)$ and orthogonal to $ \theta (C)$ divides $ C$ in two cells with equal volume $ C_{ 0},C_{ 1}$.
\item Iterate this procedure with $ C = C_{ 0} $  and $C = C_{ 1}$, and iteratively for each cell obtained.
\item Stop the subdivision at cells that are obtained as the result of $ q$ successive divisions of the initial cell $ C$, i.e. cells which have exactly volume $ 2^{ -q}$.
\end{itemize}In finite volume, one  obtains $ 2^{ q} $ cells. The hope is that, if one calls $  \mathscr  C_{ q} = \{C_{ 1},\dots ,C_{ 2^{ q}}\}$  the cells obtained with this procedure, the rescaled partition
\begin{align*}
\overline{\mathscr  C}_{ q} = \{2^{ q/d}C_{ 1},\dots ,2^{ q/d}C_{ 2^{ q}}\},
\end{align*}
made up of unit volume cells, converges weakly in law to some infinite equivariant fair partition $  \mathscr  C$ by convex polyhedra. We furthermore hope that this random tessellation will satisfy the assumptions of Theorem  \ref{thm:LotzKlatt}. 
The original STIT construction from Nagel and Weiss is similar, except that there is not the constraint that when a cell is divided in two, the two new cells are required to have the same volume.  See Figure 1 in  \cite{LotzKlatt} for $ m$-averaging sets attached to each cell, giving a $ \alpha $-hyperuniform sample with $ \alpha  = 2(m + 1)-d$ (see Section  \ref{sec:fair-partitions}).

\begin{figure}[h!]
\begin{center}
\includegraphics[width = 5cm]{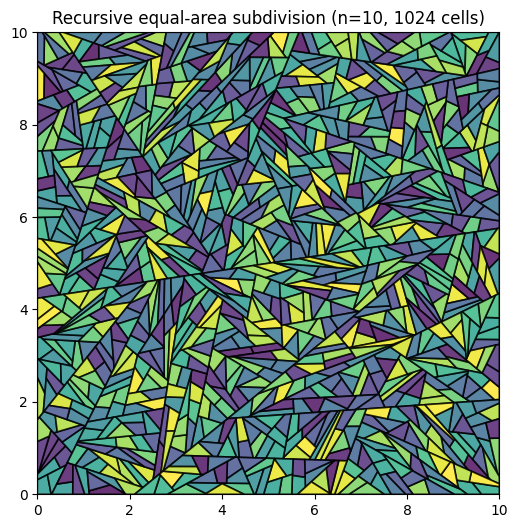}
\caption{A simulation of a fair STIT tessellation based on an initial square window in $ \mathbb{R}^{ 2}$, with $ q = 10$ levels of subdivisions, hence $ 2^{ q} = 1024$ cells with equal volume.}
\label{fig:fair-stit}
\end{center}
\end{figure}

A very important point is that the whole procedure is linear in the number of points produced, provided the number of cells is a power of $ 2$, i.e. of the form $ 2^{ q}$ for some $ q\in \mathbb{N}.$
{Let us give a point sample generated by the package \texttt{blue\_sampler}\footnote{\url{https://github.com/For-a-few-DPPs-more/rgbn}, install with \texttt{pip install blue\_sampler}},with  a variant of the method of \cite{LotzKlatt}, see Figure  \ref{fig:armand}.  For a simplified procedure, each polygonal cell is required to be quadrilateral, in the sense that a cell subdivision is discarded  if it does not produce two quadrilaters,  and the algorithm tries again until it succeeds producing two quadrilaters. The aim of the procedure is to exploit Theorem \ref{thm:LotzKlatt} with $ m = 2,n = 6$, i.e. to produce in each cell  $ B_{ i}$ independently and randomly a cluster $ \P_{ m,n,B_{ i}}$ of $ 6$ points matching the moments of $ B_{ i}$ up to order $ 2$, in order to produce a $ \alpha $-hyperuniform process with exponent $ \alpha  = 2(m + 1)-d = 4 .$}

 \begin{figure}[h]
  \centering
 \begin{subfigure}{0.3\textwidth}  
 \centering
\includegraphics[width=\linewidth]{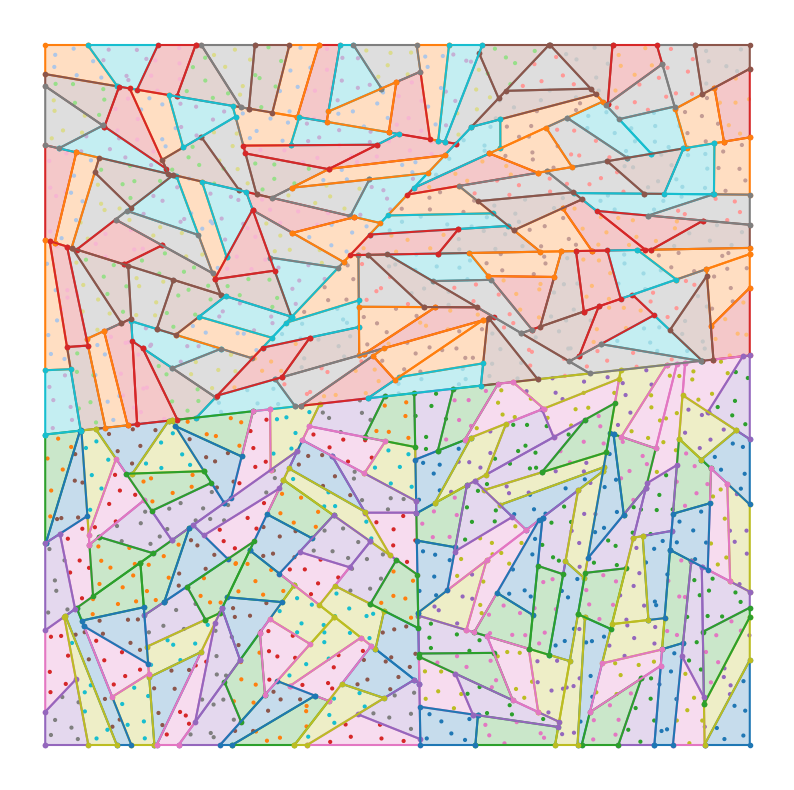}
   \end{subfigure}
     \centering
 \begin{subfigure}{0.3\textwidth}  
 \centering
\includegraphics[width=\linewidth]{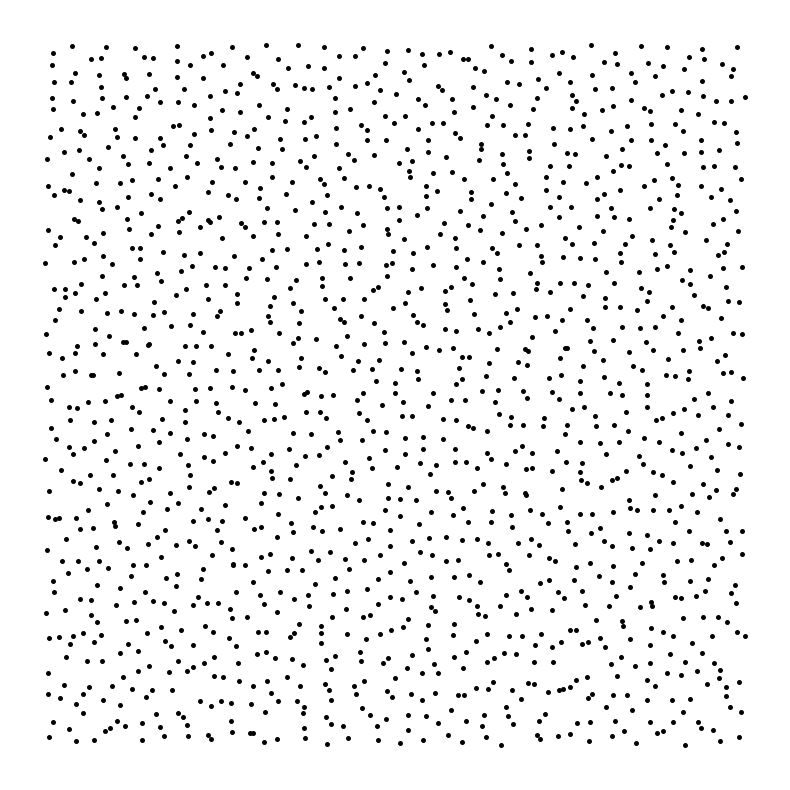}
   \end{subfigure}
\hfill
    \begin{subfigure}{0.3\textwidth}     \centering \includegraphics[width=\linewidth,
                     trim=.7cm 1cm 0cm 0cm,
                     clip]{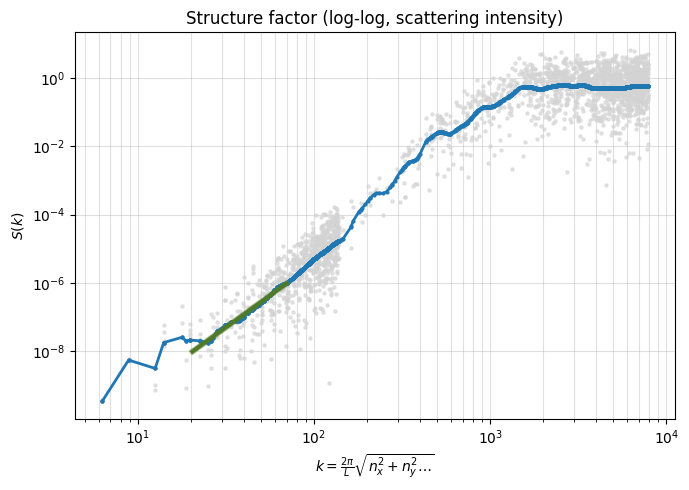}    
    \end{subfigure}    \hfill
    \begin{subfigure}{0.4\textwidth}
        \centering 
      
      \end{subfigure}
      \caption{  {\it Left:} The result of the quadrilater-forced fair STIT tessellation. Inside each cell, a cluster of 6 points matching the moments of the cell up to order $ 3$ is produced.  {\it Middle:} Same picture without the quadrilaters.  {\it Right:} Scattering intensity of this sample. A manual linear regression on the portion of curve between frequencies $ 2.10^{ 1}$ and $ 7.10^{ 1}$ (symbolised by the green segment) produces an estimation $ \alpha \approx 3.68$, probably due to the  finite volume approximation.}
      \label{fig:armand}
      \end{figure}

      \section{Simulation of empirically stealthy processes and blue noise by gradient descent}

\label{sec:grad-desc}

  As suggest many examples from this book, when there is a repelling force acting among particles, they are likely to arrange themselves in a globally ordered manner, lowering fluctuations. See for instance the example of  Gibbs measures (Theorems \ref{thm:3d-coul}, \ref{thm:Leble}, \ref{thm:riesz}, \ref{thm:Coul-infinite})  or the effect of the Coulomb force on Poisson samples  \cite{HBLr}, diminishing the variance of smooth linear statistics. We discuss here the iterative application of such forces to produce hyperuniform systems, which can be more generally formulated as a gradient descent.

     Up to applying an independent shift, for a convenient mathematical treatment we assume $ \P_{ n}$ is invariant under translations of $  \T  $  to exploit Fourier duality. It is recommanded to use periodic boundary conditions in numerical experiments to not suffer from noise caused by boundary effects.
 We also assume that the number of particles is fixed,  up to rescaling we assume unit intensity  $ \#\P_{ n} = \lambda ^{ d}$ a.s. with $ \lambda  = N^{ 1/d}$ for some $ N\in \mathbb{N}$. We apply here recursively a transformation 
\begin{align*}
 \P_{ n  +  1} =  \mathscr  T \P_{ n }
\end{align*}
that preserves shift invariance.  Even if it is applied to an unordered set of particles, $  \mathscr  T$ can be formally defined as a mapping from $ (\T )^{N}$ to $ (\T) ^{N}$ symmetric in its arguments, i.e. invariant under arguments permutation, and naturally extended to $ N$-point configurations.
The basic idea is  to  {see $ \P $ as a vector of size $ Nd$ and }define a differentiable energy function 
\begin{align*}
  \mathscr  E: (\T )^{N}  \mapsto  (\T )^{N}.
\end{align*}Start from an initial random configuration $ \P = \P_{ 0}$   in $  { \T  }$ and iterate with 
\begin{align*}
\P_{ n + 1} = \P_{ n } -   \nabla  \mathscr  E( \P_{ n} ).
\end{align*}
This procedure is sometimes adapted with a  {\it learning rate} $ \tau >0$, i.e. 
\begin{align*}
 \P_{ n + 1} = \P_{ n}-\tau  \nabla  \mathscr  E(\P_{ n}),
\end{align*}
but this amounts to multiplying $  \mathscr  E$ by $ \tau $ and we do not include $ \tau $ in  the mathematical discussions.

    The   non-periodic hyperuniform point processes provided so far have a finite exponent $ \alpha \geqslant 0$, meaning the spectral measure $ \S(du)$ vanishes in $ 0$ at speed $ \|u\|^{ \alpha }$. The iterative procedures discussed in this section  seem to converge to stealthy processes, i.e. with a gap in the spectrum, corresponding to $ \alpha  = \infty $. The involved construction are only in finite volume, i.e. 
provide after $ n$ iterations a point process $ \P_{ n} \subset  \T   , $ for some $n, \lambda  \gg 1 $,  whose expected scattering intensity  $ \S_{ \P_{ n}}$ seems to vanish around $ 0$ in numerical experiments for large $ n$ (see Section \ref{sec:hu-finite} for spectral considerations  related to finite samples).
    This experimental field of research has developped independently in physics on one hand, and in the more   numerical fields of image processing / optimal transport / numerical integration / machine learning, with interesting progress and results from all fields.   As it turns out, most successful methods rely at some point on a gradient descent.

\subsection{Laguerre centroid iteration}  
The Laguerre centroid construction  described at Section  \ref{sec:numerics-laguerre} turns out to be a gradient descent.
In this case, the ``energy'' to minimise  over the $ N$ particles is the  squared $ 2$-Wasserstein distance $$ {  \mathscr  E(\P)} =  \W_{ 2}^{ 2}(\P ,\Leb 1_{ \T})$$ or preferentially the toroidal version, to avoid boundary effects $$   \widetilde{\mathscr  E}(\P) =  \tilde \W_{ 2}^{ 2}(\P ,\Leb 1_{ \T})$$where in $  \tilde \W_{ 2}^{ 2}$ the cost function  is the square toric distance between two points. Recall that here $ \P = \sum_{i = 1}^{ N}\delta _{ x_{ i}}$ is both seen as a (atomic) measure (in $ \W_{ 2}^{ 2}$) and a vector of $ Nd$ real ordered components,  (in $  \mathscr  E$).

 This problem is studied by Mérigot and Mirebeau \cite{Merigot}, where it is proved that this function is not convex in the $ N$ points $ x_{ 1},\dots ,x_{ N}$, but it is  semi-concave, i.e. $  \mathscr  E-\kappa\|\cdot \|^{ 2} $ is concave for some $ \kappa >0$. Recall that we study the mapping $  \mathscr  T$, associating the Laguerre centroids: for $ \P = \sum_{i}\delta  _{ x_{ i}},$
\begin{align*}
  \mathscr  T(x_{ 1},\dots ,x_{ N}) = ( \text{\rm{centr}}(B_{ i}))_{ i = 1,\dots ,N},
\end{align*} where the $ B_{ i}$ are the $ \W_{ 2}^{ 2}$-Laguerre centroids.  It is   remarkable  this procedure corresponds exactly to the gradient descent of the energy.

\begin{proposition}[\cite{Sarrazin}, Proposition 1]
For $ x_{ 1},\dots ,x_{ N}\in  \mathbb  T  _{ n}^{ d}$ pairwise distinct,
\begin{align*}
  \mathscr  T(x_{ 1},\dots ,x_{ N}) = (x_{ 1},\dots ,x_{ N})-  \frac{ 1}{ 2}   \nabla  \mathscr  E(x_{ 1},\dots ,x_{ N}).
\end{align*}
\end{proposition}
See the third image of Figure  \ref{fig:sarrazin} for the visualisation of the gradients, i.e. the ``displacement vectors''. We do not provide a proof, but it is easy to check directly if for instance the $ x_{ i}$ are a small perturbation of a perfect grid on $ \T.$
It implies in particular that the more a particle is ``misplaced'', i.e. far from the centroid of its cell, the higher is the gain by moving it to its center. Mérigot, Santambrogio and Sarrazin \cite{Sarrazin} prove that a.s. this procedure converges to a ``critical'' configuration where each point is the centroid of its cell. They prove furthermore that if the initial configuration is not too cluttered, the minimisation is very effective. Denote the  {\it exclusion radius} $$\text{\rm{\color{black} radius}}(x_{ 1},\dots ,x_{ N}) = \sup\{\rho: \|x_{ i}-x_{ j}\|\geqslant \rho ,i\neq j \}.$$ 

\begin{theorem}[\cite{Sarrazin}, Thm. 3]  Let $ P$ a point configuration on $ \T$ with $ N$ distinct points.
\begin{align*}
  \mathscr  E(  \mathscr  T P)\leqslant  Nc_{ d}\text{\rm{\color{black} radius}}(P)^{ 1-d}
\end{align*}
with the constant $ c_{ d} = \frac{ 2^{ 2d-1}}{\kappa _{ d-1}} d^{ \frac{ d + 1}{2}}.$  

 \end{theorem}  
For comparison, recall from Theorem \ref{AKT}  that if $ \P $ consists in $ N$  i.i.d points uniform over $  \T  $, $  \mathscr  E(\P )\sim N\ln(N)$, and by  \eqref{eq:lower-bound-Wp}, the minimum of $  \mathscr  E$ is larger than $ c_{ d,2}N$.  

\  

 
 \subsection{Pairwise potential}
   
Another method  consists in choosing a radius $ \rho>0 $ and minimise the energy on the spectral domain over frequencies in $ B(0,\rho )  $: 
\begin{align*}
  \mathscr  E(\P) =   \S_{ \P}(1_{ B_{ \rho }}),
\end{align*}
where the spectral measure is defined at  \eqref{eq:def-spectral-density}. Recall that in all the section the intensity is fixed at $ 1$, which will induce constraints on $ \rho $ to have a flat spectrum on $ B_{ \rho }$. In any case, minimising $  \mathscr  E$ should lead to a spectrum as flat as possible on $ B_{ \rho }$. More generally, one can choose a bounded function $ \psi $ with fast decay on $ \mathbb{R}^{ d}$ and define the energy
\begin{align}
\label{eq:gauss-energy}
  \mathscr  E_{ \psi }(\P ) = \S_{ \P }(   \psi  ).
\end{align} 
We know from Proposition \ref{lm:spectral} and  \eqref{eq:exact-spectr} that  
\begin{align*}
\lambda ^{ 2d}\mathscr  E_{ \psi }(\P ) =\lambda ^{ 2d}\S_{ \P}(\psi ) = &\sum_{u\in \mathbb{Z} _{ \lambda }^{ d}}\psi (u)\lambda ^{ 2d}\S(\{u\})\\
 = &\sum_{u\in \mathbb{Z} _{ \lambda }^{ d}} \psi (u)\left(
\mathbf E \left[
 | \F \P(u) | ^{ 2}
\right] - \left | 
\mathbf E \left[
\F \P(u)
\right]
\right | ^{ 2}
\right)\\
 = &\sum_{u\in \mathbb{Z} _{ \lambda }^{ d}}\psi (u)\mathbf E\left[
 \sum_{x, y\in \P}e^{ iu\cdot (x-y)}
\right] - \sum_{u\in \mathbb{Z} _{ \lambda }^{ d}}\psi (u)\left|
\displaystyle\int_{\T}e^{ iu\cdot x}dx
\right|^{ 2}\\
 = &\mathbf E \left[
\sum_{x,y\in \P} \hatn{ \psi }(x-y)
\right] -\lambda ^{ 2d} \psi (0)\end{align*}
with the  notation 
\begin{align*}
 \hatn{\psi }(x) = \sum_{u\in \mathbb{Z} _{ \lambda }^{ d}}\psi (u)e^{ iu\cdot x}.
\end{align*}
Hence, minimising $  \mathscr  E_{ \psi }$ amounts to having particles $ x\neq y$ interact pairwise through the force $ \hatn{\psi }(x-y)$.  Seeing $  \mathscr  E_{ \psi }$ as a function symmetric in $ N$ particles $ x_{ i}$, the gradient can then be written out explicitly as a force acting on the particles:  $$ \P_{ n + 1}  = \P_{ n} -  \nabla  \mathscr  E(\P_{ n})= \{x_{ i} -   \sum_{j\neq i} \nabla \hatn{ \psi }(x_{ i}-x_{ j}),1\leqslant i\leqslant N\}.$$ Let us present two examples from the literature. \\

\draft {\color{red} Massoulié discussion - on graph $ \mathbb{Z} ^{ d}$. Let $ G$ a subset of $ \mathbb{Z} ^{ d}$.
 
 Regularity on a graph is if the variances are small 
\begin{align*}
  \textrm{Var}\left(G(f_{ R})\right).
\end{align*}
 
 Particles on a graph have a mutual repulsion $ \hat \psi (i-j)$. Using Fourier transform on $ \mathbb{Z} ^{ d}$, it gives target energy
\begin{align*}
\sum_{}| \psi  | ^{ 2}\S_{ G}
\end{align*}
 
 What about edges $ A_{ x,y}$? Repulsion 
\begin{align*}
 \sum_{x,y}A_{ x,y} \hat \psi (x-y)
\end{align*}
Then take Fourier transform on whatever sense. We abstractly need 
\begin{align*}
  \textrm{Var}\left(\sum_{x\sim y} f(x,y)\right) = \int_{} | \hat f | ^{ 2}S
\end{align*}

Can give something for a stationary model?
 }

 {\bf  Target a stealthy spectrum.}
In a vast literature going back to Batten et al.  \cite{stealthy08}, members of the Princeton team apply the previous procedure with $ \psi (u) = 1_{ B_{\rho  }}$ for some parameter radius $\rho >0$. They derive an important parameter $ \chi $, monitoring the ratio between the number of constraints  and degrees of freedom. The constraints correspond to frequencies set to $ 0$, roughly of the order $  \frac{ 1}{ 2}   \#\mathbb{Z} _{ \lambda }^{ d}\cap B_{ \rho }\sim \kappa _{ d}\rho ^{ d} (2\pi \lambda ^{ -1})^{ -d}$, and the number of degrees of freedoms is $ d\times N$, where the number of particles satisfies $ N\sim \lambda ^{ d}$, assuming intensity is fixed to $ 1$. Therefore, after removing dimensional constants, the ratio is proportional to $ \rho ^{ d}$, i.e. $ \chi  = c_{ d} \rho ^{ d}$ for some explicit $ c_{ d}>0$. 
  \begin{figure}[h!]
  \begin{center}
  \includegraphics[width = 14cm]{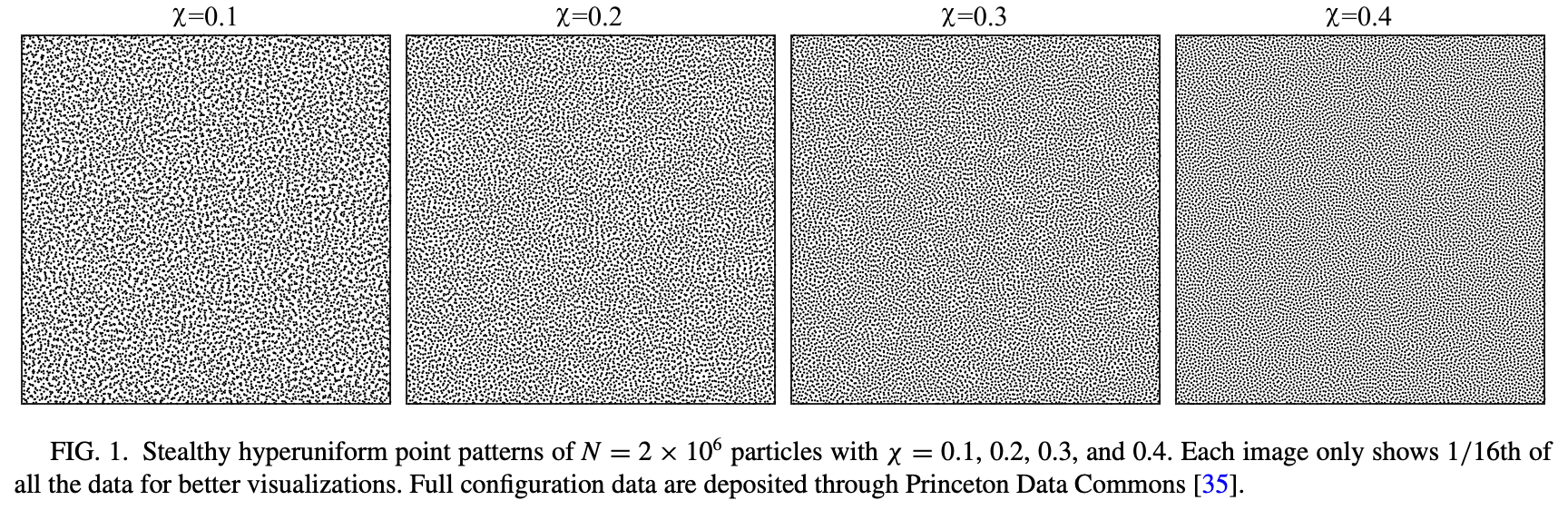}
  \label{fig}
  \end{center}\caption{  \cite[Figure 1]{Morse}Samples of $ 2\times 10^{ 6}$ particles with different values of $ \chi .$ Each image only shows $ 1/16$th of all the data for better visualization. The data is available at \url{https://doi.org/10.34770/e49n-r807}}
  \label{fig:morse}
  \end{figure}
In  \cite{stealthy09}, the same authors 
identify experimentally three regimes, separated by thresholds $ 0<\rho _{ d,1}<\rho _{ d,2}$:
\begin{itemize}
\item At $\rho <\rho  _{ d,1} $, we have regular disordered samples as in Figure  \ref{fig:morse}.  Morse et al. \cite{Morse} attain with this method the smallest attainable value within machine precision, with  $ |  \s_{ \P_{ n}}(u) | $ on $ B_{ \rho }$ as small as $ 10^{ -46}$, for $ n^{ d} \approx 2\times 10^{ 6}$ in dimension $ d = 2.$ The fact that the number of degrees of freedom minus the number of constraints grows as $ \rho ^{ d}$ pleads for the existence of a very large class of infinite volume homogeneous stealthy systems.
 Physicists claim that the manifold of minimisers is connected, and that exploring its topology is an  {\it outstanding open problem}  \cite{Tor16}. 
\item At intermediate values $ \rho  _{ d,1}<\rho <\rho _{ d,2}$, a  {\it wavy crystal} phase emerges, also called  {\it stacked-slider phase}, already noticeable in the five years older study  \cite{wavy-cryst} of the same team,  see in particular their Figures 5,7. This phase  seems to be noticeable only in dimension $ d\geqslant 2$ (i.e. $ \rho _{ 1,1} = \rho _{ 1,2}$). It is reported in the physics literature that $ \rho _{ d,2}$ corresponds to $ \chi  = 1/2$, i.e. the transition occurs when the number of degrees of freedom is twice the number of constraints, even if the experimantal values seem to differ slightly ($ \chi  \approx  0.57655$ if $ d = 2$, $ \chi \approx 0.50066$ if $ d = 3$).
\item At values $ \rho _{ d,2}<\rho <1$, the sample becomes crystalline, i.e. periodic, recalling that we have by convention a unit intensity process.
\item For $ \rho >1$, it is mathematically impossible to obtain stealthy samples with these constraints (see the theoretical justification with Proposition \ref{prop:stealthy-density}), the limit case  is actually obtained with a periodic lattice configuration. 
\end{itemize}

\draft{For a mathematical treatment, it might be wise to smoothen $ \psi $ at the corners to avoid aliasing effects, take for instance
 $ \psi (u) = 1_{ B_{ \chi  }} \ast \exp(-\|u\|^{ 2}/\sigma ^{ 2})$ with a small $ \sigma $, 
\begin{align*}
 \hat \psi (x) \sim (1 + \chi \|x\|)^{ -(d + 1)/2}\exp(-\sigma ^{ 2}x^{2}).
\end{align*}
To keep a compact spectrum, one can replace $ \exp(-t^{ 2}/\sigma ^{ 2})$ by a Schwartz function supported by $ [-\sigma ,\sigma ]$ (with superpolynomial decay). Other potentials are considered in  \cite[III]{TZS}.}

The team of   Martiniani (NYU) refined this approach in  \cite{casiulis} by targetting any kind of spectrum (not only flat on a ball) with an approach based on non-uniform fast Fourier transform (algorithm  {\bf FReSCo}, Fast Reciprocal-Space Correlator). They prioritize generating large particle sets over achieving perfectly flat spectra, enabling the creation of the largest stealthy hyperuniform samples  to date: \(10^{9}\) particles in 2D and \(10^{7}\) in 3D, with spectral precisions in the range \(10^{-20}\) to \(10^{-25}\). They argue that this level of precision is  adequate for most applications, since even a positional uncertainty on the order of \(10^{-9}\) on the particles already disrupts the effective accuracy.\\

\begin{figure}[h!]
\begin{center}
\includegraphics[width = 15cm]{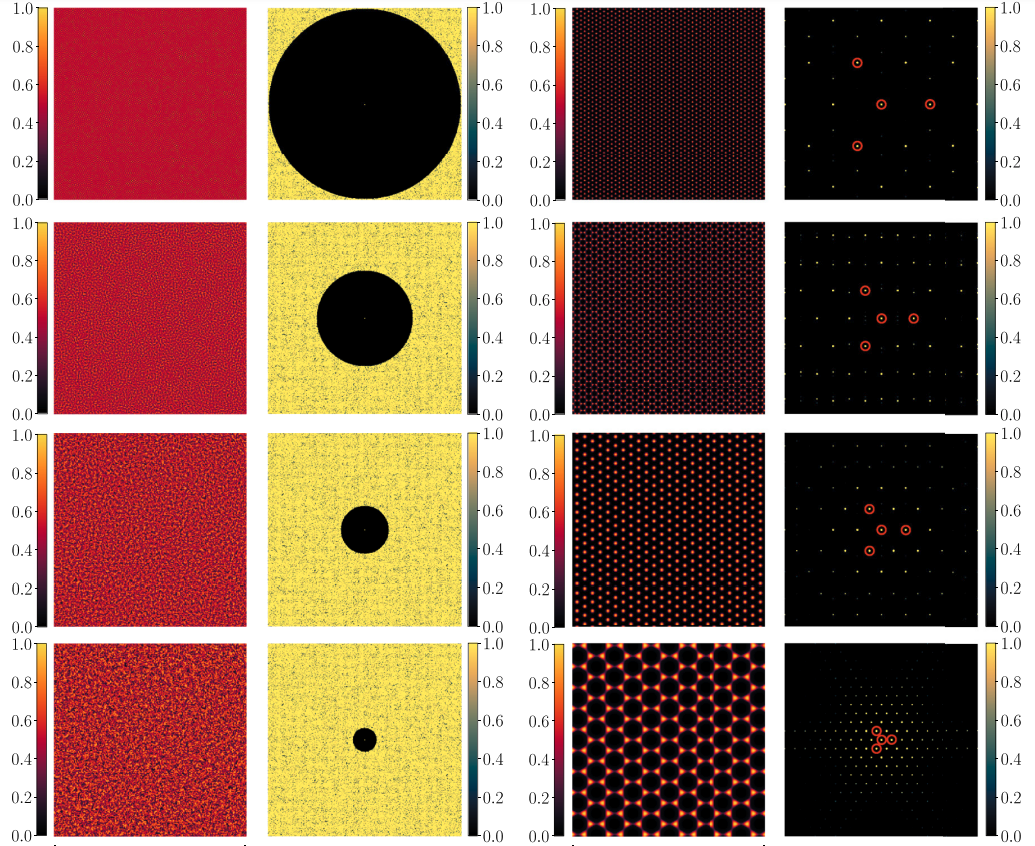}
\label{fig:casiulis-QC}\caption{  \cite{casiulis}: { Density fields of 403x403 are generated. Samples are generated to exhibit varying degrees of stealthiness, as demonstrated in the varying radius  in the structure factor. The larger the radius, the more uniform the system appears. UwNU (right) systems are generated to exhibit
sixfold symmetry, imposing only the peaks of the innermost hexagon (marked with red circles). As the radius of the innermost hexagon or
peaks changes, we sometimes get a Kagome-like tiling instead of a triangular lattice.}}
\end{center}
\end{figure}

  {\bf Gaussian blue noise.}  
A parallel branch of literature, more involved with image processing, optimal transport, numerical integration and machine learning, is concerned with producing  {\it blue noise} samples, which corresponds to the same requirements as for  { stealthy / hyperuniform processes}. We described above some algorithms based on underlying Laguerre-type tessellations, but the state of the art seems to be currently the so-called  {\it Gaussian blue noise}, which corresponds to taking pairwise repulsive Gaussian forces, i.e.  \eqref{eq:gauss-energy}  with a Gaussian kernel (recall that the Fourier transform of a Gaussian kernel is another Gaussian kernel).  The separable form of the Gaussian $ \exp(-\|u\|^{ 2}) =  \prod_{i = 1}^{ d} \exp(-u_{ i} ^{ 2})$ allows a parallelisation and an efficient scalability.  It is an old idea, already present in  \cite{DitheringBlueNoise}, and it resurfaces regularly, see the last contribution  \cite{gaussian-blue-noise} and references therein. In the latter article, they produce stealthy samples with precision in the range $ 10^{ -10}$ to $ 10^{ -15}$, much below the previous method  \cite{deGoes}, but they rather focus on a balance with the radius  of the spectral gap, not just the spectrum flatness like in physics works.  In numerical experiments, $ \exp(-\|x\|^{ 2}/\sigma ^{ 2})\leqslant 10^{ -11}$ for $ \|x\|\geqslant 5\sigma $, hence the procedure  amounts to a smooth kernel with compact support. The discussion about the right choice of $ \sigma $ is the same as the right choice of $ \chi $ / $ \rho $ above, balancing between a not too small $ \sigma $ (too much disorder) and a too large one (samples become periodic), the authors actually recommand $ \sigma  = 1$ at unit intensity.  They also demonstrate the possibility to use their method up to dimension $ 24$, whereas previous works in physics are mostly in dimension $ d = 2,3$, and the centroid algorithms below are not implemented for $ d>5$.

   \draft{
   \subsection{Discussion}

 The problem of making converge a gradient descent is crucial in many applied fields, in particular machine learning, and such algorithms have been overpolished in the last decades. Practitionners have exploited this technology to produce  hyperuniform samples $ \P_{ n}^{ \infty }$ effectively.

  {\color{blue} An interesting and deep question is to what extent the limit process obtained is indeed  {\it random}. What is looked for in practice is a  {\it disordered sample}, and a visual inspection often seems to validate this property. As in Section ..., other quantitative properties are used, in particular boundedness and isotropy of the empirical spectrum.  
The convergence of such algorithms is a very deep question of mathematics that we should not discuss here. In general, it is hoped to converge to a local minimum $ \P_{ n}^{ \infty }$, and by rotational invariance, there is no unicity as any rotation of $ \P_{ n}^{ \infty }$ is also a local minimum.
The question of wether $ \P_{ n}^{ \infty }$ is  {\it genuinely} random is also pretty deep.
 The underlying ideas is the following. The sample minimises some energy function $  \mathscr  E$, hence it is suppose to reach a minimum. Let us discuss some alternatives. \begin{itemize}
\item [(a)]When the set of constraints is not too large with respect to the number of points, typically when one only imposes  frequencies to vanish in  a ball $ B_{ \rho }$ for $ \rho $ sufficiently small, the set of global minima $  \mathscr  M_{ n}: =  \{ \mathscr  E =  \min _{  (  \mathbb  T  _{ n}^{ d})^{ n^{ d}}} \mathscr  E\}$ is large and described as the  {\it ground-state manifold}, and apparently is large enough so that the final state $ \P_{ n}^{ \infty } $ satisfies the constraints above. The whole procedure can be seen as a random sampling in $   \mathscr  M_{ n}.$ A quantitative assertion would be that $ \text{\rm{\color{black} dim}}(  \mathscr  M_{ n})\to \infty .$ It is interesting to notice that even in this large space, there might  be a universal unique model living in $  \mathscr  M_{ n}$, which means that even if the initial positions are dependent (determinantal, hard-sphere, hyperfluctuating, ...), the limit $ \P_{ n}^{ \infty }$ converges to some unique infinite stationary point process $ \P$, see  \cite{Klatt-nature} for an illustration with the cellular Lloyd algorithm (Section ...), or at least its second order structure, i.e. the spectral measure $ \S$ is unique.
\item [(b)] When the set of constraints becomes too large, e.g. when $ \rho $ increases, the manifold becomes too small, say with a bounded dimension over $ n$, and the limit configuration $ \P_{ n}^{ \infty }$ has a crystalline ``parametric'' law, similar to a lattice. 
\item [(c)] With more constraints, e.g. for  larger $ \rho $, the local minima might not be reachable (see for instance Theorem ...) and the behaviour is more unpredictable.
\end{itemize}
This heuristic has been illustrated by the team of Torquato as early as 2009?, see also Figure ... We describe in more detail below several procedures and numerical results, and leave some (likely difficult) open questions: 

 \begin{question}
\begin{itemize}
\item For $ \rho $ small enough, do we have $ \text{\rm{\color{black} dim}}(  \mathscr  M_{ n})\to \infty $? This would correspond to a satisfactory  {\it disordered behaviour}.
\item For $ \rho $ in an intermediate regime, do we have at the opposite some $ d\in \mathbb{N}$ such that $ \text{\rm{\color{black} dim}}(  \mathscr  M_{ n})\leqslant d$?
\end{itemize}
\end{question}
This discussion is essentially heuristic, as the realm of physics is populated with phase diagrams where it is not possible to prove anything rigourously.
  }

}

  \section{Hard sphere models and jammed packings}
  
  While this does not necessarily  give the most efficient simulation procedure, it is important to mention as a conclusion that the physics literature about  hyperuniformity makes numerous references and simulations of hard sphere models.  Formally, for $ \rho >0$, a $ \rho $-hard sphere model is a configuration $ P\in \N$ such that for $ x,y\in P,x\neq y  \Rightarrow\; \|x-y\|\geqslant 2\rho .$ Such models are important among other things in condensed matter physics, and also to the fundamental problem of  {\it densest packings}. We refer the reader to the survey of Torquato  \cite{Tor18} for insights about this problem from the physics point of view, see also the recent breakthrough by Viazovska and her co-authors  \cite{viazovska8,dim24packings}, but many  questions remain open in both fields. Let us present Torquato's conjecture about maximally jammed packings.
  
Let us first call  {\it saturated} a hard sphere model maximal  with respect to inclusion, i.e. such that no other point can be added to the model without breaking the hard sphere constraint.
  It is tempting to believe that  an infinite random stationary  saturated hard sphere model has low fluctuations for the number of balls in a large window, but it turns out to be false. To build a counter-example, let $  \bar \P\subset \mathbb{R}^{ d} \times \mathbb{R}_{  + }$ a ``marked'' homogeneous Poisson point process with intensity $\Leb \times (\Leb[1]1_{ \mathbb{R}_{  + }}),$ where each marked point $ (x,t)$ can be seen as a location $ x$ and a birth time $ t>0.$ Consider the random set  built by adding chronologically the points: a point $ (x,t)$ is  {\it accepted} if its sphere $ B(x,\rho )$ does not overlap   the sphere $ B(y,\rho )$ of a previously accepted marked point $ (y,s)$, with $ s<t$.
  Defining this model properly in infinite volume is not  trivial, as one must rule out the possibility of infinite causal chains to determine the acceptance status of a given point $ (x,t)$, see for instance 
  Penrose, Shreiber, Yukich  \cite{SPY07} for a formal introduction. Let $ \P_{ T}$ the set of accepted points at time $ T>0.$ It converges in law in the vague topology towards a stationary point process $ \P_{ \infty }$ called  {\it random sequential absorption model at saturation}, which is a saturated model of hard spheres and is not  hyperuniform:

    \begin{theorem}[\cite{SPY07}, Th. 1.1]

\begin{align*}
 \lim_{ R\to \infty } \frac{  \textrm{Var}\left(\#\P_{ \infty }\cap B_{ R}\right)}{R^{ d}} = \sigma ^{ 2}>0.
\end{align*}
   \end{theorem}  
   These authors furthermore prove a central limit theorem.
Torquato et al.  \cite{TS03}  observed that adding another constraint, namely  {\it jamming}, can lead to experimentally  hyperuniform configurations, if the jamming is strong enough.
  
  \begin{definition}[  \cite{Tor18}, Section 11.1]
  Say that a $ \rho $-hard sphere configuration $ P\in \N$ is  {\it locally jammed} if each $ x\in P$ has $ d + 1$ neighbours preventing it from moving, i.e. there are $ x_{ 1},\dots ,x_{ d + 1}\in P$ at distance $ 2\rho $ from $ x$ which do not lie in a half space having $ x$ on its boundary.
  
 A {\it collectively jammed} configuration is a locally jammed configuration in which no subset of particles can simultaneously be
displaced so that its members move out of contact with one another and with the remainder set fixed. See for instance  Figure \ref{fig:jamm}, left, for a counter-example.
  
  \end{definition}  
 
 Local jamming or collective jamming alone is not enough to guarantee HU. Indeed, such configurations contain  {\it rattlers}, i.e. finite group of particles that can leave the rest of the configuration unchanged, and hence generate sufficiently many variability to break HU.

 \begin{figure}[h!]
 \begin{center}
 \includegraphics[width = 8cm,  trim=0cm 4cm 0cm 0cm,
                     clip]{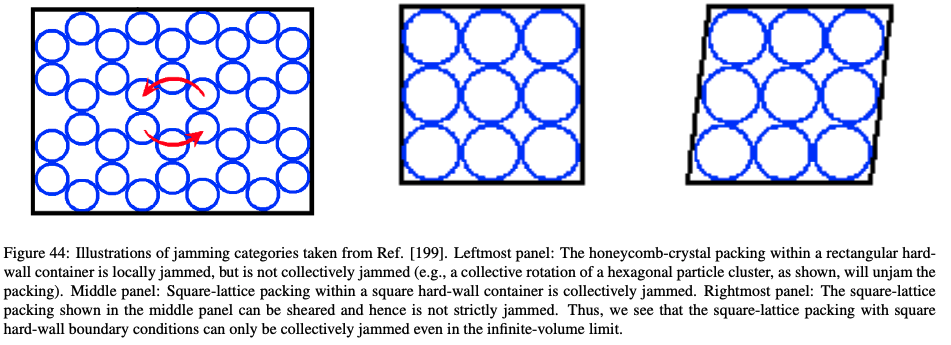}\caption{}
 \label{fig:jamm}
 \end{center}
 \end{figure}
A strictly jammed finite system is  {\it any collectively jammed configuration that disallows all uniform volume-decreasing strains of the system boundary. }Torquato and Stillinger  \cite{TS03} establish the conjecture that  infinite strictly jammed configurations are  hyperuniform.

 Let us formulate a tentative mathematical definition of  {\it strict jamming} for an infinite random stationary configuration.

\begin{definition}
Say that a hard-sphere stationary point process $ \P$ is  {\it strictly jammed} if there is no family of hard-sphere stationary point processes $ \P_{ t},t\in [0,1]$ such that $ \P_{ 0} = 0$, $ t  \mapsto  \P_{ t}$ is continuous in law in the vague topology, and the intensity $ \lambda _{ t}$ of $ \P_{ t}$ is non-decreasing with $ \lambda _{ 1}>\lambda _{ 0}$.
\end{definition}  
Remark that any infinite system achieving maximal density, such as the hexagonal lattice in dimension $ 2$ or the face-centered cubic (FCC) packing in dimension $ 3$, or Viazovska's E8 and leech lattices in dimensions $ 8$ and $ 24,$ is immediately strictly jammed since the intensity cannot increase. Such models are immediately seen to be  hyperuniform, even stealthy, as they have a locally finite spectrum.
 Figure \ref{fig:jamm}-(middle, right) shows why the square lattice is not strictly jammed, as there is a collective displacement with strictly decreasing intensity.

        \chapter{Appendix}
        \label{app}

        We give several necessary concepts and constructions for some results related to Fourier transforms.
        Most of the material is covered in the classical textbooks  \cite{Rudin,SteinWeiss,Berg}.
        \section{Schwartz functions and tempered measures}
        
         Say that a measurable function $ f: \mathbb{R}^{ d} \to  \mathbb C $ has  {\it fast decay} if  for all $k>0$, $$ \sup_{ x\in \mathbb{R}^{ d}}(1 + \|x\|)^{k } | f(x) | <\infty .$$    

\begin{definition}
Call Schwartz function a $  \mathscr  C^{ \infty }$ function $ f:\mathbb{R}^{ d}\to  \mathbb C $ such that all its derivatives have fast decay.
The space of Schwartz functions is denoted by $ \Sc(\mathbb{R}^{ d}).$ \end{definition}   
 Such functions are in $ L^{ 1}(\mathbb{R}^{ d})\cap L^{ 2}(\mathbb{R}^{ d})$, their Fourier transforms are defined unambiguously through 
\begin{align*}
 \hat f(u) = \int_{\mathbb{R}^{ d}}e^{ -iu\cdot x}f(x)dx,
\end{align*} and satisfy the duality relation $ \widehat{ \hat f}(x) = (2\pi )^{ d}f(-x).$ The typical example is Gauss kernel $ f(x) = \exp(-\|x\|^{ 2})$. The fundamental observation is that if $ f\in \Sc$, then also $ \hat f\in \Sc$  \cite{Rudin2}.

The elements of the space $ \Sc'(\mathbb{R}^{ d})$ of continuous linear functionals in the Schwartz topology over $ \Sc(\mathbb{R}^{ d})$ are called  {\it tempered distributions}. The value of the application of a distribution $ L$ to a Schwartz function $ f$ is denoted by $ \langle L,f \rangle$.  A {\it  Tempered measure} is a measure $ \mu $ such that $ \langle L,f \rangle = \mu (f)$ defines a tempered distribution $ L.$

       The  Fourier transform can be extended to $ \Sc'(\mathbb{R}^{ d})$ by duality.  The corresponding operator is denoted by $ \F$. 

\begin{theorem}
        For $ L\in \Sc'(\mathbb{R}^{ d}) $, there exists a tempered distribution denoted by $ \F L\in \Sc'(\mathbb{R}^{ d})$ and characterised by  
\begin{align*}
 \langle \F L,f \rangle =  \langle L, \hat f \rangle, f\in \Sc(\mathbb{R}^{ d}).
\end{align*}
For $ f\in \Sc, \F f = \hat f.$
        \end{theorem}  
        
     \section{The phase-space formula on a locally compact abelian group}
     \label{app:wsrm}

     In this section, we consider more generally a  locally compact abelian topological group (LCAG for short) denoted by $ \X$,  endowed with a Haar measure (i.e. invariant under translations), denoted simply by $ dx$. We follow the first chapter of Berg \& Forst  \cite{Berg} for concepts and definitions. We introduce the dual group, denoted by $ \hat \X$ and described below. We only consider for $ \X$ a subgroup or quotient of $ \mathbb{R}^{ d}$, and define the Fourier transform as 
\begin{align*}
 \hatX{f}(u) = \int_{ \X}f(x)e^{ -iu\cdot x}dx, u\in \hat\X.
\end{align*}
Most of the theory is valid for general locally  compact abelian groups, replacing functions $ x   \mapsto  e^{ ix\cdot u}$ by character functions, but it is beyond the scope of examples considered here.

In this context, the Plancherel theorem (\cite[Th. 2.5]{Berg}) states that there is a Haar measure $ \mu _{  \hat \X}$ on $ \hat \X$ such that for $ f\in L^{ 1}(\X)\cap L^{ 2}(\X)$, $ \hatX f\in L^{ 2}(   \mu _{  \hat \X})$ and
\begin{align*}
\|f\|_{ L^{ 2}(\X)} ^{ 2}= \| \hatX f\|_{ L^{ 2}(  {  \hat \X})}^{ 2}.
\end{align*}

\begin{example}
For $ \X = \mathbb{R}^{ d}, $ endowed with Lebesgue measure, $ \hat \X =  \mathbb{R}^{ d}$ and $ \mu _{ \hat\X} = (2\pi )^{ -d}\Leb$.  When there is no ambiguity and $ \X = \mathbb{R}^{ d}$, simply write $ \hatX{f} = \hat f.$

\end{example} 

\begin{example}
If $ \X = \T,\lambda >0,$ endowed with $ \Leb_{ \X}$, the Lebesgue measure restricted to $ \X$, then $ \hat\X = \mathbb{Z} _{ \lambda }^{ d}: = 2\pi \lambda ^{ -1}\mathbb{Z} ^{ d}$, and $ \mu _{ \hat\X} =\lambda ^{ -d}\mu _{ \mathbb{Z} _{ \lambda }^{ d}}$, where $ \mu _{ \mathbb{Z} _{ \lambda }^{ d}}$  is the counting measure on $ \mathbb{Z} _{ \lambda }^{ d}.$ By duality, if $ \X = \mathbb{Z} ^{ d}$  endowed with $ \mu _{ \mathbb{Z} ^{ d}}$, $, \hat \X = \T[2\pi ]$ and $ \mu _{ \hat\X} = (2\pi )^{ -d}\Leb.$
\end{example} 

 { For a complex Radon measure \(\mu\), denote its total variation measure by
\(|\mu|\). For an open set \(A\subset \X\), it is defined by
\[
|\mu|(U)
=
\sup \left\{
|\mu(f)|:
f\in \mathscr C_c^0(A),\ \|f\|_\infty\leq 1
\right\}
\]and extends by $ \sigma $-additivity on other sets. 
The space of complex Radon measures is endowed with the vague topology, i.e.
the topology generated by the maps
\[
\varphi_f:\mu\mapsto \mu(f),
\qquad f\in \mathscr C_c^0(\X),
\]
and with the corresponding Borel \(\sigma\)-algebra.

Theorem 2.6 in \cite{Berg} gives a general inversion result, valid for every
complex measure \(\mu\) with \(|\mu|(\X)<\infty\). If
\[
\hatX\mu(u):=\mu(e^{-iu\cdot})
\]
belongs to \(L^1(\hat \X,\mu_{\hat \X})\), then \(\mu\) is absolutely continuous
with respect to Haar measure on \(\X\), with density
\[
x\mapsto
\int_{\hat \X} e^{ix\cdot u}\hatX\mu(u)\,\mu_{\hat \X}(du).
\]
Consequently, if \(\bar\mu\) denotes the conjugate measure, then for a suitable
test function \(h\),
\begin{align}
\label{eq:duality}
\int_{\X} h(x)\,\bar\mu(dx)
&=
\int_{\X}\int_{\hat \X}
e^{-ix\cdot u}\overline{\hatX\mu(u)}h(x)\,
\mu_{\hat \X}(du)\,dx  \\
&=
\int_{\hat \X}
\overline{\hatX\mu(u)}\,\hatX h(u)\,\mu_{\hat \X}(du).
\end{align}}

      We then consider in this measurability framework a complex random measure $ \M$ which is locally square integrable and weakly $ \X$-stationary, i.e. such that for $ A\subset \X$ bounded, $ \mathbf E \left[
 | \M | (A)  ^{ 2}
\right]<\infty $ and for $ f\in \mathscr B_{ c}(\X),x\in \X,$ $$\mathbf E \left[
\M(\tau _{ x}^{ \X}f)
\right] = \mathbf E \left[
\M(f)
\right],\;\;   \textrm{Var}\left(\M(\tau _{ x}^{ \X}f)\right) =   \textrm{Var}\left(\M(f)\right),$$ where $ \tau _{ x}^{ \X}$ is the operator of translation by $ x$ in $ \X.$ Note that we are not restricted to non-negative random measures. Say that a measure $ \mu $ is translation bounded if for all $ f\in \mathscr B_{ c}(\X)$, 
\begin{align*}
 \sup_{ x\in \X} | \mu |  (\tau _{ x}^{ \X}f)<\infty .
\end{align*}On $\X =  \mathbb{R}^{ d}$, it implies by additivity that $ |  \mu | (B(0,R)) = O(R^{ d}) $ by covering $ B(0,R)$ by $ O(R^{ d})$ unit radius balls,   hence it implies that $ \mu $ is a tempered measure. 

\begin{theorem}
\label{thm:wsrm-X}
Let $ \M$ a locally square integrable weakly $ \X$-stationary random measure on $ \X.$
 There is a complex translation bounded measure $ \C$ on $ \X$ and a non-negative translation-bounded  measure $ \S$ on $ \hat \X$ such that for $ f,g$ bounded with compact support 
\begin{align*}
  \textrm{Cov}\left(\M(f),\M(g)\right) = \int_{ \X^{ 2}} f(x)   \overline{   g(x + y)}dx\C(dy) = \int_{ \hat \X}\hatX  f(u)    \overline{ \hatX g(u)}   \S(du) .
\end{align*}Also, $ \C$ is Hermitian, i.e. $ \C(-.) =  \bar \C$. If $ \M$ is real-valued, $ \C$ and $ \S$ are real-valued and symmetric.

This also holds for $ f,g$ Schwartz functions if $ \X = \mathbb{R}^{ d}.$
 
 \end{theorem}  
 
\begin{proof}
Denote the  {\it tilted convolution} by 
\begin{align*}
f \tconv g(y): =  \int_{ \X}f(x)  \overline{  g(x + y)}dx.
\end{align*}
 
Let
\[
\widetilde \M=\M-\mathbf E\M.
\]
Since \(\mathbf E[|\M|(K)^2]<\infty\) for $ K$ bounded, define for $ A,B$ bounded
\[
L(A\times B)
=
\mathbf E\left[\widetilde \M(A)\overline{\widetilde \M(B)}\right].
\]
Then \(L\) is a locally finite complex Radon measure on \(\X\times\X\), and
\[
\operatorname{Cov}(\M(f),\M(g))
= 
\int_{\X\times\X} f(x)\overline{g(y)}\,L(dx,dy).
\]

Stationarity implies diagonal translation-invariance of \(L\):
\[
L((A+t)\times(B+t))=L(A\times B).
\]
Push \(L\) forward by
\[
(x,y)\mapsto (x,y-x).
\]
Then the pushed-forward measure is invariant in the first coordinate, hence has the form
\[
dx\otimes \C(dz)
\]
for a Radon measure \(\C\) on \(\X\) (see \cite[Section A2.7]{DVj08}). Therefore
\[
\operatorname{Cov}(\M(f),\M(g))
=
\int_\X\int_\X f(x)\overline{g(x+z)}\,dx\,\C(dz).
\]
We immediately have with the change of variables $ y = x + z$, with $ \tilde \C =  \bar  \C(-\;.)$, and $ \tilde L$ the corresponding bilinear operator,
\begin{align*}
 L(f,g) = \displaystyle\int_{\X}\left(
\displaystyle\int_{\X}f(y-z)  \bar g(y)
dy\right)\C(dz) = \displaystyle\int_{}  \bar g(y)f(y + z)   \C(-dz) =\overline{   {  \tilde L(g,f)}
}\end{align*}
since $  L(f,g) = \overline{L(g,f)}$ by complex covariance properties, it gives $ \tilde \C = \C,$ i.e.  $ \C$ is Hermitian. When $ \M$ is real-valued, testing $ \C$ over real-valued test functions yields $  \bar \C = \C$, hence $ \C$ is symmetric.

This measure (or tempered distribution) $ \C$ is positive-definite in the sense that for $ f\in  \mathscr  C_{ c}^{ 0}(\X)$
\begin{align*}
 \int_{ \X^{ 2}} f(x)  \overline{  f(x + y)}dx\C(dy) =  \textrm{Var}\left(\M(f)\right)\geqslant 0.
\end{align*}
Therefore, by Bochner's theorem on LCAGs, \cite[Th. 4.5]{Berg},  \footnote{$ \C$ is formally a distribution and not a measure. The cited result   formally treats   actual measures on $ \X$, and in the current setup $ \C(\X)$ might not make sense because $ \C$ is not a measure, but only the fact that $ \C$  is locally a measure is used in the theory of harmonic analysis on LCAGs.}   there is a non-negative measure $ \S $ on $ \hat \X$ such that 
\begin{align*}
L(f,g) = \int_{ \hat \X } \widehat{ f \tconv g}(u)d \S_{ 0}(u) , f,g\in \mathscr C^{ 0}_{ c} (\X)
\end{align*}
with the right hand side always finite. An easy computation gives  
\begin{align*}
 \hatX{f\tconv g}(-u) = \hat f(u)\overline{ \hat g(u)},u\in \hat \X.
\end{align*}  Let $ \S = \S_{ 0}(-\;.)$. 
Then  
\begin{align*}
 \displaystyle\int_{\X} f\tconv g(y)\C(dy) = \displaystyle\int_{ \hat \X} \hat f(u) \overline{ \hat g(u)}\S(du),
\end{align*}
proving the desired relation.

 On $ \mathbb{R}^{ d}$, there is also the original generalisation of Bochner's theorem by Schwarz (\cite[Th. XVIII]{schwartz}), that gives the same result for functions $ f,g\in \Sc(\mathbb{R}^{ d}).$

 {Let us now extend the result to arbitrary bounded functions by approximation. Let \(f\in \mathscr B_c(\mathbb R^d)\). Choose
\(\delta\in\CC_c^\infty(\mathbb R^d)\), \(\delta\geq 0\),
\(\int\delta=1\), and set
\[
\delta_i(x)=i^d\delta(ix),\qquad f_i=f*\delta_i .
\]
Then \(f_i\in\CC_c^\infty(\mathbb R^d)\), the supports of the \(f_i\)'s are contained in one compact set \(K\), and
\[
\|f_i\|_\infty\leq \|f\|_\infty,\qquad f_i(x)\to f(x)
\]
at every Lebesgue point of \(f\). Let \(N\subset K\) be the exceptional Lebesgue-null set.
We first note that if \(h\in \mathscr B_c(\mathbb R^d)\) is supported by a Lebesgue-null set, then
$
\M(h)=0\quad\text{in }L^2. $
Indeed, for every compact \(A\) with positive Lebesgue measure and for fixed \(\omega\),
\[
\int_A |\M(\omega)|(N+x)\,dx=0
\]
by Fubini, since \(N\) is Lebesgue-null. Hence \(\M(\tau_x h)=0\) for Lebesgue-a.e. \(x\in A\). Using
\(\mathbf E[|\M|(K')^2]<\infty\) for compact \(K'\), Fubini gives
\[
0=\int_A \mathbf E|\M(\tau_xh)|^2\,dx.
\]
By weakly stationarity, \(\mathbf E|\M(\tau_xh)|^2\) is independent of \(x\). Therefore
\[
\mathbf E|\M(h)|^2=0.
\]

Now decompose
\[
f_i-f=(f_i-f)\mathbf 1_{K\setminus N}+(f_i-f)\mathbf 1_N .
\]
The second term is \(L^2\)-null by the previous observation. For the first term, pathwise dominated convergence with respect to the finite measure
\(|\M|_{\mid K}\) gives
\[
\M\big((f_i-f)\mathbf 1_{K\setminus N}\big)\to0
\quad\text{a.s.}
\]
Moreover,
\[
\left|\M\big((f_i-f)\mathbf 1_{K\setminus N}\big)\right|
\leq 2\|f\|_\infty |\M|(K),
\]
and the right-hand side is square-integrable. Hence
\[
\M(f_i)\to \M(f)\quad\text{in }L^2.
\]
Consequently,
\[
\operatorname{Var}(\M(f_i))\to \operatorname{Var}(\M(f)).
\]

Since the variance identity is already known for \(\CC_c^\infty\) functions,
\[
\operatorname{Var}(\M(f_i))
=
\int_{\mathbb R^d} |\widehat{f_i}(u)|^2\,\S(du).
\]
But
\[
\widehat{f_i}(u)=\widehat f(u)\widehat{\delta_i}(u),
\qquad
\widehat{\delta_i}(u)=\widehat\delta(u/i),
\]
and, because \(\delta\) is a probability density,
\[
|\widehat{\delta_i}(u)|\leq1,\qquad
\widehat{\delta_i}(u)\to1.
\]
Therefore
\[
\operatorname{Var}(\M(f_i))
=
\int_{\mathbb R^d}
|\widehat f(u)|^2|\widehat{\delta_i}(u)|^2\,\S(du).
\]
Letting \(i\to\infty\), Fatou's lemma gives
\[
\int |\widehat f(u)|^2\,\S(du)
\leq
\liminf_i \operatorname{Var}(\M(f_i))
=
\operatorname{Var}(\M(f)).
\]
The reverse inequality follows from
\[
|\widehat f(u)|^2|\widehat{\delta_i}(u)|^2
\leq
|\widehat f(u)|^2.
\]
Hence
\[
\operatorname{Var}(\M(f))
=
\int_{\mathbb R^d}|\widehat f(u)|^2\,\S(du).
\]

Finally, for arbitrary \(f,g\in\mathscr B_c(\mathbb R^d)\), use polarization. If
\[
Q(h)=\operatorname{Var}(\M(h)),
\]
then, with covariance linear in the first argument,
\[
\operatorname{Cov}(\M(f),\M(g))
=
\frac14\Big[
Q(f+g)-Q(f-g)
+iQ(f+ig)-iQ(f-ig)
\Big].
\]
Applying the variance identity to \(f+g,f-g,f+ig,f-ig\) yields
\[
\operatorname{Cov}(\M(f),\M(g))
=
\int_{\mathbb R^d}
\widehat f(u)\overline{\widehat g(u)}\,\S(du).
\]}

\end{proof}

  \section{Fourier transform of the ball and Proof of Lemma  \ref{lm:trans-bd}}
  \label{app:balls}

The more investigated examples of  linear statistics are ball  indicators $ f = 1_{ B_{ R}},R>0.$ Understanding the behaviour of their Fourier transforms is essential to study fluctuations of the number variance of  point processes, i.e. the variance of the number of points in $ B_{ R},$ as $ R\to \infty .$ It is enough to study $ R = 1$ thanks to the scaling relation $ \widehat{ 1_{ B_{ R}}} = R^{ d} \widehat{ 1_{ B_{ 1}}}(R\;\cdot )$.

\begin{lemma}
\label{lm:fourier-ball}
We have 
\begin{align*}
 \widehat{ 1_{ B_{ 1}}}(u)  = \kappa _{ d}  \mathcal J_{ d/2}(\|u\|)
\end{align*}
with the renormalised Bessel function of the first kind 
\begin{align*}
 \mathcal J_{ \nu }(t) = \sum_{k = 0}^{ \infty }(-1)^{ k}a_{ k}(t),\qquad a_{ k}(t) = \frac{ 1}{k!\Gamma (k + \nu  + 1)}\left(
\frac{ t}{2}
\right)^{ 2k},
\end{align*} recalling $ \kappa _{ d} = \Leb(B_{ 1}) =  \widehat{ 1_{ B_{ 1}}}(0).$ It implies in particular by  \cite[Sec. 9.2]{specialFunctionsBesselSec9.2}\begin{align*}
 \widehat{ 1_{ B_{ 1}}}(u) =c_{ d} \|u\|^{ -\frac{ d + 1}{2}}\sin(\|u\|-c'_{ d})(1 + o (1)),\qquad u\to \infty ,
\end{align*}
for some $ c_{ d},c'_{ d}>0,$
and
\begin{align}
\label{eq:bessel}
 |  \widehat{ 1_{ B_{ 1}}}(u) |   = O((1 + \|u\|)^{ -\frac{ d + 1}{2}}).
\end{align} Also, we have the inequality
\begin{align}
\label{eq:half-max-inequality}
 | \widehat{ 1_{ B_{ 1}}}(u) | \geqslant \frac{ \kappa _{ d}}{2}1_{ B_{ \rho _{ d}}}(u)
\end{align}where $ \rho _{ d} = { \sqrt{d + 2}} $.   
\end{lemma}

  {
  
  \begin{proof}
   
We refer to  \cite{SteinWeiss} for instance for the 
identity
$$ {\widehat{\mathbf{1}_{B_1}}(u) = \frac{J_{d/2}(\|u\|)}{\|u\|^{d/2}} \cdot (2\pi)^{d/2}}$$
where $ J_{ d/2}$ is the Bessel function of the first kind. The renormalised Bessel function is defined by $\mathcal{J}_\nu(t) = \Gamma(\nu+1)\left(\tfrac{2}{t}\right)^\nu B_\nu(t)$ with $\nu = d/2$. We have
indeed
$$\widehat{\mathbf{1}_{B_1}}(u) = \kappa_d\cdot \mathcal{J}_{d/2}(\|u\|).$$

We are looking for $\rho_d$ such that for $ \|u\|\leqslant \rho _{ d},$ $\widehat{\mathbf{1}_{B}}(u) \geq \tfrac{1}{2}\kappa_d$. We claim that $ \rho _{ d} = \sqrt{d + 2}$ works.
Let $ z\leqslant \sqrt{d + 2}.$ 
For $k \geq 1$ and $\nu = d/2 \geq 1/2$, we have $(k+1)(\nu+k+1) \geq 1 \cdot (\nu + 2) = d/2 + 2$, hence :
$$\frac{a_{k+1}(z)}{a_k(z)} \leq \frac{z^2/4}{d/2+2} = \frac{z^2}{2(d+4)} \leq \frac{d+2}{2(d+4)} < 1$$
it is hence an alternating series with decreasing terms moduli as long as $ z\leqslant \sqrt{d + 2}$.
As long as $z$ is also smaller than the smallest positive zero $ j_{ \nu ,1}$, we have therefore:

$$\mathcal{J}_\nu(z) \geq 1 - \frac{z^2}{4(\nu+1)}$$
 
  Hence $\mathcal{J}_\nu$ does not vanish on $\left(0,\, 2\sqrt{\nu+1}\right)$, and  $j_{\nu,1} > 2\sqrt{\nu+1}$.  
 Applying this to $\nu = d/2$ gives
$$j_{d/2,\,1} > 2\sqrt{\frac{d}{2}+1} = \sqrt{2(d+2)} = \sqrt{2}\cdot\sqrt{d+2}$$

Hence for $z = \|u\| \leq \rho_d = \sqrt{d+2}  < j_{d/2,\,1}$.

$$\widehat{\mathbf{1}_{B_1}}(u) \geq \kappa_d\left(1 - \frac{\|u\|^2}{4\left(\frac{d}{2}+1\right)}\right) = \kappa_d\left(1 - \frac{\|u\|^2}{2(d+2)}\right).$$
Hence$$\|u\| \leq \sqrt{d+2} \implies \widehat{\mathbf{1}_{B(0,1)}}(u) \geq \frac{1}{2}\,\kappa_d.$$

  \end{proof}

}
\begin{longversion}{}
The proof is based on an estimate that is crucial in the variance fluctuations of random measures.
 
\begin{proof}
Decompose $ u = \|u\|  \dot u.$ Let $ \sigma _{ d}$ the Haar measure on $  \mathbb  S  ^{ d-1}$.
\begin{align*}
 \widehat{ 1_{ B_{ 1}}}(u) = &\int_{B_{ 1}}e^{ i x\cdot u}du\\
  = &\int_{0}^{ 1    }\int_{  \mathbb  S  ^{ d-1}}e^{ i\rho t\cdot u}\rho ^{ d-1}d\rho \sigma (dt)\\
   = & \int_{0}^{ 1    } \rho ^{ d-1}d\rho \int_{  \mathbb  S  ^{ d-1}}e^{ i \rho \|u\| t\cdot  \dot u}\sigma _{ d}(dt)
\end{align*}
\end{proof}

Several constructions of useful functions are built on this result.

 { 
\begin{lemma}
For $ \rho >0$, there is a function $ f_{ \rho }$ such that $ 0\leqslant \hat f_{ \rho }\leqslant 1_{ B_{ \rho }}$ and for some $ 0<c_{ \rho }^{ -}<c_{ \rho }^{  + }<\infty $, 
\begin{align*}
 c_\rho ^{ -}(1 + \|u\|)^{ -\frac{ d + 1}{2}}\leqslant  \hat f_{ \rho }(u)\leqslant 
 c_\rho ^{  + }(1 + \|u\|)^{ -\frac{ d + 1}{2}}
\end{align*}
\end{lemma}

\begin{proof}
The upper bound is easy to have with the function $ 1_{ B_{ \rho }}$ from  \eqref{eq:bessel}, but the lower bound requires a finer analysis due to the frequent cancellations of $ \widehat{ 1_{ B_{ \rho }}}$. 

\draft{useful? ...}
\end{proof}

\begin{lemma}\begin{itemize}
\item 
There exists  function $ f\in \Sc(\mathbb{R}^{ d})$ such that $  {\rm supp}( \hat f)\subset B_{ 1}.$
\item
There exists $ g\in \Sc(\mathbb{R}^{ d})$ such that $ g\equiv 1$ on $ B_{ 1}.$

\end{itemize}
\end{lemma}

\begin{proof}

This is a standard construction on $ \mathbb{R}:$\begin{itemize}
\item Start from a smooth $ \CC^{ \infty }$ step function 
\begin{align*}
 h_{ 1}(t) = 1_{ t>0}e^{ -1/t}.
\end{align*}

\item Then rescale properly 
\begin{align*}
 h_{ 2}(t) = \frac{ h_{ 1}(t)}{h_{ 1}(t) + h_{ 1}(1-t)},
\end{align*}
which takes value $ 0$ on $ (-\infty ,0)$ and $ 1$ on $ (1,\infty )$, still $ \CC^{ \infty }$.
\item The main ingredients are here. For the first point, take 
\begin{align*}
 f(t) = \tau _{ -1} h_{ 2}(t)  \tau _{ 1}h_{ 2}(-t).
\end{align*}
Since it is $ \CC^{ \infty }$ with compact support, it is a Schwartz function.
\item For the second point, add a plateau in the middle 
\begin{align*}
 g(t) = \tau _{ -2}h_{ 2}(t)\tau _{ 2}h_{ 2}(-t).
\end{align*}
\item In higher dimensions, simply take tensor products $ f(t_{ 1})\dots f(t_{ d}),t = (t_{ 1},\dots ,t_{ d})\in \mathbb{R}^{ d}$, or $ g(t_{ 1})\dots g(t_{ d})$, and apply a proper rescaling.
\end{itemize}

\end{proof}

}
\end{longversion}
 
 \subsection{  Proof of Lemma  \ref{lm:trans-bd}}
\begin{proof}
Let $f=1_{B_1}$. By \eqref{eq:half-max-inequality},
\[
|\widehat f(\xi)|
=
|\widehat{1_{B_1}}(\xi)|
\ge \frac{\kappa_d}{2}1_{B_{\rho_d}}(\xi),
\qquad \rho_d=\sqrt{d+2}.
\]
Hence, for every $u\in\mathbb R^d$,
\[
1_{B(u,\rho_d)}(\xi)
\le
\frac{4}{\kappa_d^2}\,
|\widehat f(\xi-u)|^2.
\]
Therefore
\begin{align}
\label{eq:local-S-bound}
\S(B(u,\rho_d))
&\le
\frac{4}{\kappa_d^2}
\S\!\left(|\widehat f(\cdot-u)|^2\right).
\end{align}
Since
\[
\widehat{e^{iu\cdot}f}(\xi)=\widehat f(\xi-u),
\]
the spectral representation gives
\[
\S\!\left(|\widehat f(\cdot-u)|^2\right)
=
(2\pi)^d
\operatorname{Var}\!\left(\M\!\left(e^{iu\cdot}f\right)\right).
\]
Thus
\[
\sup_{u\in\mathbb R^d}\S(B(u,\rho_d))
\le
\frac{4(2\pi)^d}{\kappa_d^2}
\sup_{u\in\mathbb R^d}
\operatorname{Var}\!\left(\M\!\left(e^{iu\cdot}1_{B_1}\right)\right).
\]

Now fix $x\in\mathbb R^d$ and $T\ge 1$. There exists a finite set
$U_T\subset\mathbb R^d$ such that
\[
B(x,T)\subset \bigcup_{v\in U_T}B(v,\rho_d).
\]
Let $ 
N_d(T,\rho_d) = \sup \{\#U_{ T}\}$  the maximal cardinality of such coverings and the covering index
\begin{align*}
 c_{ d,\rho _{ d}} = \sup _{ T}\frac{N_{ d}(T,\rho _{ d}) }{T^{ d}} = \sup_{ T} \frac{ N_{ d}(T,1)}{(T/\rho _{ d})^{ d}} = {\rho _{ d}^{ d}}{ c_{ d,1}}<\infty 
.\end{align*}
Hence, since $\S$ is a positive measure,
\[
\S(B(x,T))
\le
\sum_{v\in U_T}\S(B(x+v,\rho_d))
\le
N_d(T,\rho_d)\,
\sup_{u\in\mathbb R^d}\S(B(u,\rho_d))
\le {\rho _{ d}^{ d}}
\frac{4(2\pi)^d}{\kappa_d^2}\,
 { c_{ d,1}}\,
\sup_{u\in\mathbb R^d}
\operatorname{Var}\!\left(\M\!\left(e^{iu\cdot}1_{B_1}\right)\right)
T^d.
\]
This proves \eqref{eq:adhikari}.
Finally,
\[
\operatorname{Var}\!\left(\M\!\left(e^{iu\cdot}1_{B_1}\right)\right)
\le
\mathbf E\left[
\left|\M\!\left(e^{iu\cdot}1_{B_1}\right)\right|^2
\right]
\le
\mathbf E\left[|\M|(B_1)^2\right],
\]
which is finite by the $L^2_{\rm loc}$ assumption. The lemma statement holds with 
\begin{align*}
 c_{ d} = {\rho _{ d}^{ d}}\frac{ 4(2\pi )^{ d}c_{ d,1}}{\kappa _{ d}^{ 2} }.
\end{align*}\end{proof}

       \newpage
     
{     \section*{Acknowledgements} I am thankful to Hermine Biermé, Jean-Fran\c cois Coeurjolly, and the scientific committee of the Geosto 2025 conference for offering me the opportunity to give a mini-course about  hyperuniformity and write these lecture notes. Thanks also to   Jonas Jalowy, Gabriel Mastrilli, Mattias Bylehn and Daniela Flimmel  for comments on an early version of this work.} I thank Luca Lotz for comments on an early version of Section  \ref{sec:fair-partitions} for the second version.

\addcontentsline{toc}{chapter}{Bibliography}
\bibliographystyle{abbrv}

\begin{thebibliography}{100}

\bibitem{specialFunctionsBesselSec9.2}
M.~Abramowitz and I.~A. Stegun.
\newblock {\em Handbook of mathematical functions with formulas, graphs, and
  mathematical tables}, volume~55.
\newblock US Government printing office, 1968.

\bibitem{AdiGhoshLeb}
K.~Adhikari, S.~Ghosh, and J.~Lebowitz.
\newblock Fluctuation and entropy in spectrally constrained random fields.
\newblock {\em Comm. Math. Phys.}, 386:749--780, 2021.

\bibitem{AT07}
R.~J. Adler and J.~E. Taylor.
\newblock {\em {R}andom {F}ields and {G}eometry}.
\newblock Springer, 2007.

\bibitem{gaussian-blue-noise}
A.~G.~M. Ahmed, J.~Ren, and P.~Wonka.
\newblock Gaussian blue noise.
\newblock {\em ACM Trans. Graph.}, 41(6):1--15, 2022.

\bibitem{AizMart}
M.~Aizenman and P.~A. Martin.
\newblock {S}tructure of {G}ibbs {S}tates of one {D}imensional {C}oulomb
  {S}ystems.
\newblock {\em Comm. Math. Phys.}, 78:99--116, 1980.

\bibitem{AKT}
M.~Ajtai, J.~Komlos, and G.~Tusnady.
\newblock {O}n optimal matchings.
\newblock {\em Combinatorica}, 4:259--264, 1984.

\bibitem{GuionnetBook}
G.~Anderson, A.~Guionnet, and O.~Zeitouni.
\newblock {\em An introduction to random matrices}.
\newblock Cambridge Studies in Advanced Mathematics, 2010.

\bibitem{ArmSerf}
S.~Armstrong and S.~Serfaty.
\newblock Local laws and rigidity for coulomb gases at any temperature.
\newblock {\em Ann. Prob.}, 49(1), 2021.

\bibitem{baake-book}
M.~Baake and U.~Grimm.
\newblock {\em {A}periodic order, vol. 1, {E}ncyclopedia of {M}athematics and
  its {A}pplications}, volume 149.
\newblock Cambridge University Press, 2013.

\bibitem{baccelli2020random}
F.~Baccelli, B.~B{\l}aszczyszyn, and M.~Karray.
\newblock Random measures, point processes, and stochastic geometry, 2020.

\bibitem{Balzer}
M.~Balzer, T.~Schl\"omer, and O.~Deussen.
\newblock Capacity-constrained point distributions: A variant of lloyd's
  method.
\newblock {\em ACM Trans. Graph.}, 28(3):86:1--8, 2009.

\bibitem{Bartlett}
M.~S. Bartlett.
\newblock The spectral analysis of point processes.
\newblock {\em J. Roy. Stat. Soc. B}, 25(2):264--296, 1963.

\bibitem{stealthy08}
R.~D. Batten, F.~H. Stillinger, and S.~Torquato.
\newblock Classical disordered ground states: Super-ideal gases and stealth and
  equi-luminous materials.
\newblock {\em J. Appl. Phys.}, 104(033504), 2008.

\bibitem{stealthy09}
R.~D. Batten, F.~H. Stillinger, and S.~Torquato.
\newblock Interactions leading to disordered ground states and unusual
  low-temperature behavior.
\newblock {\em Phys. Rev. E}, 80(031105), 2009.

\bibitem{BBNY}
R.~Bauerschmidt, P.~Bourgade, M.~Nikul, and H.~Yau.
\newblock {L}ocal density for two-dimensional one-component plasma.
\newblock {\em Comm. Math. Phys.}, 356:189--230, 2017.

\bibitem{Beck87}
J.~Beck.
\newblock {I}rregularities of distribution. {I}.
\newblock {\em Acta Math.}, 159:1--49, 1987.

\bibitem{Berg}
C.~Berg and G.~Forst.
\newblock {\em {P}otential {T}heory on {L}ocally {C}ompact {A}belian {G}roups}.
\newblock Ergebnisse der Mathematik und ihrer Grenzgebiete. Springer-Verlag,
  1975.

\bibitem{bylehn}
M.~Bj\"orklund and Bylehn.
\newblock Hyperuniformity of random measures on euclidean and hyperbolic
  spaces.
\newblock arXiv:2405.12737, 2024.

\bibitem{bjorklund2025hyperuniformity}
M.~Bj{\"o}rklund and M.~Byl{\'e}hn.
\newblock Hyperuniformity and hyperfluctuations of random measures in
  commutative spaces.
\newblock {\em arXiv:2503.01567}, 2025.

\bibitem{HartBjo}
M.~Bj\"orklund and T.~Hartnick.
\newblock {H}yperuniformity and non-hyperuniformity of quasicrystals.
\newblock {\em Math. Ann.}, 389:365 -- 426, 2024.

\bibitem{BYY}
B.~B{\l}aszczyszyn, D.~Yogeshwaran, and J.~Yukich.
\newblock {L}imit theory for geometric statistics of point processes having
  fast decay of correlations.
\newblock {\em Ann. Prob.}, 47:835--895, 2019.

\bibitem{BL21}
S.~Bobkov and M.~Ledoux.
\newblock {A} simple {{F}}ourier analytic proof of the {AKT} optimal matching
  theorem.
\newblock {\em Ann. Appl. Prob.}, 31(6):2567--2584, 2021.

\bibitem{boursier}
J.~Boursier.
\newblock Optimal local laws and clt for the circular riesz gas.
\newblock \url{https://arxiv.org/abs/2112.05881}.

\bibitem{BufAiry}
A.~I. Bufetov.
\newblock {R}igidity of determinantal point processes with the {{A}}iry, the
  {B}essel and the {G}amma kernel.
\newblock {\em Bull. Math. Sci.}, 6:163--172, 2016.

\bibitem{Buf-conditional}
A.~I. Bufetov.
\newblock {C}onditional {M}easures of {D}eterminantal {P}oint {P}rocesses.
\newblock {\em Funct. Anal. Appl.}, 54:7--20, 2020.

\bibitem{BufLinear}
A.~I. Bufetov, Y.~Dabrowski, and Y.~Qiu.
\newblock {L}inear rigidity of stationary stochastic processes.
\newblock {\em Ergod. Th. \& Dynam. Sys.}, 38:2493--2507, 2018.

\bibitem{BufPFaff}
A.~I. Bufetov, P.~P. Nikitin, and Y.~Qiu.
\newblock On number rigidity for {P}faffian point processes.
\newblock {\em Mosc. Math. J.}, 2:217--274, 2019.

\bibitem{BufBalanced}
A.~I. Bufetov and Y.~Qiu.
\newblock {J}-{H}ermitian determinantal point processes: balanced rigidity and
  balanced {P}alm equivalence.
\newblock {\em Math. Ann.}, 371:127--188, 2018.

\bibitem{bugeaud}
Y.~Bugeaud.
\newblock {\em Approximation by Algebraic Numbers}, volume 160.
\newblock Cambridge Tracts in Mathematics, Cambridge University Press, 2004.

\bibitem{BDG}
R.~Butez, S.~Dallaporta, and D.~Garcia-Zelada.
\newblock {O}n the {W}asserstein distance between a hyperuniform point process
  and its mean.
\newblock https://arxiv.org/pdf/2404.09549.pdf, 2024.

\bibitem{chandresakar}
S.~Chandresakar.
\newblock Stochastic problems in physics and astronomy.
\newblock {\em Rev. Mod. Phys.}, 15:1--89, 1943.

\bibitem{Cha17}
S.~Chatterjee.
\newblock {R}igidity of the three-dimensional hierarchical coulomb gas.
\newblock {\em Prob. Th. Rel. Fields}, 175:1123--1176, 2019.

\bibitem{CPPR}
S.~Chatterjee, R.~Peled, Y.~Peres, and D.~Romik.
\newblock {G}ravitational allocation to {P}oisson points.
\newblock {\em Ann. Math.}, 172(1):617--671, 2010.

\bibitem{ChhaibiNaj}
R.~Chhaibi and J.~Najnudel.
\newblock {R}igidity of the {S}ine$_\beta$ process.
\newblock {\em Elec. Comm. Prob.}, 94:1--8, 2018.

\bibitem{CES}
G.~Cipolloni, L.~Erd{\"o}s, and D.~Schr\"oder.
\newblock Mesoscopic central limit theorem for non-hermitian random matrices.
\newblock {\em Prob. Th. Rel. Fields}, 188:1131--1182, 2023.

\bibitem{dim24packings}
H.~Cohn, A.~Kumar, S.~D. Miller, D.~Radchenko, and M.~Viazovska.
\newblock The sphere packing problem in dimension 24.
\newblock {\em Ann. Math.}, 183:1017--1033, 2017.

\bibitem{Coste}
S.~Coste.
\newblock {O}rder, fluctuations, rigidities.
\newblock
  \href{https://scoste.fr/assets/survey_hyperuniformity.pdf}{https://scoste.fr/assets/survey\_hyperuniformity.pdf},
  2021.

\bibitem{DVj08}
D.~J. Daley and D.~Vere-Jones.
\newblock {\em {A}n {I}ntroduction to the {{T}}heory of {P}oint {P}rocesses,
  {V}olume I: Elementary Theory and Methods}.
\newblock Springer, Probability and its applications, 2003.

\bibitem{DVj08b}
D.~J. Daley and D.~Vere-Jones.
\newblock {\em {A}n {I}ntroduction to the {T}heory of {P}oint {P}rocesses:
  {V}olume {I}{I}: {G}eneral {T}heory and {S}tructure.}
\newblock Springer-Verlag, New-York, 2008.

\bibitem{deGoes}
F.~de~Goes, K.~Breeden, V.~Ostromoukhov, and M.~Desbrun.
\newblock Blue noise through optimal transport.
\newblock {\em ACM Trans. Graph.}, 31(6), Nov. 2012.

\bibitem{DereudreFlimmel}
D.~Dereudre and D.~Flimmel.
\newblock {N}on-hyperuniformity of {G}ibbs point processes with short-range
  interactions.
\newblock {\em J. Appl. Prob.}, 61(4):1380--1406, 2024.

\bibitem{DFHL}
D.~Dereudre, D.~Flimmel, T.~Huessman, and T.~Lebl{\'e}.
\newblock ({N}on)-hyperuniformity of perturbed lattices.
\newblock https://arxiv.org/abs/2405.19881, 2024.

\bibitem{DHLM}
D.~Dereudre, A.~Hardy, T.~Lebl{{\'e}}, and M.~Ma{\"\i}da.
\newblock {{D}{L}{R}} equations and rigidity for the sine-beta process.
\newblock {\em Comm. Pure Appl. Math.}, 74(1):172--222, 2020.

\bibitem{DereudreVasseur}
D.~Dereudre and T.~Vasseur.
\newblock {N}umber-rigidity and $\beta$-circular {R}iesz gas.
\newblock {\em Ann. Prob.}, 51(3):1025--1065, 2023.

\bibitem{betaEns}
I.~Dimitriu and A.~Edelman.
\newblock Matrix models for beta ensembles.
\newblock {\em J. Math. Phys.}, 43(11):5830--5847, 2002.

\bibitem{dyson}
F.~J. Dyson.
\newblock Correlations between the eigenvalues of a random matrix.
\newblock {\em Comm. Math. Phys.}, 19(235-250), 1970.

\bibitem{erbar2023optimal}
M.~Erbar, M.~Huesmann, J.~Jalowy, and B.~M{\"u}ller.
\newblock Optimal transport of stationary point processes: Metric structure,
  gradient flow and convexity of the specific entropy.
\newblock {\em arXiv:2304.11145}, 2023.

\bibitem{Flimmel}
D.~Flimmel.
\newblock Fitting regular point patterns with a hyperuniform perturbed lattice.
\newblock arXiv:2503.12179.

\bibitem{Forrester}
P.~J. Forrester.
\newblock {\em Log-gases and random matrices}.
\newblock London Mathematical Society Monographs. 2010.

\bibitem{ForresterHonner}
P.~J. Forrester and G.~Honner.
\newblock Exact statistical properties of the zeros of complex random
  polynomials.
\newblock {\em J. Phys. A: Math. and General}, 32(16):2961, 1999.

\bibitem{Gabrielli04}
A.~Gabrielli.
\newblock Point processes and stochastic displacement fields.
\newblock {\em Phys. Rev. E}, 70(066131), 2004.

\bibitem{torquato-tilings}
A.~Gabrielli, M.~Joyce, and S.~Torquato.
\newblock {T}ilings of space and superhomogeneous point processes.
\newblock {\em Phys. Rev. E}, 77:031125, 2008.

\bibitem{Ganguly-Sarkar}
S.~Ganguly and S.~Sarkar.
\newblock {G}round states and hyperuniformity of the hierarchical {C}oulomb gas
  in all dimensions.
\newblock {\em Prob. Th. Rel. Fields}, 177:621--675, 2020.

\bibitem{Gass-cancellation}
L.~Gass.
\newblock Spectral criteria for the asymptotics of local functionals of
  gaussian fields and their application to nodal volumes and critical.
\newblock https://arxiv.org/pdf/2501.07356, 2025.

\bibitem{GautBardValko}
G.~Gautier, R.~Bardenet, and M.~Valko.
\newblock Fast sampling from $\beta$-ensembles.
\newblock {\em Statistics and computing}, 31(7), 2021.

\bibitem{GhoshComplete}
S.~Ghosh.
\newblock {D}eterminantal processes and completeness of random exponentials:
  the critical case.
\newblock {\em Prob. Th. Rel. Fields}, 163:643--665, 2015.

\bibitem{Ghosh-conditional}
S.~Ghosh.
\newblock {P}alm measures and rigidity phenomena in point processes.
\newblock {\em Elec. Comm. Prob.}, 21:1--14, 2016.

\bibitem{GK21}
S.~Ghosh and M.~Krishnapur.
\newblock {R}igidity {H}ierarchy in {R}andom {P}oint {F}ields: {R}andom
  {P}olynomials and {D}eterminantal {P}rocesses.
\newblock {\em Comm. Math. Phys.}, 388:pp. 1205--1234, 2021.

\bibitem{GKP}
S.~Ghosh, M.~Krishnapur, and Y.~Peres.
\newblock {C}ontinuum {P}ercolation for {{G}}aussian zeroes and ginibre
  eigenvalues.
\newblock {\em Ann. Prob.}, 44(5):3357--3384, 2016.

\bibitem{GL-sufficient}
S.~Ghosh and J.~Lebowitz.
\newblock {N}umber rigidity in superhomogeneous random point fields.
\newblock {\em J. Stat. Phys.}, 166(3-4), 2017.

\bibitem{GosLeb}
S.~Ghosh and J.~L. Lebowitz.
\newblock {F}luctuations, large deviations and rigidity in hyperuniform
  systems: a brief survey.
\newblock {\em Indian J. of Pure and Appl. Math.}, 48(4):609--631, 2017.

\bibitem{GL18}
S.~Ghosh and J.~L. Lebowitz.
\newblock Generalized stealthy hyperuniform processes: Maximal rigidity and the
  bounded holes conjecture.
\newblock {\em Comm. Math. Phys.}, 363:97--110, 2018.

\bibitem{GP17}
S.~Ghosh and Y.~Peres.
\newblock {R}igidity and tolerance in point processes: {{G}}aussian zeros and
  ginibre eigenvalues.
\newblock {\em Duke Math. J.}, 166(10):1789--1858, 2017.

\bibitem{olhede}
J.~P. Grainger, T.~A. Rajala, D.~J. Murrell, and S.~C. Olhede.
\newblock Spectral estimation for spatial point processes and random fields.
\newblock https://arxiv.org/abs/2312.10176, 2024.

\bibitem{KolianderEtAl}
A.~Haimi1, G.~Koliander, and J.~L. Romero.
\newblock Zeros of gaussian weyl--heisenberg functions and hyperuniformity of
  charge.
\newblock {\em J. Stat. Phys.}, 187(22):1--41, 2022.

\bibitem{Jalowy-GAF}
B.~Hall, C.~W. Ho, J.~Jalowy, and Z.~Kabluchko.
\newblock The heat flow, gaf, and $sl(2;r)$.
\newblock \url{https://arxiv.org/abs/2304.06665}, 2023.

\bibitem{HBLr}
D.~Hawat, R.~Bardenet, and R.~Lachi\`eze-Rey.
\newblock {R}epelled point processes with application to numerical integration.
\newblock https://arxiv.org/abs/2308.04825, 2023.

\bibitem{HGBLr}
D.~Hawat, G.~Gautier, R.~Bardenet, and R.~Lachi\`{e}ze-Rey.
\newblock {O}n estimating the structure factor of a point process, with
  applications to hyperuniformity.
\newblock {\em Statistics and computing}, 33(61), 2023.

\bibitem{hawat2023estimating}
D.~Hawat, G.~Gautier, R.~Bardenet, and R.~Lachi{\`e}ze-Rey.
\newblock On estimating the structure factor of a point process, with
  applications to hyperuniformity.
\newblock {\em Statistics and Computing}, 33(3):61, 2023.

\bibitem{HeKnowles}
Y.~He and A.~Knowles.
\newblock Mesoscopic eigenvalue density correlations of wigner matrices.
\newblock {\em Prob. Th. Rel. Fields}, 177(1/2), 2020.

\bibitem{hoffman2006stable}
C.~Hoffman, A.~E. Holroyd, and Y.~Peres.
\newblock A stable marriage of {P}oisson and {L}ebesgue.
\newblock {\em Annals of Probability}, 34(4):1241--1272, 2006.

\bibitem{HolSoo}
A.~E. Holroyd and T.~Soo.
\newblock {I}nsertion and deletion tolerance of point processes.
\newblock {\em Electron. J. Probab}, 18(74):DOI: 10.1214/EJP.v18--2621, 2013.

\bibitem{HPPS}
E.~Holroyd, R.~Pemantle, Y.~Peres, and O.~Schramm.
\newblock {P}oisson matching.
\newblock {\em Ann. IHP Prob. Stat.}, 45(1):266--287, 2009.

\bibitem{BKPV}
J.~B. Hough, M.~Krishnapur, Y.~Peres, and B.~Vir\`ag.
\newblock {\em {Z}eros of {{G}}aussian Analytic Functions and Determinantal
  Point Processes}.
\newblock University Lecture Series. Institute of Mathematical Statistics,
  2009.

\bibitem{HueLebl}
M.~Huesmann and T.~Lebl{{\'e}}.
\newblock {T}he link between hyperuniformity, {{C}}oulomb energy, and
  {W}asserstein distance to {L}ebesgue for two-dimensional point processes.
\newblock {\em Prob. Math. Phys.}, 7(1):123--173, 2026.
\newblock https://arxiv.org/abs/2404.18588.

\bibitem{jalowy}
J.~Jalowy.
\newblock The {W}asserstein distance to the circular law.
\newblock {\em Annales de l'Institut Henri Poincare (B) Probabilites et
  statistiques}, 59(4):2285--2307, 2023.

\bibitem{jalowyBox}
J.~Jalowy and H.~Stange.
\newblock Box-covariances of hyperuniform point processes.
\newblock arXiv:2506.13661, 2025.

\bibitem{avian-eyes}
Y.~Jiao, T.~Lau, H.~Hatzikirou, M.~Meyer-Hermann, J.~C. Corbo, and S.~Torquato.
\newblock Avian photoreceptor patterns represent a disordered hyperuniform
  solution to a multiscale packing problem.
\newblock {\em Phys. Rev. E}, 89(022721), 2014.

\bibitem{kallenberg2002foundations}
O.~Kallenberg.
\newblock {\em Foundations of Modern Probability}.
\newblock Springer, New York, 2002.
\newblock 2nd Edition.

\bibitem{cloak}
M.~A. Klatt, J.~Kim, and S.~Torquato.
\newblock Cloaking the underlying long-range order of randomly perturbed
  lattices.
\newblock {\em Phys. Rev. E}, 101(032118), 2020.

\bibitem{KlattLast}
M.~A. Klatt and G.~Last.
\newblock {O}n strongly rigid hyperfluctuating random measures.
\newblock {\em J. Appl. Prob.}, 59(4):948--961, 2022.
\newblock https://arxiv.org/abs/2008.10907.

\bibitem{KLLY}
M.~A. Klatt, G.~Last, L.~Lotz, and D.~Yogeshwaran.
\newblock Invariant transports of stationary random measures: asymptotic
  variance, hyperuniformity, and examples, 2025.

\bibitem{KLY}
M.~A. Klatt, G.~Last, and D.~Yogeshwaran.
\newblock {H}yperuniform and rigid stable matchings.
\newblock {\em Rand. Struct. Alg.}, 57:439--473, 2020.

\bibitem{Klatt-nature}
M.~A. Klatt, J.~Lovric, D.~Chen, S.~Kapfer, F.~Schaller, P.~Schonoffer,
  B.~Gardiner, A.~Smith, G.~Schroder-Turk, and S.~Torquato.
\newblock {U}niversal hidden order in amorphous cellular geometries.
\newblock {\em Nature communications}, 10(811), 2019.

\bibitem{Kolmo-rigid}
A.~N. Kolmogorov.
\newblock {S}tationary sequences in {H}ilbert space.
\newblock {\em Bull. Moskov. Gos. Univ. Mat.}, 2:1--40, 1941.

\bibitem{Krein}
M.~G. Krein.
\newblock On a certain extrapolation problem of {A}. {N}. {K}olmogorov (in
  russian).
\newblock {\em Dokl. Akad. Nauk SSSR}, 46(8), 1945.

\bibitem{KrishVirag}
M.~Krishnapur and B.~Vir\'ag.
\newblock The ginibre ensemble and gaussian analytic functions.
\newblock {\em Int. Math. Res. Not.}, 6:1441--1464, 2014.

\bibitem{KrishYogesh}
M.~Krishnapur and D.~Yogeshwaran.
\newblock Stationary random measures: Covariance asymptotics, variance bounds
  and central limit theorems.
\newblock arXiv:2411.08848, 2024.

\bibitem{KurSarnak}
P.~Kurasov and P.~Sarnak.
\newblock Stable polynomials and crystaline measure.
\newblock {\em J. Math. Phys.}, 61(083501), 2020.

\bibitem{Lr21}
R.~Lachi\`{e}ze-Rey.
\newblock {D}iophantine {{G}}aussian excursions and random walks.
\newblock https://arxiv.org/abs/2104.07290.

\bibitem{Lr-DPPs}
R.~Lachi\`{e}ze-Rey.
\newblock Random matrices, determinantal point processes and hyperuniformity.
\newblock Master course lecture notes
  \url{https://helios2.mi.parisdescartes.fr/~rlachiez/enseignement/HU/DPPs-HU.pdf}.

\bibitem{Lac20}
R.~Lachi\`{e}ze-Rey.
\newblock {V}ariance linearity for real {{G}}aussian zeros.
\newblock {\em Ann. I. H. Poincarr\'e B}, 58(4), 2022.

\bibitem{Lr24}
R.~Lachi\`{e}ze-Rey.
\newblock {R}igidity of random stationary measures and applications to point
  processes.
\newblock https://arxiv.org/abs/2409.18519, 2024.

\bibitem{rigid-companion}
R.~Lachi\`{e}ze-Rey.
\newblock Maximal rigidity of random measure and uniqueness pairs: stealthy
  processes, quasicrystals and periodicity.
\newblock \url{https://arxiv.org/abs/2512.10686}, 2025.

\bibitem{Lr-max-rigid}
R.~Lachi\`{e}ze-Rey.
\newblock Maximal rigidity of random measure and uniqueness pairs: stealthy
  processes, quasicrystals and periodicity.
\newblock \url{https://arxiv.org/abs/2512.10686}, 2025.

\bibitem{Lr-HU-finite}
R.~Lachi\`{e}ze-Rey.
\newblock From smooth to discontinuous kernels: a variance transfer principle
  for hyperuniform systems.
\newblock 2026.

\bibitem{LrY}
R.~Lachi\`{e}ze-Rey and D.~Yogeshwaran.
\newblock {H}yperuniformity and optimal transport of point processes.
\newblock \href{https://arxiv.org/abs/2402.13705}{arXiv}, 2024.

\bibitem{rigidity-schrodinger}
P.~Y.~G. Lamarre, P.~Ghosal, and Y.~Liao.
\newblock {S}pectral rigidity of random {{S}}chr\"odinger operator via
  {F}eynman-{K}ac formulas.
\newblock {\em Ann. I. H. Poincarr\'e B}, 21:2259--2299, 2020.

\bibitem{LavRubDPP}
F.~Lavancier and E.~Rubak.
\newblock On simulation of continuous determinantal point processes.
\newblock {\em Statistics and computing}, 33(120), 2023.

\bibitem{Leble}
T.~Lebl{\'e}.
\newblock {T}he two-dimensional one-component plasma is hyperuniform.
\newblock \href{https://arxiv.org/pdf/2104.05109.pdf}{arXiv}, to appear in Duke
  Math. J., 2023.

\bibitem{leble-stationary}
T.~Lebl{{\'e}}.
\newblock {DLR} equations, number-rigidity and translation-invariance for
  infinite-volume limit points of the 2docp.
\newblock \url{https://arxiv.org/pdf/2410.04958}, 2024.

\bibitem{Leb}
J.~Lebowitz.
\newblock Charge fluctuations in coulomb systems.
\newblock {\em Phys. Rev. A}, 27(3):1491, 1983.

\bibitem{Lewin}
M.~Lewin.
\newblock {C}oulomb and {R}iesz gases: {T}he known and the unknown.
\newblock {\em J. Math. Phys.}, 63(6):061101, https://doi.org/10.1063/5.0086835
  2022.

\bibitem{LotzKlatt}
L.~Lotz and M.~A. Klatt.
\newblock Persistence of asymptotic variance under transport: from
  hyperfluctuation to stealthy hyperuniformity.
\newblock https://arxiv.org/abs/2605.22803v1.

\bibitem{LyonsSteif}
R.~Lyons and J.~E. Steif.
\newblock {S}tationary {D}eterminantal {P}rocesses: {P}hase {M}ultiplicity,
  {B}ernoullicity, {E}ntropy, and {D}omination.
\newblock {\em Duke Math. J.}, 120(3):515--575, 2003.

\bibitem{Mac75}
O.~Macchi.
\newblock {T}he coincidence approach to stochastic processes.
\newblock {\em Adv. Appl. Prob.}, 7:83--122, 1975.

\bibitem{MarkoTimar}
R.~Mark{\'o} and A.~Tim{\'a}r.
\newblock A poisson allocation of optimal tail.
\newblock {\em Ann. Prob.}, 44(2):1285--1307, 2016.

\bibitem{martin1980charge}
P.~A. Martin and T.~Yalcin.
\newblock The charge fluctuations in classical {C}oulomb systems.
\newblock {\em Journal of Statistical Physics}, 22:435--463, 1980.

\bibitem{mastrilli-minimax}
G.~Mastrilli.
\newblock Minimax estimation of the structure factor of spatial point
  processes.
\newblock PhD thesis, \url{https://arxiv.org/abs/2511.14551}.

\bibitem{MBL}
G.~Mastrilli, B.~Blaszczyszyn, and F.~Lavancier.
\newblock {E}stimating the hyperuniformity exponent of point processes.
\newblock arXiv:2407.16797, 2024.

\bibitem{fiume}
M.~McCool and E.~Fiume.
\newblock Hierarchical poisson disk sampling distributions.
\newblock {\em Proc. Graph. Interface}, pages 94--105, 1992.

\bibitem{Mehta}
M.~L. Mehta.
\newblock Pure and applied mathematics.
\newblock In {\em Random Matrices}, volume 142. Academic Press, Elsevier, 2004.

\bibitem{Merigot}
Q.~M{\'e}rigot.
\newblock Minimal geodesics along volume-preserving maps, through semidiscrete
  optimal transport.
\newblock {\em SIAM J. Num. Anal.}, 54(6):3465--3492, 2016.

\bibitem{Sarrazin}
Q.~M{\'e}rigot, F.~Santambrogio, and C.~Sarrazin.
\newblock Non-asymptotic convergence bounds for wasserstein approximation using
  point clouds.
\newblock {\em Adv. NeurIPS}, 34, 2021.

\bibitem{meyer72}
Y.~Meyer.
\newblock {\em Algebraic Numbers and Harmonic Analysis}.
\newblock North Holland, Amsterdam, 1972.

\bibitem{Mol05}
I.~Molchanov.
\newblock {\em {T}heory of random sets}.
\newblock Springer-Verlag, London, 2005.

\bibitem{Morse}
P.~K. Morse, J.~Kim, P.~J. Steinhardt, and S.~Torquato.
\newblock {G}enerating large disordered stealthy hyperuniform systems with
  ultrahigh accuracy to determine their physical properties.
\newblock {\em Phys. Rev. Res.}, 5(3):33190, 2023.

\bibitem{NagWei05}
W.~Nagel and V.~Weiss.
\newblock {C}rack {{S}{T}{I}{T}} tessellations - characterization of stationary
  random tessellations sable with respect to iteration.
\newblock {\em Adv. Appl. Prob.}, 37:859--883, 2005.

\bibitem{nazarov2012correlation}
F.~Nazarov and M.~Sodin.
\newblock Correlation functions for random complex zeroes: strong clustering
  and local universality.
\newblock {\em Communications in Mathematical Physics}, 310(1):75--98, 2012.

\bibitem{NSV}
F.~Nazarov, M.~Sodin, and A.~Volberg.
\newblock {T}ransportation to random zeroes by the gradient flow.
\newblock {\em Geom. Funct. Anal}, 17:887--935, 2007.

\bibitem{NPR}
I.~Nourdin, G.~Peccati, and M.~Rossi.
\newblock {N}odal {S}tatistics of {P}lanar {R}andom {W}aves.
\newblock {\em Comm. Math. Phys.}, 369:99--151, 2019.

\bibitem{torquato-quasi}
E.~G. Oguz, J.~E. Socolar, P.~J. Steinhardt, and S.~Torquato.
\newblock {H}yperuniformity and anti-hyperuniformity in onedimensional
  substitution tilings.
\newblock {\em Acta Crystallographica Section A: Foundations and Advances},
  75(1):3--13, 2019.

\bibitem{Osada-isde}
H.~Osada.
\newblock Infinite-dimensional stochastic differential equations related to
  random matrices.
\newblock {\em Prob. Th. Rel. Fields}, 153:471--509, 2012.

\bibitem{OsadaRigid}
H.~Osada.
\newblock {V}anishing self-diffusivity in {G}inibre interacting {B}rownian
  motions in two dimensions.
\newblock {\em Prob. Th. Rel. Fields},
  ttps://doi.org/10.1007/s00440-024-01303-2, 2024.

\bibitem{BEY}
L.~E. P.~Bourgade and H.~Yau.
\newblock Universality of general $\beta$-ensembles.
\newblock {\em Duke Math. J.}, 163(6), 2014.

\bibitem{Parnovski}
L.~Parnovski and A.~V. Sobolev.
\newblock On the bethe-sommerfeld conjecture for the polyharmonic operator.
\newblock {\em Duke Math. J.}, 107(2):209--238, 2001.

\bibitem{penrose-tilings}
R.~Penrose.
\newblock The role of aesthetics in pure and applied mathematical research.
\newblock {\em Bull. Inst. Math. Appl.}, 10:266--271, 1974.

\bibitem{Percival+Walden:2020}
D.~B. Percival and A.~T. Walden.
\newblock {\em {Spectral Analysis for Univariate Time Series}}, volume~51.
\newblock Cambridge University Press, mar 2020.

\bibitem{PS14}
Y.~Peres and A.~Sly.
\newblock {R}igidity and tolerance for perturbed lattices.
\newblock arXiv:1409.4490, 2014.

\bibitem{hypDPP}
Y.~Peres and B.~Vir\'ag.
\newblock Zeros of the i.i.d. gaussian power series: a confor- mally invariant
  determinantal process.
\newblock {\em Acta Math.}, 194:1--35, 2005.

\bibitem{PeyreCuturi}
G.~Peyr\'e and M.~Cuturi.
\newblock {C}omputational {O}ptimal {T}ransport: {W}ith {A}pplications to
  {D}ata {S}cience.
\newblock {\em Foundations and trends in Machine learning}, 11(5-6):355--607,
  2019.

\bibitem{VarianceMC}
A.~Pilleboue, G.~Singh, D.~Coeurjolly, M.~Kazhdan, and V.~Ostromoukhov.
\newblock Variance analysis for monte carlo integration.
\newblock {\em ACM Trans. Graph.}, 34(4), 2015.

\bibitem{prod2021contributions}
M.~Prod'Homme.
\newblock {\em Contributions to the optimal transport problem and its
  regularity}.
\newblock PhD thesis, Universit{\'e} Paul Sabatier-Toulouse III, 2021.
\newblock \url{https://theses.hal.science/tel-03419872/}.

\bibitem{prodhomme}
M.~Prod'Homme.
\newblock {\em {C}ontributions to the optimal transport problem and its
  regularity}.
\newblock PhD thesis, Universit{\'e} Toulouse 3 - Paul Sabatier, 2021.

\bibitem{Rudin}
W.~Rudin.
\newblock {\em Real and Complex Analysis}.
\newblock McGraw-Hill, Inc., 1987.

\bibitem{Rudin2}
W.~Rudin.
\newblock {\em Functional Analysis}.
\newblock McGraw-Hill, Inc., 1991.

\bibitem{Sant}
F.~Santambrogio.
\newblock {\em {O}ptimal {T}ransport for {A}pplied {M}athematicians}.
\newblock Birkha\"user, Basel, 2015.

\bibitem{SPY07}
T.~Schreiber, M.~Penrose, and J.~E. Yukich.
\newblock {{G}}aussian limits for multidimensional random sequential packing at
  saturation, comm.
\newblock {\em Math. Physics}, 272:167--183, 2007.

\bibitem{schwartz}
L.~Schwartz.
\newblock {\em Th{\'e}orie des distributions}.
\newblock Herman, 1950.

\bibitem{serfaty}
S.~Serfaty.
\newblock {S}ystems of points with {C}oulomb interactions.
\newblock {\em Proc. ICM 2018}, pages 935--977, 2019.

\bibitem{serfaty-HU}
S.~Serfaty.
\newblock Gaussian fluctuations and free energy expansion for coulomb gases at
  any temperature.
\newblock {\em Ann. IHP B}, 59(2):1074--1142, 2023.

\bibitem{schetchman}
D.~Shechtman, I.~Blech, D.~Gratias, and J.~Cahn.
\newblock Metallic phase with long-range orientational order and no
  translational symmetry.
\newblock {\em Phys. Rev. Lett.}, 53:1951--1953, 1984.

\bibitem{casiulis}
A.~Shih, M.~Casiulis, and S.~Martiniani.
\newblock Fast generation of spectrally shaped disorder.
\newblock {\em Phys. Rev. E}, 110(034122), 2024.

\bibitem{Skryganov}
M.~M. Skryganov.
\newblock Constructions of uniform distributions in terms of geometry of
  numbers.
\newblock {\em Algebra i Analiz}, 6(3):200--230, 1994.

\bibitem{Sodin-Calabi}
M.~Sodin.
\newblock Zeros of gaussian analytic functions.
\newblock {\em Math. Res. Lett.}, 7(4):371--381, 2000.

\bibitem{ST1}
M.~Sodin and B.~Tsirelson.
\newblock {R}andom complex zeroes {{I}}. {A}symptotic normality.
\newblock {\em Isr. J. Math.}, 144:125--149, 2004.

\bibitem{ST06}
M.~Sodin and B.~Tsirelson.
\newblock {R}andom complex zeroes, {I}{I}: {P}erturbed lattices.
\newblock {\em Israel Journal of Mathematics}, 152:105--124, 2006.

\bibitem{SodinElectric}
M.~Sodin, A.~Wennman, and O.~Yakir.
\newblock The random weierstrass zeta function i: Existence, uniqueness,
  fluctuations.
\newblock {\em J. Stat. Phys}, 190(166), 2023.

\bibitem{soshnikov}
A.~Soshnikov.
\newblock {D}eterminantal random point fields.
\newblock {\em Russ. Math Surv.}, 55(5):923--975, 2000.

\bibitem{soshnikovCLT}
A.~Soshnikov.
\newblock Gaussian limit for determinantal random point fields.
\newblock {\em Ann. Prob.}, 30:171--187, 2002.

\bibitem{SteinWeiss}
E.~M. Stein and G.~Weiss.
\newblock {\em Introduction to Fourier Analysis on Euclidean Spaces}.
\newblock Princeton, New Jersey, 1971.

\bibitem{skm}
D.~Stoyan, W.~S. Kendall, and J.~Mecke.
\newblock {\em {S}tochastic {G}eometry and its {A}pplications}.
\newblock Wiley, Chichester, second edition, 1995.

\bibitem{Thoma}
E.~Thoma.
\newblock Non-rigidity properties of the coulomb gas.
\newblock https://arxiv.org/abs/2303.11486, 2023.

\bibitem{Tor16b}
S.~Torquato.
\newblock {D}isordered hyperuniform heterogeneous materials.
\newblock {\em J. Phys.: Condens. Matter}, 28(414012), 2016.

\bibitem{Tor16}
S.~Torquato.
\newblock {H}yperuniformity and its generalizations.
\newblock {\em Phys. Rev. E}, 94(022122), 2016.

\bibitem{Tor18}
S.~Torquato.
\newblock {H}yperuniform {S}tates of {M}atter.
\newblock {\em Physics Reports}, 745:1--95, 2018.

\bibitem{TS03}
S.~Torquato and F.~H. Stillinger.
\newblock {L}ocal {D}ensity {F}luctuations, {H}yperuniform {S}ystems, and
  {O}rder {M}etrics.
\newblock {\em Phys. Rev. E}, 68(041113):1--25, 2003.

\bibitem{TZS}
S.~Torquato, G.~Zhang, and F.~H. Stillinger.
\newblock {E}nsemble theory for stealthy hyperuniform disordered ground states.
\newblock {\em Phys. Rev. X}, 5(021020):1--23, 2015.

\bibitem{wavy-cryst}
O.~U. Uche, F.~H. Stillinger, and S.~Torquato.
\newblock Constraints on collective density variables: Two dimensions.
\newblock {\em Phys. Rev. E}, 70(046122), 2004.

\bibitem{DitheringBlueNoise}
R.~A. Ullichney.
\newblock Dithering with blue noise.
\newblock {\em Proc. IEEE}, 76(1):56--79, 1988.

\bibitem{ValkoVirag}
B.~Valk{\'o} and B.~Vir\`ag.
\newblock {C}ontinuum limits of random matrices and the {B}rownian carousel.
\newblock {\em {I}nvent. math.}, 177(3):463--508, 2009.

\bibitem{viazovska8}
M.~S. Viazovska.
\newblock The sphere packing problem in dimension 8.
\newblock {\em Ann. Math.}, 185(991-1015), 2017.

\bibitem{Vil03}
C.~Villani.
\newblock {\em {T}opics in optimal transportation}.
\newblock Graduate Studies in Mathematics, Vol. 58, 2003.

\bibitem{Wiener}
N.~Wiener.
\newblock {\em {E}xtrapolation, interpolation, and smoothing of stationary time
  series. {W}ith engineering applications.}
\newblock Cambridge, New York (1949 republication), 1942.

\bibitem{BlueNoise}
D.~Yan, J.~Guo, B.~Wang, X.~Zhang, and P.~Wonka.
\newblock {A} {S}urvey of {B}lue-{N}oise {S}ampling and {I}ts {A}pplications.
\newblock {\em J. Comput. Sci. Technol.}, 30:439--452, 2015.

\bibitem{ZST-I}
G.~Zhang, F.~Stillinger, and S.~Torquato.
\newblock {G}round states of stealthy hyperuniform potentials: {I}.
  {E}ntropically favored configurations.
\newblock {\em Phys. Rev. E}, 92 - 022119(1-14), 2015.

\bibitem{ZST}
G.~Zhang, F.~H. Stillinger, and S.~Torquato.
\newblock {C}an exotic disordered ''stealthy'' particle configurations tolerate
  arbitrarily large holes?
\newblock {\em Soft matter}, 36(https://doi.org/10.1039/C7SM01028A), 2017.

\end{thebibliography}

\end{document}